\documentclass[11pt,letterpaper]{amsart}
\usepackage[T1]{fontenc}
\usepackage[latin9]{inputenc}
\usepackage{geometry}
\usepackage{wrapfig}
\usepackage{multicol}
\usepackage{enumerate}   
\usepackage{graphicx}
\usepackage{soul}
\usepackage{xcolor}
\usepackage{amssymb}

\newtheorem{theorem}{Theorem}[section]
\newtheorem{lemma}[theorem]{Lemma}
\newtheorem{proposition}[theorem]{Proposition}
\newtheorem{corollary}[theorem]{Corollary}
{ \theoremstyle{definition}
\newtheorem{definition}[theorem]{Definition}}
{ \theoremstyle{remark}
\newtheorem{remark}[theorem]{Remark}}
\usepackage{placeins}
\usepackage{cite}
\usepackage{caption}
\usepackage{stmaryrd}
\usepackage{enumerate}
\usepackage{afterpage}
\usepackage{enumitem}
\usepackage{bmpsize}
\usepackage{hyperref}

\makeatletter
\newcommand*{\rom}[1]{\expandafter\@slowromancap\romannumeral #1@}
\makeatother

\numberwithin{equation}{section}

\numberwithin{thm}{subsection}

{ \theoremstyle{definition} }

{ \theoremstyle{remark} }

{ \theoremstyle{remark}
}

\def\M{ {\kappa}}
\def\){ {\mathcal O}}

\def\lM{ {\tt M}}

\def\S{ {\mathfrak X_N^{\theta}}}

\newcommand{\diam}{\mathbin{\raisebox{0.25mm}{\scalebox{1.72}[1]{{\rotatebox[origin=c]{90}{$\diamond$}}} }}}
\newcommand{\ldiam}{\mathbin{\raisebox{1mm}{ \rotatebox[origin=c]{60}{$\diam$}}}}
\newcommand{\rdiam}{\mathbin{\raisebox{-0.5mm}{ \rotatebox[origin=c]{-60}{$\diam$}}}}
\def\i{\mathrm{i}}

\def\t {{\bf t}}

\def\vm {{\bf v}}

\def\P {{\mathbb P}}

\usepackage{tabu}
\usepackage{enumitem}   
\usepackage{tikz}
\usetikzlibrary{matrix,graphs,arrows,positioning,calc,decorations.markings,shapes.symbols}
\definecolor{dullmagenta}{rgb}{0.4,0,0.4}   
\definecolor{darkblue}{rgb}{0,0,0.4}

\begin{document}

\pagestyle{plain}

\title{Asymptotics of discrete $\beta$-corners processes via two-level discrete loop equations}
\author{Evgeni Dimitrov}

\address[Evgeni Dimitrov]{Department of Mathematics, Columbia University,
 New York, NY, USA. E-mail: esd2138@columbia.edu}

\author{Alisa Knizel}

\address[Alisa Knizel]{Department of Statistics, University of Chicago,
Chicago, IL, USA. E-mail: alknizel@gmail.com}

\begin{abstract} 
 We introduce and study a class of discrete particle ensembles that
 naturally arise in connection with classical random matrix ensembles, log-gases
 and Jack polynomials. Under technical assumptions on a general
 analytic potential we prove that the global
fluctuations of these ensembles are asymptotically Gaussian  with a universal covariance
that remarkably differs from its counterpart in random matrix theory. Our main tools are certain novel
 algebraic identities that we have discovered. They play a role of discrete multi-level analogues of loop equations.
\end{abstract}
\maketitle

\tableofcontents

%
\section{Introduction}\label{Section1}

%

\subsection{Preface}\label{Section1.1}

The goal of this paper is to introduce a new class of discrete particle ensembles and develop tools to study the fluctuations of its linear statistics. Discreteness adds an additional degree of complexity to the problem and to solve it new methods are required. The continuous counterparts of these ensembles are closely related to $\beta$ log-gases and random matrix theory and we will discuss applications of our results in this setup. This connection is one of the motivations for our work. 

A continuous $\beta$ log-gas is a probability distribution on $N$-tuples of ordered real numbers $y_1 < y_2 < \cdots < y_N$ with density proportional to
\begin{equation} \label{eq:distr1}
  \prod_{1\leq i<j \leq N} \left(y_j-y_i \right)^{\beta} \cdot \prod_{i=1}^{N}\exp( - N V(y_i)), 
\end{equation}
where $V(y)$ is a continuous function called the {\em potential}. The study of continuous log-gases for general potentials is a rich subject that is of special interest to random matrix theory, see e.g. \cite{Meh,Forr,AGZ,PS}.  For example, when $V(y) = \frac{\beta y^2}{4}$ and $\beta = 1,2,4$, the distribution (\ref{eq:distr1}) is the joint density of the eigenvalues of random matrices from the Gaussian Orthogonal/Unitary/Symplectic ensembles (GOE, GUE and GSE) \cite{Forr, AGZ}.

The above random matrix ensembles at $\beta = 1, 2, 4$ come with an additional structure,
which is a natural coupling between the distributions \eqref{eq:distr1} with varying number of particles $N$.
In the case of the Gaussian Unitary Ensemble, take $M = [M_{ij}]^N_{i,j=1}$ to be a random Hermitian
matrix with probability density proportional to $\exp\left(-\textup{Trace}(M^2/2)\right).$
Let $y^k_1 \leq y^k_2\leq \dots \leq y^k_k$ for
$k = 1,\dots, N$, denote the set of (real) eigenvalues of the top-left $k \times k$ corner $[M_{ij} ]^k_{i,j=1}.$ The
eigenvalues satisfy the interlacing conditions 
$y^j_i\leq y^{j-1}_i \leq y^j_{i+1}$ for all meaningful values of
$i$ and $j$, with the inequalities being strict almost surely.

In this way (\ref{eq:distr1}) is canonically extended to the measure on the Gelfand-Tsetlin cone
$$GT_N := \{y\in\mathbb R^{N(N+1)/2}:y_{i}^{j}<y_i^{j-1}<y_{i+1}^{j}, \mbox{ $i = 1, \dots, j$, $j = 1,\dots, N$}\}$$ formed by the eigenvalues of corner submatrices. This ensemble is known as the {\it GUE-corners process} (the term {\it GUE-minors} process is also used) \cite{N, GelN, Bar}.

Similar constructions are available for GOE and GSE ($\beta=1$, $\beta=4$). One can notice that in the resulting formulas for the distribution of the corners process $\{y^j_i\}$,
$i = 1, \dots, j$, $j = 1,\dots, N$, the parameter $\beta$ enters in a simple way (see e.g. \cite[Proposition
1.1]{Ne}). This readily leads to the generalization of the definition of the corners process to the
case of general $\beta > 0$ and general potential $V$ and it is given by
the formula
\begin{equation}\label{eq:gen_cont_beta}
\frac{1}{Z}\prod\limits_{1 \leq i<j \leq N}(y_j^N-y_i^N) \prod_{k = 1}^{N-1}\left(\prod\limits_{1\leq i<j\leq k}
(y_j^k-y_i^k)^{2-\beta} \prod\limits_{a=1}^k\prod_{b=1}^{k+1}|y_a^k-y_b^{k+1}|^{\beta/2-1}\right) \cdot \prod_{i = 1}^N e^{-NV(y^N_i)}, 
\end{equation}
where $Z$ is a normalization constant, see \cite{Ne} and \cite{OO}. The fact that the projection of (\ref{eq:gen_cont_beta}) to the top level $y^N$ is given by (\ref{eq:distr1}) can be deduced from the Dixon-Anderson integration formula (see \cite{Dixon, Anderson}), which has been studied in the context of {\em Selberg integrals} (see \cite{S}, \cite{Meh} and \cite[Chapter 4]{Forr}). For $V(y) = y^2$ the above ensemble also carries the name {\em Hermite $\beta$-corners process}. If $e^{-NV(y)} =y^p (1 -y)^q$ where $p,q > - 1$ then (\ref{eq:gen_cont_beta}) is called the {\em $\beta$-Jacobi corners process} and describes the joint distribution of the eigenvalues of a (different type) random matrix corners ensemble \cite{Sun16}.\\

The aim of the present paper is to initiate a detailed study of the fluctuations of general $\beta$-corner processes and their discrete analogues that we introduce next. In \cite{BGG} the authors proposed the following integrable discretizations of (\ref{eq:distr1}), called {\em discrete $\beta$-ensembles} or {\em discrete log-gases}. These are probability distributions that depend on a parameter $\theta = \beta/2 > 0$ and a positive function $w(x;N)$, and have the form
\begin{equation} \label{eq:distr2}
\begin{split}
& \mathbb{P}(\ell_1, \dots, \ell_N)  \propto \prod_{1\leq i<j \leq N} \frac{\Gamma(\ell_i - \ell_j + 1)\Gamma(\ell_i - \ell_j+ \theta)}{\Gamma(\ell_i - \ell_j)\Gamma(\ell_i - \ell_j-\theta + 1)}  \prod_{i=1}^{N}w(\ell_i; N),
\end{split}
\end{equation}
where $\ell_i = \lambda_i + (N - i) \cdot \theta$ and $\lambda_N \leq \lambda_{N-1} \leq \dots \leq \lambda_1$ are integers \footnote{Note that we have reversed the order of the indices so that $\ell_1$ is now largest and $\ell_N$ is the smallest -- this convention is more consistent with the symmetric function origin of (\ref{eq:distr2}).}. When $\theta=1$ the interaction term becomes $ \prod_{1\leq i<j \leq N} (\ell_j-\ell_i)^2.$

It is important to stress that although the discretization of the form

\begin{equation} \label{eq:distr3}
\begin{split}
& \mathbb{P}(\ell_1, \dots, \ell_N)  \propto \prod_{1\leq i<j \leq N} (\ell_j-\ell_i)^\beta \prod_{i=1}^{N}w(\ell_i; N),
\end{split}
\end{equation}
looks more natural, it seemingly lacks the integrability that is present in the distribution (\ref{eq:distr2}). In particular, we are not aware of any techniques currently available to study fluctuations for this discretization for general $\beta.$ The same holds for the multi-level generalization that we consider. On the other hand, there is also a lot of interest in measures of the form (\ref{eq:distr2}) coming from integrable probability; specifically, due to their connection to uniform random tilings, $(z,w)$-measures, Jack measures --- see \cite[Section 1]{BGG} for more details. 

Central objects of interest to us in the present paper are certain extensions of (\ref{eq:distr2}) to multi-level settings that are natural discrete analogues of (\ref{eq:gen_cont_beta}) the same way that (\ref{eq:distr2}) is a discrete analogue of (\ref{eq:distr1}). The models we study depend on a parameter $\theta = \beta/2 >0,$ and $N, M \in \mathbb N$ as well as a positive function $w(x;N)$. We explain the construction below.

 A Gelfand-Tsetlin scheme (pattern) is a triangular array of integers integers $\lambda_i^k$ such that: 
\begin{align}
	\begin{array}{c}
	\begin{tikzpicture}[scale=1]
		\def\h{0.8}
		\def\x{2}
		\node at (-5,5) {$\lambda^N_N$};
		\node at (-3,5) {$\lambda^N_{N-1}$};
		\node at (0-\x/2,5) {$\ldots\ldots\ldots\ldots$};
		\node at (3-\x,5) {$\lambda^N_{2}$};
		\node at (5-\x,5) {$\lambda^N_{1}$};
		\node at (-4,5-\h) {$\lambda^{N-1}_{N-1}$};
		\node at (-2,5-\h) {$\lambda^{N-1}_{N-2}$};
		\node at (0-\x/2+.1,5-\h) {$\ldots$};
		\node at (2-\x,5-\h) {$\lambda^{N-1}_{2}$};
		\node at (4-\x,5-\h) {$\lambda^{N-1}_{1}$};
		\node at (-3,5-2*\h) {$\lambda^{N-2}_{N-2}$};
		\node at (3-\x,5-2*\h) {$\lambda^{N-2}_{1}$};
		\foreach \LePoint in {(4.5-\x,5-\h/2),(2.5-\x,5-\h/2),(-3.5,5-\h/2),
		(-2.5,5-3*\h/2),(3.5-\x,5-3*\h/2),(2.5-\x,5-5*\h/2),(1.7-\x,5-7*\h/2)} {
    		\node [rotate=45] at \LePoint {$\le$};
    	};
    	\foreach \GePoint in {(3.5-\x,5-\h/2),(-4.5,5-\h/2),(-2.5,5-\h/2),
    	(-3.5,5-3*\h/2),(-2.5,5-5*\h/2),(.5,5-3*\h/2),(-1.7,5-7*\h/2)} {
    		\node [rotate=135] at \GePoint {$\ge$};
    	};
    	\node at (0-\x/2,5-3*\h) {$\ldots\ldots\ldots\ldots$};
    	\node at (0-\x/2+.08,5-4*\h) {$\lambda^{1}_1$};
	\end{tikzpicture}
	\end{array}
	\label{GT_scheme}
\end{align}
Any two  adjacent levels of a Gelfand-Tsetlin pattern satisfy the following interlacing property $\lambda^{i+1}_{i+1} \leq  \lambda^i_i \leq \lambda_{i}^{i+1} \leq \cdots \leq \lambda_1^i \leq \lambda_{1}^{i+1}$, which we denote by $\lambda^i \preceq\lambda^{i+1}$. The state space of our measure consists of Gelfand-Tsetlin patterns such that $0 \leq \lambda_N^N\leq \lambda^N_1\leq M$ (here $M \in \mathbb{N}$ is an additional parameter of the model). The measure we put on this space is more easily expressible in terms of the shifted coordinates $ \ell^i_j = \lambda^i_j + (N - j)\cdot\theta$ and we will frequently use these coordinates, going back and forth without mention. We write $\ell^i \preceq \ell^{i+1}$ to mean $\lambda^i \preceq \lambda^{i+1}$. 

With the above data we define the following measure
\begin{equation}\label{eq:measure_k}
\mathbb{P}^{\theta}_N(\ell^1, \dots, \ell^N) \propto \prod_{1 \leq i < j \leq N} \frac{\Gamma(\ell^N_i - \ell^N_j + 1)}{\Gamma(\ell^N_i - \ell^N_j- \theta + 1)}  \cdot \prod_{k = 1}^{N-1}  I(\ell^{k+1}, \ell^k) \cdot \prod\limits_{i=1}^{N}w(\ell_i^{N}; N), \mbox{ with }
\end{equation}
\begin{align}\label{tm}
\begin{split}
I(\ell^s, \ell^{s-1}) = &\prod_{1 \leq i < j \leq s}\frac{\Gamma(\ell^s_i - \ell^s_j -\theta + 1 )}{\Gamma(\ell^s_i - \ell^s_j) } \cdot \prod_{1 \leq i < j \leq s-1} \frac{\Gamma(\ell^{s-1}_i - \ell^{s-1}_j + 1)}{\Gamma(\ell^{s-1}_i - \ell^{s-1}_j + \theta)} \\
&\times \prod_{1 \leq i < j \leq s} \frac{\Gamma(\ell^{s-1}_i - \ell^s_j)}{ \Gamma(\ell^{s-1}_i - \ell^s_j  - \theta + 1)}  \cdot \prod_{1 \leq i \leq j \leq s-1} \frac{\Gamma(\ell^{s}_i - \ell^{s-1}_j + \theta)}{\Gamma(\ell^s_i - \ell^{s-1}_j + 1)}.
\end{split}
\end{align}
The fact that the projection of (\ref{eq:measure_k}) to the top level $\ell^N$ is given by (\ref{eq:distr2}) is a consequence of the branching relations for Jack symmetric functions \cite{Mac} and can be deduced from \cite[Section 2]{GS}, see Proposition \ref{PropExtension} in the main text below.

Note that when $\theta = \beta/2 = 1$ the terms $I(\ell^s, \ell^{s-1}) $ all become $1$ and the conditional distribution of $\ell^1, \ell^2, \dots, \ell^{N-1}$, given $\ell^N$ becomes uniform on the discrete set specified by the interlacing conditions $\ell^N \succeq \ell^{N-1} \succeq \cdots \succeq \ell^1$. We view such an extension as a canonical {\em Gibbsian} extension of the measures in (\ref{eq:distr2}) to multiple levels. One reason to view (\ref{tm}) as a canonical or integrable extension of (\ref{eq:distr2}) for general $\theta > 0$ is due to connections to different integrable models of $2d$ statistical mechanics, such as random tilings, ensembles of non-intersecting paths and Jack ascending processes  \cite{Go, J, GS}. We will refer to the measures defined by (\ref{eq:measure_k}) as {\em discrete $\beta$-corners processes}. 

Note that if we set $\ell^i_j=T \cdot y^i_{i+1-j}$ then by \cite{TE} we have
\begin{align*}
&\prod_{1 \leq i < j \leq N} \frac{\Gamma(\ell^N_i - \ell^N_j + 1)}{\Gamma(\ell^N_i - \ell^N_j - \theta + 1)}  \cdot \prod_{k = 1}^{N-1}  I(\ell^{k+1}, \ell^k)  \sim T^{N(N-1) (\theta - 1/2)}  \\ 
& \times \prod_{1 \leq i<j \leq N}(y_j^N-y_i^N) \prod_{k = 1}^{N-1}\left(\prod\limits_{1\leq i<j\leq k} (y_j^k-y_i^k)^{2-2 \theta} \prod\limits_{a=1}^k\prod_{b=1}^{k+1}|y_a^k-y_b^{k+1}|^{\theta-1}\right), \text{ as }T\rightarrow\infty
\end{align*}
 which mimics \eqref{eq:gen_cont_beta} for $\theta=\frac{\beta}{2}$  and is another reason one might consider (\ref{eq:measure_k}) as a reasonable discretization. \\

In the present paper we investigate the projections of (\ref{eq:measure_k}) to the top two levels $(\ell^N, \ell^{N-1})$ and certain generalizations of the latter. More specifically, we consider measures on pairs $(\ell, m),$ where
$$\ell=(\ell_1, \dots, \ell_N), \quad m=(m_1, \dots, m_{N-1}),\text { with }$$
$$ \ell_j = \lambda_j + (N - j)\cdot\theta \text{ for } i=1\dots, N, \quad m_i = \mu_i + (N- i)\cdot\theta \text{ for } i=1\dots, N-1\text{ and }$$
$$M\geq\lambda_1\geq \mu_1\geq \lambda_2  \geq \dots \geq  \mu_{N-1}\geq \lambda_N\geq 0,\quad \lambda_i, {\mu_i}\in \mathbb Z $$ 
of the form

\begin{equation}\label{S1PDef}
\mathbb{P}^\theta_N(\ell, m) = Z_N^{-1} \cdot H^t(\ell) \cdot H^b(m) \cdot I(\ell, m), \mbox{ where }
\end{equation}

\begin{align}\label{S1PDef2}
\begin{split}
&H^t(\ell) = \prod_{1 \leq i < j \leq N} \frac{\Gamma(\ell_i - \ell_j + 1)}{\Gamma(\ell_i - \ell_j  - \theta + 1)} \cdot \prod_{i = 1}^N  w(\ell_i;N); \\
&   H^b(m) =  \prod_{1 \leq i < j \leq N-1} \frac{\Gamma(m_i - m_j + \theta)}{\Gamma(m_i - m_j)} \cdot \prod_{i = 1}^{N-1} \tau(m_i;N);\\
&I(\ell, m) = \prod_{1 \leq i < j \leq N} \frac{\Gamma(\ell_i - \ell_j  - \theta + 1)}{\Gamma(\ell_i - \ell_j) }\cdot \prod_{1 \leq i < j \leq N-1} \frac{\Gamma(m_i - m_j + 1)}{\Gamma(m_i - m_j + \theta)} \\
& \times \prod_{1 \leq i < j \leq N} \frac{\Gamma(m_i - \ell_j)}{ \Gamma(m_i - \ell_j - \theta + 1)}  \cdot \prod_{1 \leq i \leq j \leq N-1} \frac{\Gamma(\ell_i - m_j + \theta)}{\Gamma(\ell_i - m_j + 1)}.
\end{split}
\end{align}

In (\ref{S1PDef2}) $w(x;N)$ is a positive function on $[0, M + (N-1) \cdot \theta]$ and $\tau(x; N)$ is a positive function on $[\theta, M + (N-1) \cdot \theta]$. In all the asymptotics results that we prove we set $\tau=1$ (then the distribution of $(\ell, m)$ in (\ref{S1PDef}) is exactly the same as that of $(\ell^N, \ell^{N-1})$ in (\ref{eq:measure_k}), see Section \ref{Section7} for a proof of this fact). However, from the algebraic perspective (see Section \ref{Section1.3}) we can add different potentials attached to the levels of the Gelfand-Tsetlin scheme. It would be interesting to investigate the influence of the interplay between these potentials on the asymptotics of the measures in (\ref{S1PDef}), but we will not do this in the present paper. 

One can also think about $\ell$ and $m$ as the locations of a collection of $2$ types of particles on a line and then \eqref{S1PDef} can be viewed as a certain discrete version of a
two-component plasma on a line with charges required to alternate in space. This model has attracted a significant interest in the physics literature due to its equivalence to
impurity Kondo problem \cite{Sa, FA, AY1, AY2}. \\

In this paper, we focus on the study of the asymptotics of the global fluctuations of the measures in \eqref{S1PDef}. Our main tools are certain novel two-level analogues of the discrete loop equations from \cite{BGG} that we have discovered. Loop equations (also known as Schwinger-Dyson equations) have proved to be a very efficient tool in the study of global fluctuations of discrete and continuous log-gases,  see \cite{JL, BoGu, S, BoGu2, KS} and the references therein. In the discrete setup they are known as Nekrasov's equations (see \cite{BGG, DK1}). These are functional equations for certain observables of the log-gases \eqref{eq:distr1} that are related to the Stieltjes transforms of the empirical measure and their cumulants.  Since their introduction loop equations have become a powerful tool for studying not only global fluctuations but also establishing {\em local} and {\em edge universality} for continuous and discrete $\beta$ ensembles \cite{BEY, BFG, GH}. 

A priori there was no evidence that multi-level loop equations even exist, and one of the main contributions of this paper is the construction of these objects. Currently, we are only able to do this for two levels but our hope is that we can construct such equations for arbitrary number of levels. In the two-level setting our loop equations can be thought of as functional equations that relate the Stieltjes transforms of the empirical measures on the two levels and their cumulants, and can be used to extract meaningful probabilistic information for various systems. To demonstrate their potential we use our discrete loop equations to study the global fluctuations of the measures in (\ref{S1PDef}) for a large class of weights, establishing Gaussian fluctuations. In a different direction, through a diffuse limit of the equations we obtain two-level analogues of the loop equations in \cite{BoGu}, which to our knowledge are novel and are of separate interest. An important feature of our approach is that both in the discrete and continuous level the multi-level loop equations we derive hold for arbitrary $\beta$ and analytic potential functions.

It is worth mentioning that in the continuous setting the problem of deriving asymptotic two-level global fluctuations was investigated in \cite{ES}, but only in the context when the measures in (\ref{eq:gen_cont_beta}) come from a random matrix theory model, like the GUE. In the discrete setting or in the continuous setting when the measure in (\ref{eq:gen_cont_beta}) is given for an arbitrary potential function $V$ there is no random matrix model giving rise to the corners process. This prohibits one from using any tools from random matrix theory to study the general $\beta$ corners processes, and one of the purposes of the multi-level loop equations is to become a toolbox that meets this purpose. Our two-level loop equations are notably more complex than the single-level loop equations that were previously known, and this reflects the increased complexity of going from single to two-level log-gases. Despite the increased complexity one can still use the equations to obtain the two-level global fluctuations of discrete $\beta$ corners processes, which to our knowledge are not accessible by any other means.

%
\subsection{Main results} \label{Section1.2}
We next turn to explaining our asymptotic results, for which we assume that $\tau(x; N)\equiv1$. We start by listing the limiting regime and corresponding regularity assumptions.\\

  We assume that we are given parameters $\theta > 0$, $\lM > 0$. In addition, we assume that we have a sequence of parameters $M_N \in \mathbb{N}$ such that $M_N/N\rightarrow \lM$ as $N \rightarrow \infty.$ We assume further that the weight function $w(x;N)$ in the interval $[0, M_N + (N-1) \cdot \theta]$ has the form
$$w(x;N) = \exp\left( - N V(x/N)\right),$$
for a function $V$ that is bounded and continuous on $ [0, \lM + \theta]$. We also require that $V(s)$ is differentiable on $(0, \lM + \theta)$ and for some $C > 0$
\begin{equation}\label{S1DerPot}
\left| V'(s) \right| \leq C \mbox{ for } s \in [0, \lM + \theta].
\end{equation}
We denote by  $\mathbb{P}^{\theta}_N$ measures defined by (\ref{S1PDef}) for $ M= M_N$ and  $\tau(x;N)=1.$ The above assumptions can be refined further as explained in Section \ref{Section3}, see Assumptions 1 and 2.

Let us briefly explain the significance of the above assumptions. Consider the random probability measure $\mu_N$ on $\mathbb{R}$ given by
\begin{equation}\label{S1EmpMeas}
\mu_N = \frac{1}{N} \sum_{i = 1}^N \delta \left( \frac{\ell_i}{N} \right), \mbox{ where $\ell = (\ell_1, \dots, \ell_N)$ is the marginal of the $\mathbb{P}^{\theta}_N$-distributed $(\ell,m)$}.
\end{equation}

Define for a compactly supported probability measure $\rho$ the functional
\begin{equation}\label{S1energy}
I_V[\rho] = \theta \iint_{x\neq y} \log|x - y| \rho(x) \rho(y)dxdy - \int_{\mathbb{R}} V(x) \rho(x)dx.
\end{equation}
Then the above assumptions ensure that the empirical measures $\mu_N$ concentrate with large probability around the maximizer of the energy functional $I_V[\rho]$ on the space of all probability measures on $[0, \lM + \theta]$ with density bounded by $\theta^{-1}$ (the density constraint is a consequence of the discrete nature of our problem, which prohibits two particles $\ell_i, \ell_j$ from coming closer than distance $\theta$). One can show, see Proposition \ref{LLN} in the main text, that $I_V$ has a unique maximizer $\mu,$ which is called  the {\em equilibrium measure}, and under the above assumptions the measures $\mu_N$ converge weakly in probability to $\mu.$ Proposition \ref{LLN}, which is proved in \cite{BGG}, requires some regularity assumption on the potential function $V$ -- this is the significance of (\ref{S1DerPot}) above.

In order to study the fluctuations we need to assume that we have an open set $\mathcal{M}  \subseteq \mathbb{C}$, such that $[0, \lM + \theta] \subset \mathcal{M}$ and $V$ is analytic in $\mathcal{M}$ and real-valued on $\mathcal{M} \cap \mathbb{R}$. A refined version of this assumption can be found as Assumption 3 in Section \ref{Section3}. In addition to the above assumptions we require two technical assumptions, which can be found as Assumptions 4 and 5 in Section \ref{Section3}. We forgo stating these assumptions here but mention that these assumptions imply that the equilibrium measure $\mu$ has a continuous density, which we write as $\mu(x)$, and the set of points where $\mu(x) \in (0,\theta^{-1})$ form a single open interval $(\alpha, \beta)$ --- this interval is sometimes referred to as a {\em band}, see \cite{BGG}. In the case when $\mu$ is given by the semi-circle law the closure of the band is precisely the support of the measure.

We may now state our main asymptotic result.
\begin{theorem} \label{CLTfun_main}
Suppose that the above assumptions and the technical Assumptions 4-5 (see Section \ref{Section3}) all
hold. For $n \geq 1$ let $f_1, \dots, f_n$ be analytic functions in a complex neighborhood $\mathcal{M}_1$ of $[0, \lM + \theta]$ and real-valued on $\mathcal{M}_1 \cap \mathbb{R}$ and define 
$$\mathcal L^t_{f_k}= \sum_{i = 1}^N f_k(\ell_i/N) - \mathbb{E}  \left[ \sum_{i = 1}^Nf_k(\ell_i/N)  \right] \mbox{, } \mathcal L^b_{f_k}= \sum_{i = 1}^{N-1} f_k(m_i/N) - \mathbb{E}  \left[ \sum_{i = 1}^{N-1} f_k(m_i/N)  \right] \mbox{ and }$$
$$\mathcal L^m_{f_k}=N^{1/2} \cdot \left[\mathcal L^t_{f_k} -  \mathcal L^b_{f_k}\right] \mbox{ for $k = 1, \dots, n$} .$$
Then the random variables $\{\mathcal{L}^m_{f_i} \}_{i = 1}^n$, $\{\mathcal{L}^t_{f_i} \}_{i = 1}^n$, $\{\mathcal{L}^b_{f_i} \}_{i = 1}^n$ converge jointly in the sense of moments to a $3n$-dimensional centered Gaussian vector $\xi = (\xi^m_1,\dots, \xi^m_n, \xi^t_1, \dots, \xi^t_n, \xi^b_n, \dots, \xi^b_n)$ with covariance
\begin{align}\label{eq:cov}
\begin{split}
& Cov(\xi^t_i, \xi^m_j )  =  Cov(\xi^b_i, \xi^m_j )   = 0,   \\
&Cov(\xi^{t}_i, \xi^{t}_j) = Cov(\xi^{b}_i, \xi^{b}_j) = Cov(\xi^{t}_i, \xi^{b}_j) = \frac{1}{(2\pi \i )^2} \oint_{\Gamma_1} \oint_{\Gamma_1} f_i(s) f_j(t) \mathcal{C}_{\theta, \mu}(s,t)dsdt,  \\
&Cov(\xi^{m}_i, \xi^m_j) = \frac{1}{(2\pi \i )^2} \oint_{\Gamma_1} \oint_{\Gamma_1} f_i(s) f_j(t) \Delta \mathcal{C}_{\theta, \mu}(s,t)dsdt, \\
\end{split}
\end{align}
for all $i,j = 1, \dots , n$, where
\begin{align}\label{eq:var_main}
\begin{split}
& \mathcal{C}_{\theta, \mu}(z_1, z_2) = -\frac{\theta^{-1}}{2(z_1-z_2)^2} \left(1 - \frac{(z_1 - \alpha)(z_2- \beta) + (z_2 - \alpha )(z_1- \beta)}{2\sqrt{(z_1 -\alpha )(z_1- \beta)}\sqrt{(z_2 - \alpha)(z_2- \beta)}} \right) \mbox{ and }\\
& \Delta \mathcal{C}_{\theta, \mu} (z_1, z_2) = \frac{1}{2\pi \i}\int_{\Gamma_2}  \frac{dz}{e^{\theta G_\mu(z)} - 1} \cdot \left[ - \frac{1}{(z-z_2)^2(z-z_1)^2}\right],
\end{split}
\end{align}
where $(\alpha,\beta)$ denotes the unique band of the equilibrium measure $\mu$, $G_\mu(z)$ is the Stieltjes transform of the equilibrium measure and $\Gamma_1, \Gamma_2$ are positively oriented contours that contain the segment $[0, \lM + \theta]$, are contained in $\mathcal{M} \cap \mathcal{M}_1$ with $\mathcal{M}$ as in Assumption 3 and $\Gamma_1$ contains $\Gamma_2$ in its interior.
\end{theorem}
\begin{remark} In (\ref{eq:var_main}) we pick the branch of the square root so that $\sqrt{(z -\alpha )(z- \beta)}$ is analytic in $\mathbb{C} \setminus [\alpha, \beta]$ and $\sqrt{(z -\alpha )(z- \beta)} \sim z$ as $|z| \rightarrow \infty$. See also Section \ref{Section1.5}. Also $\i = \sqrt{-1}$.
\end{remark}
\begin{remark} Note that $\mathcal{C}_{\theta, \mu}$ depends on the equilibrium
measure $\mu$ only through the quantities $\alpha, \beta$ -- the endpoints of the unique by assumption band of $\mu$. On the other hand,
$\Delta \mathcal{C}_{\theta, \mu}$ depends on the Stieltjes transform of $\mu$ and not just $(\alpha, \beta)$, see also Remark \ref{RemPrefactor}. In this sense the covariance is {\it universal} for systems that have the same limiting equilibrium measure $\mu$. We also remark that the form of the covariance $ \mathcal{C}_{\theta, \mu}$ is exactly the same as in the continuous setting, which was observed in related contexts in \cite{BGG,DK1}. Consequently, for fixed levels the global fluctuations of discrete and continuous corners processes are the same.

The situation is different for the differences of two consecutive levels. While the order of rescaling one needs to do to get a non-trivial limit for the differences of two adjacent levels is the same as in the continuous setting and the limiting behavior is still Gaussian, the covariance $\Delta \mathcal{C}_{\theta, \mu}$ is different. For example \cite{ES} studied the analogue of our setup for Wigner matrices, and more specifically the GUE. Our computations show that if we substitute $\mu$ with the semicircle law in $\Delta \mathcal{C}_{\theta, \mu}$ we do not recover the same formula in \cite{ES}. In this sense, it appears that for differences of two adjacent levels the limiting behavior {\em feels} the discreteness of the model and behaves differently from the case of continuous corners processes. We perform a comparison between the discrete and continuous setting for quadratic potential in detail in Section \ref{Section7.2.3}.

Similar quantities as in Theorem \ref{CLTfun_main} have been considered in the context of random matrix theory also in \cite{GZ, Ahn}. In the law of large numbers setting they were previously considered in  \cite{Buf} and \cite{Ker}. 
\end{remark}

%
\subsection{Methods}\label{Section1.3}

We next explain the main algebraic component of our argument, which we call {\em the two-level Nekrasov's equations}. The equations take a different form depending on whether $\theta = 1$ or $\theta \neq 1$. We forgo stating the equation for the case $\theta = 1$ until the main text, see Theorem \ref{TN1Theta1}.

\begin{theorem}\label{Nekrasov} Let $\mathbb{P}^{\theta}_N$ be a probability distribution as in (\ref{S1PDef}) for $\theta \neq 1$, $M \geq 1$ and $N \geq 2$. Let $\mathcal{M} \subseteq \mathbb{C}$ be open and $[0 , M + 1 + (N-1)\cdot \theta] \subset \mathcal{M}.$
Suppose there exist six functions $\phi_i^t$, $\phi_i^b$ and $\phi_i^m$ for $i = 1,2$ that are analytic in $\mathcal{M}$ and such that
\begin{equation}\label{S1ratioNE}
\frac{\phi_1^t(z)}{ \phi_1^m(z)} =   \frac{w(z-1)}{w(z)}, \hspace{2mm}  \frac{\phi_1^b(\hat{z})}{\phi_1^m(\hat{z})} =  \frac{\tau(\hat{z})}{\tau(\hat{z}-1)} \mbox{ and }\frac{\phi^{t}_2(z)}{\phi^{m}_2(z)} =  \frac{w(z)}{w(z-1)}, \hspace{2mm} \frac{\phi^{b}_2(\hat{z} - \theta)}{\phi^{m}_2(\hat{z} - \theta)}= \frac{\tau (\hat{z} - 1)}{\tau(\hat{z})},
\end{equation}
for all $z \in [ 1, M + (N-1) \cdot \theta]$ and $\hat{z} \in [ \theta + 1, M + (N-1)\cdot \theta]$. In addition we suppose that
\begin{equation*}\label{vansihEnd}
\phi_1^t(0) = \phi_1^b(M + 1 + (N-1) \cdot \theta) = \phi^m_1(M + 1 + (N-1) \cdot \theta) = \phi^t_2( M + 1 + (N-1) \cdot \theta) = \phi_2^m(0) = \phi_2^b(0) = 0.
\end{equation*}
Define $R_1(z)$ and $R_2(z)$ through
\begin{align}\label{S1mN}
\begin{split}
R_1(z)= & \phi_1^t(z) \cdot \mathbb{E} \left[ \prod_{i = 1}^N\frac{z- \ell_i -\theta}{z - \ell_i}  \right] + \phi_1^b(z) \cdot \mathbb{E} \left[ \prod_{i = 1}^{N-1}\frac{z- m_i + \theta - 1}{z - m_i - 1} \right]  \\
& + \frac{\theta}{1 - \theta} \cdot \phi_1^m(z) \cdot \mathbb{E} \left[  \prod_{i = 1}^N \frac{z- \ell_i -\theta}{z - \ell_i - 1}  \prod_{i = 1}^{N-1}\frac{z- m_i + \theta - 1}{z - m_i} \right],
\end{split}
\end{align}
\begin{align}\label{S1mNv2}
\begin{split}
R_2(z) = &\phi^{t}_2(z)  \cdot \mathbb{E} \left[ \prod_{i = 1}^N\frac{z - \ell_i + \theta - 1}{z - \ell_i - 1} \right] + \phi^{b}_2(z) \cdot \mathbb{E} \left[ \prod_{i = 1}^{N-1}\frac{z - m_i}{z - m_i + \theta} \right] \\
&+ \frac{\theta}{1 - \theta} \cdot  \phi^{m}_2(z) \cdot  \mathbb{E} \left[  \prod_{i = 1}^N \frac{z - \ell_i + \theta - 1}{z - \ell_i} \prod_{i = 1}^{N-1} \frac{z - m_i}{z - m_i + \theta - 1} \right].
\end{split}
\end{align}
Then $R_1(z)$ and $R_2(z)$ are analytic in $\mathcal{M}$. 
\end{theorem}
We call the above expressions that define $R_{1,2}$ equations because once we multiply both sides by any holomorphic function and integrate it around a closed contour we get 0 due to analyticity. Theorem \ref{Nekrasov} is a two-level analogue of the discrete loop equations from \cite{BGG} for the measures (\ref{eq:distr2}), and we refer to the latter as {\em one-level or single-level Nekrasov's equations} so as to distinguish them from our equations.  While our proof is similar in spirit to the one in \cite{BGG}, essentially performing a careful residue calculation, we remark that the computation is much more subtle in the two level-case and the main difficulty was in constructing $R_{1,2}(z).$ In \cite{BGG} the construction comes from the work of \cite{N}, however, we are not aware of a proper analogue for $R_{1,2}$ in the physics literature.

We readily see from (\ref{S1mN}) and (\ref{S1mNv2}) that the case $\theta = 1$ is special because of the vanishing of the denominator $1-\theta$. We believe that it was crucial for us to find the correct observables in the Theorem \ref{Nekrasov} for general $\theta$ first and then specialize to $\theta=1.$
\begin{remark} One has the following asymptotic expansion
\begin{equation*}\label{exp_intro}
\begin{split}
\prod_{ i = 1}^N \frac{Nz - \ell_i - \theta}{Nz - \ell_i} =\exp \left[ \sum_{i = 1}^N \log \left( 1 - \frac{1}{N}\frac{\theta}{z - \ell_i/N } \right)\right] =\exp\left[ -\frac{\theta G^t_N(z) }{N} + \frac{\theta^2\partial_z G^t_N(z)}{2N^2} +  O(N^{-2})\right],
\end{split}
\end{equation*}
where $G^t_N(z) =\sum_{i=1}^N\frac{1}{z-\ell_i/N}$ and the error term is uniform in $z$ over compact subsets of $\mathbb{C} \setminus \mathbb{R}$. If we denote $G^b_N(z) =\sum_{i=1}^{N-1}\frac{1}{z-m_i/N}$ we can perform similar expansions for the other products in our two-level Nekrasov's equations. In this sense, the expansion of the two-level equations leads to certain functional equations involving $G^t_N(z)$ and  $G^b_N(z)$, and our asymptotic results are a consequence of a careful analysis of the lower order terms of this expansion. 
\end{remark}

\begin{remark}The structure of the discrete loop equations is intimately related to the structure of the discrete space that underlies it. In particular, the form of the equations we have written in Theorem \ref{Nekrasov} depends on the fact that the underlying space is given by shifted integer lattices. One can extend the one-level loop equations to the case of shifted $q$-lattices \cite{BGG} or even shifted quadratic lattices \cite{DK1}. We hope to extend our two-level equations to such lattices in the future.
\end{remark}

The next result we present is obtained by studying the diffuse limits of our two-level discrete loop equations, which lead to two-level analogues of the loop equations in \cite{BoGu}. We call these objects {\em two level-loop equations} and similarly to \cite{BoGu} they come with various ranks. While the usual loop equations have rank parametrized by $n \in \mathbb{N}$, the two level equations are parametrized by $(m,n) \in \mathbb{Z}_{\geq 0} \times \mathbb{Z}_{\geq 0}$. As the formulas are quite involved we forgo stating them here, and only write down the equations of rank $(0,0)$. The interested reader is referred to Section \ref{Section6} for more details.

\begin{theorem}\label{S1TLLoop}
Suppose that $\theta > 0$,  $a_-, a_+ \in \mathbb{R}$ with $a_- < a_+$ and let $V^t, V^b$ be analytic in a neighborhood $\mathcal{M}$ of $[a_-, a_+]$ that are real-valued on $\mathcal{M} \cap \mathbb{R}$. Suppose $(X_1, \dots, X_N, Y_1, \dots, Y_{N-1})$ is a random $(2N-1)$-dimensional vector with density
\begin{equation}\label{S1twologgas}
f(x, y) = \frac{1}{Z^c} \prod_{1 \leq i < j \leq N} \hspace{-4mm}(x_j - x_i) \hspace{-4mm} \prod_{1 \leq i < j \leq N-1}  \hspace{-4mm} (y_j - y_i) \prod_{i = 1}^{N-1} \prod_{j = 1}^N |y_i - x_j|^{\theta - 1} \prod_{i = 1}^{N-1} e^{-N\theta V^b(y_i)} \prod_{i = 1}^N e^{-N\theta V^t(x_i)},
\end{equation}
where $f(x,y)$ is supported on $\mathcal{G} = \{(x,y) \in \mathbb{R}^{2N-1}: a_- < x_1 < y_1 < x_2 < y_2 < \cdots < y_{N-1} < x_N < a_+\}$ and $Z^c$ is a normalization constant such that the integral of $f(x,y)$ over $\mathcal{G}$ is $1$. Then 
\begin{align}\label{S1TLRank00}
\begin{split} 
& 0 =  \frac{N\theta}{2\pi \i}\int_{\Gamma} \frac{ dz (z - a_-)(z - a_+) }{(z-v)(v - a_-) (v-a_+)} \left[  \mathbb{E}[G^t(z)] \partial_z V^t(z) +   \mathbb{E}[G^b(z)] \partial_z V^b(z) \right] + \frac{\mathbb{E}[G^t(v)^2]}{2}  \\
&    + \hspace{-1mm}\frac{\mathbb{E}[G^b(v)^2]}{2} - \frac{ \partial_v \mathbb{E}[G^t(v)] + \partial_v \mathbb{E}[G^b(v)]}{2} \hspace{-1mm} - (1-\theta) \mathbb{E}[G^t(v)G^b(v)] -  \frac{N^2 - (1-\theta)N(N-1)}{(v - a_-) (v-a_+)} ,
\end{split}
\end{align}
where $G^t(z) = \sum_{i = 1}^N \frac{1}{z - X_i}$, $G^b(z) = \sum_{i = 1}^{N-1} \frac{1}{z- Y_i}$ and $\Gamma$ is a positively oriented contour inside $\mathcal{M}$ that contains the segment $[a_-, a_+]$ in its interior. 
\end{theorem}
\begin{remark} The loop equation of rank $1$, see e.g. \cite{BoGu}, is the single level analogue of the equation in (\ref{S1TLRank00}). It is the main ingredient in establishing the global fluctuations of continuous $\beta$ log-gases in \cite{JL, BoGu}. We are hopeful that (\ref{S1TLRank00}) and its higher rank versions  in Theorem \ref{TwoLevelLoopThm} can be used to study fluctuations of general $\beta$-corners processes.
\end{remark}

%
\subsection{Structure of the paper}\label{Section1.4}

The paper consists of seven sections and two appendices. The dependence of the different sections on each other is depicted in Figure \ref{S2_2}. Section \ref{Section1} is an introductory section that motivates our work, defines the basic objects that we study and presents (simplified versions of) the main statements we prove: Theorem \ref{CLTfun_main} is the main asymptotic statement we prove, Theorem \ref{Nekrasov} is the main algebraic tool we use and Theorem \ref{S1TLLoop} is a continuous degeneration of Theorem \ref{Nekrasov} that we believe to be of separate interest. The proof of the two-level Nekrasov's equations is given in Section \ref{Section2}. Section \ref{Section3} defines the two-level interacting particle system whose global analysis we are interested in establishing, and lists our assumptions on it. This section contains several statements, whose proofs are the content of Appendices A and B.

\begin{figure}[h]
\centering
\scalebox{0.4}{\includegraphics{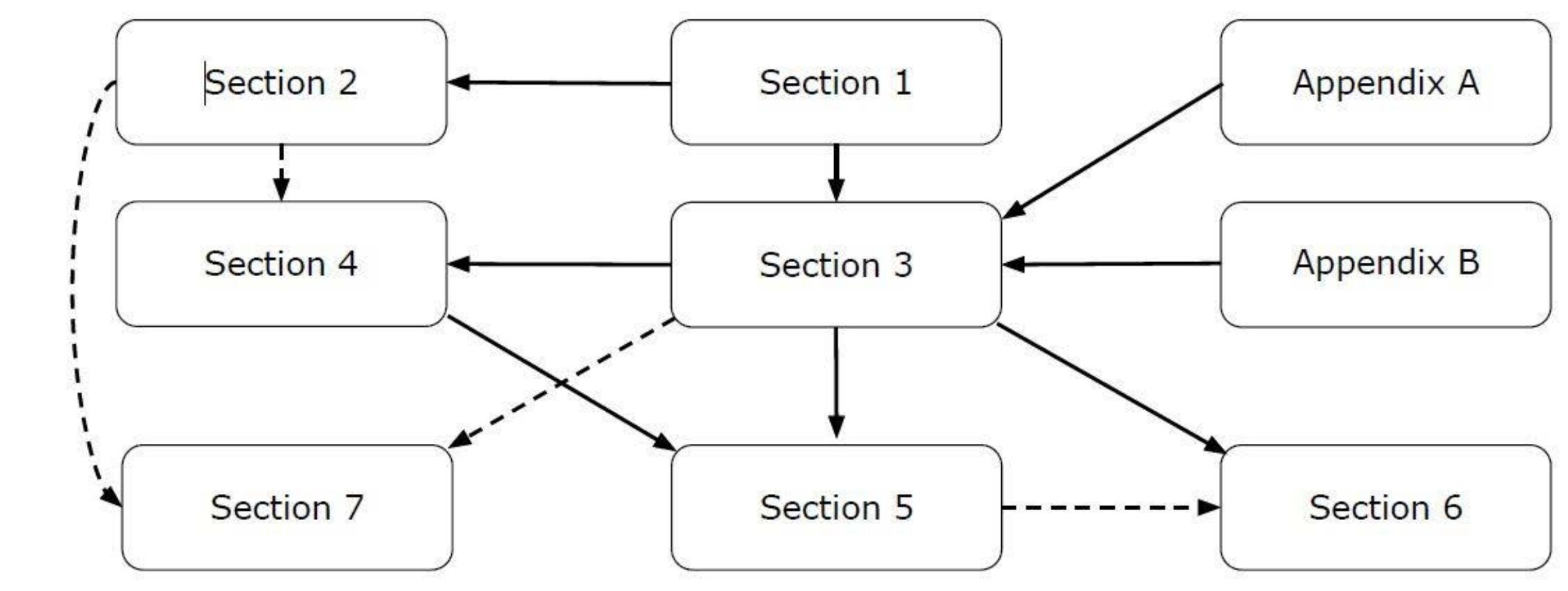}}
\caption{Dependence of sections in the paper. The filled arrows indicate a strong dependence of one section on another. Dashed arrows indicate very weak dependence. For example: Section \ref{Section6} requires only the statement of Theorem \ref{CLTfun} from Section \ref{Section5} and nothing else. Also the dependence on Section \ref{Section2} is only through the statements of Theorems \ref{TN1} and \ref{TN1Theta1}. }
\label{S2_2}
\end{figure}

In Section \ref{Section4} we use our two-level equations from Section \ref{Section2} to obtain integral equations relating certain observables (in this case joint cumulants) for the measures in Section \ref{Section3}. These integral equations are then used in Section \ref{Section5} to prove that the global fluctuations of the systems in Section \ref{Section3} are asymptotically Gaussian with an explicit covariance. This statement is Theorem \ref{CLTfun} in the text and is proved in Section \ref{Section5}. After that we demonstrate in Section \ref{Section7} a generic way to obtain measures of the type in Section \ref{Section3} and give two examples of models that fit into our framework and explain what our results say about them. In particular, the examples illustrate how our result fits into the context of related problems studied in integrable probability and random matrix theory.

In Section \ref{Section6} we prove a generalization of Theorem \ref{S1TLLoop}, namely we derive rank $(m,n)$ continuous loop equations for continuous $\beta$-corners processes -- this is Theorem \ref{TwoLevelLoopThm}. The latter are two-level analogues of the loop equations derived in \cite{BoGu}. To our knowledge these rank $(m,n)$ continuous loop equations have not appeared before in the literature and provide novel  identities even for the GUE corners process. The way we derive these equations is by first demonstrating how one can obtain continuous $\beta$-corners processes as diffuse limits of our discrete $\beta$-corners processes. This is done in Section \ref{Section6.2}. Armed with this result, obtaining the two-level loop equations in Theorem \ref{TwoLevelLoopThm} is relatively straightforward (although technical) and is the content of Section \ref{Section6.4}. Section \ref{Section6.3} reproduces the result in \cite{BoGu} using our framework and we included it due to its conceptual importance.\\

%
\subsection{Notation} \label{Section1.5} In this section we summarize some notation used throughout the paper. We write $\i = \sqrt{-1}$. If $\mu$ is a probability measure on $\mathbb{R}$, which is absolutely continuous with respect to Lebesgue measure, we write $\mu(x)$ for its density and also denote $\mu$ by $\mu(x)dx$. All complex logarithms in the paper are defined with respect to the principal branch. We denote by $\mathrm{arccos}(x)$ the function, which is $\pi$ on $(-\infty, -1]$, $0$ on $[1 ,\infty)$, and the usual arccosine function on $(-1,1)$. 

Given $a, b \in \mathbb{R}$ with $a<b$, we write $f(z) = \sqrt{(z - a)(z- b)} $ to mean $$f(z) = \begin{cases} \exp \left( \frac{1}{2} \log (z-a) + \frac{1}{2} \log (z-b)  \right)&\mbox{ when $z \in \mathbb{C} \setminus (-\infty, b]$ }, \\ -\sqrt{a -z }\sqrt{b - z} &\mbox{ when $z \in (-\infty, a)$ }. \end{cases}$$ 
Observe that in this way $f$ is holomorphic on $\mathbb{C} \setminus [a,b]$. Indeed, the analyticity on $\mathbb{C} \setminus (-\infty, b]$ is immediate and by direct inspection the function extends continuously on $(-\infty, a)$, which implies the analyticity on $\mathbb{C} \setminus [a,b]$ by the Symmetry principle (Theorem 2.5.5 in \cite{SS}).

\subsection*{Acknowledgments}  The authors are deeply grateful to
Alexei Borodin, Vadim Gorin, Konstantin Matetski, Nikita Nekrasov, Andrei Okounkov and
Oleksandr Tsymbaliuk for very helpful discussions. The authors would also like to thank the anonymous referees for their careful reading of the article and their thoughtful comments and suggestions.

E.D. is partially supported by the Minerva Foundation Fellowship. A.K. is partially supported through NSF grant DMS:1704186.

%

\section{Nekrasov's equations}\label{Section2}
In this section we present the main algebraic component of our argument, which we call {\em the two-level Nekrasov's equations}. The equations have different formulations when $\theta \neq 1$ -- see Theorem \ref{TN1}, and when $\theta = 1$ -- see Theorem \ref{TN1Theta1}. 

%
\subsection{Preliminary computations}\label{Section2.1} For $N, M \in \mathbb{N}$ and $\theta  > 0$  we define
\begin{align}\label{S2GenState}
\begin{split}
&\Lambda_N = \{  (\lambda_1, \dots, \lambda_N) : \lambda_1\geq  \lambda_2 \geq \cdots \geq \lambda_N, \lambda_i \in \mathbb{Z}  \mbox{ and }0  \leq \lambda_i  \leq M \}, \\
&\mathbb{W}^\theta_{N,k} = \{ (\ell_1, \dots, \ell_k):  \ell_i = \lambda_i + (N - i)\cdot\theta, \mbox{ with } (\lambda_1, \dots, \lambda_k) \in \Lambda_k\}, \\
&\mathfrak{X}^\theta_N = \{ (\ell, m) \in \mathbb{W}^{\theta}_{N, N} \times \mathbb{W}^{\theta}_{N, N-1}: \ell \succeq m \},
\end{split}
\end{align}
where $\ell \succeq m$ means that if $ \ell_i = \lambda_i + (N - i)\cdot\theta$ and $m_i = \mu_i + (N - i) \cdot \theta$ then $\lambda_1 \geq \mu_1 \geq \lambda_2 \geq \mu_2 \geq \cdots \geq \mu_{N-1} \geq \lambda_N$. The set $\mathfrak{X}^\theta_N$ is the state space of our point configuration $(\ell, m)$. In Section \ref{Section2} we work with complex measures $\mathbb{P}^\theta_N$ on the state space $\mathfrak{X}^\theta_N$ which have the form (\ref{S1PDef}). We assume that $w(x;N) \neq 0$ for $x \in [0, M + (N-1) \cdot \theta]$ and $\tau(x; N) \neq 0$ for $x \in [\theta, M + (N-1) \cdot \theta]$. We also assume that the normalization constant $Z_N$ in (\ref{S1PDef}) is non-zero, so that the total mass of $\mathbb{P}^\theta_N$ is one.

Fix $i\in\{1, \dots, N\}$ and  $\ell \in \mathbb{W}^{\theta}_{N,N} $. If $\ell_i = s$ we define
$$\ell^-= (\ell_1, \dots, \ell_{i-1}, s-1,
\ell_{i+1}, \dots, \ell_N) \text{ and }
\ell^+= (\ell_1, \dots, \ell_{i-1}, s+1, \ell_{i+1}, \dots,
\ell_N),$$
where we suppress the dependence on $i$ from the notation, whenever it is clear from context. We assume in the first case that $s>m_i$ and in the second one that $m_{i-1}-\theta>s,$ as otherwise $(\ell^\pm, \mu) \not\in \S$. Using the functional equation $\Gamma(z+1) = z\Gamma(z)$ and (\ref{S1PDef2}) one readily observes that
\begin{equation}\label{eq:l_ratio}
\frac{\mathbb P^{\theta}_N(\ell^-, m)}{\mathbb P^{\theta}_N(\ell, m)} 
=\prod_{j \neq i}^N \frac{s-\ell_j  - 1}{s - \ell_j }  \cdot \prod_{j = 1}^{N-1} \frac{s - m_j}{s - m_j + \theta - 1}\cdot \frac{w(s-1)}{w(s)},
\end{equation}
where, to ease notation, we suppress the dependence of $w$ and $\tau$ on $N$. Analogously, for $m \in \mathbb{W}^\theta_{N,N-1}$ and $i\in \{1,\dots, N-1\}$ we define 
  $$m^- = (m_1, \dots, m_{i-1}, t -1, m_{i+1}, \dots,
m_{N-1}) \text{ and } m^+ = (m_1, \dots, m_{i-1}, t +1, m_{i+1},
\dots, m_{N-1}),$$ 
where $m_i = t.$ We assume $t-\theta>\ell_{i+1}$ in the first case and $\ell_i>t$ in the second, as, otherwise $(\ell, \mu^\pm) \not \in \S$. As before we use $\Gamma(z+1) = z\Gamma(z)$ and (\ref{S1PDef2}) to get
\begin{equation}\label{S2eq:m_ratio}
\frac{\mathbb P^{\theta}_N(\ell, m^-)}{\mathbb P^{\theta}_N(\ell, m)} 
=\prod_{j \neq i}^{N-1} \frac{t -m_j - 1 }{t - m_j } \cdot \prod_{j = 1}^{N}\frac{t - \ell_j - \theta}{t  - \ell_j - 1}\cdot  \frac{\tau(t-1)}{\tau(t)} .
\end{equation}
Finally, suppose we have $\ell_i=m_i= s$ and assume $s>\ell_{i+1}+\theta.$ Then we get from (\ref{eq:l_ratio}) and (\ref{S2eq:m_ratio})
\begin{align}
\begin{split}\label{eq:lm_ratio}
\frac{\mathbb P_N^{\theta}(\ell^-, m^-)}{\mathbb P_N^{\theta}(\ell, m)}&  =\frac{\mathbb P_N^{\theta}(\ell, m^-)}{\mathbb P_N^{\theta}(\ell, m)} \cdot  \frac{\mathbb P_N^{\theta}(\ell^-, m^-)}{\mathbb P_N^{\theta}(\ell, m^-)} \\
& = \prod_{j \neq i}^{N-1} \frac{s -m_j - 1 }{s - m_j + \theta -1} \cdot \prod_{j \neq i}^{N} \frac{s-\ell_j  - \theta}{s - \ell_j }\cdot\frac{ w(s-1) \tau(s-1)}{w(s) \tau(s)}.
\end{split}
\end{align}

%

\subsection{Two-level Nekrasov's equations: $\theta \neq 1$.} 
The goal of this section is to state and prove the two-level Nekrasov's equations for the case $\theta \neq 1$. We write the equations and proof for $\theta = 1$ in the next section.
\begin{theorem}\label{TN1} Let $\mathbb{P}^{\theta}_N$ be as in Section \ref{Section2.1} for $\theta \neq 1$, $M \geq 1$ and $N \geq 2$. Let $\mathcal{M} \subseteq \mathbb{C}$ be open and $[0 , M + 1 + (N-1) \cdot \theta] \subset \mathcal{M}.$
Suppose there exist six functions $\phi_i^t$, $\phi_i^b$ and $\phi_i^m$ for $i = 1,2$ that are analytic in $\mathcal{M}$ and such that
\begin{equation}\label{ratioNE}
\frac{\phi_1^t(z)}{ \phi_1^m(z)} =   \frac{w(z-1)}{w(z)}, \hspace{2mm}  \frac{\phi_1^b(\hat{z})}{\phi_1^m(\hat{z})} =  \frac{\tau(\hat{z})}{\tau(\hat{z}-1)} \mbox{ and }\frac{\phi^{t}_2(z)}{\phi^{m}_2(z)} =  \frac{w(z)}{w(z-1)}, \hspace{2mm} \frac{\phi^{b}_2(\hat{z} - \theta)}{\phi^{m}_2(\hat{z} - \theta)}= \frac{\tau (\hat{z} - 1)}{\tau(\hat{z})},
\end{equation}
for all $z \in [ 1, M + (N-1) \cdot \theta]$ and $\hat{z} \in [ \theta + 1, M + (N-1) \cdot \theta]$. Define $R_1(z)$ and $R_2(z)$ through
\begin{align}\label{S2mN}
\begin{split}
R_1(z)= & \phi_1^t(z) \cdot \mathbb{E} \left[ \prod_{i = 1}^N\frac{z- \ell_i -\theta}{z - \ell_i}  \right] + \phi_1^b(z) \cdot \mathbb{E} \left[ \prod_{i = 1}^{N-1}\frac{z- m_i + \theta - 1}{z - m_i - 1} \right]   \\
& + \frac{\theta}{1 - \theta} \cdot \phi_1^m(z) \cdot \mathbb{E} \left[  \prod_{i = 1}^N \frac{z- \ell_i -\theta}{z - \ell_i - 1}  \prod_{i = 1}^{N-1}\frac{z- m_i + \theta - 1}{z - m_i} \right] - \frac{r_1^-(N)}{ z} - \frac{r_1^+(N)}{z - s_N},
\end{split}
\end{align}
\begin{align}\label{S2mNv2}
\begin{split}
R_2(z) = &\phi^{t}_2(z)  \cdot \mathbb{E} \left[ \prod_{i = 1}^N\frac{z - \ell_i + \theta - 1}{z - \ell_i - 1} \right] + \phi^{b}_2(z) \cdot \mathbb{E} \left[ \prod_{i = 1}^{N-1}\frac{z - m_i}{z - m_i + \theta} \right]  \\
& + \frac{\theta}{1 - \theta} \cdot  \phi^{m}_2(z) \cdot  \mathbb{E} \left[  \prod_{i = 1}^N \frac{z - \ell_i + \theta - 1}{z - \ell_i} \prod_{i = 1}^{N-1} \frac{z - m_i}{z - m_i + \theta - 1} \right]- \frac{r_2^-(N)}{ z } - \frac{r_2^+(N)}{z - s_N},
\end{split}
\end{align}
where $s_N = M + 1 + (N-1) \cdot \theta$ and
\begin{align}\label{S2NE1Res}
\begin{split}
&r_1^-(N) = \phi_1^t(0) \cdot (-\theta) \cdot \mathbb{P}^{\theta}_N(\ell_N = 0) \cdot \mathbb{E} \left[ \prod_{i = 1}^{N-1} \frac{ \ell_i + \theta}{  \ell_i } \Big{\vert} \ell_N = 0 \right], \\
& r_1^+(N) =   \phi_1^m(s_N) \cdot  \theta  \cdot  \mathbb{P}^{\theta}_N(\ell_1 = s_N - 1) \cdot \mathbb{E} \hspace{-1mm} \left[ \hspace{-0.5mm} \prod_{i = 2}^{N} \frac{s_N - \ell_i  - \theta }{s_N  - \ell_i - 1 } \hspace{-1mm} \prod_{ i = 1}^{N-1} \hspace{-1mm} \frac{s_N  - m_i + \theta - 1}{s_N - m_i }\Big{\vert} \ell_1 = s_N - 1\right] \\
&+ \phi_1^b(s_N) \cdot \theta \cdot \mathbb{P}^{\theta}_N(m_1 = s_N - 1) \cdot \mathbb{E} \left[ \prod_{i = 2}^{N-1} \frac{s_N - m_i + \theta - 1 }{s_N- m_i - 1 } \Big{\vert} m_1 = s_N - 1 \right], \\
&r_2^-(N) = \phi_2^b(0) \cdot (- \theta) \cdot \mathbb{P}^{\theta}_N(m_{N-1} = \theta) \cdot \mathbb{E} \left[ \prod_{i = 1}^{N-2} \frac{  m_i }{  m_i - \theta } \Big{\vert} m_{N-1} = \theta \right]  \\
& + \phi_2^m(0) \cdot  (- \theta) \cdot  \mathbb{P}^{\theta}_N(\ell_N = 0 ) \cdot \mathbb{E} \left[ \prod_{i = 1}^{N-1} \frac{ \ell_i  - \theta + 1}{  \ell_i} \prod_{i = 1}^{N-1} \frac{   m_i}{ m_i  - \theta + 1}\Big{\vert} \ell_N = 0 \right], \\
&r_2^+(N) = \phi^t_2(s_N ) \cdot \theta \cdot \mathbb{P}^\theta_N( \ell_1 = s_N-1) \cdot \mathbb{E} \left[ \prod_{i = 2}^N\frac{s_N - \ell_i + \theta - 1}{s_N- \ell_i - 1}\Big{\vert} \ell_1 = s_N - 1 \right] .
\end{split} 
\end{align}
Then $R_1(z)$ and $R_2(z)$ are analytic in $\mathcal{M}$. 
\end{theorem}
\begin{remark}
Note that if we put $\phi_1^t(0) = \phi_1^b(M + 1 + (N-1) \cdot \theta) = \phi^m_1(M + 1 + (N-1) \cdot \theta) = 0$ then $r_1^+ = r_1^- = 0$  (\ref{S2NE1Res}). Analogously, if $\phi^t_2( M + 1 + (N-1) \cdot \theta) = \phi_2^m(0) = \phi_2^b(0) = 0$ then $r_2^+ = r_2^- = 0$. Consequently, Theorem \ref{TN1} implies Theorem \ref{Nekrasov}.
\end{remark}

%
\subsubsection{Analyticity of $R_1(z)$} The function $R_1(z)$ has possible poles at $s = a +(N- i)\cdot \theta$, where $i= 1, \dots, N$ and $a\in \{0,\dots, M + 1\}$. Note that all of these poles are simple. This is
 obvious for the expectations in the first line of (\ref{S2mN}). In addition, for the third expectation since for $(\ell, m )\in \S$ we have that $\ell_i, m_i$ separately are {\em strictly} increasing and so each of the two products has simple poles. We still need to consider the case when the two products share a pole. The latter would be possible only if  $m_i = \ell_j+ 1 = s$. Therefore, 
$$\mu_i -i \theta=(\lambda_j+1) -j \theta \iff  \mu_i-\lambda_j-1=(i-j)\theta. $$
Since $\mu_i\geq \lambda_j \text{ for } i< j$ and $\mu_i\leq \lambda_j \text{ for } i\geq j$ 
the above equality can only happen for $i<j$ and $\mu_i=\lambda_j.$ But then by interlacing $\mu_{i} = \mu_{i+1} = \cdots = \mu_{j-1}=\lambda_j$ and so
necessarily $m_{j-1}=\ell_j+\theta = s - 1 + \theta$ producing an extra zero in the second product and reducing the order of the pole by $1.$ Next, we compute the residues at these
possible poles and show that they sum to zero.\\

Fix a possible pole $s$ and assume $s\neq 0, (M+1)+(N-1) \cdot \theta.$ The expectation $\mathbb E$ is a sum over elements  $(\ell, m)\in \S.$ Such an element
contributes to a residue if
\begin{enumerate}
\item $\ell_i=s \text{ or } \ell_i=s-1,$ for some $i = 1, \dots , N$ or 
\item $m_i=s \text{ or }m_i=s-1$ for some $i=1,\dots, N - 1.$ 
\end{enumerate}
Then the residue at $s$ is given by
\begin{equation}\label{BigSum}
\theta \left( \sum_{i  = 1}^{N-1} (A_i + B_i + C_i) +  A_N + E_N + D \right) , \mbox{ where }
\end{equation}

\begin{align*}
A_i = & \sum\limits_{(\ell, m)\in \mathcal  G^i_1} \hspace{-4mm} - \phi_1^t(s)\mathbb P^{\theta}_N(\ell, m) \hspace{-1mm} \left[\prod_{ j\neq i}^N\frac{s-\ell_j-\theta}{s-\ell_j}\right ]\hspace{-1mm}  +\phi_1^{m}(s) \mathbb P^{\theta}_N(\ell^-,m) \hspace{-1mm} \left[\prod_{ j\neq i}^{N}\frac{s-\ell_j-\theta}{s-\ell_j-1}\prod_{j=1}^{N-1}  \frac{s-m_j+\theta-1}{s-m_j}\right] ;
\end{align*}
\begin{align*}
B_i = &  \sum\limits_{(\ell, m)\in \mathcal G^i_2}- \phi_1^{m}(s) \mathbb P^{\theta}_N(\ell, m^+)\left[\prod_{j=1}^{N} \frac{s-\ell_j-\theta}{s-\ell_j - 1} \prod_{ j\neq i}^{N-1}\frac{s-m_j+\theta-1}{s-m_j}\right]\\
& + \phi_1^{b}(s)\mathbb P^{\theta}_N(\ell, m)\left[\prod_{ j\neq i}^{N-1}\frac{s-m_j+\theta-1}{s-m_j-1}\right]; \\
C_i = &\sum\limits_{(\ell, m)\in \mathcal G^i_3}- \phi_1^t(s)\mathbb P^{\theta}_N(\ell, m)
\left[\prod_{ j\neq  i}^N\frac{s-\ell_j-\theta}{s-\ell_j}\right ] +\sum\limits_{(\ell, m)\in \mathcal G^i_4} \phi_1^b(s)\mathbb P^{\theta}_N(\ell^-, m^-)
\left[\prod_{ j\neq i}^{N-1}\frac{s-m_j+\theta-1}{s-m_j-1}\right]; \\
D = &   \phi_1^{m}(s) \sum\limits_{(\ell, m)\in \mathcal{G}}  \mathbb P^{\theta}_N(\ell,m) \mbox{Res}_{z = s} \left[  \prod_{i = 1}^N \frac{z- \ell_i -\theta}{z - \ell_i - 1}  \prod_{i = 1}^{N-1}\frac{z- m_i + \theta - 1}{z - m_i} \right]; \\
E_N =  & \sum\limits_{(\ell, m)\in \mathcal G^N_3}- \phi_1^t(s)\mathbb P^{\theta}_N(\ell, m) \left[\prod_{ j\neq  i}^N\frac{s-\ell_j-\theta}{s-\ell_j}\right ] , \mbox{ where }
\end{align*}
$$  \mathcal G^i_1=\{(\ell, m) | \ell_i=s, (\ell, m)  \hspace{-0.5mm} \in \hspace{-0.5mm}\S, (\ell^- \hspace{-2mm}, m) \hspace{-0.5mm} \in \hspace{-0.5mm} \S \}, \mathcal G^i_2 = \{(\ell, m) | m_i=s - 1, (\ell, m) \in \S, (\ell, m^+) \hspace{-0.5mm}\in \hspace{-0.5mm}\S \}, $$
$$ \mathcal G^i_3=\{(\ell, m) | \ell_i=s, (\ell^-\hspace{-2mm}, m) \not \in \S, (\ell, m) \in \S \}, \mathcal G^i_4=\{(\ell, m) | m_i = s, (\ell^- \hspace{-2mm}, m) \not \in \S , (\ell^- \hspace{-2mm}, m^-) \in \S\}, $$
$$ \mathcal{G} = \{ (\ell, m) \in \mathfrak{X}_N^\theta | (\ell^{+,i}, m) \not \in  \mathcal{G}^i_1 \mbox{ for $i = 1, \dots, N$}, \mbox{ and } (\ell, m^{-,i}) \not \in \mathcal{G}^i_2 \mbox{ for $i = 1, \dots, N-1$ } \}. $$
In the above expressions the operations $(\ell, m)\rightarrow (\ell, m^{\pm}) \text{ and }(\ell, m)\rightarrow (\ell^{\pm}, m)$ are performed in the $i$-th component. Also in the definition of $\mathcal{G}$ the notations $(\ell^{+,i}, m)$ and $(\ell, m^{-,i})$ indicate, in which coordinate the operation is performed. In the second sum in the definition of $C_i$ we have performed a relabeling of the summation variables from $(\ell, m)$ to $(\ell^-, m^-)$, which is reflected in $\mathcal G^i_4$. 

We mention that the first sums in $A_i, C_i$ and $E_N$ together give the contribution of the residue coming from $\ell_i = s$ from the first expectation in (\ref{S2mN}), the second sums in $B_i, C_i$ together give the residue coming from $m_i = s-1$ from the second expectation in (\ref{S2mN}) and the remaining sums together give the residue coming from $m_i = s$ of $\ell_i = s-1$ from the third expectation in (\ref{S2mN}). We have split these sums in a way that will make the cancelation of these terms easier to follow.\\

We first observe that $A_i = 0$ for $i = 1, \dots, N$ and $B_i = 0$ for all $i = 1, \dots, N-1$ because each summand equals zero as follows from  (\ref{eq:l_ratio}), (\ref{S2eq:m_ratio}) and (\ref{ratioNE}). 

We next observe that we can restrict the two sums in $C_i$ to be over $\mathcal{G}_3^i \cap \mathcal G^i_4$ without affecting the values of each sum. For the first sum, we note that if $(\ell,m) \in \mathcal{G}^i_3$ then $\ell_i = s = m_i$ and so $(\ell, m) \not \in  \mathcal{G}^i_4$ implies that $\ell_{i + 1} = s -1$. The latter means that the product in the first sum vanishes for all $(\ell,m) \in \mathcal{G}^i_3$ such that $(\ell,m) \not \in \mathcal{G}^i_4$ and so we do not affect the value of the sum by removing such terms. Similarly, if $(\ell, m) \in \mathcal{G}^i_4$ then $m_i = s = \ell_i$ and so  $(\ell, m) \not \in  \mathcal{G}^i_3$ implies either $s = M + 1 + (N-1) \cdot \theta$ if $i = 1$ (which we ruled out) or $m_{i - 1} = s$. The latter means that the product in the second sum vanishes for all $(\ell,m) \in \mathcal{G}^i_4$ such that $(\ell,m) \not \in \mathcal{G}^i_3$ and so we do not affect the value of the sum by removing such terms. Finally, once we restrict the two sums in $C_i$ to the same set $\mathcal{G}^i_3 \cap \mathcal{G}^i_4$ we can cancel the two sums term by term using (\ref{eq:lm_ratio}) and  (\ref{ratioNE}). 

We see next that $E_N = 0$ since $\mathcal{G}_3^N =\varnothing$ by our assumption that $s \neq 0$.

Let us explain why each summand in $D$ is zero. If $(\ell, m)$ is such that $\ell_i \neq s - 1$ for all $i = 1,\dots,N$ and $m_j \neq s$ for all $j = 1, \dots, N-1$ then the double product is analytic near $s$ and the residue is zero. If $\ell_i = s-1$ for some $i$ and $m_j \neq s$  for all $j$ then by the fact that $(\ell^{+,i},m) \not \in \mathcal{G}^i_1 $ we see that either $s = M+1 + (N-1) \cdot \theta$ if $i = 1$ (which we ruled out) or $m_{i-1} = s-1 + \theta$. It then follows that in the double product, both the numerator and denominator have one factor of $(z- s)$ coming from $(z - m_{i-1} + \theta - 1)$ and $(z - \ell_i - 1)$ respectively --  they cancel and so the residue is zero. Similarly, if $\ell_i \neq s-1$ for all $i = 1, \dots, N$ and $m_j = s$ for some $j$ then since $(\ell, m^{-,j}) \not \in \mathcal{G}^j_2$ we conclude that $\ell_{j+1} = s -\theta$.  It then follows that in the double product, both the numerator and denominator have one factor of $(z- s)$, coming from $(z - \ell_{j+1} - \theta)$ and $(z - m_j)$ respectively -- they cancel and so the residue is zero.

Suppose $(\ell, m)$ is such that $\ell_i = s-1$ for some $i \in \{1, \dots, N\}$ and $m_j = s$ for some $j\in \{1, \dots, N-1\}$. This is only possible when $\theta$ has the form $\theta = k^{-1}$ with $k \geq 2$ (we know $\theta \neq 1$), which we assume in the remainder. If $\theta = k^{-1}$ then the previous situation implies $i- j = k \geq 2$ and the denominator of the product has two factors $(z -s)$ coming from $(z- \ell_i - 1)$ and $(z - m_{j})$. But because $\ell \succeq  m$, we know that $m_{i-1} = s + \theta - 1$ and $\ell_{j+1} = s - \theta$, which means that the numerator in the double product has two factors $(z-s)$ coming from $(z - \ell_{j+1} - \theta)$ and $(z - m_{i - 1} + \theta - 1)$. These factors cancel with those in the denominator and so the residue is zero. Overall we conclude that $R_1(z)$ has no pole at $s$ if $s\neq 0, (M+1)+(N-1)\cdot \theta.$ \\

We finally consider the residues at $s = 0$ and $s = M+1 + (N-1)\cdot \theta$ starting with the former. If $s = 0$ we get a contribution from the first expectation in (\ref{S2mN}) only when $\ell_N = 0$ and we get no contributions from the other two expectations. Consequently, the residue is given by
$$ \phi_1^t(0) \cdot (-\theta) \cdot \mathbb{P}^{\theta}_N(\ell_N = 0) \cdot \mathbb{E} \left[ \prod_{i = 1}^{N-1} \frac{ \ell_i + \theta }{ \ell_i } \Big{\vert} \ell_N = 0 \right]- r_1^-(N),$$
which vanishes by definition of $r_1^-(N)$. If $s =M + 1 + (N-1)\cdot \theta$ then we get a contribution from the second expectation in (\ref{S2mN}) only when $m_1 = M + (N-1)\cdot \theta$ and from the third expectation only when $\ell_1 = M + (N-1)\cdot \theta$, and there is no contribution from the first expectation. Thus the residue is 
\begin{align*}
\begin{split}
& \phi_1^b(s) \cdot \theta \cdot \mathbb{P}^{\theta}_N(m_1 = s - 1) \cdot \mathbb{E} \left[ \prod_{i = 2}^{N-1} \frac{s - m_i + \theta-1 }{s- m_i - 1 } \Big{\vert} m_1 = s - 1 \right] \\
& + \phi_1^m(s) \cdot  \theta \cdot  \mathbb{P}^{\theta}_N(\ell_1 = s - 1) \cdot \mathbb{E} \left[ \prod_{i = 2}^{N} \frac{s - \ell_i  - \theta }{s - \ell_i-1 } \prod_{ i = 1}^{N-1} \frac{s - m_i + \theta-1}{s - m_i }\Big{\vert} \ell_1 = s - 1\right] - r^+_1(N),
\end{split}
\end{align*}
which is zero by the definition of $r_1^+(N)$. This proves the analyticity of $R_1(z)$.

%
\subsubsection{Analyticity of $R_2(z)$.} Our strategy is exactly the same as in the previous section, namely we compute the residue of the right side of (\ref{S2mNv2}) at all possible poles and show that they are all zero. $R_2(z)$ has possible poles at $s = a+(N- i)\cdot \theta$ where $a = 0, \dots, M+1$ and $i  = 1, \dots, N$. Note that all of these poles are simple. This is again  obvious for the first two expectations in (\ref{S2mNv2}). In addition, for the last  expectation since for $(\ell, m )\in \S$ and  $\ell_i, m_i$  separately are {\em strictly} increasing each of the two products has  simple poles. We still need to consider the case when the two products share a pole. The latter would be possible only if  $\ell_i = m_j+1-\theta.$ Therefore, 
$$\lambda_i -i \theta=(\mu_j+1) -(j+1) \theta \iff \lambda_i-\mu_j-1=(i-j-1)\theta.$$
Since $\mu_i\geq \lambda_j \text{ for } i< j$ and $\lambda_i \geq \mu_j \text{ for } i\leq j$
the above equality can only happen for $i\leq j$ and $\mu_j=\lambda_i.$ But then $m_{i}=\ell_i$ producing an extra zero
in the second product and reducing the order of the pole by $1.$ Next, we compute the residues at these possible poles and show that they sum to zero.\\

Fix a possible pole $s$ and assume $s\neq 0, M+1+(N-1) \cdot \theta.$ A pair $(\ell, m)\in \S$ contributes to a residue if
\begin{enumerate}
\item $\ell_i=s\text{ or }s-1,$ for some $i \in \{1,\dots, N\}$ or
\item $m_{i-1}=s+\theta \text{ or }s+\theta-1$ for some $i \in \{2,\dots, N\}.$
\end{enumerate}

Then the residue at $s$ is given by
\begin{equation}\label{BigSumv2}
\theta \left( \sum_{i  = 2}^{N} A_i + B_i + C_i +  A_1 + E_1 + D \right), \mbox{ where }
\end{equation}

\begin{equation*}
\begin{split}
&A_i = \sum\limits_{(\ell, m)\in \mathcal G^i_1}\phi_2^t(s)\mathbb P^{\theta}_N(\ell, m) \left[\prod_{ j\neq i}^N\frac{s-\ell_j+\theta-1}{s-\ell_j-1}\right ]-\phi_2^{m}(s) \mathbb P^{\theta}_N(\ell^+,m)\left[  \prod_{ j\neq i}^{N}\frac{s-\ell_j+\theta-1}{s-\ell_j}\prod_{j=1}^{N-1}  \frac{s-m_j}{s-m_j+\theta-1}\right]; \\
&B_i = \sum_{(\ell,m) \in \mathcal{G}^i_2}\phi_2^{m}(s) \mathbb P^{\theta}_N(\ell, m^-)\left[\prod_{j=1}^{N}  \frac{s-\ell_j+\theta-1}{s-\ell_j}   \prod_{ j\neq i-1}^{N-1}\frac{s-m_j}{s-m_j+\theta-1}\right]- \phi_2^b(s)\mathbb P^{\theta}_N(\ell, m) \left[\prod_{ j\neq i-1}^{N-1}\frac{s-m_j}{s-m_j+\theta}\right]; \\
&C_i = \sum_{(\ell,m) \in \mathcal{G}^i_3} \phi_2^t(s)\mathbb P^{\theta}_N(\ell^-, m^-) \left[\prod_{ j\neq  i}^N\frac{s-\ell_j+\theta-1}{s-\ell_j-1}\right ] - \sum\limits_{(\ell, m)\in \mathcal G^i_{4}}\hspace{-2mm}\phi_2^b(s)\mathbb P^{\theta}_N(\ell, m)\left[\prod_{ j\neq i-1}^{N-1}\frac{s-m_j}{s-m_j+\theta}\right]; \\
&D =   \phi_1^{m}(s) \sum\limits_{(\ell, m)\in \mathcal{G}}  \mathbb P^{\theta}_N(\ell,m) \mbox{Res}_{z = s} \left[   \prod_{i = 1}^N \frac{z - \ell_i + \theta - 1}{z - \ell_i} \prod_{i = 1}^{N-1} \frac{z - m_i}{z - m_i + \theta - 1} \right];\\
&E_1  = \sum_{(\ell,m) \in \mathcal{G}^1_3}\phi_2^t(s)\mathbb P^{\theta}_N(\ell, m) \left[\prod_{ j\neq i}^N\frac{s-\ell_j+\theta-1}{s-\ell_j-1}\right ],  \text{ where}
\end{split}
\end{equation*}
$$  \mathcal G^i_1:=\{(\ell, m) | \ell_i=s - 1, (\ell, m) \in \mathfrak{X}_N^1, (\ell^+, m) \in \mathfrak{X}_N^1 \}, \mathcal G^i_2 := \{(\ell, m) | m_{i-1}=s +\theta , (\ell, m) \in \mathfrak{X}_N^1, (\ell, m^-) \in \mathfrak{X}_N^1 \}, $$
$$ \mathcal G^i_3:=\{(\ell, m) | \ell_i=s, (\ell^-, m^-) \in \mathfrak{X}_N^1,  (\ell, m^-) \not \in \mathfrak{X}_N^1\}, \mathcal G^i_4:=\{(\ell, m) | m_{i-1} = s + \theta, (\ell, m^-) \not \in \mathfrak{X}_N^1 , (\ell, m) \in \mathfrak{X}_N^1\}, $$
$$ \mathcal{G} := \{ (\ell, m) \in \mathfrak{X}_N^\theta | (\ell^{-,i}, m) \not \in  \mathcal{G}^i_1 \mbox{ for $i = 1, \dots, N$ and } (\ell, m^{+,j-1}) \not \in  \mathcal{G}^j_2 \mbox{ for $j = 2, \dots, N$} \}. $$
In the above expressions the operations $\ell^{\pm}$ are applied to the $i$-th component, while $m^{\pm}$ to the $(i -1)$-th. Also in the definition of $\mathcal{G}$ the notations $(\ell^{-,i}, m)$ and $(\ell, m^{+,i})$ indicate, in which coordinate the operation is performed. 

We first observe that $A_i = 0$ for $i = 1, \dots, N$ and $B_i = 0$ for all $i = 1, \dots, N-1$ because each summand equals zero as follows from  (\ref{eq:l_ratio}), (\ref{S2eq:m_ratio}) and (\ref{ratioNE}). 

We next observe that we can restrict the two sums in $C_i$ to be over $\mathcal{G}_3^i \cap \mathcal G^i_4$ without affecting the values of each sum. For the first sum, we note that if $(\ell,m) \in \mathcal{G}^i_3$ then $\ell_i = s$ and $m_{i-1} = s + \theta$ and so $(\ell, m) \not \in  \mathcal{G}^i_4$ implies that $\ell_{i - 1} = s -1 +\theta$. The latter means that the product in the first sum vanishes for all $(\ell,m) \in \mathcal{G}^i_3$ such that $(\ell,m) \not \in \mathcal{G}^i_4$ and so we do not affect the value of the sum by removing such terms. Similarly, if $(\ell, m) \in \mathcal{G}^i_4$ then $m_{i - 1} =s +\theta$ and $\ell_i = s$ and so  $(\ell, m) \not \in  \mathcal{G}^i_3$ implies either $s = 0$ if $i = N$ (which we ruled out) or $m_{i} = s$. The latter means that the product in the second sum vanishes for all $(\ell,m) \in \mathcal{G}^i_4$ such that $(\ell,m) \not \in \mathcal{G}^i_3$ and so we do not affect the value of the sum by removing such terms. Finally, once we restrict the two sums in $C_i$ to the same set $\mathcal{G}^i_3 \cap \mathcal{G}^i_4$ we can cancel the two sums term by term using (\ref{eq:lm_ratio}) and  (\ref{ratioNE}). 

We see next that $E_1 = 0$ since $\mathcal{G}_3^1 =\varnothing$ by our assumption that $s \neq M + 1 + (N-1)\cdot \theta$.

Let us explain why each summand in $D$ is zero. If $(\ell, m)$ is such that $\ell_i \neq s$ for all $i = 1,\dots,N$ and $m_j \neq s + \theta - 1$ for all $j = 1, \dots, N-1$ then the double product is analytic near $s$ and the residue is zero. If $\ell_i = s$ for some $i$ and $m_j \neq s + \theta - 1$  for all $j$ then by the fact that $(\ell^{-,i},m) \not \in \mathcal{G}^i_1 $ we see that either $s = 0$ if $i = N$ (which we ruled out) or $m_{i} = s$. It then follows that in the double product, both the numerator and denominator have one factor of $(z- s)$ coming from $(z - m_{i})$ and $(z - \ell_i )$ respectively --  they cancel and so the residue is zero. Similarly, if $\ell_i \neq s$ for all $i = 1, \dots, N$ and $m_j = s + \theta - 1$ for some $j$ then since $(\ell, m^{j,+}) \not \in \mathcal{G}^{j+1}_2$ we conclude that $\ell_{j} = s - 1 + \theta$.  It then follows that in the double product, both the numerator and denominator have one factor of $(z- s)$, coming from $(z - \ell_{j} + \theta - 1)$ and $(z - m_j + \theta - 1)$ respectively -- they cancel and so the residue is zero.

Suppose $(\ell, m)$ is such that $\ell_i = s$ for some $i \in \{1, \dots, N\}$ and $m_j = s + \theta - 1$ for some $j\in \{1, \dots, N-1\}$. This is only possible when $\theta$ has the form $\theta = k^{-1}$ with $k \geq 2$ (we know $\theta \neq 1$), which we assume in the remainder. If $\theta = k^{-1}$ then the previous situation implies $j- i - 1 = k \geq 2$ and the denominator of the product has two factors $(z -s)$ coming from $(z- \ell_i)$ and $(z - m_{j} + \theta - 1)$. But because $\ell \succeq  m$, we know that $m_{i} = s$ and $\ell_{j} = s + \theta - 1$, which means that the numerator in the double product has two factors $(z-s)$ coming from $(z - \ell_{j} + \theta - 1)$ and $(z - m_{i})$. These factors cancel with those in the denominator and so the residue is zero. Overall we conclude that $R_2(z)$ has no pole at $s$ if $s\neq 0, (M+1)+(N-1)\cdot  \theta.$ \\

We finally consider the residues at $s = 0$ and $s = M+1 + (N-1) \cdot \theta$ starting with the former. If $s = 0$ we get no contribution from the first expectation in (\ref{S2mNv2}) and we get a contribution from the second expectation only when $m_{N-1} = \theta$ and from the third expectation only when $\ell_N = 0$. Consequently, the residue is 
\begin{equation*}
\begin{split}
& \phi_2^b(0) \cdot (- \theta) \cdot \mathbb{P}^{\theta}_N(m_{N-1} = \theta) \cdot \mathbb{E} \left[ \prod_{i = 1}^{N-2} \frac{ m_i }{ m_i - \theta } \Big{\vert} m_{N-1} = \theta \right] + \\
& \phi_2^m(0) \cdot  (- \theta) \cdot  \mathbb{P}^{\theta}_N(\ell_N = 0 ) \cdot \mathbb{E} \left[ \prod_{i = 1}^{N-1} \frac{  \ell_i  - \theta + 1}{ \ell_i} \prod_{i = 1}^{N-1} \frac{  m_i}{  m_i  - \theta + 1}\Big{\vert} \ell_N = 0 \right] - r^-_2(N),
\end{split}
\end{equation*}
which is zero by the definition of $r_2^-(N)$. 

If $s = M + 1 + (N-1) \cdot \theta$ then we get a contribution from the first expectation in (\ref{S2mNv2}) only when $\ell_1 = M + (N-1) \cdot \theta$ and we get no contributions from the other two expectations. Consequently, the residue is given by
$$ \phi^t(s ) \cdot \mathbb{P}^\theta_N(\ell_1 = s-1) \theta \cdot \mathbb{E} \left[ \prod_{i = 2}^N\frac{s - \ell_i + \theta - 1}{s- \ell_i - 1}\Big{\vert} \ell_1 = s - 1 \right]  - r_2^+(N)$$
which is zero by the definition of $r_2^+(N)$. This proves the analyticity of $R_2(z)$.

%

\subsection{Two-level Nekrasov's equations: $\theta= 1$}\label{Section2.3} We state the $\theta = 1$ version of Theorem \ref{TN1}.

\begin{theorem}\label{TN1Theta1} Let $\mathbb{P}^{1}_N$ be as in Section \ref{Section2.1} for $\theta = 1$, $M \geq 1$ and $N \geq 2$. Let $\mathcal{M} \subseteq \mathbb{C}$ be open and $[0 , M  + N] \subset \mathcal{M}.$ Suppose there exist six functions $\phi_i^t$, $\phi_i^b$ and $\phi_i^m$ for $i = 1,2$ that are analytic in $\mathcal{M}$ and such that
\begin{equation}\label{ratioNETheta1}
\frac{\phi_1^t(z)}{ \phi_1^m(z)} =   \frac{w(z-1)}{w(z)}, \hspace{2mm}  \frac{\phi_1^b(\hat{z})}{\phi_1^m(\hat{z})} =  \frac{\tau(\hat{z})}{\tau(\hat{z}-1)} \mbox{ and }\frac{\phi^{t}_2(z)}{\phi^{m}_2(z)} =  \frac{w(z)}{w(z-1)}, \hspace{2mm} \frac{\phi^{b}_2(\hat{z}- 1)}{\phi^{m}_2(\hat{z}-1)}= \frac{\tau (\hat{z}- 1)}{\tau(\hat{z} )},
\end{equation}
for all $z \in [1, M +N -1]$ and $\hat{z} \in [2, M +N -1]$. Define $R_1(z)$ and $R_2(z)$ through
\begin{align}\label{S2mNTheta1}
\begin{split}
R_1(z)= & \phi_1^t(z) \cdot \mathbb{E} \left[ \prod_{i = 1}^N\frac{z- \ell_i -1}{z - \ell_i}  \right] + \phi_1^b(z) \cdot \mathbb{E} \left[ \prod_{i = 1}^{N-1}\frac{z- m_i }{z - m_i - 1} \right]  \\
& + \phi_1^m(z) \cdot \mathbb{E} \left[ \sum_{i = 1}^N \frac{1}{z - \ell_i - 1} - \sum_{i = 1}^{N-1} \frac{1}{z - m_i}  \right] - \frac{r_1^-(N)}{ z } - \frac{r_1^+(N)}{z - s_N},
\end{split}
\end{align}
\begin{align}\label{S2mNv2Theta1}
\begin{split}
R_2(z) = &\phi^{t}_2(z)  \cdot \mathbb{E} \left[ \prod_{i = 1}^N\frac{z - \ell_i}{z - \ell_i - 1} \right] + \phi^{b}_2(z) \cdot \mathbb{E} \left[ \prod_{i = 1}^{N-1}\frac{z - m_i}{z - m_i + 1} \right] \\
& + \phi^{m}_2(z) \cdot  \mathbb{E} \left[ \sum_{i = 1}^{N-1} \frac{1}{z - m_i} - \sum_{i = 1}^N \frac{1}{z - \ell_i} \right]- \frac{r_2^-(N)}{ z } - \frac{r_2^+(N)}{z - s_N} .
\end{split}
\end{align}
where $s_N = M + N$ and 
\begin{align}\label{S2NE1ResTheta1}
\begin{split}
&r_1^-(N) =  \phi_1^t(0) \cdot (-1) \cdot \mathbb{P}^{1}_N(\ell_N = 0) \cdot \mathbb{E} \left[ \prod_{i = 1}^{N-1} \frac{\ell_i + 1}{ \ell_i } \Big{\vert} \ell_N = 0 \right],\\
& r_1^+(N) =   \phi_1^b(s_N)   \mathbb{P}^{1}_N(m_1 = s_N - 1)  \mathbb{E} \left[ \prod_{i = 2}^{N-1} \frac{s_N - m_i }{s_N - m_i - 1 } \Big{\vert} m_1 = s_N - 1 \right] + \phi_1^m(s_N)    \mathbb{P}^{1}_N(\ell_1 = s_N - 1),\\
&r_2^-(N) = \phi_2^b(0) \cdot (- 1) \cdot \mathbb{P}^{1}_N(m_{N-1} = 1) \cdot \mathbb{E} \left[ \prod_{i = 1}^{N-2} \frac{m_i}{ m_i - 1 } \Big{\vert} m_{N-1} = 1 \right] + \phi_2^m(0) \cdot  (- 1) \cdot  \mathbb{P}^{1}_N(\ell_N = 0 ),\\
&r_2^+(N) = \phi^t(s_N ) \cdot \mathbb{P}^1_N(\ell_1 = s_N-1) \cdot \mathbb{E} \left[ \prod_{i = 2}^N\frac{s_N - \ell_i }{s_N- \ell_i - 1}\Big{\vert} \ell_1 = s_N - 1 \right].
\end{split} 
\end{align}
Then $R_1(z)$ and $R_2(z)$ are analytic in $\mathcal{M}$. 
\end{theorem}

\begin{remark}
From (\ref{S2NE1ResTheta1}) we see that if $\phi_1^t(0) = \phi_1^b(M + N) = \phi^m_1(M + N) = 0$ then $r_1^+(N) = r_1^-(N) = 0$. Analogously, if $\phi^t_2( M  + N) = \phi_2^b(0) = \phi_2^b(0) = 0$ then $r_2^+(N) = r_2^-(N) = 0$. 
\end{remark}

\begin{proof}
We will deduce the theorem from Theorem \ref{TN1} by performing an appropriate $\theta \rightarrow 1$ limit transition. We begin with the analyticity of $R_1$.

From the first line in (\ref{ratioNETheta1}) we know that $w, \tau$  have (unique) analytic continuations to a complex neighborhood $U \subseteq \mathcal{M}$ of $[\theta, M +N]$, which we continue to call $w, \tau$. Let $\theta$ be sufficiently close to $1$ so that $[0, M+1 +(N -1)\theta]\subset U$ and write $\P$ for the measure as in (\ref{S1PDef}) with the weights $w, \tau$ we just introduced. It follows from (\ref{S2mN}) that 
\vspace{-1mm}
\begin{align}\label{S23E1}
\begin{split}
R^\theta_1(z)= & \phi_1^t(z) \cdot \mathbb{E}^\theta \left[ \prod_{i = 1}^N\frac{z- \ell_i -\theta}{z - \ell_i}  \right] + \phi_1^b(z) \cdot \mathbb{E}^\theta \left[ \prod_{i = 1}^{N-1}\frac{z- m_i + \theta - 1}{z - m_i - 1} \right]   \\
& + \frac{\theta}{1 - \theta} \cdot \phi_1^m(z) \cdot \mathbb{E}^\theta \left[  \prod_{i = 1}^N \frac{z- \ell_i -\theta}{z - \ell_i - 1}  \prod_{i = 1}^{N-1}\frac{z- m_i + \theta - 1}{z - m_i} \right] - \frac{r_1^-(N)}{ z} - \frac{r_1^+(N)}{z - s_N},
\end{split}
\end{align}
is analytic in $U$. In (\ref{S23E1}) we have inserted $\theta$ into the notation to indicate the dependence of the expressions on it. Subtracting $G^{\theta}_1(z) := \frac{\theta}{1-\theta}\phi^{m}_1(z)  $
from both sides of (\ref{S23E1}) and letting $\theta \rightarrow 1$, we see that the right side converges uniformly over compact subsets of $U \setminus [0, M + N]$ to the right side of (\ref{S2mNTheta1}). In particular, we see that for $z \in U \setminus [0, M+N]$ we have 
\begin{equation}\label{S23E2}
\lim_{\theta \rightarrow 1} [R_1^{\theta}(z) - G^{\theta}_1(z)]  = R_1(z).
\end{equation}

Let $\epsilon > 0$ be sufficiently small so that $V_{\epsilon} = \{z \in \mathbb{C}: d(z,[0, M+N]) < \epsilon \} \subseteq U$ and let $\gamma$ denote a positively oriented contour that encloses $V_{\epsilon}$, and is contained in $U$. From Cauchy's theorem and (\ref{S23E2}) we have for $z \in V_{\epsilon/2} $
$$\lim_{\theta \rightarrow 1} [R_1^{\theta}(z) - G^{\theta}_1(z)] = \lim_{\theta \rightarrow 1} \frac{1}{2\pi \i}\int_{\gamma} \frac{R_1^{\theta}(\zeta) - G^{\theta}_1(\zeta)}{\zeta -z} d\zeta = \frac{1}{2\pi \i}\int_{\gamma} \frac{R_1(\zeta) }{\zeta -z} d\zeta,$$
and the latter convergence is uniform on $V_{\epsilon/2} $. From Theorem \ref{TN1} for $\theta \neq 1$ we know that $R_1^{\theta}$ are analytic and $G^{\theta}_1$ are analytic. This means that 
$$\frac{1}{2\pi \i}\int_{\gamma} \frac{R_1(\zeta) }{\zeta -z} d\zeta$$
defines an analytic function on $V_{\epsilon/2}$ as the uniform limit of analytic functions, see \cite[Chapter 2, Theorem 5.2]{SS}. In addition, by (\ref{S23E2}) we know that this function agrees with $ R_1(z)$ on $V_{\epsilon/2} \setminus [0, M+N].$ The latter shows that $R_1(z)$, which from (\ref{S2mNTheta1}) is clearly analytic in $\mathcal{M} \setminus [0, M+N]$,  has an analytic continuation to $\mathcal{M}$. Since from (\ref{S2mNTheta1}) we know that $R_1(z)$ is meromorhpic on $\mathcal{M}$, we conclude that it is in fact analytic there as desired.

For the analyticity of $R_2$ we argue in a similar fashion.
\end{proof}

%

\section{Particular setup} \label{Section3}

In this section we specialize the measures $\mathbb{P}_N^{\theta}$ that were defined in Section \ref{Section1.1} to the case when $\tau(\cdot) \equiv 1$ -- this is the main object in our asymptotic analysis. In Section \ref{Section3.1} below we list some properties of this measure as well as the particular assumptions we make about the way its parameters are scaled. In Section \ref{Section3.2} we introduce certain deformations of the measure $\mathbb{P}_N^\theta$ that will be useful in our analysis.

%
\subsection{Properties of the system} \label{Section3.1}
We begin with some necessary notation, some of it being recalled from Sections \ref{Section1} and \ref{Section2}. Let $\theta > 0, M \in \mathbb{Z}_{\geq 0}$ and $N \in \mathbb{N}$. For such parameters we set
\begin{align}\label{GenState}
\begin{split}
&\Lambda_N = \{  (\lambda_1, \dots, \lambda_N) : \lambda_1\geq  \lambda_2 \geq \cdots \geq \lambda_N, \lambda_i \in \mathbb{Z}  \mbox{ and }0  \leq \lambda_i  \leq M \}, \\
&\mathbb{W}^\theta_{N,k} = \{ (\ell_1, \dots, \ell_k):  \ell_i = \lambda_i + (N - i)\cdot\theta, \mbox{ with } (\lambda_1, \dots, \lambda_k) \in \Lambda_k\}, \\
&\mathfrak{X}^\theta_N = \{ (\ell, m) \in \mathbb{W}^{\theta}_{N, N} \times \mathbb{W}^{\theta}_{N, N-1}: \ell \succeq m \},
\end{split}
\end{align}
where $\ell \succeq m$ means that if $ \ell_i = \lambda_i + (N - i)\cdot\theta$ and $m_i = \mu_i + (N - i) \cdot \theta$ then $\lambda_1 \geq \mu_1 \geq \lambda_2 \geq \mu_2 \geq \cdots \geq \mu_{N-1} \geq \lambda_N$. We interpret $\ell_i$'s and $m_i$'s as locations of two classes of particles. If $\theta = 1$ then all particles live on the integer lattice, while for general $\theta$ the particles of index $i$ live on the shifted lattice $\mathbb{Z} +(N - i) \cdot \theta$. Throughout the text we will frequently switch from $\ell_i$'s to $\lambda_i$'s and from $m_i$'s to $\mu_i$'s without mention using the formulas 
\begin{equation}\label{Eqcoord}
\ell_i = \lambda_i + (N- i)\cdot\theta \mbox{ and } m_i = \mu_i + (N - i) \cdot \theta. 
\end{equation}

We define a probability measure $\mathbb{P}_N$ on $\mathfrak{X}^\theta_N$ through
\begin{equation}\label{PDef}
\mathbb{P}_N(\ell, m) = Z_N^{-1} \cdot H^t(\ell) \cdot H^b(m) \cdot I(\ell, m), \mbox{ where }
\end{equation}
\begin{align}\label{PDef2}
\begin{split}
H^t(\ell) = &\prod_{1 \leq i < j \leq N} \frac{\Gamma(\ell_i - \ell_j + 1)}{\Gamma(\ell_i - \ell_j  - \theta + 1)}  \prod_{i = 1}^N w(\ell_i;N), \hspace{2mm} H^b(m) = \prod_{1 \leq i < j \leq N-1} \frac{\Gamma(m_i - m_j + \theta)}{\Gamma(m_i - m_j)} \\
I(\ell, m) = &\prod_{1 \leq i < j \leq N} \frac{\Gamma(\ell_i - \ell_j - \theta + 1)}{\Gamma(\ell_i - \ell_j) } \cdot \prod_{1 \leq i < j \leq N-1} \frac{\Gamma(m_i - m_j + 1)}{\Gamma(m_i - m_j + \theta)} \\
&\times \prod_{1 \leq i < j \leq N} \frac{\Gamma(m_i - \ell_j)}{ \Gamma(m_i - \ell_j - \theta + 1)}  \cdot \prod_{1 \leq i \leq j \leq N-1} \frac{\Gamma(\ell_i - m_j + \theta)}{\Gamma(\ell_i - m_j + 1)}.
\end{split}
\end{align}
Here $Z_N$ is a normalization function (called the {\em partition function}) and $w(x,N)$ is a weight function, which is assumed to be positive for $x \in [0 , M + (N-1) \cdot \theta]$. 
This is precisely the measure from (\ref{S1PDef}) when $\tau \equiv 1$. 

The measure $\mathbb{P}_N$ satisfies the following important property. The projection of $\mathbb{P}_N$ on the particles $\ell_1, \cdots, \ell_N$ satisfies
\begin{equation}\label{SingleLevelMeasure}
\mathbb{P}_N(\ell_1, \dots, \ell_N) \propto\prod_{1 \leq i < j \leq N} \frac{\Gamma(\ell_i - \ell_j + 1)\Gamma(\ell_i - \ell_j + \theta)}{\Gamma(\ell_i - \ell_j)\Gamma(\ell_i - \ell_j -\theta + 1)}  \prod_{i = 1}^N w(\ell_i; N).
\end{equation}
The law in (\ref{SingleLevelMeasure}) is known as a {\em discrete $\beta$-ensemble} and its global fluctuations were studied in \cite{BGG}. We will prove (\ref{SingleLevelMeasure}) later in Proposition \ref{PropExtension}.

The fact that the projection of $\mathbb{P}_N$ to the $\ell_i$'s is given by (\ref{SingleLevelMeasure}) is the main reason we consider $\tau \equiv 1$, since it allows us to use the results that were already established in \cite{BGG} about these measures. We believe that one should be able to extend the results of \cite{BGG} to the more general measures in (\ref{S1PDef}), but we do not pursue this direction in this paper as it deviates significantly from our main goal. We mention here that $\tau$ plays the role of an external potential for the particles $(m_1, \dots, m_{N-1})$ and setting $\tau \equiv 1$ corresponds to having no external potential acting on these particles. \\

We are interested in obtaining asymptotic statements about $\mathbb{P}_N$ as $N \rightarrow \infty$. This requires that we scale our parameter $M$ and impose some regularity assumptions on the weight functions $w(x;N)$. We list these assumptions below.\\

{\raggedleft \bf Assumption 1.} We assume that we are given parameters $\theta > 0$, $\lM > 0$. In addition, we assume that we have a sequence of parameters $M_N \in \mathbb{N}$ such that
\begin{equation}\label{GenPar}
\mbox{$M_N \geq 0$ and }, \left| M_N - N\lM\right| \leq A_1, \mbox{ for some $A_1 > 0$.}
\end{equation}
The measures $\mathbb{P}_N$ will then be as in (\ref{PDef}) for $ M= M_N, \theta$ and $N$.\\

{\raggedleft \bf Assumption 2.} We assume that $w(x;N)$ in the interval $[0, M_N + (N-1) \cdot \theta]$ has the form
$$w(x;N) = \exp\left( - N V_N(x/N)\right),$$
for a function $V_N$ that is continuous in the interval $I_N = [0, M_N \cdot N^{-1} + (N-1) \cdot N^{-1} \cdot \theta]$. In addition, we assume that there is a continuous function $V(s)$ on $I = [0, \lM + \theta]$ such that 
\begin{equation}\label{GenPot}
\left| V_N(s) - V(s) \right| \leq A_2 \cdot N^{-1}\log(N), \mbox{ for $s \in I_N \cap I$ and $|V(s)| \leq A_3$ for $s \in I$, }
\end{equation}
for some constants $A_2,A_3 > 0$. We also require that $V(s)$ is differentiable and for some $A_4 > 0$
\begin{equation}\label{DerPot}
\left| V'(s) \right| \leq A_4 \cdot \left[ 1 + \left| \log |s | \right|  + | \log |s -  \lM - \theta||  \right], \mbox{ for } s \in \left[0, \lM + \theta \right], \mbox{ for $s \in I$}.
\end{equation}
\begin{remark}
In plain words, Assumption 2 states that the weights $w(x;N)$ underlying the discrete model $\mathbb{P}_N$ asymptotically look like $e^{-NV(x/N)}$ for some function $V$ that plays the role of an external potential in our model. This external potential is assumed to be differentiable on the interval $(0, \lM + \theta)$, but its derivative is allowed to have logarithmic singularities near the endpoints $0$ and $\lM + \theta$ -- some of the applications we have in mind satisfy this condition. Also in applications (see \cite[Section 9]{BGG}) the function $V$ may depend on $N$, therefore, we keep this dependence in the notation.  We believe that one can take more general remainders in the above two assumptions, without significantly influencing the arguments in the later parts of the paper. However, we do not pursue this direction due to the lack of natural examples.
\end{remark}

Consider the random probability measure $\mu_N$ on $\mathbb{R}$ given by
\begin{equation}\label{EmpMeas}
\mu_N = \frac{1}{N} \sum_{i = 1}^N \delta \left( \frac{\ell_i}{N} \right), \mbox{ where $(\ell_1, \dots, \ell_N)$ is $\mathbb{P}_N$-distributed}.
\end{equation}
Under Assumptions 1 and 2 we have the following result, which can be found in \cite[Theorem 5.3]{BGG}.
\begin{proposition}\label{LLN}
Suppose that Assumptions 1 and 2 hold. Then the measures $\mu_N$ converge weakly in probability to a certain deterministic probability measure $\mu(x)dx$. More precisely, for each Lipschitz function $f(x)$ defined in a real neighborhood of $[0, \lM + \theta]$ and each $\varepsilon > 0$ the random variables
$$N^{1/2 - \varepsilon} \left| \int_{\mathbb{R}} f(x) \mu_N(dx) - \int_{\mathbb{R}}f(x) \mu(x)dx \right|$$
converge to $0$ in probability and in the sense of moments. The measure $\mu(x)dx$ 
\begin{enumerate}
\item is supported on a subset of $[0, \lM + \theta]$;
\item has density $\mu(x)$ such that $0 \leq \mu(x) \leq \theta^{-1}$ 
\end{enumerate}
and is the unique maximizer of the functional
\begin{equation}\label{energy}
I_V[\rho] = \theta \iint_{x\neq y} \log|x - y| \rho(x) \rho(y)dxdy - \int_{\mathbb{R}} V(x) \rho(x)dx
\end{equation}
among all measures that satisfy the above two properties.
\end{proposition}
We call the measure $\mu$ in Proposition \ref{LLN} the {\em equilibrium measure}. 

{\raggedleft \bf Assumption 3.}  We assume that we have an open set $\mathcal{M}  \subseteq \mathbb{C}$, such that $[0, \lM + \theta] \subset \mathcal{M}$.
 In addition, for all large $N$ one is provided with holomorphic functions $\Phi^+_N, \Phi^-_N$ on $\mathcal{M}$ such that 
\begin{equation}\label{eqPhiN}
\begin{split}
&\frac{w(Nx;N)}{w(Nx-1;N)}=\frac{\Phi_N^+(Nx)}{\Phi_N^-(Nx)},
\end{split}
\end{equation} 
whenever $x \in [ N^{-1}, M_N \cdot N^{-1}+  (N-1) \cdot N^{-1} \cdot \theta]$. Moreover, 
\begin{equation*}
\begin{split}
&\Phi^{-}_N(Nz) = \Phi^{-}(z) + O \left(N^{-1} \right) \mbox{ and } \Phi^{+}_N(z) = \Phi^{+}(Nz) + O \left(N^{-1} \right),
\end{split}
\end{equation*}
where the constants in the big $O$ notation are uniform over $z$ in compact subsets of $\mathcal{M}$. All aforementioned functions are holomorphic in $\mathcal{M}$, $\Phi^{\pm}$ do not depend on $N$ and are positive (in particular real) on $(0, \lM + \theta)$.\\

Under Assumptions 1-3 we have the following single level Nekrasov's equation for $\mathbb{P}_N$, which is an analogue of \cite[Theorem 4.1]{BGG}. Since the way we state the result is a bit different we also supply the proof in Section \ref{Section10}.
\begin{proposition}\label{SingleLevelNekrasov}
 Suppose that Assumptions 1-3 hold. Define
\begin{equation}\label{SingleLevelEquation}
R_N(z) = \Phi^-_N(z) \cdot \mathbb{E}_{\mathbb{P}_N} \left[ \prod_{i = 1}^N  \frac{z - \ell_i - \theta}{z - \ell_i}\right] + \Phi^+_N(z) \cdot \mathbb{E}_{\mathbb{P}_N} \left[ \prod_{i = 1}^N  \frac{z - \ell_i + \theta - 1}{z - \ell_i - 1} \right] - \frac{r^-(N)}{z} - \frac{r^+(N)}{z - s_N}, 
\end{equation}
where $s_N = M_N+ 1 + (N-1) \cdot \theta$ and 
\begin{align}\label{RemSL}
\begin{split}
&r^-(N) = \Phi^-_N(0) \cdot (-\theta) \cdot \mathbb{P}_N(\ell_N = 0) \cdot \mathbb{E}_{\mathbb{P}_N} \left[ \prod_{i= 1}^{N-1}  \frac{ \ell_i + \theta }{ \ell_i}\Big{\vert}  \ell_N = 0 \right],\\
&r^+(N) = \Phi^+_N(s_N) \cdot \theta \cdot \mathbb{P}_N(\ell_1 = s_N -1 ) \cdot \mathbb{E}_{\mathbb{P}_N} \left[\prod_{i = 2}^N  \frac{s_N - \ell_i  + \theta -1}{s_N - \ell_i - 1} \Big{\vert}  \ell_1 = s_N - 1 \right].
\end{split}
\end{align}
Then $R_N(z)$ is a holomorphic function on the rescaled domain $N \cdot \mathcal{M}$ in Assumption 3.
\end{proposition}

The next assumption we require aims to upper bound the quantities $r^{\pm}(N)$ in (\ref{RemSL}).\\
{\raggedleft \bf Assumption 4.} We assume that there are constants $C, c, a > 0$ such that for all large $N$
\begin{align}\label{RUB}
\begin{split}
&\mathbb{P}_N(\ell_N = 0) \cdot \left|\Phi_N^-(0)\right| \leq C \exp ( - c N^a) \mbox{ and } \\
&\mathbb{P}_N(\ell_1= M_N + (N-1) \cdot \theta ) \cdot \left| \Phi_N^+(M_N + 1 +(N-1) \cdot \theta) \right| \leq C \exp ( - c N^a).
\end{split}
\end{align}
Note that by the interlacing condition we have $\mathbb{P}_N(m_{N-1} = 0) \leq \mathbb{P}_N(\ell_{N} = 0)$ and $\mathbb{P}_N(m_{1} = M_N + (N-1) \cdot \theta) \leq \mathbb{P}_N(\ell_{1} = M_N + (N-1) \cdot \theta)$, which implies from (\ref{RUB}) that
\begin{align*}
\begin{split}
&\mathbb{P}_N(m_{N-1} = 0) \cdot \left|\Phi_N^-(0)\right| \leq C \exp ( - c N^a) \mbox{ and } \\
&\mathbb{P}_N(m_1= M_N + (N-1) \cdot \theta ) \cdot \left| \Phi_N^+(M_N + 1 +(N-1) \cdot \theta) \right| \leq C \exp ( - c N^a).
\end{split}
\end{align*}
\begin{remark} The significance of Assumption 4 is that the last two terms in equation (\ref{SingleLevelEquation}) are very small and can be essentially ignored when using the $N \rightarrow \infty$ limit of that equation.  One case when these terms can be completely ignored is if $\Phi^-_N(0) = 0$ and $\Phi^+_N(M_N + 1 +(N-1) \cdot \theta) = 0$. If these functions do not vanish then, in view of Assumption 3, we know that $\Phi_N^-(0)$ and $\Phi_N^+(M_N + 1 +(N-1) \cdot \theta)$ are bounded for all large $N$ and so (\ref{RUB}) would hold if $\mathbb{P}_N(\ell_N = 0) \leq C \exp ( - c N^a) $ and $\mathbb{P}_N(\ell_1= M_N + (N-1) \cdot \theta )\leq  C \exp ( - c N^a)$, i.e. if it is very unlikely to find particles at the ends of the interval $I_N = [0, M_N +(N-1) \cdot \theta]$.
\end{remark}

The final assumption we require is about the equilibrium measure $\mu$. A convenient way to encode $\mu$ is through its Stieltjes transform
\begin{equation}\label{GmuDef}
G_{\mu}(z) := \int_\mathbb{R} \frac{\mu(x)dx}{z - x}.
\end{equation}
The following two functions naturally arise in the asymptotic study of the measures $\mu_N$
\begin{align}\label{QRmu}
\begin{split}
&R_{\mu}(z) = \Phi^-(z) \cdot e^{- \theta G_{\mu} (z) }+  \Phi^+(z) \cdot e^{ \theta G_{\mu} (z) }\\
&Q_{\mu}(z) = \Phi^-(z) \cdot e^{- \theta G_{\mu} (z) } -  \Phi^+(z) \cdot e^{ \theta G_{\mu} (z) }.
\end{split}
\end{align}
For the above functions we have the following result. 
\begin{lemma}\label{S3AnalRQ}
If Assumptions 1-4 hold then $R_\mu$ and $Q_\mu^2$ in (\ref{QRmu}) are analytic on $\mathcal{M}$ and real-valued on $\mathcal{M} \cap \mathbb{R}$. 
\end{lemma}
\begin{proof}
See the proof of Lemma \ref{AnalRQ}.
\end{proof}
We detail the relationship between $R_\mu$ and the equilibrium measure $\mu$ in the following statement.
\begin{lemma}\label{S3Lsupp} If Assumptions 1-4 hold then $\mu$ has density
\begin{equation}\label{S3eqMForm}
 \mu(x) =  \frac{1}{\theta \pi } \cdot \mathrm{arccos} \left( \frac{R_\mu(x)}{2 \sqrt{\Phi^-(x)  \Phi^+(x)}}\right),
\end{equation}
for $x \in [0, \lM + \theta]$ and $0$ otherwise. We recall that $\mathrm{arccos}$ is as in Section \ref{Section1.5} and $\Phi^+(x) \Phi^-(x) > 0$ on $(0, \lM+ \theta)$ by Assumption 3. In particular, $\mu(x)$ is continuous in $[0, \lM + \theta]$. 
\end{lemma}
\begin{proof}
See the proof of Lemma \ref{Lsupp}.
\end{proof}

Similarly to \cite[Section 4]{BGG} we impose the following technical condition on the function $Q_\mu(z)$ and refer to that paper for a discussion on its significance.

{\raggedleft \bf Assumption 5.} Assume there are a holomorphic function $H(z)$ and numbers $\alpha, \beta$ such that
\begin{itemize}
\item $0 \leq \alpha < \beta \leq \lM + \theta$;
\item $H(z) \neq 0$ for all $z \in [0, \lM + \theta]$;
\item $Q_\mu(z) = H(z) \sqrt{(z - \alpha)(z- \beta)}$.
\end{itemize}
We recall that $\sqrt{(z - \alpha)(z- \beta)}$ is as in Section \ref{Section1.5} and note that $H$ here is different from $H^t$ and $H^b$ in (\ref{PDef2}), which should cause no confusion. We remark that the form $Q_\mu(z) = H(z) \sqrt{(z - \alpha)(z- \beta)}$ ensures that $\mu$ has a single interval of support in $[0, \lM + \theta]$ and also the set of points where $\mu(x) \in (0,\theta^{-1})$ is precisely the interval $(\alpha,\beta)$ -- see the proofs of Lemmas \ref{Lsupp} and \ref{NonVanish}. Maximal connected intervals where $\mu(x) \in (0,\theta^{-1})$ are referred to as {\em bands}, see \cite{BGG}, and Assumption 5 implies that $\mu$ has a single band $(\alpha, \beta)$. In reality, Assumption 5 is stronger than the single band assumption and its precise form is made in a way that is suitable for the application of the loop equations and dates back to \cite[Assumption 4]{BGG}.  \\

We isolate an important consequence of Assumption 5 and postpone its proof until Section \ref{Section9}.
\begin{lemma}\label{S3NonVanish}
Under Assumptions 1-5, we have that $\Phi^-(z) + \Phi^+(z) - R_\mu(z) \neq 0$ for all $z \in [0, \lM + \theta]$. 
\end{lemma}
\begin{proof}
See the proof of Lemma \ref{NonVanish}.
\end{proof}
\begin{remark}
In view of Lemma \ref{S3NonVanish} by possibly making $\mathcal{M}$ in Assumption 3 smaller we can ensure that $\Phi^-(z) + \Phi^+(z) - R_\mu(z) \neq 0$  and $H(z) \neq 0$ for all $z \in \mathcal{M}$ where $H(z)$ is as in Assumption 5. We will assume this to always be the case.
\end{remark}

\begin{remark} Assumptions 1-5 are essentially the same as those in \cite[Section 3]{BGG} for the case $k = 1$ and note that Assumption 4 is formulated more closely after \cite[Section 8, Assumption 6]{BGG} than the stronger \cite[Section 3, Assumption 5]{BGG}.
\end{remark}

%
\subsection{Deformed measures}\label{Section3.2}
If $(\ell, m)$ is distributed according to (\ref{PDef}) we denote 
\begin{equation}\label{RegG}
 G^{t}_N(z)= \sum_{i = 1}^N \frac{1}{z- \ell_i/ N } \mbox{ and } G^b_N(z) = \sum_{i = 1}^{N-1} \frac{1}{z - m_i/N}. 
\end{equation}
Our asymptotic analysis of $\mathbb{P}_N$ goes through a detailed study of the joint distribution of $G^t$ and $G^b$. From \cite[Theorem 7.1]{BGG} we have the following result about $G^t$.
\begin{proposition}\label{CLT}
Suppose that Assumptions 1-5 hold. As $N \rightarrow \infty$ the random field $G^t_N(z) - \mathbb{E}_{\mathbb P_N} [G^t_N(z)]$, $z \in \mathbb{C} \setminus [0, \lM + \theta]$ converges (in the sense of joint moments uniformly in $z$ over compact subsets of $\mathbb{C} \setminus [0, \lM + \theta]$) to a centered complex Gaussian field with second moment
\begin{equation}\label{eq:GField}
\lim_{N\rightarrow \infty}  \left(\mathbb E_{\mathbb P_N} \left[G^t_N(z_1) G^t_N(z_2)\right]-\mathbb E_{\mathbb P_N} \left[G^t_N(z_1)\right] \mathbb E_{\mathbb P_N} \left[G^t_N(z_2)\right] \right)=: \mathcal C_\theta(z_1, z_2), \mbox{ where }
\end{equation}
\begin{equation}\label{eq:var}
\begin{split}
 \mathcal{C}_\theta(z_1, z_2) = -\frac{\theta^{-1}}{2(z_1-z_2)^2} \left(1 - \frac{(z_1 - \alpha)(z_2- \beta) + (z_2 - \alpha )(z_1- \beta)}{2\sqrt{(z_1 -\alpha )(z_1- \beta)}\sqrt{(z_2 - \alpha)(z_2- \beta)}} \right),
\end{split}
\end{equation}
where $\alpha,\beta$ are as in Assumption 5.
\end{proposition}
We point out that the proof of Proposition \ref{CLT} in \cite{BGG} relies on the following moment bounds statement, which is not written out explicitly in that paper.
\begin{proposition}\label{MomentBoundSingleLevel}
Suppose that Assumptions 1-5 hold and let $U = \mathbb{C} \setminus [0, \lM + \theta]$. Then for each $k \in \mathbb{N}$ we have
\begin{equation}\label{MBSLE}
\mathbb{E} \left[ \left| G^t_N(z) - N G_\mu(z) \right|^k \right] = O(1),
\end{equation}
where $G_\mu(z)$ is as in (\ref{GmuDef}) and the constant in the big $O$ notation depends on $k$ and is uniform as $z$ varies over compact subsets of $U$. 
\end{proposition}
We will require Proposition \ref{MomentBoundSingleLevel} in our paper, and since its proof has not been written out in \cite{BGG} we will give it in Section \ref{Section10}, where we will also give the proof of Proposition \ref{CLT}. We remark that the arguments presented in \cite{BGG} are sufficient to establish these propositions and we will follow them quite closely; however, there are a few inaccuracies in the proofs in \cite{BGG} and for the sake of completeness we will supply the full proofs of the above two propositions in the present paper.

While Proposition \ref{CLT} gives a complete answer to the question of the asymptotic distribution of $G^t_N$, we need to study the joint distribution of $G^t_N$ and $G^b_N$. Below we present a method for obtaining the joint cumulants of these fields, which is inspired by \cite{BGG}.\\

Take $2m+2n$ complex parameters $\t^1 = (t^1_1,\dots,t^1_m)$, $\vm^1 = (v^1_1,\dots,v^1_m)$, $\t^2 = (t_1^2, \dots, t^2_n)$, $\vm^2 = (v^2_1, \dots, v^2_n)$ and such that $v^i_a + t^i_a - y \neq 0$ for all meaningful $i,a$ and all $y \in [0, M_N \cdot N^{-1} + \theta]$ and we require that the numbers $t^i_a$ are sufficiently close to zero as we specify shortly. Let the deformed distribution $\mathbb P^{\t, \vm}_N$  be defined through
\begin{equation} \label{eq:distrgen_deformed}
\begin{split}
\mathbb{P}_N^{\t, \vm}(\ell, m)=Z(\t, \vm)^{-1} \mathbb{P}_N(\ell, m)\prod_{i =1}^{N} \prod^m_{a=1}  \left(  \hspace{-1mm} 1+ \frac{t^1_a}{v^1_a-\ell_i/N} \right) \cdot \prod_{i =1}^{N-1}\prod^n_{a=1}  \left(  \hspace{-1mm} 1+ \frac{t^2_a}{v^2_a-m_i/N} \right) .
\end{split}
\end{equation}
 If $m = n = 0$ we have $\mathbb P_N^{\t, \vm} = \mathbb{P}_N$ is the undeformed measure. In general, $\mathbb P_N^{\t, \vm}$ may be a complex-valued measure but we always choose the normalization constant $Z(\t, \vm)$ so that $\sum_{\ell, m} \mathbb P^{\t, \vm}_N(\ell, m) = 1$. The requirement that the numbers $t^i_a$ are sufficiently close to zero ensures that $Z(\t, \vm) \neq 0$ and the deformed measure is well-defined.

For $n$ bounded random variables $\xi_1, \dots, \xi_n$ we let $M(\xi_1, \dots, \xi_n)$ denote their joint cumulant, see \cite[Chapter 3]{Taqqu}. If $n = 1$ the latter expression stands for the expectation $\mathbb{E}[\xi_1]$.

The definition of the deformed measure $\mathbb P_N^{\t, \vm}$ is motivated by the following observation. 
\begin{lemma}\label{LemCum1} Let $\xi$ be a bounded random variable. For any $m, n\geq 0$ we have
\begin{equation}\label{eq:derivative_k}
\frac{\partial^{m+n}}{\partial t^1_1 \cdots \partial t^1_m \partial t^2_1 \cdots \partial t^2_n}\mathbb E_{\mathbb P_N^{\t, \vm}}\left[\xi\right]\bigg\rvert_{t^i_a = 0} = M( \xi, G^t_N(v^1_1),\dots ,G^t_N(v^1_m),G^b_N(v^2_1), \dots, G^b_N(v^2_n)),
\end{equation}
where the right side is the joint cumulant of the given random variables with respect to $\mathbb P_N$.
\end{lemma}
\begin{remark}
The above result is analogous to \cite[Lemma 2.4]{BGG}, which in turn is based on earlier related work in random matrix theory. We present a proof below for the sake of completeness.
\end{remark}
\begin{proof}
One way to define the joint cumulant of bounded random variables $\xi ,\xi^1_1, \dots, \xi^1_m, \xi_1^2, \dots, \xi_n^2$ is through
$$ \frac{\partial^{m+n+1}}{\partial t_0 \partial t^1_1 \cdots \partial t^1_m \partial t^2_1 \cdots \partial t^2_n}  \log \left(\mathbb{E} \exp \left(t_0 \xi +  \sum_{ i = 1}^m t^1_i \xi^1_i + \sum_{i = 1}^n t^2_i \xi^2_i\right) \right) \Bigg\rvert_{t_0= 0, t_a^i=0}. $$
Performing the differentiation with respect to $t_0$ we can rewrite the above as
$$ \frac{\partial^{m+n}}{ \partial t^1_1 \cdots \partial t^1_m \partial t^2_1 \cdots \partial t^2_n}  \frac{\mathbb{E}  \left[\xi \exp \left(  \sum_{ i = 1}^m t^1_i \xi^1_i + \sum_{i = 1}^n t^2_i \xi^2_i\right) \right] }{\mathbb{E}  \left[ \exp \left(  \sum_{ i = 1}^m t^1_i \xi^1_i + \sum_{i = 1}^n t^2_i \xi^2_i\right)  \right]} \Bigg \rvert_{t^i_a=0}. $$
Set $\xi^1_i = G^t_N(v^1_i)$ for $i = 1, \dots, m$ and $\xi^2_i = G^b(v^2_i)$ for $i = 1, \dots, n$ and observe that 
$$\exp \left( t G^t_N(z)\right) = \prod_{i = 1}^N \left(1+ \frac{t}{z-\ell_i/N} \right) + O(t^2) \mbox{ and }$$
$$  \exp \left( t G^b_N(z)\right) = \prod_{i = 1}^{N-1} \left(1+ \frac{t}{z-m_i/N} \right) + O(t^2) \mbox{ as $t \rightarrow 0$. }$$
The last statements imply the desired statement.
\end{proof}

%

\section{Application of Nekrasov's equations} \label{Section4}
We continue with the same notation as in Section \ref{Section3}. If $G_N^t(z), G_N^b(z)$ are as in (\ref{RegG}) and $G_\mu(z)$ is as in (\ref{GmuDef}) we define
\begin{align}\label{DefX}
\begin{split}
&X^t_N(z) = G^t_N(z) - N G_\mu(z), \hspace{2mm} X^b_N(z) = G^b_N(z) - N G_\mu(z) \mbox{, and } \\
& \Delta X_N(z) = N^{1/2} (X^t_N(z) - X^b_N(z)).
\end{split}
\end{align}
The goal of this section is to use the Nekrasov's equations to obtain formulas relating the joint cumulants of the above random variables for different values of $z$. These formulas, given in (\ref{NekrasovOutput}), will play a central role in our asymptotic analysis in Section \ref{Section5}.

%
\subsection{Moment bounds for $X^t_N(z)$ and $X^t_b(z)$} \label{Section4.1}
In this section we establish that $\mathbb{E}_{\mathbb{P}_N} \left[ |X^t_N(z)|^k \right] = O(1) $ and $\mathbb{E}_{\mathbb{P}_N} \left[ |X^b_N(z)|^k \right] = O(1)$ for all $k \geq 1$. We will need these estimates together with some others for the remainder of the paper.

\begin{proposition}
Suppose that Assumptions 1-5 hold. Then for each $k \geq 1$ we have
\begin{equation}\label{momentBound}
\mathbb{E}_{\mathbb{P}_N} \left[ |X^t_N(z)|^k \right] = O(1) \mbox{ and }  \mathbb{E}_{\mathbb{P}_N} \left[ |X^b_N(z)|^k \right] = O(1),
\end{equation}
where the constants in the big $O$ notation depend on $k$ but not on $N$ (provided it is sufficiently large) and are uniform as $z$ varies over compact subsets of $\mathbb{C} \setminus [0, \lM + \theta]$. 
\end{proposition}

The fact that $\mathbb{E}_{\mathbb{P}_N} \left[ |X^t_N(z)|^k \right] = O(1)$ for each $k \geq 1$ uniformly as $z$ varies over compact subsets of $\mathbb{C} \setminus [0, \lM + \theta]$ follows from Proposition \ref{MomentBoundSingleLevel} (which is proved in Section \ref{MomentBoundSingleLevel}). The fact that $\mathbb{E}_{\mathbb{P}_N} \left[ |X^b_N(z)|^k \right] = O(1)$ follows from $\mathbb{E}_{\mathbb{P}_N} \left[ |X^t_N(z)|^k \right] = O(1)$ and the following lemma.

\begin{lemma} Suppose that $K$ is a compact subset of $\mathbb{C} \setminus [0, \lM + \theta]$. Then there exists $N_0 \in \mathbb{N}$ and $C > 0$ that depend on $K$ and $\lM$, $\theta$, $A_1$ as in Assumption 1 such that if $N \geq N_0$, $z \in K$ then $\mathbb{P}_N$ almost surely 
\begin{equation}\label{PointwiseDiff}
\left| X^t_N(z) - X^b_N(z)\right| \leq C.
\end{equation}
\end{lemma}
\begin{proof}
Let $d> 0$ be the distance between $K$ and $[0, \lM + \theta]$. We choose $N_0$ to be sufficiently large so that the distance between $K$ and $[\theta \cdot N^{-1}, M_N \cdot N^{-1} + \theta]$ is at least $d/2$ and assume $N \geq N_0$. Let $z \in K$ and write $z = x+ \i y$ with $x, y \in \mathbb{R}$. For brevity we set $a_i = \ell_i/N$ for $i = 1, \dots, N$ and $b_i = m_i/N$ for $i = 1, \dots, N-1$. Then by definition we have
$$X^t_N(z) - X^b_N(z) = F(z) + \i G(z), \mbox{ where } F(z) =  \sum_{i = 1}^N \frac{x - a_i}{(x-a_i)^2 + y^2} - \sum_{i = 1}^{N-1}\frac{x - b_i}{(x-b_i)^2 + y^2} \mbox{ and }  $$
$$ G(z) =  \sum_{i = 1}^N\frac{y}{(x-a_i)^2 +y^2} -  \sum_{i = 1}^{N-1}\frac{y}{(x-b_i)^2 + y^2} .$$
In addition, by the interlacing property $\ell \succeq m$ we know that $a_1 > b_1 > a_2 > b_2 > \cdots > b_{N-1} > a_N$. 

Let us denote $F_1(z) = \{ i \in \{1, \dots, N\}: x- a_i \leq - |y|\}$, $F_2(z) = \{ i \in \{1, \dots, N\}: -|y|< x- a_i \leq |y|\}$ and $F_3(z) =  \{ i \in \{1, \dots, N\}: x- a_i > |y|\}$. Observe that since the sequence $a_i$ is decreasing we know that for each $i \in \{1,2,3\}$ the set $F_i(z)$ is either empty or consists of consecutive integers. Moreover, $F_1(z), F_2(z), F_3(z)$ are pairwise disjoint and $F_1(z) \cup F_2(z) \cup F_3(z) = \{1, \dots, N\}$. We next define the sets $E_i(z)$ as follows. If $|F_i(z)| \leq 1$ then $E_i(z) = \emptyset$, and if $F_i(z) = \{{s_1}, \dots, {s_1 + k +1} \}$ for $k \geq 0$ then $ E_i(z) = \{{s_1}, \dots, {s_1 + k} \}$. We observe that by construction we have that $E_1(z), E_2(z), E_3(z)$ are pairwise disjoint and $N-1 \geq |E_1(z)| + |E_2(z)| + |E_3(z)| \geq N - 4.$ 

If $y \in \mathbb{R} \setminus \{0\}$ then the function $f_y(t) = \frac{t}{y^2 + t^2}$ has the derivative $f'_y(t) = \dfrac{y^2 - t^2}{\left(y^2 + t^2\right)^2}$ and so we conclude that $f_y(t)$ is decreasing on $(-\infty, -|y|)$,  is increasing on $(-|y|, |y|)$ and is decreasing on $(|y|, \infty)$. This combined with the statement $a_1 > b_1 > a_2 > b_2 > \cdots > b_{N-1} > a_N$ means that
$$ \left| \sum_{i \in F_1(z)} \frac{x-a_i}{(x-a_i)^2 + y^2} - \sum_{i \in E_1(z)} \frac{x - b_i}{(x-b_i)^2 + y^2} \right|\leq \frac{|x- a_1|}{(x-a_1)^2 + y^2} \leq \frac{2}{d} .$$
One shows analogously that for $j = 2,3$ we have
$$\left| \sum_{i \in F_j(z)} \frac{x-a_i}{(x-a_i)^2 + y^2} - \sum_{i \in E_j(z)} \frac{x - b_i}{(x-b_i)^2 + y^2} \right|\leq \frac{2}{d},$$
and so provided $y \neq 0$ we have
\begin{equation}\label{FIneq}
|F(z)| \leq \frac{9}{d}
\end{equation}
If $y = 0$ then $F(z) = \sum_{i = 1}^N \frac{1}{x-a_i} - \sum_{i = 1}^{N-1}\frac{1}{x-b_i}$ and we must have either $x > a_1$ or $x < a_N$. Using the monotonicity of the function $f_0(t) =t^{-1}$ on $\mathbb{R}_+$ or $\mathbb{R}_-$ we see that (\ref{FIneq}) also holds in this case. \\

Let us denote $F_+(z) = \{ i \in \{1, \dots, N\}: x- a_i \geq 0\}$, $F_-(z) = \{ i \in \{1, \dots, N\}: x- a_i < 0\}$. As before the sets $F_+(z), F_{-}(z)$ are either empty or consist of consecutive integers. Moreover, $F_+(z), F_-(z)$ are pairwise disjoint and $F_+(z) \cup F_-(z)  = \{1, \dots, N\}$. We next define the sets $E_{\pm}(z)$ as follows. If $|F_{\pm}(z)| \leq 1$ then $E_{\pm}(z) = \emptyset$, and if $F_{\pm}(z) = \{{s_1}, \dots, {s_1 + k +1} \}$ for $k \geq 0$ then $ E_{\pm}(z) = \{{s_1}, \dots, {s_1 + k} \}$. We observe that by construction we have that $E_+(z), E_-(z)$ are pairwise disjoint and $N-1 \geq |E_+(z)| + |E_-(z)| \geq N - 3.$ 

If $y \in \mathbb{R} \setminus \{0\}$ then function $g_y(t) = \frac{1}{t^2 + y^2}$ is increasing on $(-\infty, 0]$ and decreasing on $(0, \infty)$.  This combined with the statement $a_1 > b_1 > a_2 > b_2 > \cdots > b_{N-1} > a_N$ means that
$$ \left| \sum_{i \in F_\pm(z)} \frac{1}{(x-a_i)^2 + y^2} - \sum_{i \in E_\pm(z)} \frac{1}{(x-b_i)^2 + y^2} \right|\leq \frac{1}{(x-a_1)^2 + y^2} \leq \frac{4}{d^2} .$$
The latter implies
\begin{equation}\label{GIneq}
|G(z)| \leq \frac{10R_0}{d^2} ,
\end{equation}
where $R_0$ is sufficiently large so that $K$ is contained in the disk of radius $R_0$ at the origin. If $y = 0$ then (\ref{GIneq}) holds trivially. Combining (\ref{FIneq}) and (\ref{GIneq}) we conclude (\ref{PointwiseDiff}) with $C = \frac{10(d + R_0)}{d^2 }$.
\end{proof}

We finish this section by remarking that by Cauchy's inequalities, see e.g. \cite[Corollary 4.3]{SS}, we have that for each $k \geq 1$ the quantities
\begin{equation}\label{derBound}
\mathbb{E}_{\mathbb{P}_N} \left[ |X^t_N(z)|^k \right], \mathbb{E}_{\mathbb{P}_N} \left[ |X^b_N(z)|^k \right], \mathbb{E}_{\mathbb{P}_N} \left[ |\partial_z X^t_N(z)|^k \right], \mathbb{E}_{\mathbb{P}_N} \left[ |\partial_z X^b_N(z)|^k \right]
\end{equation}
are all $O(1)$ uniformly over compact subsets of $\mathbb{C} \setminus [0, \lM + \theta]$. 

%
\subsection{Asymptotic expansions} \label{Section4.2} 
In this section we summarize several asymptotic expansion formulas as well as derive some estimates for various products that appear in the Nekrasov's equations -- Theorems \ref{TN1} and \ref{TN1Theta1}. Below we write $\zeta_N(z)$ to mean a generic random analytic on $\mathbb{C} \setminus [N^{-1} (1 + \theta) ,  \max( \lM + \theta, (M_N + 1 + N\theta) N^{-1} ) ]$ function such that for each $k \geq 1$ we have $\mathbb{E} [ |\zeta_N(z)|^k] = O(1)$ uniformly over compact subsets of $\mathbb{C} \setminus [0, \lM + \theta]$. 

If $x, y \in [-\theta - 1, \theta + 1]$ we observe that 
\begin{align*}
\begin{split}
 \prod_{ i = 1}^N \frac{Nz - \ell_i + x}{Nz - \ell_i + y} &= \exp \left( \sum_{i = 1}^N \log \left( 1 + \frac{N^{-1}(x-y)}{z - \ell_i/N + y/N} \right)\right) \\
&  =\exp \left( \frac{x-y}{N} G_N^t(z) + \frac{x^2 - y^2}{2N^2} \partial_z G_N^t(z)  +O(N^{-2})  \right),
\end{split}
\end{align*}
where the error is $\mathbb{P}_N$-almost sure and uniform over compact subsets of $\mathbb{C} \setminus [0, \lM + \theta]$. Combining the latter with (\ref{DefX}) and (\ref{derBound}) we conclude that 
\begin{equation}\label{AEP1v2}
\begin{split}
&\prod_{ i = 1}^N \frac{Nz - \ell_i + x}{Nz - \ell_i + y}= e^{(x -y)G_\mu(z)} \cdot \left[1   + \frac{(x-y) X^t_N(z)}{N} + \frac{(x^2 - y^2) \partial_z G_\mu(z)}{2N} \right] +  \frac{\zeta_N(z)}{N^2}.
\end{split}
\end{equation}
Analogous considerations give 
\begin{equation}\label{AEP1v3}
\begin{split}
&\prod_{ i = 1}^{N-1} \frac{Nz - m_i + x}{Nz - m_i + y}  = e^{(x -y)G_\mu(z)} \cdot \left[1   + \frac{(x-y) X^b_N(z)}{N} + \frac{(x^2 - y^2) \partial_z G_\mu(z)}{2N} \right] +  \frac{\zeta_N(z)}{N^2}.
\end{split}
\end{equation}
We also have the following simple equality
\begin{equation}\label{AEP1v4}
\begin{split}
&\sum_{i = 1}^N \frac{1}{N z - \ell_i + x} - \sum_{i = 1}^{N-1} \frac{1}{N z - m_i + y} = \frac{X_N^t(z) - X_N^b(z)}{N} + \frac{(x-y) \partial_zG_\mu(z)}{N} + \frac{\zeta_N(z)}{N^2}.
\end{split}
\end{equation}

We next derive some estimates on the following random variables
\begin{align}\label{S4xiN}
\begin{split}
& \xi_N^{t,1} = \frac{\theta \cdot \Phi^{-}_N(0)}{N} \prod_{i = 1}^{N-1} \frac{ \ell_i + \theta}{  \ell_i } \cdot {\bf 1}\{ \ell_N = 0\}, \xi_N^{b,1} =  \frac{ \theta \cdot \Phi^+_N(s_N)}{N}  \prod_{i = 2}^{N-1} \frac{s_N - m_i + \theta - 1 }{s_N- m_i - 1 } \cdot {\bf 1}\{ m_1 = s_N - 1\}, \\
&  \xi_N^{m,1} = \frac{ \theta \cdot \Phi^+_N(s_N)}{N}  \prod_{i = 2}^{N} \frac{s_N - \ell_i  - \theta }{s_N  - \ell_i - 1 } \prod_{ i = 1}^{N-1} \frac{s_N  - m_i + \theta - 1}{s_N - m_i } {\bf 1} \{\ell_1 = s_N - 1\} \\
&\xi_N^{t,2} = \frac{ \theta \cdot \Phi^+_N(s_N)}{N} \prod_{i = 2}^N\frac{s_N - \ell_i + \theta - 1}{s_N- \ell_i - 1} {\bf 1}\{ \ell_1 = s_N - 1\}, \xi_N^{b,2} =  \frac{\theta \cdot \Phi^{-}_N(0)}{N}  \prod_{i = 1}^{N-2} \frac{  m_i }{  m_i - \theta } \cdot {\bf 1}\{ m_{N-1} = \theta\}, \\
&  \xi_N^{m,2} =   \frac{\theta \cdot \Phi^{-}_N(0)}{N}  \prod_{i = 1}^{N-1} \frac{ \ell_i - \theta + 1}{  \ell_i} \prod_{i = 1}^{N-1} \frac{   m_i}{ m_i - \theta + 1} \cdot {\bf 1}\{ \ell_N = 0\},
\end{split}
\end{align}
where $s_N = M_N + 1 +(N-1)\cdot \theta$. We claim that $\mathbb{P}_N$ - almost surely we have
\begin{equation}\label{S4xiN2}
\left| \xi_N^{t/m/b, 1/2} \right| = O(1).
\end{equation}
 Observe that since $\ell_i, m_i \geq (N-i)\theta$ and $s_N - \ell_i - 1, s_N - m_i - 1 \geq (i-1) \theta$ each of the products
 $$ \prod_{i= 1}^{N-1}  \frac{ \ell_i + \theta }{ \ell_i}, \quad\prod_{i= 2}^{N-1} \frac{s_N - m_i +\theta - 1}{s_N - m_i - 1}, \quad\prod_{i = 1}^{N-2} \frac{  m_i }{  m_i - \theta },\quad \prod_{i = 2}^N\frac{s_N - \ell_i + \theta - 1}{s_N- \ell_i - 1} $$
 is greater than or equal to $1$ and less than or equal to $N$.
 
 In view of Assumption 3 in Section \ref{Section3}, the above inequalities establish (\ref{S4xiN2}) for the variables $\xi_N^{t/b, 1/2}$. For $\theta \geq 1$ we have
 $$0 \leq \prod_{i = 1}^{N-1} \frac{ \ell_i - \theta + 1}{  \ell_i}  \leq 1, \hspace{2mm}0 \leq  \prod_{i = 2}^N \frac{s_N - \ell_{i}  - \theta }{s_N  - \ell_{i} - 1 } \leq 1\text{ and } 0 \leq  \frac{s_N  - m_1 + \theta - 1}{s_N - m_1 }\leq \theta, \quad\mbox{ and so }$$
$$ 0 \leq  \prod_{i = 2}^N  \frac{s_N - \ell_{i}  - \theta }{s_N  - \ell_{i} - 1 } \cdot  \prod_{ i = 1}^{N-1} \frac{s_N  - m_i + \theta - 1}{s_N - m_i } \leq \theta \cdot \prod_{i= 2}^{N-1} \frac{s_N - m_i +\theta - 1}{s_N - m_i - 1} \leq \theta N,$$ 
$$ \mbox{ and } 0 \leq \prod_{i = 1}^{N-1} \frac{ \ell_i - \theta + 1}{  \ell_i} \cdot \frac{   m_i}{ m_i - \theta + 1} \leq \theta \cdot \prod_{i= 1}^{N-2} \frac{   m_i}{ m_i  - \theta} \leq \theta  N.$$
On the other hand, if $\theta \in (0,1)$ we have using $\ell \succeq m$ that
$$ 0 \leq  \prod_{i = 1}^{N-1} \left[\frac{ \ell_i  - \theta + 1}{  \ell_i} \cdot \frac{   m_i}{ m_i - \theta + 1} \right] \leq 1 \mbox{ and } 0 \leq \prod_{ i = 1}^{N-1}\left[ \frac{s_N - \ell_{i+1}  - \theta }{s_N  - \ell_{i+1} - 1 } \cdot \frac{s_N  - m_i + \theta - 1}{s_N - m_i }\right] \leq 1.$$
Combining the above inequalities we conclude (\ref{S4xiN2}) for $\xi_N^{m, 1/2}$. 

We end this section by introducing some useful notation for the expressions in the Nekrasov's equations that change depending on whether $\theta = 1$ or $\theta \neq 1$
\begin{align}\label{S4Pi}
\begin{split}
&\Pi^\theta_{1}(z) =  \begin{cases} \frac{\theta}{1-\theta} \prod_{i = 1}^N \frac{z- \ell_i -\theta}{z - \ell_i - 1}  \prod_{i = 1}^{N-1}\frac{z- m_i + \theta - 1}{z - m_i} &\mbox{ if } \theta \neq 1, \\
 \sum_{i = 1}^N \frac{1}{z - \ell_i - 1} - \sum_{i = 1}^{N-1} \frac{1}{z - m_i} &\mbox{ if } \theta = 1,
\end{cases} \\
&\Pi^\theta_{2}(z) =  \begin{cases} \frac{\theta}{1-\theta} \prod_{i = 1}^N \frac{z - \ell_i + \theta - 1}{z - \ell_i} \prod_{i = 1}^{N-1} \frac{z - m_i}{z - m_i + \theta - 1} &\mbox{ if } \theta \neq 1, \\ 
\sum_{i = 1}^{N-1} \frac{1}{z - m_i} - \sum_{i = 1}^N \frac{1}{z - \ell_i}  &\mbox{ if } \theta = 1.
\end{cases}
\end{split}
\end{align}

%
\subsection{Application of Nekrasov's equations} \label{Section4.3} In this section we derive formulas relating the joint cumulants of $X^t_N(z)$, $X^b_N(z)$, $\Delta X_N(z)$ from (\ref{DefX}) at different values of $z$. Our main tool is the Nekrasov's equations -- Theorems \ref{TN1} and \ref{TN1Theta1}. 

Let us fix a compact subset $K$ of $\mathbb{C} \setminus [0, \lM + \theta]$ and suppose $\epsilon > 0$ is sufficiently small so that $K$ is at least distance $\epsilon$ from $[0, \lM + \theta]$. We also let $v_a^1$ for $a = 1, \dots, m$ and $v_a^2$ for $a = 1, \dots, n$ be $m+n$ points in $K$. Let us define
\begin{align}\label{S4AB}
\begin{split}
&A_1(z) = \prod_{a = 1}^m \left[ v_a^1 +t_a^1 - z + \frac{1}{N} \right] \left[ v_a^1 - z \right], B_1(z) =  \prod_{a = 1}^m \left[ v_a^1 +t_a^1 - z \right] \left[ v_a^1 - z + \frac{1}{N}\right], \\
&A_2(z) = \prod_{a = 1}^n \left[ v_a^2 +t_a^2 - z \right] \left[ v_a^2 - z + \frac{1}{N} \right], B_2(z) =  \prod_{a = 1}^n \left[ v_a^2 +t_a^2 - z + \frac{1}{N} \right] \left[ v_a^2 - z\right], \\
&A_3(z) =  \prod_{a = 1}^n \left[ v_a^2 +t_a^2 - z  - \frac{\theta}{N}\right] \left[ v_a^2 - z + \frac{1 - \theta}{N} \right], B_3(z) =  \prod_{a = 1}^n \left[ v_a^2 +t_a^2 - z + \frac{1-\theta}{N} \right] \left[ v_a^2 - z - \frac{\theta}{N}\right] \hspace{-1mm}.
\end{split}
\end{align}
Then we readily check that when $\theta \neq 1$ Theorem \ref{TN1} and when $\theta = 1$ Theorem \ref{TN1Theta1} is satisfied for the measures $\mathbb{P}_N^{{\bf t}, {\bf v}}$ from Section \ref{Section3.2} with 
$$\phi_1^t(Nz) = \Phi^{-}_N(Nz)   A_1(z)  B_2(z), \phi^b_1(Nz) = \Phi^{+}_N(Nz)    A_2(z) B_1(z) , \phi^m_1(Nz) =  \Phi^{+}_N(Nz)   B_1(z) B_2(z)$$
$$\phi_2^t(Nz) = \Phi^{+}_N(Nz) A_3(z)  B_1(z) , \phi^b_2(Nz) = \Phi^{-}_N(Nz)  A_1(z) B_3(z), \phi^m_2(Nz) = \Phi^{-}_N(Nz)  A_1(z) A_3(z) ,$$
where $\Phi^{\pm}_N$ are as in Assumption 3. Also the domain $\mathcal{M}$ in Theorem \ref{TN1} or Theorem \ref{TN1Theta1} is $N \cdot \mathcal{M}$ with $\mathcal{M}$ as in Assumption 3. We conclude from Theorem \ref{TN1} or Theorem \ref{TN1Theta1} that the following $R_N^1$ and $R_N^2$ are analytic in $\mathcal{M}$ as in Assumption 3
\begin{align}\label{S4TLNE1v1}
\begin{split}
R^1_N(Nz) = &\Phi^{-}_N(Nz) A_1(z)  B_2(z)  \mathbb{E} \left[ \prod_{i = 1}^N\frac{Nz- \ell_i -\theta}{Nz - \ell_i} \right] + \Phi^{+}_N(Nz)    A_2(z)  B_1(z) \\
& \times   \mathbb{E} \left[ \prod_{i = 1}^{N-1}\frac{Nz- m_i + \theta - 1}{Nz - m_i - 1} \right] +  \Phi^{+}_N(Nz)  B_1(z) B_2(z) \cdot  \mathbb{E} \left[  \Pi^\theta_{1}(Nz)  \right] +  V_N^1(z),
\end{split}
\end{align}

\begin{align}\label{S4TLNE2v1}
\begin{split}
R^2_N(Nz) = &\Phi^{+}_N(Nz)  A_3(z) B_1(z)  \mathbb{E} \left[ \prod_{i = 1}^N\frac{Nz - \ell_i + \theta - 1}{Nz - \ell_i - 1} \right] + \Phi^{-}_N(Nz)  A_1(z) B_3(z)\\
&\times \mathbb{E} \left[ \prod_{i = 1}^{N-1}\frac{Nz - m_i}{Nz - m_i + \theta} \right]  +     \Phi^{-}_N(Nz)  A_1(z) A_3(z)  \mathbb{E} \left[\Pi^\theta_{2}(Nz)  \right] +  V_N^2(z),
\end{split}
\end{align}
where $A_i(z),B_i(z)$ are as in (\ref{S4AB}), $\Pi^\theta_i(z)$ are as in (\ref{S4Pi}), and
\begin{align}
\begin{split}
& V_N^1(z) = \frac{A_1(0)  B_2(0)  \mathbb{E} \left[\xi_N^{t,1}\right]  }{z}  - \frac{ B_1(s_N/N) A_2(s_N/N) \mathbb{E} \left[ \xi_N^{b,1} \right] + B_1(s_N/N)B_2(s_N/N)\mathbb{E} \left[ \xi_N^{m,1} \right]}{z - s_N/N}, \\
& V_N^2(z) = \frac{A_1(0) B_3(0)   \mathbb{E} \left[\xi_N^{b,2} \right] + A_1(0) A_3(0) \mathbb{E} \left[\xi_N^{m,2} \right]  }{z}  - \frac{ B_1(s_N/N) A_3(s_N/N) \mathbb{E}\left[\xi_N^{t,2}\right]}{z - s_N/N},
\end{split}
\end{align}
with $\Phi_N^{\pm}$ as in Assumption 3, $s_N = M_N + 1 +(N-1)\cdot \theta$ and $\xi_N^{t/m/b, 1/2}$ as in (\ref{S4xiN}). All the expectations above are with respect to the measure $\mathbb{P}_N^{{\bf t}, {\bf v}}$, and we have assumed $\max_{i,j} |t^i_j|$ is small enough, depending on $K, n, m$ alone, so that $\mathbb{P}_N^{{\bf t}, {\bf v}}$ is well-defined and we also assume $\max_{i,j} |t^i_j| < \epsilon/2$ (the latter ensures the deformed weights are non-vanishing, which is required for the application of the results of Section \ref{Section2}).

\begin{definition}\label{CumDef}
We summarize some notation in this definition. Let $K$ be a compact subset of $\mathbb{C} \setminus [0, \lM + \theta]$. In addition, we fix integers $r,s,t \geq 0$ and let $m = r+ s$ and $n = r + t$. We fix points $\{v_a\}_{a = 1}^r, \{u_b \}_{b = 1}^s, \{w_c \}_{c = 1}^t \subseteq K$ and set $v_a^1 = v_a^2 = v_a$ for $a = 1, \dots, r$, $v_{b+r}^1 = u_b$ for $b = 1, \dots, s$ and $v_{c+r}^2 = w_c$ for $c = 1, \dots, t$.
In addition, we fix $v \in K$ and let $\Gamma$ be a positively oriented contour, which encloses the segment $[0 , \lM + \theta]$, is contained in $\mathcal{M}$ as in Assumption 3 and avoids $K$.

For integers $p \leq  q$, we will denote by $\llbracket p, q \rrbracket$ the set of integers $\{p, p+1, \dots ,q\}$.

For a bounded random variable $\xi$ and sets $A,B,C$ we let  $M(\xi; A,B,C)$ be the joint cumulant of the random variables $\xi$, $\Delta X_N(v_a)$, $X_N^t(u_b)$, $X_N^b(w_c)$ from (\ref{DefX}) for $a \in A$,  for $b \in B$ and $c \in C$. If $A = B= C = \emptyset$ then $M(\xi; A,B,C) = \mathbb{E}[\xi]$. 
\end{definition}
In the remainder of the section we use the notation in Definition \ref{CumDef}. Our goal is to use (\ref{S4TLNE1v1}), (\ref{S4TLNE2v1}) and our work from Sections \ref{Section4.1} and \ref{Section4.2} to derive the following result
\begin{align}\label{NekrasovOutput}
\begin{split} 
& M\left( \Delta X_N(v) ; \llbracket 1,r \rrbracket, \llbracket 1, s \rrbracket, \llbracket 1, t \rrbracket \right) =   \frac{1}{2\pi \i}\int_{\Gamma} dz W_N(z)  \\
&+ \sum_{ B \subseteq \llbracket 1, s \rrbracket} \sum_{C \subseteq \llbracket 1 ,t \rrbracket}  \frac{ \Phi^+(z) e^{\theta G_\mu(z)}  M  \left( X^b_N(z); \llbracket 1, r \rrbracket, B, C  \right) {\bf 1}\{ |B|+|C| < s + t\} }{S(z) N^{-1/2}N^{s - |B|} N^{t - |C|} }  \prod_{b \in B^c} \hspace{-1mm} \frac{1 }{(u_b-z)^2 } \hspace{-1mm} \prod_{c \in C^c} \hspace{-1mm} \frac{1}{(w_c - z)^2} \\
&- \sum_{A \subseteq \llbracket 1,r \rrbracket} \sum_{B \subseteq \llbracket 1, s \rrbracket}  \sum_{C \subseteq \llbracket 1, t \rrbracket}   \frac{ \Phi^+(z)e^{\theta G_\mu(z)} M \left(X_N^t(z)   ; A, B, C \right) {\bf 1}\{ |A| + |B|+|C| < r+ s + t\}   }{S(z) N^{r/2 - |A|/2-1/2} N^{r - |A|} N^{s - |B|} N^{t - |C|}}      \\
&\times  \prod_{a \in A^c} \left[ \frac{-2\theta }{(v_a -z)^3}\right] \prod_{b \in B^c} \frac{1}{ (u_b -z)^2} \prod_{c \in C^c}  \frac{1}{ (w_c -z)^2} \\
&+ \sum_{A \subseteq \llbracket 1,r \rrbracket}  \sum_{C \subseteq \llbracket 1, t \rrbracket}  \frac{ \Phi^-(z) M\left( \Delta X_N(z) ; A, \llbracket 1, s \rrbracket , C\right) {\bf 1}\{ |A| +|C| < r + t\}  }{S(z) N^{r/2 - |A|/2} N^{t - |C|}} \prod_{a \in A^c} \left[ \frac{-1}{  (v_a -z)^2}\right]   \prod_{c \in C^c} \frac{1}{ (w_c -z)^2} \\
&+   \sum_{A \subseteq \llbracket 1, r \rrbracket, B \subseteq \llbracket 1, s \rrbracket}\frac{  \Phi^+(z)  M\left( \Delta X_N(z) ; A, B, \llbracket 1, t \rrbracket  \right) {\bf 1}\{ |A| + |B| < r+ s \}   }{ S(z)  N^{r/2 - |A|/2} N^{s - |B|}}  \prod_{a \in A^c}   \frac{1}{ (v_a -z)^2}\prod_{b \in B^c}  \frac{ 1}{ (u_b -z)^2} \\
& + \sum_{A \subseteq \llbracket 1,r \rrbracket}\sum_{B \subseteq \llbracket 1,s \rrbracket}\sum_{C \subseteq \llbracket 1,t \rrbracket}   \frac{ M \left(  \zeta^{\Gamma}_N(z); A, B,C\right)}{N^{r/2 - |A|/2 + 1/2}} = 0.
\end{split}
\end{align}
where  $S(z) = (z-v) \cdot (\Phi^+(z) + \Phi^-(z) - R_\mu(z))$ and $\zeta^{\Gamma}_N(z)$ stands for a generic random analytic function such that for each $k \geq 1$ we have $\mathbb{E} [ |\zeta^{\Gamma}_N(z)|^k] = O(1)$ uniformly over $z \in \Gamma$. In addition,  $A^c = \llbracket 1,r \rrbracket \setminus A$, $B^c = \llbracket 1, s \rrbracket \setminus B$ and $C^c = \llbracket 1,t \rrbracket \setminus C$, and $W_N(z)$ is given by
\begin{align}\label{WNDef}
\begin{split}
&N^{-(r+3)/2} W_N(z) =  \frac{\theta \Phi_N^-(Nz) e^{-\theta G_\mu(z)}  \partial_z G_\mu(z) {\bf 1} \{ r + s + t = 0\} }{S(z) N  }    +   \frac{\Phi_N^+(Nz) e^{\theta G_\mu(z)} {\bf 1} \{ r = 0\} }{S(z) N^{s+t} } \\
&  \times \prod_{b = 1}^s \left[ \frac{1 }{(u_b-z)(u_b - z^-) } \right] \prod_{c  = 1}^t  \left[ \frac{1}{(w_c - z)(w_c - z^-)}\right] \cdot \left[\frac{1}{\theta} + \frac{[\theta - 2] \partial_z G_\mu(z) }{2N}\right]  \\
&-  \frac{\theta \Phi_N^+(Nz)  \partial_z G_\mu(z) {\bf 1} \{ t = 0\} }{ S(z) N^{r + s + 1} }  \prod_{a = 1}^r  \left[ \frac{1}{ (v_a -z)(v_a - z^- )}\right] \prod_{b = 1}^s  \left[ \frac{1}{ (u_b -z)(u_b - z^-)}\right]    \\
&- \frac{\Phi_N^+(Nz)e^{\theta G_\mu(z)} }{S(z) N^{2r + s+t}}  \prod_{a = 1}^r \left[ \frac{-\theta  (2 v_a - z_{\theta} - z^-)}{(v_a -z)(v_a -z^-)(v_a - z_{\theta})(v_a -z_{\theta}^-) }\right]    \\
&\times  \prod_{b = 1}^s \left[ \frac{1}{ (u_b -z)(u_b - z^-)}\right]   \prod_{c = 1}^t \left[ \frac{1}{ (w_c -z_\theta)(w_c - z^-_\theta)}\right]  \left[\frac{1}{\theta} + \frac{[\theta - 2] \partial_z G_\mu(z)}{2N} \right]   \\
\end{split}
\end{align}
with $z^{\pm} = z \pm N^{-1}$, $z^{\pm}_\theta = z^{\pm} + \theta/N$ and $z_\theta = z + \theta/N$. Equation (\ref{NekrasovOutput}) is the main output of the Nekrasov's equations that we will need in this paper, and the rest of the section is devoted to establishing it. The overall approach is as follows. We will take each of the two equations (\ref{S4TLNE1v1}) and (\ref{S4TLNE2v1}), divide them by a suitable factor and integrate over the contour $\Gamma$ in Definition \ref{CumDef}. The resulting expression will then be differentiated with respect to the $t$ variables and the latter will be set to $0$. Afterwards we will be able to rewrite the resulting expressions using the results in Section \ref{Section4.2}, Assumption 4 and (\ref{S4xiN2}). Ultimately the first Nekrasov's equation will lead us to (\ref{expandedNE1v2}) and the second one will lead us to (\ref{expandedNE2v2}). Equation (\ref{NekrasovOutput}) is then derived from the difference of the resulting two equations. We supply the details below. \\

Dividing both sides of (\ref{S4TLNE1v1}) by $2\pi\i \cdot S(z) \cdot A_1(z) \cdot B_2(z)$ and integrating over $\Gamma$ gives
\begin{align*}
\begin{split}
&\frac{1}{2\pi \i }\int_{\Gamma} \frac{dz R^1_N(Nz) }{ S(z) \cdot A_1(z) \cdot B_2(z)}=  \frac{1}{2\pi \i }\int_{\Gamma} \frac{dz \Phi^{-}_N(Nz) }{ S(z)}  \cdot \mathbb{E} \left[ \prod_{i = 1}^N\frac{Nz- \ell_i -\theta}{Nz - \ell_i} \right] \\
&+ \frac{1}{2\pi \i }\int_{\Gamma} \frac{dz \Phi^{+}_N(Nz) \cdot A_2(z) \cdot B_1(z)}{S(z) \cdot A_1(z)\cdot B_2(z) } \cdot   \mathbb{E} \left[ \prod_{i = 1}^{N-1}\frac{Nz- m_i + \theta - 1}{Nz - m_i - 1} \right]  \\
& + \frac{1}{2\pi \i }\int_{\Gamma} \frac{dz \Phi^{+}_N(Nz) \cdot B_1(z)}{S(z) \cdot A_1(z)} \cdot \mathbb{E} \left[ \Pi_1^{\theta}(Nz) \right] + \frac{1}{2\pi \i }\int_{\Gamma} \frac{dz V^1_N(z) }{ S(z) \cdot A_1(z) \cdot B_2(z)}  \times \\
\end{split}
\end{align*}
Notice that by Lemmas \ref{S3AnalRQ} and \ref{S3NonVanish} we have that $\Phi^+(z) + \Phi^-(z) - R_\mu(z) \neq 0$ and is analytic in a neighborhood of the region enclosed by $\Gamma$ and so by Cauchy's theorem the left side of the above expression vanishes. Furthermore, the integrand in the last term on the right has two simple poles in the interior of $\Gamma$ at $z = 0$ and $z = s_N N^{-1}$. The last two observations show
\begin{align}\label{S4ExpBRQ}
\begin{split}
0= &\frac{1}{2\pi \i }\int_{\Gamma} \frac{dz \Phi^{-}_N(Nz) }{ S(z)}  \cdot \mathbb{E} \left[ \prod_{i = 1}^N\frac{Nz- \ell_i -\theta}{Nz - \ell_i} \right]  + \frac{1}{2\pi \i }\int_{\Gamma} \frac{dz \Phi^{+}_N(Nz) \cdot A_2(z) \cdot B_1(z) }{S(z) \cdot A_1(z)\cdot B_2(z) }      \\
&\times   \mathbb{E} \left[ \prod_{i = 1}^{N-1}\frac{Nz- m_i + \theta - 1}{Nz - m_i - 1} \right]  + \frac{1}{2\pi \i }\int_{\Gamma} \frac{dz \Phi^{+}_N(Nz) \cdot B_1(z)}{S(z) \cdot A_1(z)} \cdot \mathbb{E} \left[ \Pi_1^{\theta}(Nz) \right] \\
&+   \frac{ \mathbb{E} \left[\xi_N^{1,t} \right]}{S(0)}   - \frac{ B_1(s_N/N)A_2(s_N/N)  \mathbb{E} \left[ \xi_N^{b,1}\right]  }{ S(s_N/N) A_1(s_N/N)B_2(s_N/N) }  - \frac{B_1(s_N/N)\mathbb{E} \left[ \xi_N^{m,1}\right]}{  S(s_N/N) A_1(s_N/N) }.
\end{split}
\end{align}

We next apply the operator 
$$\mathcal{D} : = [\partial_{t^1_1} - \partial_{t_1^2}] \cdots [\partial_{t_r^1} - \partial_{t_r^2}] \cdot \partial_{t_{r+1}^{1}} \cdots  \partial_{t_{r+s}^{1}} \cdot \partial_{t_{r+1}^{2}} \cdots  \partial_{t_{r+t}^{2}}  $$
 to both sides and set $t_a^1 = 0$ for $a = 1, \dots, m$ and $t_a^2 = 0$ for $a = 1, \dots ,n$. Notice that when we perform the differentiation to the right some of the derivatives could land on the products and some on the measure in the expectations. We will split the result of the differentiation based on subsets $A,B,C$. The set $A$ consists of indices $a$ in $\{1, \dots, r\}$ such that $[\partial_{t_a^1} - \partial_{t_a^2}] $ differentiates the expectation. Similarly, $B$ denotes the set of indices $b$ in $\{1, \dots, s\}$ such that $\partial_{t_{r+b}^1}$ differentiates the expectation and $C$ the set of indices in $\{1, \dots, t\}$ such that $\partial_{t^2_{r+c}}$ differentiates the expectation. The result of applying $\mathcal{D}$ is then
\begin{align}\label{S4ExpB}
\begin{split} 
& \frac{1}{2\pi \i}\int_{\Gamma} dz  \frac{\Phi_N^-(Nz) }{S(z)  N^{r/2}}   M \left(\prod_{i = 1}^N\frac{Nz- \ell_i -\theta}{Nz - \ell_i} ; \llbracket 1, r \rrbracket, \llbracket 1 , s \rrbracket , \llbracket 1, t \rrbracket \right) + \sum_{ B \subseteq \llbracket 1, s \rrbracket} \sum_{C \subseteq \llbracket 1 ,t \rrbracket}  \frac{\Phi_N^+(Nz) }{S(z) N^{r/2}}  \\
& \times \prod_{b \in B^c} \hspace{-1mm} \left[ \frac{N^{-1} }{(u_b-z)(u_b - z^-) } \right] \hspace{-1mm} \prod_{c \in C^c} \hspace{-1mm} \left[ \frac{N^{-1}}{(w_c - z)(w_c - z^-)}\right] \hspace{-1mm} M  \hspace{-1mm}\left( \prod_{i = 1}^{N-1}\frac{Nz- m_i +\theta - 1}{Nz - m_i - 1} ; \llbracket 1, r \rrbracket, B, C \hspace{-1mm} \right) \\
&+    \sum_{A \subseteq \llbracket 1, r \rrbracket, B \subseteq \llbracket 1, s \rrbracket}\frac{ \Phi_N^+(Nz)}{ S(z) N^{|A|/2}}  \prod_{a \in A^c}  \left[ \frac{N^{-1}}{ (v_a -z)(v_a - z^- )}\right] \prod_{b \in B^c}  \left[ \frac{ N^{-1}}{ (u_b -z)(u_b - z^-)}\right] \\
&\times  M\left( \Pi_1^{\theta}(Nz); A, B, \llbracket 1, t \rrbracket  \right) = - \tilde{V}^1_N,
\end{split}
\end{align}
where 
\begin{align*}
\begin{split} 
&\tilde{V}_N^1 =  \frac{M \left(\xi_N^{1,t} ; \llbracket 1, r \rrbracket ,  \llbracket 1, s \rrbracket, \llbracket 1, t \rrbracket \right)}{S(0)  N^{r/2}}   - \sum_{ B \subseteq \llbracket 1, s \rrbracket} \sum_{C \subseteq \llbracket 1 ,t \rrbracket} \prod_{b \in B^c} \hspace{-1mm} \left[ \frac{N^{-1} }{(u_b-s_N/N)(u_b - s_N/N + 1/N) } \right]  \\
& \times  \prod_{c \in C^c} \hspace{-1mm} \left[ \frac{N^{-1}}{(w_c - s_N/N)(w_c - s_N/N + 1/N)}\right] \cdot \frac{ M\left(\xi_N^{b,1} ; \llbracket 1, r\rrbracket, B, C\right) }{ S(s_N/N) N^{r/2}} \\
&-  \sum_{A \subseteq \llbracket 1, r \rrbracket, B \subseteq \llbracket 1, s \rrbracket}\prod_{a \in A^c}  \left[ \frac{N^{-1}}{ (v_a -s_N/N)(v_a - s_N/N +1/N )}\right] \\
& \times \prod_{b \in B^c}  \left[ \frac{ N^{-1}}{ (u_b -s_N/N)(u_b -s_N/N + 1/N)}\right]\frac{ M\left(\xi_N^{b,1} ; A, B, \llbracket 1, t \rrbracket \right) }{ S(s_N/N) N^{|A|/2}}.
\end{split}
\end{align*}
We mention that the term $\tilde{V}_N^1$ comes from the action of $\mathcal{D}$ on the last line of (\ref{S4ExpBRQ}).

Recall from (\ref{S4xiN2}) that $\xi^{b/m/t,1}_N= O(1)$ $\mathbb{P}_N$-almost surely. Combining the latter with the fact that $X_N^t(v_a)$ are $O(N)$ almost surely,  and Assumption 4, we see that $\tilde{V}_N^1 = O(N^{r+s+t} \exp(-cN^a))$. Combining the latter with  (\ref{AEP1v2}, \ref{AEP1v3}, \ref{AEP1v4}) we may rewrite (\ref{S4ExpB}) as
\begin{equation*}
\begin{split} 
&0 = \frac{1}{2\pi \i}\int_{\Gamma} dz  \frac{\Phi_N^-(Nz) e^{-\theta G_\mu(z)} }{S(z)  N^{r/2}}   M \left(1 - \frac{\theta X_N^t(z)}{N} + \frac{\theta^2 \partial_z G_\mu(z)}{2N} ; \llbracket 1, r \rrbracket, \llbracket 1 , s \rrbracket , \llbracket 1, t \rrbracket \right)  \\
& + \sum_{ B \subseteq \llbracket 1, s \rrbracket} \sum_{C \subseteq \llbracket 1 ,t \rrbracket}  \frac{\Phi_N^+(Nz) e^{\theta G_\mu(z)} }{S(z) N^{r/2}}  \prod_{b \in B^c} \hspace{-1mm} \left[ \frac{N^{-1} }{(u_b-z)(u_b - z^-) } \right] \hspace{-1mm} \prod_{c \in C^c} \hspace{-1mm} \left[ \frac{N^{-1}}{(w_c - z)(w_c - z^-)}\right] \\
&\times   M  \left(1 + \frac{\theta X^b_N(z)}{N} + \frac{[\theta^2 - 2\theta] \partial_z G_\mu(z) }{2N}; \llbracket 1, r \rrbracket, B, C \hspace{-1mm} \right) \\
&+    \sum_{A \subseteq \llbracket 1, r \rrbracket, B \subseteq \llbracket 1, s \rrbracket}\frac{\theta \Phi_N^+(Nz)  }{ S(z) N^{|A|/2}}  \prod_{a \in A^c}  \left[ \frac{N^{-1}}{ (v_a -z)(v_a - z^- )}\right] \prod_{b \in B^c}  \left[ \frac{ N^{-1}}{ (u_b -z)(u_b - z^-)}\right] \\
&\times  M\left( \frac{{\bf 1 } \{ \theta \neq 1\}}{1-\theta} + \frac{\Delta X_N(z)}{N^{3/2}} - \frac{\theta \partial_z G_\mu(z)}{N}; A, B, \llbracket 1, t \rrbracket  \right)  + \sum_{A \subseteq \llbracket 1,r \rrbracket}\sum_{B \subseteq \llbracket 1,s \rrbracket}\sum_{C \subseteq \llbracket 1,t \rrbracket}   \frac{ M \left(  \zeta^{\Gamma}_N(z); A, B,C\right)}{N^{r - |A|/2 + 2}},
\end{split}
\end{equation*}
where we recall that $\zeta^{\Gamma}_N(z)$ stands for a generic random analytic function such that for each $k \geq 1$ we have $\mathbb{E} [ |\zeta^{\Gamma}_N(z)|^k] = O(1)$ uniformly over $z \in \Gamma$ and $v^i_j \in K$. 

We can now use the linearity of cumulants together with the fact that the joint cumulant of any non-empty collection of  bounded random variables and a constant is zero. For example this allows us to remove the term ${\bf 1}\{\theta \neq 1\}$ above if $|A| + |B| + t > 0$ and otherwise we see that we can remove it as it integrates to $0$ by Cauchy's theorem. In addition, we know from Assumption 3 that uniformly as $z$ varies over $\Gamma$ and $v \in K$ we have
$$\frac{1}{ (v -z)(v - z +N^{-1})} = \frac{1}{ (v -z)^2} + O(N^{-1}) \mbox{ and } \Phi^{\pm}_N(Nz) = \Phi^{\pm}(z) + O(N^{-1}).$$
Applying the last few statements and (\ref{derBound}) we get
\begin{align}\label{expandedNE1v2}
\begin{split} 
& \frac{1}{2\pi \i}\int_{\Gamma} dz  \frac{\Phi_N^-(Nz) e^{-\theta G_\mu(z)} {\bf 1} \{ r + s + t = 0\} }{S(z)  }      \left[1 + \frac{\theta^2 \partial_z G_\mu(z)}{2N} \right] +   \frac{\Phi_N^+(Nz) e^{\theta G_\mu(z)} {\bf 1} \{ r = 0\} }{S(z) } \\
&  \times \prod_{b = 1}^s \left[ \frac{N^{-1} }{(u_b-z)(u_b - z^-) } \right] \prod_{c  = 1}^t  \left[ \frac{N^{-1}}{(w_c - z)(w_c - z^-)}\right] \cdot \left[1 + \frac{[\theta^2 - 2\theta] \partial_z G_\mu(z) }{2N}\right]  \\
&+   \frac{\theta \Phi_N^+(Nz)  {\bf 1} \{ t = 0\} }{ S(z)}  \prod_{a = 1}^r  \left[ \frac{N^{-1}}{ (v_a -z)(v_a - z^- )}\right] \prod_{b = 1}^s  \left[ \frac{ N^{-1}}{ (u_b -z)(u_b - z^-)}\right]   \left[- \frac{\theta \partial_z G_\mu(z)}{N}\right] \\
& - \frac{\theta \Phi^-(z) e^{-\theta G_\mu(z)} M \left( X_N^t(z) ; \llbracket 1, r \rrbracket, \llbracket 1 , s \rrbracket , \llbracket 1, t \rrbracket \right)  }{S(z)  N^{r/2 + 1}}   \\
& + \sum_{ B \subseteq \llbracket 1, s \rrbracket} \sum_{C \subseteq \llbracket 1 ,t \rrbracket}  \frac{\theta \Phi^+(z) e^{\theta G_\mu(z)}  M  \left( X^b_N(z); \llbracket 1, r \rrbracket, B, C  \right) }{S(z) N^{r/2 + 1}N^{s - |B|} N^{t - |C|} }  \prod_{b \in B^c} \hspace{-1mm} \frac{1 }{(u_b-z)^2 } \hspace{-1mm} \prod_{c \in C^c} \hspace{-1mm} \frac{1}{(w_c - z)^2} \\
&+    \sum_{A \subseteq \llbracket 1, r \rrbracket, B \subseteq \llbracket 1, s \rrbracket}\frac{ \theta \Phi^+(z)  M\left( \Delta X_N(z) ; A, B, \llbracket 1, t \rrbracket  \right)   }{ S(z)  N^{r - |A|/2 + 3/2} N^{s - |B|}}  \prod_{a \in A^c}   \frac{1}{ (v_a -z)^2}\prod_{b \in B^c}  \frac{ 1}{ (u_b -z)^2} \\
&  + \sum_{A \subseteq \llbracket 1,r \rrbracket}\sum_{B \subseteq \llbracket 1,s \rrbracket}\sum_{C \subseteq \llbracket 1,t \rrbracket}   \frac{ M \left(  \zeta^{\Gamma}_N(z); A, B,C\right)}{N^{r - |A|/2 + 2}} = 0.
\end{split}
\end{align}
We mention that the way we use (\ref{derBound}) is when replacing $z^-$ with $z$ in the products above, and $\Phi_N^{\pm}(Nz)$ with $\Phi^{\pm}(z)$. Indeed, this replacement produces an error, which is 
$$\frac{O(N^{-1}) }{N^{r/2 + 1}}  M \left(  X_N^t(z) ;  \llbracket 1, r \rrbracket, \llbracket 1 , s \rrbracket , \llbracket 1, t \rrbracket \right)  \mbox{, } \frac{O(N^{-1}) }{N^{r/2 + 1} N^{s - |B|} N^{t - |C|}}   M\left(  X^b_N(z) ;  \llbracket 1,r \rrbracket,B, C\right) \mbox{ or }$$
$$  \frac{O(N^{-1}) }{N^{r -|A|/2 +3/2} N^{s - |B|}}   M\left( \Delta X_N(z) ;  A,B, \llbracket 1, t \rrbracket \right) $$
in the fourth, fifth and sixth lines of (\ref{expandedNE1v2}), and the latter can be absorbed into the sum on the last line of (\ref{expandedNE1v2}). Equation (\ref{expandedNE1v2}) is the expression we need from the first Nekrasov's equation.\\

We next divide both sides of (\ref{S4TLNE2v1} by $2\pi \i \cdot S(z) \cdot A_1(z) \cdot B_3(z)$ and integrate over $\Gamma$ to get
\begin{align*}
\begin{split}
&\frac{1}{2\pi \i }\int_{\Gamma} \frac{dz R^2_N(Nz) }{ S(z) \cdot A_1(z) \cdot B_3(z)}= \frac{1}{2\pi \i }\int_{\Gamma} \frac{dz \Phi^{+}_N(Nz) \cdot A_3(z)  \cdot B_1(z) }{S(z) \cdot A_1(z) \cdot B_3(z) }  \cdot \mathbb{E} \left[ \prod_{i = 1}^N\frac{Nz - \ell_i + \theta - 1}{Nz - \ell_i - 1} \right] \\
&+\frac{1}{2\pi \i }\int_{\Gamma} \frac{dz \Phi^{-}_N(Nz)}{ S(z)} \cdot   \mathbb{E} \left[\prod_{i = 1}^{N-1}\frac{Nz - m_i}{Nz - m_i + \theta} \right] + \frac{1}{2\pi \i }\int_{\Gamma} \frac{dz \Phi^{-}_N(Nz) \cdot A_3(z) }{S(z) \cdot B_3(z)  }   \mathbb{E} \left[ \Pi^\theta_2(Nz) \right]  \\
&+  \frac{1}{2\pi \i }\int_{\Gamma} \frac{dz V^2_N(z)}{ S(z) \cdot A_1(z) \cdot B_3(z)}.
\end{split}
\end{align*}

As before the left side of the above equation vanishes by Cauchy's theorem and the last term on the right has simple poles at $z =0$ and $z = s_N \cdot N^{-1}$. Thus we can rewrite the above as
\begin{align*}
\begin{split}
& 0 = \frac{1}{2\pi \i }\int_{\Gamma} \frac{dz \Phi^{+}_N(Nz) \cdot A_3(z) \cdot B_1(z)  }{S(z) \cdot A_1(z) \cdot B_3(z) }  \cdot \mathbb{E} \left[ \prod_{i = 1}^N\frac{Nz - \ell_i + \theta - 1}{Nz - \ell_i - 1} \right] \\
&+\frac{1}{2\pi \i }\int_{\Gamma} \frac{dz \Phi^{-}_N(Nz) }{ S(z)} \cdot   \mathbb{E} \left[\prod_{i = 1}^{N-1}\frac{Nz - m_i}{Nz - m_i + \theta} \right] + \frac{1}{2\pi \i }\int_{\Gamma} \frac{dz \Phi^{-}_N(Nz) \cdot A_3(z) }{S(z) \cdot B_3(z) }   \mathbb{E} \left[ \Pi^\theta_2(Nz) \right]  \\
&+  \frac{\mathbb{E} \left[\xi_N^{b,2} \right]}{S(0)} + \frac{A_3(0)\mathbb{E} \left[\xi_N^{m,2} \right] }{ S(0)B_3(0) } - \frac{ B_1(s_N/N) A_3(s_N/N) \mathbb{E}\left[\xi_N^{t,2}\right]}{S(s_N/N)   A_1(s_N/N) B_3(s_N/N) }.
\end{split}
\end{align*}
Then we can apply $\mathcal{D}$ to both sides and set $t_a^1 = 0$ for $a = 1, \dots, m$ and $t_a^2 = 0$ for $a = 1, \dots ,n$. This gives
\begin{align}\label{S4ExpB2}
\begin{split} 
& \frac{1}{2\pi \i}\int_{\Gamma} dz  \sum_{A \subseteq \llbracket 1,r \rrbracket} \sum_{B \subseteq \llbracket 1, s \rrbracket}  \sum_{C \subseteq \llbracket 1, t \rrbracket}   \frac{\Phi_N^+(Nz) }{S(z)  N^{|A|/2}}  \prod_{a \in A^c} \left[ \frac{-\theta N^{-2} (2 v_a - z_{\theta} - z^-)}{(v_a -z)(v_a -z^-)(v_a - z_{\theta})(v_a -z_{\theta}^-) }\right]    \\
&\times  \prod_{b \in B^c} \left[ \frac{N^{-1}}{ (u_b -z)(u_b - z^-)}\right]   \prod_{c \in C^c} \left[ \frac{N^{-1}}{ (w_c -z_\theta)(w_c - z^-_\theta)}\right]  M \left(\prod_{i = 1}^N\frac{Nz - \ell_i + \theta - 1}{Nz - \ell_i - 1} ; A, B, C \right)   \\
&  + \frac{\Phi^{-}_N(Nz) }{ S(z)N^{r/2}}    M \left(\prod_{i = 1}^{N-1}\frac{Nz - m_i}{Nz - m_i + \theta}; \llbracket 1, r \rrbracket, \llbracket 1, s \rrbracket, \llbracket 1, t \rrbracket \right) \\
&+ \sum_{A \subseteq \llbracket 1,r \rrbracket}  \sum_{C \subseteq \llbracket 1, t \rrbracket}  \frac{\Phi_N^-(Nz) }{S(z) N^{|A|/2}} \prod_{a \in A^c} \left[ \frac{-N^{-1}}{  (v_a -z_\theta)(v_a - z^-_\theta)}\right]  \prod_{c \in C^c} \left[ \frac{N^{-1}}{ (w_c -z_\theta) (w_c - z^-_\theta)}\right] \\
&  \times  M\left( \Pi^{\theta}_2(Nz) ; A, \llbracket 1, s \rrbracket , C\right) =  - \tilde{V}_N^2,
\end{split}
\end{align}
where 
\begin{align*}
\begin{split} 
& \tilde{V}_N^2  =  \frac{M (\xi_N^{b,2}; \llbracket 1, r \rrbracket, \llbracket 1, s \rrbracket, \llbracket 1, t \rrbracket )}{S(0)N^{r/2}}+ \sum_{A \subseteq \llbracket 1,r \rrbracket}  \sum_{C \subseteq \llbracket 1, t \rrbracket}  \prod_{a \in A^c} \left[ \frac{-N^{-1}}{  (v_a -\theta/N)(v_a - \theta/N + 1/N )}\right]  \\
&  \times \prod_{c \in C^c} \left[ \frac{N^{-1}}{ (w_c -\theta/N) (w_c - \theta/N + 1/N)}\right]  \frac{M\left( \xi_N^{m,2} ; A, \llbracket 1, s \rrbracket , C\right) }{S(0) N^{|A|/2}}  - \sum_{A \subseteq \llbracket 1,r \rrbracket}\sum_{B \subseteq \llbracket 1, s \rrbracket}  \sum_{C \subseteq \llbracket 1, t \rrbracket} \frac{M(\xi_N^{t,2}; A,B,C) }{S(s_N/N)  N^{|A|/2}}   \\
&  \times   \prod_{a \in A^c} \left[ \frac{-\theta N^{-2} (2 v_a - 2s_N/N  - (\theta - 1)/N)}{(v_a -s_N/N)(v_a -s_N/N + 1/N)(v_a - s_N/N - \theta/N)(v_a -s_N/N - (\theta- 1)/N) }\right]    \\
&\times  \prod_{b \in B^c} \left[ \frac{N^{-1}}{ (u_b -s_N/N)(u_b - s_N/N + 1/N)}\right]   \prod_{c \in C^c} \left[ \frac{N^{-1}}{ (w_c - s_N/N - \theta/N)(w_c - s_N/N - (\theta - 1)/N)}\right].
\end{split}
\end{align*}
Arguing as before, we have from (\ref{S4xiN2}) and Assumption 4 that $\tilde{V}_N^2 =  O(N^{r+s+t} \exp(-cN^a))$. Combining the latter with  (\ref{AEP1v2}, \ref{AEP1v3}, \ref{AEP1v4}) we can rewrite (\ref{S4ExpB2}) as 
\begin{align*}
\begin{split} 
& \frac{1}{2\pi \i}\int_{\Gamma} dz  \sum_{A \subseteq \llbracket 1,r \rrbracket} \sum_{B \subseteq \llbracket 1, s \rrbracket}  \sum_{C \subseteq \llbracket 1, t \rrbracket}   \frac{\Phi_N^+(Nz) e^{\theta G_\mu(z)} }{S(z)  N^{|A|/2}}  \prod_{a \in A^c} \left[ \frac{-\theta N^{-2} (2 v_a - z_{\theta} - z^-)}{(v_a -z)(v_a -z^-)(v_a - z_{\theta})(v_a -z_{\theta}^-) }\right]    \\
&\times  \prod_{b \in B^c} \left[ \frac{N^{-1}}{ (u_b -z)(u_b - z^-)}\right]   \prod_{c \in C^c} \left[ \frac{N^{-1}}{ (w_c -z_\theta)(w_c - z^-_\theta)}\right]  M \left(1 + \frac{\theta X_N^t(z)}{N} + \frac{[\theta^2 - 2\theta] \partial_z G_\mu(z)}{2N}  ; A, B, C \right)   \\
& + \frac{\Phi^{-}_N(Nz)e^{-\theta G_\mu(z)}}{ S(z)N^{r/2}}    M \left(1 - \frac{\theta X_N^b(z)}{N} - \frac{\theta^2 \partial_z G_\mu(z)}{2N}; \llbracket 1, r \rrbracket, \llbracket 1, s \rrbracket, \llbracket 1, t \rrbracket \right) \\
&+ \sum_{A \subseteq \llbracket 1,r \rrbracket}  \sum_{C \subseteq \llbracket 1, t \rrbracket}  \frac{\theta \Phi_N^-(Nz) }{S(z) N^{|A|/2}} \prod_{a \in A^c} \left[ \frac{-N^{-1}}{  (v_a -z_\theta)(v_a - z^-_\theta)}\right]  \prod_{c \in C^c} \left[ \frac{N^{-1}}{ (w_c -z_\theta) (w_c - z^-_\theta)}\right] \\
&  \times  M\left( \frac{{\bf 1}\{\theta \neq 1\}}{1-\theta} - \frac{ \Delta X_N(z)}{N^{3/2}} ; A, \llbracket 1, s \rrbracket , C\right)  +  \sum_{A \subseteq \llbracket 1,r \rrbracket}\sum_{B \subseteq \llbracket 1,s \rrbracket}\sum_{C \subseteq \llbracket 1,t \rrbracket}   \frac{ M \left(  \zeta^{\Gamma}_N(z); A, B,C\right)}{N^{r - |A|/2 + 2}}= 0.
\end{split}
\end{align*}
As before, we may remove constant terms from second and higher order cumulants, the ${\bf 1}\{\theta \neq 1\}$ term, and also simplify the above expression to get
\begin{align}\label{expandedNE2v2}
\begin{split} 
& \frac{1}{2\pi \i}\int_{\Gamma} dz     \frac{\Phi_N^+(Nz)e^{\theta G_\mu(z)} }{S(z) }  \prod_{a = 1}^r \left[ \frac{-\theta N^{-2} (2 v_a - z_{\theta} - z^-)}{(v_a -z)(v_a -z^-)(v_a - z_{\theta})(v_a -z_{\theta}^-) }\right]    \\
&\times  \prod_{b = 1}^s \left[ \frac{N^{-1}}{ (u_b -z)(u_b - z^-)}\right]   \prod_{c = 1}^t \left[ \frac{N^{-1}}{ (w_c -z_\theta)(w_c - z^-_\theta)}\right]  \left[1 + \frac{[\theta^2 - 2\theta] \partial_z G_\mu(z)}{2N} \right]   \\
& + \frac{\Phi_N^{-}(Nz)  e^{-\theta G_\mu(z)} {\bf 1}\{ r + s + t = 0\}}{ S(z)}    \left[1  - \frac{\theta^2 \partial_z G_\mu(z)}{2N}\right]  \\
&+ \sum_{A \subseteq \llbracket 1,r \rrbracket} \sum_{B \subseteq \llbracket 1, s \rrbracket}  \sum_{C \subseteq \llbracket 1, t \rrbracket}   \frac{\theta \Phi^+(z)e^{\theta G_\mu(z)} M \left(X_N^t(z)   ; A, B, C \right)   }{S(z)  N^{(|A|+2)/2} N^{2r - 2|A|} N^{s - |B|} N^{t - |C|}}  \prod_{a \in A^c} \left[ \frac{-2\theta }{(v_a -z)^3}\right]    \\
&\times  \prod_{b \in B^c} \frac{1}{ (u_b -z)^2} \prod_{c \in C^c}  \frac{1}{ (w_c -z)^2}  - \frac{\theta \Phi^{-}(z)  e^{-\theta G_\mu(z)} M \left(  X_N^b(z) ; \llbracket 1, r \rrbracket, \llbracket 1, s \rrbracket, \llbracket 1, t \rrbracket \right)}{ S(z)N^{r/2 + 1} }     \\
&- \sum_{A \subseteq \llbracket 1,r \rrbracket}  \sum_{C \subseteq \llbracket 1, t \rrbracket}  \frac{\theta \Phi^-(z) M\left( \Delta X_N(z) ; A, \llbracket 1, s \rrbracket , C\right)  }{S(z) N^{|A|/2 + 3/2} N^{r - |A|} N^{t - |C|}} \prod_{a \in A^c} \left[ \frac{-1}{  (v_a -z)^2}\right]   \prod_{c \in C^c} \frac{1}{ (w_c -z)^2} \\
&+  \sum_{A \subseteq \llbracket 1,r \rrbracket}\sum_{B \subseteq \llbracket 1,s \rrbracket}\sum_{C \subseteq \llbracket 1,t \rrbracket}   \frac{ M \left(  \zeta^{\Gamma}_N(z); A, B,C\right)}{N^{r - |A|/2 + 2}}= 0.
\end{split}
\end{align}
Equation (\ref{expandedNE2v2}) is the expression we need from the second Nekrasov's equation.\\

We next subtract (\ref{expandedNE2v2}) from (\ref{expandedNE1v2}) to get
\begin{align}\label{expandedDiffNE1}
\begin{split} 
&  \frac{1}{2\pi \i}\int_{\Gamma} dz \theta N^{-(r+3)/2} W_N(z)  \\
&+ \sum_{ B \subseteq \llbracket 1, s \rrbracket} \sum_{C \subseteq \llbracket 1 ,t \rrbracket}  \frac{\theta \Phi^+(z) e^{\theta G_\mu(z)}  M  \left( X^b_N(z); \llbracket 1, r \rrbracket, B, C  \right) }{S(z) N^{r/2 + 1}N^{s - |B|} N^{t - |C|} }  \prod_{b \in B^c} \hspace{-1mm} \frac{1 }{(u_b-z)^2 } \hspace{-1mm} \prod_{c \in C^c} \hspace{-1mm} \frac{1}{(w_c - z)^2} \\
&- \sum_{A \subseteq \llbracket 1,r \rrbracket} \sum_{B \subseteq \llbracket 1, s \rrbracket}  \sum_{C \subseteq \llbracket 1, t \rrbracket}   \frac{\theta \Phi^+(z)e^{\theta G_\mu(z)} M \left(X_N^t(z)   ; A, B, C \right)   }{S(z)  N^{(|A|+2)/2} N^{2r - 2|A|} N^{s - |B|} N^{t - |C|}}  \prod_{a \in A^c} \left[ \frac{-2\theta }{(v_a -z)^3}\right]    \\
&\times  \prod_{b \in B^c} \frac{1}{ (u_b -z)^2} \prod_{c \in C^c}  \frac{1}{ (w_c -z)^2} \\
&+ \sum_{A \subseteq \llbracket 1,r \rrbracket}  \sum_{C \subseteq \llbracket 1, t \rrbracket}  \frac{\theta \Phi^-(z) M\left( \Delta X_N(z) ; A, \llbracket 1, s \rrbracket , C\right)  }{S(z) N^{|A|/2 + 3/2} N^{r - |A|} N^{t - |C|}} \prod_{a \in A^c} \left[ \frac{-1}{  (v_a -z)^2}\right]   \prod_{c \in C^c} \frac{1}{ (w_c -z)^2} \\
&+   \sum_{A \subseteq \llbracket 1, r \rrbracket, B \subseteq \llbracket 1, s \rrbracket}\frac{ \theta \Phi^+(z)  M\left( \Delta X_N(z) ; A, B, \llbracket 1, t \rrbracket  \right)   }{ S(z)  N^{r - |A|/2 + 3/2} N^{s - |B|}}  \prod_{a \in A^c}   \frac{1}{ (v_a -z)^2}\prod_{b \in B^c}  \frac{ 1}{ (u_b -z)^2} \\
& - \frac{\theta \Phi^-(z) e^{-\theta G_\mu(z)} M \left( \Delta X_N(z) ; \llbracket 1, r \rrbracket, \llbracket 1 , s \rrbracket , \llbracket 1, t \rrbracket \right)  }{S(z)  N^{r/2 + 3/2}}  \\
& + \sum_{A \subseteq \llbracket 1,r \rrbracket}\sum_{B \subseteq \llbracket 1,s \rrbracket}\sum_{C \subseteq \llbracket 1,t \rrbracket}   \frac{ M \left(  \zeta^{\Gamma}_N(z); A, B,C\right)}{N^{r - |A|/2 + 2}} = 0.
\end{split}
\end{align}
Finally, we can extract the terms corresponding to $B = \llbracket 1, s\rrbracket$, $C = \llbracket 1, t \rrbracket$ from the first double sum in (\ref{expandedDiffNE1}), to $A = \llbracket 1,r \rrbracket$, $B = \llbracket 1, s\rrbracket$, $C = \llbracket 1, t \rrbracket$ from the first triple sum in (\ref{expandedDiffNE1}), to $A = \llbracket 1,r \rrbracket$, $C = \llbracket 1, t \rrbracket$ from the second double sum in (\ref{expandedDiffNE1}), to $A = \llbracket 1,r \rrbracket$, $B = \llbracket 1, s\rrbracket$ from the third double sum in (\ref{expandedDiffNE1}), and combine them with the next to last line of (\ref{expandedDiffNE1}) to form
$$( \Phi^-(z) + \Phi^+(z) - R_\mu(z))\frac{\theta M \left( \Delta X_N(z) ; \llbracket 1, r \rrbracket, \llbracket 1 , s \rrbracket , \llbracket 1, t \rrbracket \right)  }{S(z)  N^{r/2 + 3/2}}  = \frac{\theta M \left( \Delta X_N(z) ; \llbracket 1, r \rrbracket, \llbracket 1 , s \rrbracket , \llbracket 1, t \rrbracket \right)  }{ (z-v)N^{r/2 + 3/2}},$$
where we used that $R_\mu(z) = \Phi^-(z) e^{-\theta G_\mu(z)} + \Phi^+(z) e^{\theta G_\mu(z)}$ from (\ref{QRmu}) and $S(z) = (z-v) \cdot (\Phi^+(z) + \Phi^-(z) - R_\mu(z))$. The latter expression is analytic outside of $\Gamma$ and decays at least like $|z|^{-2}$ as $|z| \rightarrow \infty$, which means that we can compute that integral as (minus) the residue at $z = v$ as there is no residue at $\infty$. After performing this computation and multiplying everything by $\theta^{-1} N^{(r+3)/2}$ we arrive at (\ref{NekrasovOutput}).

%

\section{Central limit theorem} \label{Section5}
In this section we formulate the main technical result of the paper as Theorem \ref{TGField} and prove it in Section \ref{Section5.3} after establishing some necessary moment bounds in Section \ref{Section5.2}. We continue with the notation from Sections 3 and 4.

%
\subsection{Main technical result} \label{Section5.1} In this section we isolate the main technical result of the paper as Theorem \ref{TGField} and deduce Theorem \ref{CLTfun_main}, recalled as Corollary \ref{CLTfun} below, from it.

\begin{theorem} \label{TGField}
Suppose Assumptions 1-5 hold and let $U : = \mathbb{C} \setminus [0, \lM + \theta]$. For $z \in U$ we define the random field $(Y^t_N(z), Y^b_N(z), \Delta Y_N(z))$ through 
\begin{equation}\label{DefDiscreteField}
Y^t_N(z) = G^t_N(z) - \mathbb{E} \left[G^t_N(z)  \right], Y^b_N(z) = G^b_N(z) - \mathbb{E} \left[ G^b_N(z)\right], \Delta Y_N(z) = N^{1/2} [Y^t_N(z) - Y^b_N(z) ].
\end{equation}
Then as $N \rightarrow \infty$ the random field $(Y^t_N(z), Y^b_N(z), \Delta Y_N(z)), z \in U$, converges  (in the sense of joint moments, uniformly in $z$ in compact subsets of $U$) to a centered complex Gaussian random field, whose covariance structure is given by
\begin{align}\label{eq:GField}
\begin{split}
&\lim_{N \rightarrow \infty} Cov(Y_N^t(z_1), \Delta Y_N(z_2) ) =\lim_{N \rightarrow \infty} Cov(Y_N^b(z_1), \Delta Y_N(z_2) ) = 0, \hspace{2mm} \\
& \mathcal C_{\theta, \mu}(z_1, z_2)  = \lim_{N\rightarrow \infty} Cov(Y_N^{\zeta_1}(z_1), Y^{\zeta_2}_N(z_2))  \mbox{ for $\zeta_1, \zeta_2 \in \{t , b\}$}  \mbox{ and }\\
\end{split}
\end{align}
\begin{equation}\label{eq:GFieldv2}
\begin{split}
&\lim_{N \rightarrow \infty} Cov(\Delta Y_N(z_1), \Delta Y_N(z_2) ) = \Delta \mathcal{C}_{\theta, \mu} (z_1, z_2), \mbox{ where }
\end{split}
\end{equation}
\begin{align}\label{eq:var}
\begin{split}
& \mathcal{C}_{\theta, \mu}(z_1, z_2) = -\frac{\theta^{-1}}{2(z_1-z_2)^2} \left(1 - \frac{(z_1 - \alpha)(z_2- \beta) + (z_2 - \alpha )(z_1- \beta)}{2\sqrt{(z_1 -\alpha )(z_1- \beta)}\sqrt{(z_2 - \alpha)(z_2- \beta)}} \right) \mbox{ and }\\
& \Delta \mathcal{C}_{\theta, \mu} (z_1, z_2) = \frac{1}{2\pi \i}\int_{\Gamma}  \frac{dz}{e^{\theta G_\mu(z)} - 1} \cdot \left[ - \frac{1}{(z-z_2)^2(z-z_1)^2}\right],
\end{split}
\end{align}
where $\alpha$ and $\beta$ are as in Assumption 5, $G_\mu(z)$ is as in (\ref{GmuDef}) and $\Gamma$ is a positively oriented contour that contains the segment $[0, \lM + \theta]$, $ \Gamma \subset \mathcal{M}$ as in Assumption 3 and excludes the points $z_1, z_2$.
\end{theorem}
\begin{remark}
Since $\overline{G^{t/b}_N(z)} = G^{t/b}_N(\overline{z})$, we can use (\ref{eq:var}) to completely characterize the asymptotic covariance of the field $(Y^t_N(z), Y^b_N(z), \Delta Y_N(z))$.  In particular, convergence of the joint moments in Theorem \ref{TGField} implies finite-dimensional convergence.
\end{remark}

\begin{remark}
It is worth pointing out that $\mathcal{C}_{\theta, \mu}$ depends on the equilibrium measure $\mu$ only through the quantities $\alpha, \beta$ , while $\Delta \mathcal{C}_{\theta, \mu}$ depends on its Stieltjes transform.
\end{remark}

\begin{corollary}\label{CLTfun}
Assume the same notation as in Theorem \ref{TGField}. For $n \geq 1$ let $f_1, \dots, f_n$ be analytic functions in a complex neighborhood $\mathcal{M}_1$ containing $[0, \lM + \theta]$, whose restriction to $\mathbb{R} \cap \mathcal{M}_1$ is real-valued. Define 
$$\mathcal L^t_{f_k}= \sum_{i = 1}^N f_k(\ell_i/N) - \mathbb{E}  \left[ \sum_{i = 1}^Nf_k(\ell_i/N)  \right] \mbox{, } \mathcal L^b_{f_k}= \sum_{i = 1}^{N-1} f_k(m_i/N) - \mathbb{E}  \left[ \sum_{i = 1}^{N-1} f_k(m_i/N)  \right] \mbox{ and }$$
$$\mathcal L^m_{f_k}=N^{1/2} \cdot \left[\mathcal L^t_{f_k} -  \mathcal L^b_{f_k}\right] \mbox{ for $k = 1, \dots, n$} .$$
Then the random variables $\{\mathcal{L}^m_{f_i} \}_{i = 1}^n$, $\{\mathcal{L}^t_{f_i} \}_{i = 1}^n$, $\{\mathcal{L}^b_{f_i} \}_{i = 1}^n$ converge jointly in the sense of moments to a $3n$-dimensional centered Gaussian vector $\xi = (\xi^m_1,\dots, \xi^m_n, \xi^t_1, \dots, \xi^t_n, \xi^b_n, \dots, \xi^b_n)$ with covariance
\begin{equation*}
\begin{split}
& Cov(\xi^t_i, \xi^m_j )  =  Cov(\xi^b_i, \xi^m_j )   = 0, \hspace{2mm} \\
&Cov(\xi^{t}_i, \xi^{t}_j) =Cov(\xi^{b}_i, \xi^{b}_j)   = Cov(\xi^{t}_i, \xi^{b}_j)  = \frac{1}{(2\pi \i )^2} \oint_{\Gamma} \oint_{\Gamma} f_i(s) f_j(t) \mathcal{C}_{\theta, \mu}(s,t)dsdt,  \\
&Cov(\xi^{m}_i, \xi^m_j) = \frac{1}{(2\pi \i )^2} \oint_{\Gamma} \oint_{\Gamma} f_i(s) f_j(t) \Delta \mathcal{C}_{\theta, \mu}(s,t)dsdt, \\
\end{split}
\end{equation*}
for all $i,j = 1, \dots , n$, where $\mathcal{C}_{\theta, \mu}$ and $ \Delta \mathcal{C}_{\theta, \mu}$ are as in (\ref{eq:var}) and $\Gamma$ is a positively oriented contour that is contained in $\mathcal{M} \cap \mathcal{M}_1$ and encloses $[0, \lM + \theta]$.
\end{corollary}
\begin{remark}\label{RemPrefactor}
In Corollary \ref{CLTfun} we chose to rescale our particle positions by $N$ and one could instead rescale them by $N \theta$ -- we briefly explain here how this affects the statements above. Let $x_i = \frac{\ell_i}{N \theta}$ for $i = 1,\dots, N$ and $y_i = \frac{m_i}{N\theta}$ then by Proposition \ref{LLN} we have that $\mu_N^{\theta} := \frac{1}{N} \sum_{i = 1}^N \delta (x_i)$ converges weakly in probability to $\mu^{\theta}$, which satisfies $\mu^{\theta}(x) = \theta \mu(x \theta)$. In particular, we have that 
$$G_\mu(z) = \theta^{-1} \cdot G_{\mu^{\theta}}(\theta^{-1}z).$$

Suppose that $f_1, \dots, f_n$ are analytic functions in a complex neighborhood of $[0, \lM \theta^{-1} + 1]$, whose restriction to $\mathbb{R}$ is real-valued, and define
$$\mathcal L^{t, \theta}_{f_k}= \sum_{i = 1}^N f_k(x_i) - \mathbb{E}  \left[ \sum_{i = 1}^Nf_k(x_i)  \right] \mbox{, } \mathcal L^{b,\theta}_{f_k}= \sum_{i = 1}^{N-1} f_k(y_i) - \mathbb{E}  \left[ \sum_{i = 1}^{N-1} f_k(y_i)  \right] \mbox{ and }$$
$$\mathcal L^{m, \theta}_{f_k}=N^{1/2} \cdot \left[\mathcal L^{t,\theta}_{f_k} -  \mathcal L^{b,\theta}_{f_k}\right] \mbox{ for $k = 1, \dots, n$} .$$
Then the random variables $\{\mathcal{L}^{m,\theta}_{f_i} \}_{i = 1}^n$, $\{\mathcal{L}^{t,\theta}_{f_i} \}_{i = 1}^n$, $\{\mathcal{L}^{b,\theta}_{f_i} \}_{i = 1}^n$ converge jointly in the sense of moments to a $3n$-dimensional centered Gaussian vector $\xi = (\xi^{m, \theta}_1,\dots, \xi^{m,\theta}_n, \xi^{t,\theta}_1, \dots, \xi^{t,\theta}_n, \xi^{b,\theta}_n, \dots, \xi^{b,\theta}_n)$ with covariance
\begin{align}\label{eq:covV2}
\begin{split}
& Cov(\xi^{t,\theta}_i, \xi^{m,\theta}_j )  =  Cov(\xi^{b,\theta}_i, \xi^{m,\theta}_j )   = 0, \\
&Cov(\xi^{t, \theta}_i, \xi^{t, \theta}_j) = Cov(\xi^{b, \theta}_i, \xi^{b, \theta}_j) = Cov(\xi^{t, \theta}_i, \xi^{b, \theta}_j) = \frac{1}{(2\pi \i )^2} \oint_{\Gamma^{\theta}_1} \oint_{\Gamma^{\theta}_1} f_i(s) f_j(t) \mathcal{C}^{\theta}_{\theta, \mu}(s,t)dsdt,  \\
&Cov(\xi^{m,\theta}_i, \xi^{m,\theta}_j) = \frac{1}{(2\pi \i )^2} \oint_{\Gamma^{\theta}_1} \oint_{\Gamma^{\theta}_1} f_i(s) f_j(t) \Delta \mathcal{C}^{\theta}_{\theta, \mu}(s,t)dsdt, \\
\end{split}
\end{align}
for all $i,j = 1, \dots , n$, where
\begin{align}\label{eq:var_mainV2}
\begin{split}
& \mathcal{C}^{\theta}_{\theta, \mu}(z_1, z_2) = \frac{-\theta^{-1}}{2(z_1-z_2)^2} \left(1 - \frac{(z_1 - \theta^{-1}\alpha)(z_2- \theta^{-1}\beta) + (z_2 -  \theta^{-1}\alpha )(z_1- \theta^{-1}\beta)}{2\sqrt{(z_1 -\theta^{-1}\alpha )(z_1- \theta^{-1}\beta)}\sqrt{(z_2 - \theta^{-1}\alpha)(z_2- \theta^{-1}\beta)}} \right) \mbox{ and }\\
& \Delta \mathcal{C}^{\theta}_{\theta, \mu} (z_1, z_2) = \frac{\theta^{-1}}{2\pi \i}\int_{\Gamma^{\theta}_2}  \frac{dz}{e^{ G_{\mu^{\theta}}(z)} - 1} \cdot \left[ - \frac{1}{(z-z_2)^2(z-z_1)^2}\right],
\end{split}
\end{align}
and $\Gamma^{\theta}_i = \theta^{-1} \Gamma_i$ with $\Gamma_i$ and $\alpha, \beta$ as in Theorem \ref{CLTfun_main}. The main point here is that if we do this alternative scaling, then in the formulas for the covariances $ \mathcal{C}^{\theta}_{\theta, \mu}$ and $ \Delta \mathcal{C}^{\theta}_{\theta, \mu}$, the parameter $\theta$ enters simply as a linear prefactor, as opposed to (\ref{eq:var}) where in $\Delta \mathcal{C}_{\theta, \mu}$ the $\theta$ appears inside the integral in a non-trivial way. This linearity of the covariances in $\theta^{-1}$ is in strong agreement with the results in \cite{GZ} for the $\beta$-Jacobi corners process and in \cite{Hu} for Jack processes.
\end{remark}

\begin{proof}
Observe that when $f$ is analytic in $\mathcal{M}_1$, and real-valued on $\mathbb{R} \cap \mathcal{M}_1$ we have for all large $N$
$$\mathcal{L}^{t/b/m}_f = \frac{1}{2\pi \i} \oint_\Gamma f(z) \cdot Y^{t/b/m}_N(z) dz \mbox{ where we write $Y^m_N(z) := \Delta Y_N(z)$ for convenience}  ,$$
and $\Gamma$ is as in the statement of the theorem. Therefore, for any $r,s,t \geq 0$ with $r + s + t \geq 1$ and $i_1, \dots, i_r, j_1, \dots, j_s,$ $k_1, \dots, k_t \in \{1, \dots, n\}$ we have
\begin{align}\label{QRCLTeq1}
\begin{split}
 \mathbb{E} \left[ \prod_{a = 1}^r\mathcal{L}^m_{ f_{i_a} }\prod_{b = 1}^s\mathcal{L}^t_{ f_{j_b} } \prod_{c = 1}^t \mathcal{L}^b_{ f_{k_c} } \right] = &\frac{1}{(2\pi \i )^{r+s+t}} \oint_\Gamma \cdots \oint_\Gamma   \mathbb{E} \left[\prod_{a = 1}^rY^m_N(x_a)\prod_{b = 1}^sY^t_N(y_b) \prod_{c = 1}^t Y^b_N(z_c) \right] \times \\
& \prod_{a = 1}^rf_{i_a}(x_a) dx_a \prod_{b = 1}^sf_{j_b}(y_b) dy_b  \prod_{c = 1}^t f_{k_c}(z_c)dz_c. 
\end{split}
\end{align}
Since cumulants of centered random variables are linear combinations of moments and vice versa, we conclude that all third and higher order cumulants of $\{\mathcal{L}^m_{f_i} \}_{i = 1}^n$, $\{\mathcal{L}^t_{f_i} \}_{i = 1}^n$, $\{\mathcal{L}^b_{f_i} \}_{i = 1}^n$ vanish as $N \rightarrow \infty$ (here we used Theorem \ref{TGField}, which implies the third and higher order joint cumulants of $Y^{t/b/m}_N(z_i)$ vanish uniformly when $z_i \in \Gamma$). This proves the Gaussianity of the limiting vector $\xi$. Since $\mathcal{L}^{t/b/m}_{f_i}$ are centered for each $N$ the same is true for $\xi$. To get the covariance, we can set $r+s +t = 2$ in (\ref{QRCLTeq1}) and use (\ref{eq:GField}) and (\ref{eq:GFieldv2}).
\end{proof}

%
\subsection{Moment bounds} \label{Section5.2}
We establish the following moment estimates for $\Delta X_N(z) -\mathbb{E}[\Delta X_N(z)]$.
\begin{proposition}\label{S5S1T1}
Suppose that Assumptions 1-5 hold. Then for each $k \geq 1$ we have
\begin{equation}\label{S5S1E1}
  \mathbb{E} \left[ | Y_N(z) |^k \right] = O(1), \mbox{ with $Y_N(z) = \Delta X_N(z) -\mathbb{E}[\Delta X_N(z)]$},
\end{equation}
where the constants in the big $O$ notation depend on $k$ but not on $N$ (provided it is sufficiently large) and are uniform as $z$ varies over compact subsets of $\mathbb{C} \setminus [0, \lM + \theta]$. Moreover, if $G_\mu, R_\mu, \Phi^+(z), \Phi^-(z)$ are as in Section \ref{Section3.1} then
\begin{equation}\label{expectationDiff}
\mathbb{E}[\Delta X_N(v) ]= \frac{N^{1/2}}{2\pi \i}\int_{\Gamma}\frac{ \theta \cdot [  \Phi^-(z) e^{-\theta G_\mu(z)} - \Phi^+(z) ] \partial_z G_\mu(z) }{(z-v) [ \Phi^+(z) + \Phi^-(z) - R_\mu(z)] }dz + O(N^{-1/2}),
\end{equation}
where $\Gamma$ is a positively oriented contour that encloses the segment $[0, \lM + \theta]$, is contained in $\mathcal{M}$ as in Assumption 3 and excludes $v$; the constant in the big $O$ notation is uniform as $v$ varies over compact subsets in $\mathbb{C} \setminus [0, \lM + \theta]$.
\end{proposition}
\begin{remark}
The proof we present below is an adaptation of the one in \cite[Section 6.2]{DK2}, which in turn relies on ideas from \cite{BGG} that go back to \cite{BoGu}.
\end{remark}

For the sake of clarity we split the proof into several steps.\\

{\bf \raggedleft Step 1.} In this step we prove (\ref{expectationDiff}). From (\ref{NekrasovOutput}) for the case $r = s = t = 0$ we get

\begin{equation}\label{S5NekrasovOutput}
\begin{split} 
 \mathbb{E}\left[ \Delta X_N(v) \right] = \frac{1}{2\pi \i}\int_{\Gamma}W_N(z)dz +  N^{-1/2} \mathbb{E} \left[  \zeta^{\Gamma}_N(z) \right], \mbox{ where }
\end{split}
\end{equation}
\begin{equation}\label{S5WNDef}
\begin{split}
&W_N(z)=    \frac{\theta [  \Phi_N^-(Nz) e^{-\theta G_\mu(z)} - \Phi_N^+(Nz) ] \partial_z G_\mu(z) }{ S(z)N^{-1/2} }.
\end{split}
\end{equation}
Equations (\ref{S5NekrasovOutput}) and (\ref{S5WNDef}) give (\ref{expectationDiff}) once we use $\Phi^{\pm}_N(Nz) = \Phi^{\pm}(z) + O(N^{-1})$ from Assumption 3 and $S(z) = (z-v)[ \Phi^+(z) + \Phi^-(z) - R_\mu(z)]$ . Here we also implicitly used that $R_\mu(z)$ is continuous on $\Gamma$ as follows from Lemma \ref{S3AnalRQ} and that $ \Phi^+(z) + \Phi^-(z) - R_\mu(z) \neq 0$ on $\Gamma$ as follows from Lemma \ref{S3NonVanish}. \\

{\bf \raggedleft Step 2.} In this step we reduce the proof of (\ref{S5S1E1}) to the establishment of the following self-improvement estimate claim.\\

{\bf \raggedleft Claim:} Suppose that for some $n, H \in \mathbb{N}$ we have that 
\begin{equation}\label{S5S1E2}
\mathbb{E} \left[ \prod_{a = 1}^h |Y_N(v_a)| \right]= O(1) + O(N^{h/2 + 1 - H/2}) \mbox{ for $h = 1, \dots, 4n + 4$,}
\end{equation}
where the constants in the big $O$ notation depend on $n$ and are uniform as $v_a$ vary over compacts in $\mathbb{C} \setminus [0, \lM + \theta]$. Then
\begin{equation}\label{S5S1E3}
\mathbb{E} \left[ \prod_{a = 1}^h |Y_N(v_a)| \right]= O(1) + O(N^{h/2 + 1 - (H+1)/2}) \mbox{ for $h = 1, \dots, 4n $.}
\end{equation}
We prove the above claim in the following steps. For now we assume its validity and finish the proof of (\ref{S5S1E1}). 

Notice that (\ref{derBound}) and (\ref{expectationDiff}) imply that (\ref{S5S1E2}) holds for the pair $n = 2k$ and $H = 1$. The conclusion is that (\ref{S5S1E2}) holds for the pair $n = 2k- 1$ and $H = 2$. Iterating the argument an additional $k$ times we conclude that (\ref{S5S1E2}) holds with $n = k - 1$ and $H = k+2$, which implies (\ref{S5S1E1}).\\

{\bf \raggedleft Step 3.} In this step we prove that 
\begin{equation}\label{S5S1E4}
M( Y_N(v_0), Y_N(v_1), \dots, Y_N(v_h)) = O(1) + O( N^{h/2 + 1 - H/2}) \mbox{ for $h = 1, \dots, 4n+2$}.
\end{equation}
The constants in the big $O$ notation are uniform over $v_0, \dots, v_h$ in compact subsets of $\mathbb{C} \setminus [0, \lM + \theta]$. 

We start by fixing $\mathcal{V}$ to be a compact subset of $\mathbb{C} \setminus [0, \lM + \theta]$, which is invariant under conjugation. We also fix $\Gamma$ to be a positively oriented contour, which encloses the segment $[0, \lM + \theta]$, is contained in $\mathcal{M}$ as in Assumption 3 and excludes the set $\mathcal{V}$. 

From (\ref{NekrasovOutput}) for $s = t = 0$, $r = h = 1, \dots, 4n+2$, $v_0, v_1 \dots, v_h \in \mathcal{V}$ and $v = v_0$ we have
\begin{align}\label{S5S1E5}
\begin{split} 
& M\left( Y_N(v_0), Y_N(v_1), \dots, Y_N(v_h)  \right) = \frac{1}{2\pi \i}\int_{\Gamma}dz W_N(z) + \sum_{A \subseteq \llbracket 1,h \rrbracket}  \frac{ M \left(  \zeta^{\Gamma}_N(z); A\right)}{N^{h/2 - |A|/2 + 1/2}}   \\
& - \sum_{A \subsetneqq \llbracket 1,h \rrbracket} \frac{\Phi^+(z) e^{\theta G_\mu(z)} M(X_N^t(z);A)}{S(z) N^{3h/2 - 3|A|/2 - 1/2} } \prod_{a \in A^c} \left[ \frac{ -2\theta}{ (v_a -z)^3}\right] \\
& +\sum_{A \subsetneqq \llbracket 1,h \rrbracket} \frac{[ \Phi^+(z) + (-1)^{h - |A|} \Phi^-(z)] \cdot M(\Delta X_N(z); A)}{S(z) N^{h/2 -|A|/2}}  \prod_{a \in A^c}\frac{ 1}{ (v_a -z)^2} .
\end{split}
\end{align}
In deriving the above we have suppressed the sets $B$ and $C$ from the notation, since $s = t = 0$.

We next use the fact that cumulants can be expressed as linear combinations of products of moments, see \cite[Chapter 3]{Taqqu}. This means that $M(\xi_1, \dots, \xi_r)$ can be controlled by the quantities $1$ and $\mathbb{E} \left[ |\xi_i|^r \right]$ for $ i = 1, \dots, r$. We use the latter and (\ref{S5S1E2}) to get
\begin{equation}\label{S5S1E6}
\sum_{A \subseteq \llbracket 1, h \rrbracket} N^{(-h + |A| - 1)/2} \cdot M (\zeta^{\Gamma}_N(z); A) = O(1) + O(N^{h/2 + 1 - H/2}).
\end{equation}
One can analogously show using (\ref{derBound}) that for each $A \subseteq \llbracket 1, h \rrbracket$ 
\begin{equation}\label{S5S1E7}
 M (X_N^t(z) ; A) = O(1) + O(N^{|A|/2 + 3/2 - H/2}) \mbox{, } M (X_N^b(z) ; A) = O(1) + O(N^{|A|/2 + 3/2 - H/2}) .
\end{equation}
Substituting (\ref{S5S1E6}) and (\ref{S5S1E7}) into (\ref{S5S1E5}), and using $\Delta X_N(z) = N^{1/2}(X_N^t(z) - X_N^b(z))$ we conclude
\begin{equation}\label{S5S1E8}
\begin{split} 
& M\left( Y_N(v_0), Y_N(v_1), \dots, Y_N(v_h)  \right) = \frac{1}{2\pi \i}\int_{\Gamma}W_N(z)dz + O(1) + O(N^{h/2 + 1 - H/2}).
\end{split}
\end{equation}
From (\ref{WNDef}) we have for $r \geq 1$, $s = t = 0$ that 
\begin{equation}\label{S5S1E9}
W_N(z) = O(1),
\end{equation}
where we also used that $\Phi_N^{\pm}(Nz) = \Phi^{\pm}(z) + O(N^{-1})$ from Assumption 3. Combining the latter with (\ref{S5S1E8}) gives (\ref{S5S1E4}).\\

{\bf \raggedleft Step 4.} In this step we will prove (\ref{S5S1E3}) except for a single case, which will be handled separately in the next step. We mention that the work in this step will rely mostly on the estimates from Step 3 and properties of moments of random variables and not the cumulant equations we obtained from the Nekrasov's equations.

 Notice that by H{\"o}lder's inequality we have
\begin{equation*}
\sup_{v_1, \dots, v_h \in \mathcal{V}} \mathbb{E} \left[ \prod_{a = 1}^h |Y_N(v_a)| \right]\leq \sup_{v \in \mathcal{V}} \mathbb{E} \left[ |Y_N(v)|^h \right],
\end{equation*}
and so to finish the proof it suffices to show that for $h = 1, \dots, 4n$ we have
\begin{equation}\label{S5S1E11}
 \mathbb{E} \left[ |Y_N(v)|^h \right] = O(1)  + O(N^{h/2 + 1 - (H+1)/2}).
\end{equation}
Using the fact that for centered random variables one can express joint moments as linear combinations of products of joint cumulants, see \cite[Section 3]{Taqqu}, we deduce from (\ref{S5S1E4}) that
\begin{equation}\label{S5S1E12}
\sup_{ v_0, v_1, \dots, v_{h-1} \in \mathcal{V}} \mathbb{E} \left[ \prod_{a = 0}^{h-1}Y_N(v_a) \right] = O(1) + O( N^{(h-1)/2 + 1 - H/2}) \mbox{ for $h = 1, \dots, 4n+2$}.
\end{equation}
If $h = 2m_1$ we set $v_0 = v_1 = \cdots = v_{m_1 - 1} = v$ and $v_{m_1} = \cdots = v_{2m_1 - 1} = \overline{v}$ in (\ref{S5S1E12}), which yields 
\begin{equation}\label{S5S1E13}
\sup_{ v \in \mathcal{V}} \mathbb{E} \left[ |Y_N(v)|^h \right] = O(1) + O( N^{h/2 + 1/2 - H/2}) \mbox{ for $h = 1, \dots, 4n+2$}.
\end{equation}
In deriving the above we used that $Y_N(\overline{v}) = \overline{Y_N(v)}$ and so $Y_N(v) \cdot Y_N(\overline{v}) = |Y_N(v)|^2$. 

We next let $h = 2m_1 + 1$ be odd and notice that by the Cauchy-Schwarz inequality and (\ref{S5S1E13}) 
\begin{align}\label{S5S1E14}
\begin{split}
&\sup_{ v \in \mathcal{V}} \mathbb{E} \left[ |Y_N(v)|^{2m_1 + 1} \right] \leq \sup_{ v \in \mathcal{V}} \mathbb{E} \left[ |Y_N(v)|^{2m_1 + 2} \right]^{1/2} \cdot \mathbb{E} \left[ |Y_N(v)|^{2m_1} \right]^{1/2} =  \\  
&O(1) + O( N^{h/2 + 1 - H/2}) + O(N^{m_1/2 + 3/4 - H/4}) .
\end{split}
\end{align}
We note that the bottom line of (\ref{S5S1E14}) is $O(1) + O(N^{m_1 + 1 - H/2})$ except when $H = 2m_1 + 2$, since
$$m_1/2 + 3/4 - H/4 \leq \begin{cases} m_1 + 1 - H/2 &\mbox{ when $H \leq 2m_1 + 1$,} \\ 0 &\mbox{ when $H \geq 2m_1 + 3$.} \end{cases}$$
Indeed, the first inequality implies that $O(N^{m_1/2 + 3/4 - H/4})$ can be absorbed into $O(N^{m_1 + 1 - H/2})$ and the second that it can be absorbed into $O(1)$. If $H = 2m_1 + 2$ then $O(N^{m_1/2 + 3/4 - H/4})$ cannot be absorbed into either of these terms.

From (\ref{S5S1E13}) and (\ref{S5S1E14}) we conclude (\ref{S5S1E11}), except when $H = 2m_1 +2$ and $h = 2m_1 + 1$. We will handle this case in the next step.\\

{\bf \raggedleft Step 5.} In this step we will show that (\ref{S5S1E11}) even when $H = 2m_1 + 2$ and $4n > h = 2m_1 + 1$. In the previous step we showed in (\ref{S5S1E11}) that
$\sup_{v \in \mathcal{V}} \mathbb{E} \left[ |Y_N(v)|^{2m_1 + 2} \right] = O(N^{1/2})$, and below we will improve this estimate to
\begin{equation}\label{S5S1E15}
\sup_{v \in \mathcal{V}} \mathbb{E} \left[ |Y_N(v)|^{2m_1 + 2} \right] = O(1).
\end{equation}
The trivial inequality $x^{2m_1 + 2} + 1 \geq |x|^{2m_1 + 1}$ together with (\ref{S5S1E15}) imply
$$\sup_{v \in \mathcal{V}} \mathbb{E} \left[ |Y_N(v)|^{2m_1 + 1} \right]  = O(1).$$
Consequently, we have reduced the proof of the claim to establishing (\ref{S5S1E15}). \\

Let us list the relevant estimates we will need
\begin{equation}\label{S5S1E16}
\begin{split}
& \mathbb{E} \left[ \prod_{a = 1}^{2m_1 + 2} |Y_N(v_a)| \right] = O(N^{1/2})\mbox{ and } \mathbb{E} \left[ \prod_{a = 1}^j |Y_N(v_a)| \right] = O(1) \mbox{ for $0 \leq j \leq 2m_1$.}
\end{split}
\end{equation}
The above identities follow from (\ref{S5S1E11}), which we showed to hold unless $h = 2m_1 + 1$ in the previous step. All constants are uniform over $v_a \in \mathcal{V}$. Below we feed the improved estimates of (\ref{S5S1E16}) into Steps 3 and 4, which ultimately yield (\ref{S5S1E15}).\\

In Step 3 we have the following improvement over (\ref{S5S1E6}) using the estimates of (\ref{S5S1E16})
\begin{equation}\label{S5S1E17}
\sum_{A \subseteq \llbracket 1, h \rrbracket} N^{(-1 + |A| - h)/2} \cdot M (\zeta^{\Gamma}_N(z); A) = O(1) \mbox{ for $h =1, 2, \dots, 2m_1+ 1$}.
\end{equation}
In addition we will need the following improvement over (\ref{S5S1E7}) 
\begin{equation}\label{S5S1E18}
 M (X_N^t(v) ; A) = O(1) \mbox{  and }  M ( X_N^b(v) ; A) = O(1) ,
\end{equation}
where $A \subsetneqq \llbracket 1, 2m_1 + 1\rrbracket$ and the constants in the big $O$ notation is uniform as $v, v_1, \dots, v_{2m_1 + 1}$ vary over compact subsets of $\mathbb{C} \setminus [0, \lM + \theta]$. We will prove (\ref{S5S1E18}) in the next step. For now we assume its validity and finish the proof of (\ref{S5S1E15}). 

Substituting (\ref{S5S1E17}), (\ref{S5S1E18}) into (\ref{S5S1E5}) we obtain the following improvement over (\ref{S5S1E4})
\begin{equation}\label{S5S1E19}
\begin{split} 
& M\left( Y_N(v_0), Y_N(v_1), \dots, Y_N(v_{2m_1 + 1})  \right) = \frac{1}{2\pi \i}\int_{\Gamma}dz W_N(z) + O(1) = O(1),
\end{split}
\end{equation}
where in the last line we used (\ref{S5S1E9}). We may now repeat the arguments in Step 4 and note that by using (\ref{S5S1E19}) in place of (\ref{S5S1E4}) we obtain the following improvement over (\ref{S5S1E12})
\begin{equation}\label{S5S1E20}
\sup_{ v_0, v_1, \dots, v_{2m_1 + 1} \in \mathcal{V}} \mathbb{E} \left[ \prod_{a = 0}^{2m_1 + 1}Y_N(v_a) \right] = O(1).
\end{equation}
Setting $v_0 = v_1 = \cdots = v_{m_1} = v$ and $v_{m_1+1} = \cdots = v_{2m_1 + 1} = \overline{v}$ in (\ref{S5S1E20}) we get (\ref{S5S1E15}).\\

{\bf \raggedleft Step 6.} In this step we establish (\ref{S5S1E18}). We have by (\ref{derBound}) and (\ref{S5S1E16}) that (\ref{S5S1E18}) holds for all subsets $A \subsetneqq \llbracket 1, 2m_1 + 1\rrbracket$ such that $|A| \leq 2m_1 - 1$ or if $A$ is empty. We thus only focus on the case when $|A| = 2m_1 \geq 2$.

We use (\ref{NekrasovOutput}) with $r = 2m_1-1$ and $s = 1$ and $t = 0$. This gives 
\begin{align*}
\begin{split} 
& M\left( \Delta X_N(v)  ; \llbracket 1,r \rrbracket, \{1\}, \varnothing \right) =\frac{1}{2\pi \i}\int_{\Gamma}\hspace{-1mm}  dz W_N(z) +  \frac{ \Phi^+(z) e^{\theta G_\mu(z)}  M  \left( X^b_N(z); \llbracket 1, r \rrbracket, \varnothing, \varnothing  \right)  }{S(z) N^{1/2} } \hspace{-1mm} \prod_{b \in B^c} \hspace{-1mm} \frac{1 }{(u_b-z)^2 }   \\
&- \sum_{A \subseteq \llbracket 1,r \rrbracket} \sum_{B \subseteq \{1\}}  \hspace{-2mm}   \frac{ \Phi^+(z)e^{\theta G_\mu(z)} M \left(X_N^t(z)   ; A, B, \varnothing \right) {\bf 1}\{ |A| + |B| < r+ 1 \}   }{S(z) N^{r/2 - |A|/2-1/2} N^{r - |A|} N^{1- |B|}}  \hspace{-1mm}  \prod_{a \in A^c}\hspace{-1mm} \left[ \frac{-2\theta }{(v_a -z)^3}\right]\hspace{-1mm} \prod_{b \in B^c} \frac{1}{ (u_b -z)^2} \\
&+ \sum_{A \subsetneqq \llbracket 1,r \rrbracket}  \frac{ \Phi^-(z) M\left( \Delta X_N(z) ; A, \{1\} , \varnothing \right)   }{S(z) N^{r/2 - |A|/2} } \prod_{a \in A^c} \left[ \frac{-1}{  (v_a -z)^2}\right]   \\
&+   \sum_{A \subseteq \llbracket 1, r \rrbracket, B \subseteq \{1\} }\frac{  \Phi^+(z)  M\left( \Delta X_N(z) ; A, B, \varnothing  \right) {\bf 1}\{ |A| + |B| < r+ 1 \}   }{ S(z)  N^{r/2 - |A|/2} N^{1 - |B|}}  \prod_{a \in A^c}   \frac{1}{ (v_a -z)^2}\prod_{b \in B^c}  \frac{ 1}{ (u_b -z)^2} \\
& + \sum_{A \subseteq \llbracket 1,r \rrbracket}\sum_{B \subseteq \{1\}}  \frac{ M \left(  \zeta^{\Gamma}_N(z); A, B, \varnothing\right)}{N^{r/2 - |A|/2 + 1/2}} = 0.
\end{split}
\end{align*}
Applying (\ref{S5S1E16}) and (\ref{derBound}) above we get
\begin{equation*}
M\left( \Delta X_N(v)  ; \llbracket 1,r \rrbracket, \{1\}, \varnothing \right) = \frac{1}{2\pi \i}\int_{\Gamma}W_N(z)dz + O(1).
\end{equation*}
In view of (\ref{WNDef}) we have $W_N(z) = O(N^{-1})$, which gives
\begin{equation}\label{S5TRQ1}
M\left( \Delta X_N(v)  ; \llbracket 1,r \rrbracket, \{1\}, \varnothing \right) = O(1),
\end{equation}
which is the first statement in (\ref{S5S1E18}).

Using (\ref{NekrasovOutput}) with $r = 2m_1-1$ and $s = 0$ and $t = 1$ instead, and repeating the same arguments we obtain
\begin{equation}\label{S5TRQ2}
M\left( \Delta X_N(v)  ; \llbracket 1,r \rrbracket, \varnothing, \{1\} \right) = O(1),
\end{equation}
which concludes the proof of (\ref{S5S1E18}) and hence the proposition.

%
\subsection{Proof of Theorem \ref{TGField} } \label{Section5.3}
In this section we prove Theorem \ref{TGField}. As we are dealing with centered random variables it suffices to show that second and higher order cumulants of $(Y^t_N(z), Y^b_N(z), \Delta Y_N(z))$ converge to those specified in the statement of the theorem. In the sequel we fix a compact set $K \subset U$ and a positively oriented contour $\Gamma$ that contains $[0, \lM + \theta]$, is contained in $\mathcal{M}$ as in Assumption 3 and excludes $K$. 

We begin by utilizing Proposition \ref{S5S1T1} and (\ref{derBound}) to rewrite (\ref{NekrasovOutput}) in a way that is convenient for us. Let us fix $r \geq 0, s \geq 0$ and $t \geq 0$ such that $r + s + t \geq 1$. In addition, we fix $v_0, \dots, v_r, u_1, \dots, u_s, w_1, \dots ,w_t \in K$. To ease our notation we write 
$$M^Y_N(v_0, \dots, v_r; u_1, \dots, u_s; w_1, \dots, w_t)$$
for the joint cumulant of $\Delta Y_N(v_0), \dots, \Delta Y_N (v_r), Y_N^t(u_1), \dots, Y_N^t(u_s), Y_N^b(w_1), \dots, Y_N^b(w_t)$.
We now apply (\ref{NekrasovOutput}) with $v = v_0$, replacing in all second and higher order cumulants on the right side of (\ref{NekrasovOutput}) $\Delta X_N(\cdot)$ with $\Delta Y_N (\cdot)$. Notice that since $\Delta Y_N (z) - \Delta X_N(z)$ is deterministic (see (\ref{DefX}) and (\ref{DefDiscreteField})) this does not affect the right side of (\ref{NekrasovOutput}). Furthermore, we replace on the left side of (\ref{NekrasovOutput}) $X_N^{t/b}(\cdot)$ with $Y_N^{t/b}(\cdot)$, which again does not affect the left side of (\ref{NekrasovOutput}) since $r + s + t \geq 1$. Finally, since cumulants can be expressed as linear combinations of products of moments, see \cite[Chapter 3]{Taqqu}, we can utilize Proposition \ref{S5S1T1} and (\ref{derBound}) to bound most of the cumulants on the right side of (\ref{NekrasovOutput}) by $O(N^{-1/2})$. Overall, we obtain the following form of (\ref{NekrasovOutput}) when $r + s + t \geq 1$
\begin{align}\label{S5HFE}
\begin{split}
&M^Y_N(v_0, \dots, v_r; u_1, \dots, u_s; w_1, \dots, w_t) =  \frac{1}{2\pi \i} \int_{\Gamma}dz W_N(z) ,\\
&+ {\bf 1}\{ r =1, s= t =  0\} \cdot \frac{[ \Phi^+(z) - \Phi^-(z)]\mathbb{E} \left[\Delta X_N(z) \right]}{N^{1/2} S(z) (v_1-z)^2} +  O(N^{-1/2}),
\end{split}
\end{align}
where
\begin{align}\label{S5HFE2}
\begin{split}
 W_N(z) = &{\bf 1} \{ r = 1, s = t = 0\} \cdot \left[ \frac{2 \Phi^+(z) e^{\theta G_\mu(z)} }{S(z)(v_1-z)^3} -\frac{\theta \Phi^+(z) \partial_zG_\mu(z) }{S(z)(v_1-z)^2}  \right] + O(N^{-1/2}), \\
\end{split}
\end{align}
and the constants in the big $O$ notations depend on $K,r,s,t$. 

We next compute the limits of the second and higher order cumulants in the statement of Theorem \ref{TGField} using Proposition \ref{CLT}, (\ref{S5HFE}) and (\ref{S5HFE2}).
\subsubsection{Second order cumulants} \label{Section5.3.1}
We fix $z_1, z_2 \in K$. We first have by Proposition \ref{CLT} that 
\begin{equation}\label{S5S3E1V}
\lim_{N\rightarrow \infty} Cov(Y_N^{t}(z_1), Y^{t}_N(z_2)) = \mathcal C_\theta(z_1, z_2).
\end{equation}
Next we have by (\ref{S5HFE}) applied to $r = 0, s = 1$, $t = 0$, $v_0= z_1$ and $u_1 = z_2$ that
\begin{equation}\label{S5S3E2V}
\begin{split} 
&Cov\left( \Delta Y_N(z_1) , Y^t_N(z_2) \right) = O(N^{-1/2}),
\end{split}
\end{equation}
Similarly, from (\ref{S5HFE}) applied to $r = 0, s = 0$, $t = 1$, $v_0= z_1$ and $w_1 = z_2$ we have
\begin{equation}\label{S5S3E3V}
\begin{split} 
&Cov\left( \Delta Y_N(z_1) , Y^b_N(z_2) \right) = O(N^{-1/2}).
\end{split}
\end{equation}
Equations (\ref{S5S3E2V}) and (\ref{S5S3E3V}) imply the first line in (\ref{eq:GField}). \\

Using that $\Delta Y_N(z_1)= N^{1/2} [Y_N^t(z) - Y_N^b(z)]$, (\ref{S5S3E1V}) and (\ref{S5S3E2V}) we have
\begin{align}\label{S5S3E4V}
\begin{split}
&\lim_{N\rightarrow \infty} Cov(Y_N^{b}(z_1), Y^{t}_N(z_2)) = \\
& \lim_{N\rightarrow \infty} Cov(Y_N^{t}(z_1), Y^{t}_N(z_2)) - \lim_{N\rightarrow \infty}N^{-1/2}Cov(\Delta Y_N(z_1), Y^{t}_N(z_2)) = \mathcal C_\theta(z_1, z_2).
\end{split}
\end{align}
Using that $\Delta Y_N(z_1)= N^{1/2} [Y_N^t(z) - Y_N^b(z)]$, (\ref{S5S3E4V}) and (\ref{S5S3E3V}) we have
\begin{align}\label{S5S3E5V}
\begin{split}
&\lim_{N\rightarrow \infty} Cov(Y_N^{b}(z_1), Y^{b}_N(z_2)) = \\
& \lim_{N\rightarrow \infty} Cov(Y_N^{b}(z_1), Y^{t}_N(z_2)) - \lim_{N\rightarrow \infty}N^{-1/2}Cov( Y^b_N(z_1), \Delta Y_N(z_2)) = \mathcal C_\theta(z_1, z_2).
\end{split}
\end{align}
Equations (\ref{S5S3E1V}), (\ref{S5S3E4V}) and (\ref{S5S3E5V}) imply the second line in (\ref{eq:GField}). \\

We next focus on establishing (\ref{eq:GFieldv2}). From (\ref{S5HFE}) applied to $r = 1$, $s=t = 0$, $v_0 = z_1$ and $v_1 = z_2$ and (\ref{expectationDiff}) we have
\begin{align}\label{S5S3E6V}
\begin{split}
&\lim_{N\rightarrow \infty} Cov(\Delta Y_N(z_1), \Delta Y_N(z_2)) = \frac{1}{2\pi \i} \int_{\Gamma}  dz \left[ \frac{2 \Phi^+(z) e^{\theta G_\mu(z)} }{S(z)(z_2-z)^3} -\frac{\theta \Phi^+(z) \partial_zG_\mu(z) }{S(z)(z_2-z)^2}  \right] + \\
&\frac{[ \Phi^+(z)- \Phi^-(z)]}{ S(z) (z_2-z)^2} \cdot \frac{1}{2 \pi \i}  \int_{\Gamma_1} \frac{ \theta \cdot [  \Phi^-(\zeta) e^{-\theta G_\mu(\zeta)} - \Phi^+(\zeta) ] \partial_z G_\mu(\zeta) }{(\zeta-z) [ \Phi^+(\zeta) + \Phi^-(\zeta) - R_\mu(\zeta)] }d\zeta,
\end{split}
\end{align}
where $\Gamma_1$ is a positively oriented contour inside $\Gamma$, which encloses the segment $[0, \lM + \theta]$.

By computing the integral over $\Gamma$ as a residue at the simple pole at $z = \zeta$ we get
\begin{align}\label{S5S3E7V}
\begin{split}
& \frac{1}{(2\pi \i)^2} \int_{\Gamma}    \int_{\Gamma_1}  \frac{[ \Phi^+(z) - \Phi^-(z)]}{ S(z) (z_2-z)^2}  \frac{ \theta \cdot [  \Phi^-(\zeta) e^{-\theta G_\mu(\zeta)} - \Phi^+(\zeta) ] \partial_z G_\mu(\zeta) }{(\zeta-z) [ \Phi^+(\zeta) + \Phi^-(\zeta) - R_\mu(\zeta)] }d\zeta dz  \\
& = \frac{-1}{2\pi \i}    \int_{\Gamma_1}   \frac{ \theta \cdot [ \Phi^+(\zeta) - \Phi^-(\zeta)] \cdot [  \Phi^-(\zeta) e^{-\theta G_\mu(\zeta)} - \Phi^+(\zeta) ] \partial_z G_\mu(\zeta) }{ S(\zeta) (z_2-z)^2 [ \Phi^+(\zeta) + \Phi^-(\zeta) - R_\mu(\zeta)] }d\zeta \\
& =\frac{-1}{2\pi \i}    \int_{\Gamma}   \frac{ \theta \cdot [ \Phi^+(z) - \Phi^-(z)] \cdot [  \Phi^-(z) e^{-\theta G_\mu(z)} - \Phi^+(z) ] \partial_z G_\mu(z) }{ S(z) (z_2 -z)^2 [ \Phi^+(z) + \Phi^-(z) - R_\mu(z)] }dz,
\end{split}
\end{align}
where in going from the second to the third line, we deformed $\Gamma_1$ to $\Gamma$ (by Cauchy's theorem this does not affect the value of the integral) and relabeled the integration variable from $\zeta$ to $z$. We mention that in going from the first to the second line in (\ref{S5S3E7V}) we used that $\Phi^{\pm}$ are analytic and $S(z) = (z-z_1)[ \Phi^+(z) + \Phi^-(z) - R_\mu(z)]$ is non-vanishing and analytic in the closure of the region enclosed by $\Gamma$ as follows from Lemma \ref{S3NonVanish}.

We next substitute (\ref{S5S3E7V}) into (\ref{S5S3E6V}) to get
\begin{align}\label{S5S3E8V}
\begin{split}
\lim_{N\rightarrow \infty} Cov(\Delta Y_N(z_1), \Delta Y_N(z_2)) = &\frac{1}{2\pi \i} \int_{\Gamma}  dz \frac{2 \Phi^+(z) e^{\theta G_\mu(z)} }{(z_2-z)^3(z - z_1)[ \Phi^+(z) + \Phi^-(z) - R_\mu(z)]} + \\
&\frac{\theta \partial_z G_\mu(z) }{(z_2-z)^2 (z - z_1)[ e^{\theta G_\mu(z)} - 2 + e^{-\theta G_\mu(z)}]}, \\
\end{split}
\end{align}
where we used that $S(z) = (z-z_1)[ \Phi^+(z) + \Phi^-(z) - R_\mu(z)]$, $R_\mu(z) = \Phi^+(z) e^{\theta G_\mu(z)}+ \Phi^-(z) e^{-\theta G_\mu(z)}$ which implies
\begin{align*}
&-\frac{\theta \Phi^+(z) \partial_zG_\mu(z) }{S(z)(z_2-z)^2}    -  \frac{ \theta \cdot [ \Phi^+(z) - \Phi^-(z)] \cdot [  \Phi^-(z) e^{-\theta G_\mu(z)} - \Phi^+(z) ] \partial_z G_\mu(z) }{ S(z) (z_2 -z)^2 [ \Phi^+(z) + \Phi^-(z) - R_\mu(z)] } \\
& = \frac{\theta \partial_z G_\mu(z) }{(z_2-z)^2 (z - z_1)[ e^{\theta G_\mu(z)} - 2 + e^{-\theta G_\mu(z)}]}.
\end{align*}
Notice that 
$$\frac{\theta \partial_z G_\mu(z)}{[e^{\theta G_\mu(z)} - 2 + e^{-\theta G_\mu(z)} ]} = - \partial_z \left[ \frac{1}{e^{\theta G_\mu(z)} - 1}\right],$$
and so using integration by parts for the second term on the right side of (\ref{S5S3E8V}) we arrive at
\begin{align*}
\begin{split} 
& \lim_{N\rightarrow \infty} Cov(\Delta Y_N(z_1), \Delta Y_N(z_2)) = \frac{1}{2\pi \i}\int_{\Gamma}   \frac{2\Phi^+(z)e^{\theta G_\mu(z)} }{  (z_2 -z)^3  (z - z_1) [ \Phi^{+}(z) + \Phi^-(z) - R_\mu(z)] } + \\
& \frac{1}{e^{\theta G_\mu(z)} - 1} \cdot \left[ - \frac{2}{(z-z_2)^3(z-z_1)} - \frac{1}{(z-z_1)^2(z-z_2)^2}\right].
\end{split}
\end{align*}
Next we can add $- \frac{2 [R_\mu(z) - \Phi^+(z)] }{ (z_2 - z)^3 (z - z_1) [\Phi^{+}(z) + \Phi^-(z) - R_\mu(z)]}$ to the above integrand without affecting the value of the integral by Cauchy's theorem (here we used Lemma \ref{S3NonVanish}). The benefit is that
$$\frac{2\Phi^+(z)e^{\theta G_\mu(z)} }{  (z_2 -z)^3 (z - z_1) [ \Phi^{+}(z) + \Phi^-(z) - R_\mu(z)] } - \frac{2 [R_\mu(z) - \Phi^+(z)] }{  (z_2 - z)^3 (z - z_1) [\Phi^{+}(z) + \Phi^-(z) - R_\mu(z)]} = $$
$$ =\frac{2}{(z_2 -z)^3 (z - z_1) [ 1 - e^{\theta G_\mu(z)}] }  .$$
Substituting this above yields
\begin{equation*}
\begin{split} 
&\lim_{N\rightarrow \infty} Cov(\Delta Y_N(z_1), \Delta Y_N(z_2)) =  \frac{1}{2\pi \i}\int_{\Gamma}  \frac{1}{e^{\theta G_\mu(z)} - 1} \cdot \left[ - \frac{1}{(z-z_1)^2(z-z_2)^2}\right],
\end{split}
\end{equation*}
which clearly implies (\ref{eq:GFieldv2}).

\subsubsection{Third and higher order cumulants} \label{Section5.3.2} Let us fix $r \geq -1, s \geq 0$ and $t \geq 0$ such that $r + s + t \geq 2$. Our goal is to show that
\begin{equation}\label{S5S3E3.5}
\lim_{ N \rightarrow \infty} M^Y_N(v_0, \dots, v_r; u_1, \dots, u_s; w_1, \dots, w_t) = 0,
\end{equation}
where the convergence is uniform over $K$. \\

We first have by Proposition \ref{CLT} that if $u_1, \dots, u_s \in K$ and $s \geq 3$ then
\begin{equation}\label{S5S3E4}
\lim_{N\rightarrow \infty} M^Y_N(\varnothing; u_1, \dots,  u_s; \varnothing) = 0.
\end{equation}
Furthermore for each $1 \leq i \leq s-1$ we have
\begin{align}\label{S5S3E5}
\begin{split}
&M^Y_N(\varnothing; u_1, \dots, u_i; u_{i+1}, \dots , u_s) - M^Y_N(\varnothing; u_1, \dots, u_{i-1}; u_i, u_{i+1}, \dots ,u_s) = \\
&N^{-1/2} M^Y_N(u_{i};  u_1, \dots, u_{i-1};  u_{i+1}, \dots ,u_s) = O(N^{-1/2}),
\end{split}
\end{align}
where in the second line we used Proposition \ref{S5S1T1} and (\ref{derBound}). Combining (\ref{S5S3E4}) and (\ref{S5S3E5}) we conclude (\ref{S5S3E3.5}) provided $r = -1$ and $s + t \geq 3$.

We next suppose that $r \geq 0$. From (\ref{S5HFE}) we have $M^Y_N(v_0, \dots, v_r; u_1, \dots, u_s; w_1, \dots, w_t) = O(N^{-1/2})$ when $r + s + t \geq 2$, which concludes the proof of (\ref{S5S3E3.5}) and hence the theorem.\\

%
\section{Multilevel extensions and examples}\label{Section7}
In this section we demonstrate how any discrete $\beta$-ensemble can be extended to a multilevel system of the type presented in Section \ref{Section1.1}. Using this connection we can construct many measures that fit into the general framework of Section \ref{Section3} and we discuss some of them in Section \ref{Section7.2}.

%
\subsection{Multilevel extension}\label{Section7.1} In this section we provide a method for extending any discrete $\beta$-ensemble to a multilevel system as in (\ref{eq:measure_k}). The construction uses Jack symmetric functions and mimics the construction of the ascending Macdonald processes of \cite{BorCor}.

Let us recall some notation from earlier sections of the text. Fix $\theta > 0$ and $N \in \mathbb{N}$. Then we define
\begin{align}\label{MultiState}
\begin{split}
&\mathbb{W}^\theta_{N,k} = \{ (\ell_1, \dots, \ell_k):  \ell_i = \lambda_i + (N - i)\cdot\theta, \mbox{ with } \lambda_1 \geq \lambda_2 \cdots \geq \lambda_k \mbox{ and } \lambda_i \in \mathbb{Z} \}.\\
&\mathfrak{X}^{\theta}_{N,N}:= \{ (\ell^1 , \cdots, \ell^N): \ell^k \in \mathbb{W}^\theta_{N,k} \mbox{ for $k = 1, \dots,N$ and } \ell^1 \preceq \ell^2 \preceq \cdots \preceq \ell^N \}
\end{split}
\end{align}
where we recall that $\ell^{k} \preceq \ell^{k+1}$ if $\lambda^{k+1}_{k+1} \leq \lambda_k^k \leq \lambda^{k+1}_{k} \leq \cdots \leq \lambda_1^k\leq \lambda_1^{k+1}$ with $\ell^k_i = \lambda_i^k + (N- i + 1)\theta$ for $i = 1, \dots, k$ and $\ell^{k+1}_i = \lambda_i^{k+1} + (N- i) \cdot \theta$ for $i = 1, \dots, k+ 1$.

\begin{proposition}\label{PropExtension} Suppose that $w(x;N)$ is a non-negative function on $\mathbb{R}$ such that 
$$Z: = \sum_{\ell \in \mathbb{W}^\theta_{N,N}} \frac{\Gamma(\ell_i - \ell_j + 1)\Gamma(\ell_i - \ell_j + \theta)}{\Gamma(\ell_i - \ell_j)\Gamma(\ell_i - \ell_j +1-\theta)}  \prod_{i = 1}^N w(\ell_i; N) \in (0, \infty).$$
For each $(\ell^1, \cdots, \ell^N) \in \mathfrak{X}^{\theta}_{N,N}$ we define
\begin{equation}\label{S7eq:measure_k}
\mathbb{P}^{\theta, N}_N(\ell^1, \dots, \ell^N) =\prod_{i = 1}^{N} \frac{\Gamma( i \theta)}{\Gamma(\theta)}  \cdot  \frac{1}{Z} \prod_{1 \leq i < j \leq N} \frac{\Gamma(\ell^N_i - \ell^N_j + 1)}{\Gamma(\ell^N_i - \ell^N_j + 1- \theta)}  \cdot \prod_{k = 1}^{N-1}  I(\ell^{k+1}, \ell^k) \cdot \prod\limits_{i=1}^{N}w(\ell_i^{N}; N), \mbox{  }
\end{equation}
\begin{align}\label{S7tm}
\begin{split}
I(\ell^s, \ell^{s-1}) = & \prod_{1 \leq i < j \leq s}\frac{\Gamma(\ell^s_i - \ell^s_j + 1 - \theta)}{\Gamma(\ell^s_i - \ell^s_j) } \cdot \prod_{1 \leq i < j \leq s-1} \frac{\Gamma(\ell^{s-1}_i - \ell^{s-1}_j + 1)}{\Gamma(\ell^{s-1}_i - \ell^{s-1}_j + \theta)} \times\\
&\prod_{1 \leq i < j \leq s} \frac{\Gamma(\ell^{s-1}_i - \ell^s_j)}{ \Gamma(\ell^{s-1}_i - \ell^s_j + 1 - \theta)}  \cdot \prod_{1 \leq i \leq j \leq s-1} \frac{\Gamma(\ell^{s}_i - \ell^{s-1}_j + \theta)}{\Gamma(\ell^s_i - \ell^{s-1}_j + 1)}.
\end{split}
\end{align}
Then $\mathbb{P}^{\theta, N}_N$ is a probability measure on $\mathfrak{X}^{\theta}_{N,N}$ as in (\ref{MultiState}). Moreover, the projection of $\mathbb{P}^{\theta, N}_N$ on $\ell^N$ is
\begin{equation}\label{S7SingleLevM}
\mathbb{P}_N^{\theta,N}(\ell_1^N, \cdots, \ell_N^N) = \frac{1}{Z} \cdot \frac{\Gamma(\ell^N_i - \ell^N_j + 1)\Gamma(\ell^N_i - \ell^N_j + \theta)}{\Gamma(\ell^N_i - \ell^N_j)\Gamma(\ell^N_i - \ell^N_j +1-\theta)}  \prod_{i = 1}^N w(\ell^N_i; N),
\end{equation}
and the projection of $\mathbb{P}_N^{\theta, N}$ on the top two levels $(\ell^N, \ell^{N-1})$ is given by
\begin{equation}\label{S7PDef}
\mathbb{P}_N^{\theta,N}(\ell^{N}, \ell^{N-1}) = \frac{\Gamma( N  \theta)}{\Gamma(\theta)} \cdot \frac{1}{Z} \cdot H^t(\ell^N) \cdot H^b(\ell^{N-1}) \cdot I(\ell^N, \ell^{N-1}), \mbox{ where }
\end{equation}
\begin{equation}\label{S7PDef2}
H^t(\ell) = \prod_{1 \leq i < j \leq N} \frac{\Gamma(\ell_i - \ell_j + 1)}{\Gamma(\ell_i - \ell_j + 1 - \theta)}  \prod_{i = 1}^N w(\ell_i;N), \hspace{2mm} H^b(m) = \prod_{1 \leq i < j \leq N-1} \frac{\Gamma(m_i - m_j + \theta)}{\Gamma(m_i - m_j)},
\end{equation}
\end{proposition}
We give the proof of Proposition \ref{PropExtension} at the end of this subsection.\\

We begin by introducing some useful notation for symmetric functions, using \cite{Mac} as a main reference. A {\em partition} is a sequence $\lambda =(\lambda_1, \lambda_2, \cdots)$ of non-negative integers such that $\lambda_1 \geq \lambda_2 \geq \cdots$ and all but finitely many elements are zero. We denote the set of all partitions by $\mathbb{Y}$ and by $\varnothing$ the empty partition $\lambda$ such that $\lambda_i = 0$ for all $i \in \mathbb{N}$. The {\em length} $\ell(\lambda)$ of a partition is the number of non-zero $\lambda_i$ and the {\em weight} of a partition $\lambda$ is given by $|\lambda| = \lambda_1 + \lambda_2 + \cdots$. An alternative representation is given by $\lambda = 1^{m_1} 2^{m_2} \cdots$, where $m_j(\lambda) = |\{i \in \mathbb{N}: \lambda_i = j \}$ is called the {\em multiplicity} of $j$ in the partition $\lambda$. There is a natural ordering on the space of partitions, called the {\em reverse lexicographic order}, given by
$$\lambda> \mu \iff \exists k \in \mathbb{N} \mbox{ such that $\lambda_i = \mu_i$ whenever $i < k$ and $\lambda_k > \mu_k$.}$$

A {Young diagram} is a graphical representation of a partition $\lambda$, with $\lambda_1$ left justified boxes in the top row, $\lambda_2$ in the second row and so on. In general, we do not distinguish between a partition $\lambda$ and the Young diagram representing it. The {\em conjugate} of a partition $\lambda$ is the partition $\lambda'$ whose Young diagram is the transpose of the diagram $\lambda$. In particular, we have the formula $\lambda_i' = |\{j \in \mathbb{N}: \lambda_j \geq i\}|$.

Given two diagrams $\lambda$ and $\mu$ such that $\mu \subseteq \lambda$ (as a collection of boxes), we call the difference $\kappa = \lambda - \mu$ a {\em skew Young diagram}. A skew Young diagram $\kappa$ is a {\em horizontal $m$-strip} if $\kappa$ contains $m$ boxes and no two lie in the same column. If $\lambda - \mu$ is a horizontal strip we write $\lambda \succeq \mu$. We observe that $\lambda \succeq \mu \iff \lambda_1 \geq \mu_1 \geq \lambda_2 \geq \mu_2 \geq \cdots$. Some of these concepts are illustrated in Figure \ref{S2_1}.
\begin{figure}[h]
\centering
\scalebox{0.45}{\includegraphics{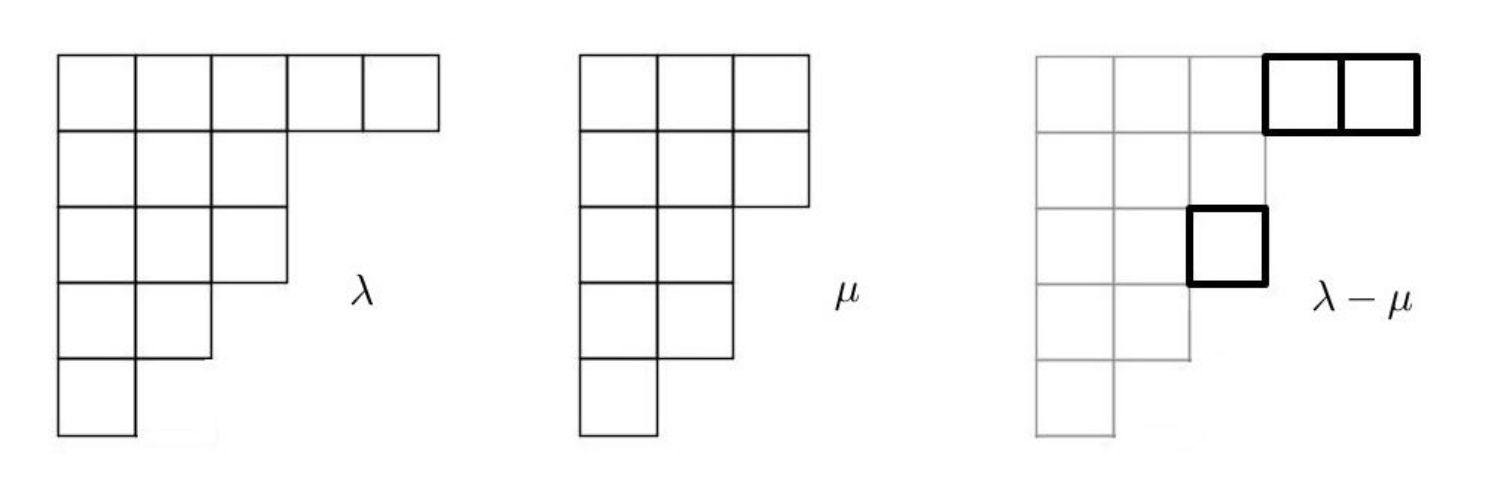}}
\caption{The Young diagram $\lambda = (5,3,3,2,1)$ and its transpose (not shown) $\lambda' = ( 5,4,3,1,1)$. The length $\ell(\lambda) = 5$ and weight $|\lambda| = 14$. The Young diagram $\mu = (3,3,2,2,1)$ is such that $\mu \subseteq \lambda$. The skew Young diagram $\lambda - \mu$ is shown in {\em black bold lines} and is a horizontal $3$-strip.}
\label{S2_1}
\end{figure}

For a box $\square=(i, j)$ of $\lambda$ (that is, a pair $(i, j)$ such that $\lambda_i\geq j$) we denote by $a(i, j)$ and $l(i,j)$ its {\em arm} and {\em leg lengths}: 
$$a(i,j)=\lambda_i-j, \quad l(i, j)=\lambda'_j-i.$$
Further, $a'(i,j)$  and $\ell'(i,j)$ denote the {\em co-arm} and {\em co-leg lengths}:
$$a'(i,j)=j-1, \quad l'(i, j)=i-1.$$

Let $\Lambda_X$ be the algebra of symmetric functions in variables $X=(x_1, x_2, \dots)$. An element of $\Lambda_X$ can be viewed as a formal symmetric power series of bounded degree in the variables $x_1, x_2, \dots .$  One way to view $\Lambda_X$ is as an algebra of polynomials in Newton power sums
$$ p_k(X) = \sum_{i = 1}^\infty x_i^k, \hspace{4mm} \mbox{ for $k \geq 1$}.$$
For any partition $\lambda$ we define
$$p_\lambda(X) = \prod_{i = 1}^{\ell(\lambda)} p_{\lambda_i} (X),$$
and note that $p_\lambda(X)$, $\lambda \in \mathbb{Y}$ form a basis in $\Lambda_X$. 

In what follows we fix a parameter $\theta$. Unless the dependence on $\theta$ is important we will suppress it from the notation, similarly for the variable set $X$. We consider the following scalar product $\langle \cdot, \cdot \rangle$ on $\Lambda \otimes \mathbb{Q}(\theta)$
\begin{equation}
\langle p_\lambda, p_\mu \rangle = \delta_{\lambda,\mu} \cdot \theta^{- \ell(\lambda)} \prod_{i = 1}^{\lambda_1} i ^{m_i(\lambda)} m_i(\lambda)!,
\end{equation}
where $\delta_{\lambda, \mu} = 1$ if $\lambda = \mu$ and is zero otherwise.

\begin{proposition}\label{ThmJack}\cite{Mac} There are unique symmetric functions $J_\lambda \in \Lambda\otimes \mathbb Q(\theta)$ for all $\lambda \in \mathbb Y$ such that
\begin{itemize}
\item  $\langle J_\lambda ,J_\mu \rangle =0$ unless $\lambda = \mu,$
\item the leading (with respect to reverse lexicographic order) monomial in $J_\lambda$ is $\prod_{i = 1}^{\ell(\lambda)} x_i^{\lambda_i}$.
\end{itemize}
\end{proposition}
The functions $J_\lambda$ in Proposition \ref{ThmJack} are called {\em Jack symmetric functions} and they form a homogeneous basis of $\Lambda$ that is different from the $p_\lambda$ above. Given $\lambda \in \mathbb{Y}$ we define the {\em dual Jack symmetric functions} $\tilde{J}_\lambda$ through
$$\tilde J_\lambda =J_\lambda \prod\limits_{\Box\in \lambda}\frac{a(\square)+\theta \cdot \ell (\square)+\theta}{a(\square)+\theta \cdot \ell(\square)+1}.$$
\begin{definition}
Take the two infinite sequences of variables $X=(x_1, x_2, \dots)$,  $Y=(y_1, y_2, \dots)$ and let $(X,Y)$ denote their concatenation. For two partitions $\lambda, \mu$ define the {\em skew Jack symmetric functions} $J_{\lambda / \mu}(X)$ and $\tilde{J}_{\lambda/ \mu}(X)$ as the coefficients in the expansion
\begin{equation}\label{S7Branch}
J_\lambda(X, Y)= \sum_{\mu \in \mathbb{Y}} J_\mu(Y) J_{\lambda/ \mu}(X) \mbox{ and }\tilde{J}_\lambda(X, Y)= \sum_{\mu \in \mathbb{Y}} \tilde{J}_\mu(Y) \tilde{J}_{\lambda/ \mu}(X) .
\end{equation}
\end{definition}
The equations in (\ref{S7Branch}) are called {\em branching relations} for the Jack symmetric functions.
\begin{definition}
A {\it specialization} $\rho$ is an algebra homomorphism from $\Lambda$ to the set of complex numbers.  A specialization is called Jack-positive if its values on all (skew) Jack polynomials with a fixed parameter $\theta> 0$ are real
and non-negative. 
\end{definition}

We will mostly work with simple specializations in this paper but point out the following important result.
\begin{proposition}\cite{KOO} For any fixed $\theta > 0$, Jack-positive specializations can be parameterized
by triplets $(\alpha, \beta, \gamma)$, where $\alpha, \beta$ are sequences of real numbers with
$$\alpha_1\geq \alpha_2\geq \dots \geq 0, \quad \beta_1\geq \beta_2\geq \dots \geq 0, \quad \sum\limits_i(\alpha_i+\beta_i)<\infty$$
and $\gamma$ is a non-negative real number. The specialization corresponding to a triplet $(\alpha, \beta, \gamma)$ is given by its values on $p_k$
$$p_1\rightarrow p_1(\alpha, \beta, \gamma)=\gamma+\sum\limits_i(\alpha_i+\beta_i),$$
$$p_k\rightarrow p_k(\alpha, \beta, \gamma)=\sum\limits_i \alpha_i^k+(-\theta)^{k-1}\sum\limits_i \beta_i^k, \hspace{3mm} k\geq 2.$$
\end{proposition}

We write $1^N$ for the specialization $\rho$ with $\alpha_1 = \cdots = \alpha_N = 1$ and all other parameters being set to $0$. For these specializations we have the following formula, which is \cite[Chapter VI, (10.20)]{Mac} 
\begin{equation}\label{jackpoly}
J_\lambda(1^{N})  = {\bf 1} \{ \ell(\lambda) \leq N\}  \cdot  \prod_{i = 1}^{N} \prod_{j =1}^{\lambda_i} \frac{N \theta + (j-1) - (i-1) \theta}{\lambda_i - j + \theta(\lambda_j' - i) + \theta }.
\end{equation}
We introduce the shifted coordinates $\ell_i = \lambda_i + (N - i ) \cdot \theta$ for $i = 1, \dots, N$. It will be more convenient to rewrite the formula in (\ref{jackpoly}) in terms of $\ell_i$'s.

The denominator in (\ref{jackpoly}) can be rewritten as 
$$\prod_{1 \leq i \leq k \leq N} \prod_{j =\lambda_{k+1} + 1}^{\lambda_k} \frac{1}{\lambda_i - j + \theta(k - i + 1) } = \prod_{1 \leq i \leq k \leq N}  \frac{\Gamma(\lambda_i + \theta(k - i + 1) - \lambda_k)}{\Gamma(\lambda_i + \theta(k-i+1) - \lambda_{k+1})} = $$
$$\prod_{1 \leq i < k \leq N}  \frac{\Gamma(\lambda_i - \lambda_k + \theta(k - i + 1) )}{\Gamma(\lambda_i - \lambda_{k}+ \theta(k-i) )} \cdot \prod_{i = 1}^{N} \frac{\Gamma(\theta)}{\Gamma(\lambda_i + \theta(N-i + 1))} = \hspace{-3mm} \prod_{1 \leq i < j \leq N} \frac{\Gamma(\ell_i - \ell_j + \theta)}{\Gamma(\ell_i - \ell_j)}  \prod_{i = 1}^{N}  \frac{ \Gamma(\theta)}{\Gamma( \ell_i + \theta)},$$
where $\lambda_{N+1} = 0$. The numerator in (\ref{jackpoly}) can be rewritten as
$$ \prod_{i = 1}^{N} \prod_{j =1}^{\lambda_i}[N \theta + (j-1) - (i-1) \theta] = \prod_{i = 1}^{N}\frac{\Gamma(N\theta + \lambda_i - (i-1)\theta)}{\Gamma((N-i + 1)\theta )} = \prod_{i = 1}^{N}\frac{\Gamma(\ell_i + \theta)}{\Gamma((N-i + 1)\theta )}.$$
Overall, we have
\begin{equation}\label{jackpoly2}
J_\lambda(1^{N}) = \prod_{i = 1}^{N} \frac{\Gamma(\theta)}{\Gamma( i \theta)} \times \prod_{1 \leq i < j \leq N} \frac{\Gamma(\ell_i - \ell_j + \theta)}{\Gamma(\ell_i - \ell_j)} .
\end{equation}

In addition, we have from \cite[Chapter VI, (7.14')]{Mac} (see also \cite[Section 2]{GS}) 
\begin{align}\label{SkewJackPoly}
\begin{split}
J_{\lambda/ \mu}(1) = {\bf 1} \{ \lambda \succeq \mu\} \cdot & \prod_{1 \leq i < j \leq N} \frac{\Gamma(\ell_i - \ell_j + 1 - \theta)}{\Gamma(\ell_i - \ell_j) } \cdot \prod_{1 \leq i < j \leq N-1} \frac{\Gamma(m_i - m_j + 1)}{\Gamma(m_i - m_j + \theta)} \times\\
&\prod_{1 \leq i < j \leq N} \frac{\Gamma(m_i - \ell_j)}{ \Gamma(m_i - \ell_j + 1 - \theta)}  \cdot \prod_{1 \leq i \leq j \leq N-1} \frac{\Gamma(\ell_i - m_j + \theta)}{\Gamma(\ell_i - m_j + 1)},  
\end{split}
\end{align}
where $\ell_i = \lambda_i + (N - i ) \cdot \theta$ and $m_i = \mu_i + (N - i)\cdot \theta$.

We remark that while the formulas (\ref{jackpoly2}) and (\ref{SkewJackPoly}) were initially defined for partitions $\lambda$ and $\mu$ they can be naturally extended to {\em signatures} of length $N$ and $N-1$ respectively (a signature of length $N$ is a sequence of integers $\lambda_1 \geq \lambda_2 \geq \cdots \geq \lambda_N$), since the expressions remain unchanged if we shift all the elements by the same integer. In particular, we have the following version of the branching relation for a given signature $\lambda$
\begin{equation}\label{S7Branchv2}
J_\lambda(1^N)= \sum_{\lambda^1 \preceq \lambda^2 \preceq \cdots \preceq \lambda^{N-1}\preceq \lambda} J_{\lambda/ \lambda^{N-1}}(1)  J_{\lambda^{N-1}/ \lambda^{N-2}}(1) \cdot J_{\lambda^{N-2}/ \lambda^{N-3}}(1) \cdots J_{\lambda^{2}/ \lambda^{1}}(1) ,
\end{equation}
where $\lambda^i$ are summed over signatures of length $i$.\\

We are finally ready to give the proof of Proposition \ref{PropExtension}.
\begin{proof} (Proposition \ref{PropExtension}) Let us write $\ell^k = (\ell^k_1, \dots, \ell^k_k)$ for $k = 1, \dots N$ and define $\lambda_i^j$ through $\ell_i^j = \lambda_i^j + (N- i)\cdot\theta$. We start by proving that (\ref{S7eq:measure_k}) defines a probability measure on $\mathfrak{X}^{\theta}_{N,N}$. Clearly, $\mathbb{P}^{\theta, N}_N(\ell^1, \dots, \ell^N) \geq 0$ by definition and so it suffices to show that 
\begin{equation}\label{S7E1}
\sum_{ (\ell^1, \dots, \ell^N) \in \mathfrak{X}^{\theta}_{N,N}}\mathbb{P}^{\theta, N}_N(\ell^1, \dots, \ell^N) = 1.
\end{equation}
Using the definition of $\mathbb{P}^{\theta, N}_N$ as well as (\ref{SkewJackPoly}) we see that
\begin{align*}
\begin{split}
& \sum_{ (\ell^1, \dots, \ell^N) \in \mathfrak{X}^{\theta}_{N,N}}\mathbb{P}^{\theta, N}_N(\ell^1, \dots, \ell^N)  = \prod_{i = 1}^{N} \frac{\Gamma( i \theta)}{\Gamma(\theta)} \cdot \frac{1}{Z}\cdot \sum_{\ell^N\in \mathbb{W}_{N,N}^\theta} \prod_{1 \leq i < j \leq N} \frac{\Gamma(\ell^N_i - \ell^N_j + 1)}{\Gamma(\ell^N_i - \ell^N_j + 1 - \theta)}   \prod_{i = 1}^N w(\ell^N_i; N) \\
&\times \sum_{\lambda^1 \preceq \lambda^2 \preceq \cdots \preceq \lambda^N} J_{\lambda^N/ \lambda^{N-1}}(1) \cdot J_{\lambda^{N-1}/ \lambda^{N-2}}(1) \cdots J_{\lambda^{2}/ \lambda^{1}}(1)  \\
&= \prod_{i = 1}^{N} \frac{\Gamma( i \theta)}{\Gamma(\theta)}\cdot \frac{1}{Z} \cdot \sum_{\ell^N\in \mathbb{W}_{N,N}^\theta} \prod_{1 \leq i < j \leq N} \frac{\Gamma(\ell^N_i - \ell^N_j + 1)}{\Gamma(\ell^N_i - \ell^N_j + 1 - \theta)}   \prod_{i = 1}^N w(\ell^N_i; N) J_{\lambda^N}(1^{N})   \\
& = \frac{1}{Z} \cdot \sum_{\ell^N\in \mathbb{W}_{N,N}^\theta} \prod_{1 \leq i < j \leq N} \frac{\Gamma(\ell^N_i - \ell^N_j + 1)}{\Gamma(\ell^N_i - \ell^N_j + 1 - \theta)} \frac{\Gamma(\ell^N_i - \ell^N_j + \theta)}{\Gamma(\ell^N_i - \ell^N_j)}   \prod_{i = 1}^N w(\ell^N_i; N)  = 1,
\end{split}
\end{align*}
where in the second equality we used the branching relations (\ref{S7Branchv2}), in the third equality we used (\ref{jackpoly2}) and in the last one we used the definition of $Z$. This proves (\ref{S7E1}). Furthermore, performing the same summation above but fixing $\ell^N$ shows (\ref{S7SingleLevM}). 

Finally, let us fix $\ell^N \in \mathbb{W}^\theta_{N,N}$ and $\ell^{N-1} \in \mathbb{W}^\theta_{N,N-1}$ such that $\ell^N \succeq \ell^{N-1}$. Using (\ref{SkewJackPoly}) we get
\begin{equation*}
\begin{split}
&\mathbb{P}_N^{\theta,N}(\ell^{N}, \ell^{N-1})  = \prod_{i = 1}^{N} \frac{\Gamma( i \theta)}{\Gamma(\theta)} \cdot Z^{-1} \cdot H^t(\ell^N) \cdot I(\ell^N, \ell^{N-1}) \\
&\times \sum_{\lambda^1 \preceq \lambda^2 \preceq \cdots \preceq \lambda^{N-1}} J_{\lambda^{N-1}/ \lambda^{N-2}}(1) \cdot J_{\lambda^{N-2}/ \lambda^{N-3}}(1) \cdots J_{\lambda^{2}/ \lambda^{1}}(1)  \\
& = \prod_{i = 1}^{N} \frac{\Gamma( i \theta)}{\Gamma(\theta)} \cdot Z^{-1} \cdot H^t(\ell^N) \cdot I(\ell^N, \ell^{N-1}) \cdot J_{\lambda^{N-1}}(1^{N-1})  \\
& = \frac{\Gamma( N  \theta)}{\Gamma(\theta)} \cdot Z^{-1}  \cdot H^t(\ell^N) \cdot I(\ell^N, \ell^{N-1})  \cdot  \prod_{1 \leq i < j \leq N - 1} \frac{\Gamma(\ell^{N-1}_i - \ell^{N-1}_j + \theta)}{\Gamma(\ell^{N-1}_i - \ell^{N-1}_j)} ,
\end{split}
\end{equation*}
where in the second equality we used the branching relations (\ref{S7Branchv2}) and in the third equality we used (\ref{jackpoly2}). This proves (\ref{S7PDef}).
\end{proof}

%
\subsection{Applications}\label{Section7.2}
In this section we discuss several applications of Theorem \ref{CLTfun}. In view of our work in Section \ref{Section7.1} we have that essentially all single-band $\beta$-ensembles that satisfy the assumptions in \cite{BGG} have a multi-level extension for which Theorem \ref{CLTfun} is applicable. We remark that the only difference between the assumptions in Section \ref{Section3} and those in \cite{BGG} is that in Assumption 3 we assume that $\Phi^{\pm}$ are positive on $(0, \lM + \theta)$, which will automatically be the case for all the models we consider. 

%
\subsubsection{Krawtchouk ensemble}\label{Section7.2.1}
In this section we study the Krawtchouk orthogonal polynomial ensemble with $\theta=1$ -- this is probably the simplest case that one can consider in our framework.

The Krawtchouk ensemble is a probability distribution that depends on two parameters $M, N$ with $M \in \mathbb{Z}_{\geq 0}$ and $N \in \mathbb{N}$. The state space of the model is the set of $N$-tuples of integers $(\ell_1, \dots, \ell_N)$ that satisfy $M+N-1 \geq \ell_1 > \ell_2 > \cdots \ell_N \geq 0$ and the measure is given by
\begin{equation}
\mathbb{P}_N \left(\ell_1, \dots, \ell_N\right) = \frac{1}{Z} \prod_{1 \leq i < j \leq N} (\ell_i - \ell_j)^2 \cdot \prod_{i = 1}^N \binom{M+N - 1}{\ell_i}.
\end{equation}
The two-level measure is obtained using Proposition \ref{PropExtension}. Since $\theta = 1$ the extension can be considered as first sampling $(\ell_1, \cdots, \ell_N)$ from the above measure, then sampling uniformly from the set of (half-strict) Gelfand Tsetlin patterns whose top level is given by $(\ell_1, \cdots, \ell_N)$ and forgetting the bottom $N-2$ levels. The resulting $2$-level distribution is given by
\begin{equation}\label{S7Krawtchouk}
\mathbb{P}(\ell, m) =\Gamma(N)\cdot  \frac{1}{Z} \cdot \prod_{1 \leq i < j \leq N} (\ell_i - \ell_j)  \cdot \prod_{1 \leq i < j \leq N-1} (m_i - m_j) \cdot \prod_{i = 1}^N \binom{M+N-1}{\ell_i}.
\end{equation}
We fix $\mathfrak m> 0 $ (independent on $N$), set $M=\left \lfloor{\mathfrak m N}\right \rfloor$ and discuss the limit of (\ref{S7Krawtchouk}) as $N \rightarrow \infty$.

In \cite{BGG} the authors showed that the above measure satisfies Assumptions 1-5 as we explain here. Assumptions 1 and 2 can be easily deduced using Stirling's formula.
Moreover for this example 
$$\frac{w(z-1;N )}{w(z;N)}=\frac{z}{M +N-z},$$
and so we can take
\begin{equation}
\Phi^-_N(z)=\frac{z}{N},\quad\Phi_N^+(z)=\frac{M+ N}{N}-\frac{z}{N};
\end{equation}
We conclude that Assumption 3 is satisfied with $\mathcal M=\mathbb C,$ $\Phi^-(z)=z$
and $\Phi^+(z)=\mathfrak m + 1-z$ and $\Phi^{\pm}_N$ as above. Moreover, we have $\Phi^-_N(0)=0$ and $
\Phi^+_N(M+ N)$ so Assumption 4 is also valid.

By a direct limit of the single level Nekrasov's equations, Proposition \ref{SingleLevelNekrasov}, the following formulas for $R_\mu$ and $Q_\mu$ were found in \cite{BGG}
$$R_{\mu}(z)=\mathfrak m-1, \quad Q_{\mu}(z)=2\sqrt{(z-(\mathfrak m+1)/2)^2-\mathfrak m},$$
so Assumption 5 is also verified. We remark that the square root is as in Section \ref{Section1.5}.

The conclusion is that the Krawtchouk ensemble satisfies all the
assumptions and then Theorem \ref{CLTfun} is valid with $\alpha=(\mathfrak m + 1)/2-\sqrt{\mathfrak m}$  and $\beta=(\mathfrak m + 1)/2+\sqrt{\mathfrak m}$. We also remark that in view of (\ref{QRmu}) we have that 
$$e^{G_\mu(z)} = \frac{R_\mu(x) - \sqrt{R^{2}_\mu(z) - 4 \Phi^-(z) \Phi^+(z)}}{2 \Phi^+(z)} = \frac{R_\mu(z) - Q_\mu(z)}{2\Phi^+(z)}$$

%
\subsubsection{Lozenge tilings of the hexagon}\label{Section7.2.2}
In this section we apply our result to the model of uniform random lozenge tilings of the $A \times B \times C$ hexagon. This is a well-studied model, with many results available, cf. \cite{CLP, P1, P2, JN, Go}. We explain below the definition of the model, how it fits into our framework and what our results say about it. Afterwards we compare our results with those in \cite{BuGo} and \cite{GZ}.

\vspace{-2mm}
\begin{figure}[h]
\centering
\scalebox{0.45}{\includegraphics{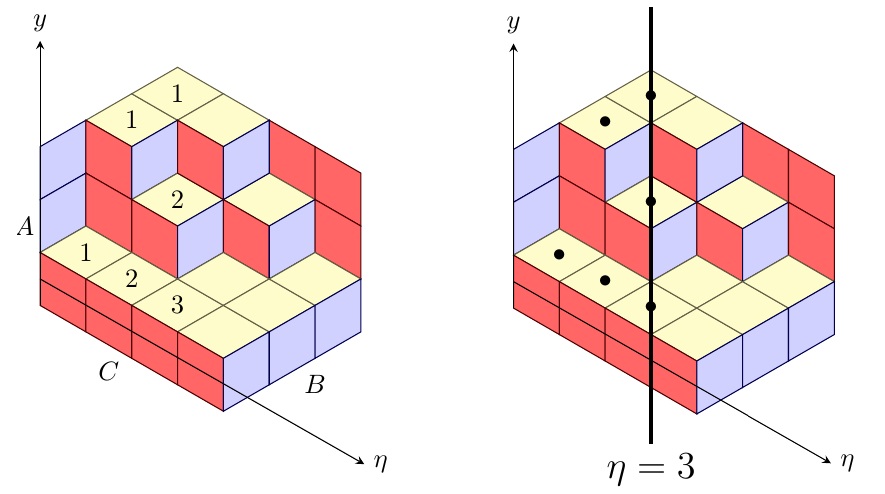}}
\caption{The left part shows a lozenge tiling of the $3 \times 3 \times 4$ hexagon and its height function. The right part shows the vertical line through $(0,3)$, which intersects $3$ horizontal lozenges in the tiling. The dots indicate the location of $\ell^i_j$ and for the picture we have $\ell^3_1 =  5, \ell^2_1 = 4, \ell^3_2 = 3, \ell_3^3 = \ell_2^2 = \ell_1^1 = 1$.} 
\label{S7_1}
\end{figure}
\vspace{-2mm}

Fix integers $A,B,C \geq 1$ and consider the $A \times B \times C$ hexagon drawn on the triangular lattice, see Figure \ref{S7_1}. We are interested in random tilings of such a hexagon by rhombi, also called lozenges (these are obtained by gluing two neighboring triangles together). There are three types of rhombi that arise in such a way: horizontal $\diam$, and two others \hspace{-3mm} $\ldiam$, \hspace{-4mm} $\rdiam$. We will work with the standard coordinate axes $(y, \eta)$, whose origin is located at the base of the left-most side of the hexagon, see Figure \ref{S7_1}. If we fix any $B+C-1 \geq \eta \geq 1$ and look at the vertical line through $(0, \eta)$ we see that this line intersects some fixed (depending on $A,B,C,\eta$) number $N$ of horizontal lozenges. In particular, if $\min (B,C) \geq \eta \geq 1$ then $\eta = N$ and if we let $\ell_1^N > \ell_2^N > \cdots > \ell_N^N$ denote the $y$-coordinate of these horizonal lozenges then their distribution is given by
\begin{equation}\label{S72E1}
\begin{split}
&\mathbb{P}_N (\ell^N_1, \dots, \ell^N_N) = \frac{1}{Z_N} \prod_{1 \leq i < j \leq N}(\ell^N_i - \ell^N_j)^2 \cdot \prod_{i = 1}^N w(\ell^N_i;N), \mbox{ where } \\
&w(y;N) = \frac{(y + C- N)!(A+B-y-1)!}{y!(A+N - y - 1)!},
\end{split}
\end{equation}
provided that $0 \leq \ell_N^N$ and $\ell_1^N \leq A+ N-1$ and is $0$ otherwise. In the above $Z_N$ is a normalization constant. The computation of $\mathbb{P}_N$ is possible by noticing that a tiling can be viewed as two {\em Gelfand-Tsetlin patterns} glued together and utilizing some well-known techniques of enumerating the latter. See for example \cite{CLP, Go,BuGo,BP}, in particular (\ref{S72E1}) can be found as \cite[Proposition 2.6]{BP}.

Let us denote by $\ell^i = (\ell^i_1 > \ell^i_2 > \cdots > \ell_i^i)$ the $y$-coordinates of the horizontal lozenges on the vertical line through $(0,i)$ for $i = 1, \dots, N$. Then the combinatorics of the model imply that $\ell^1 \preceq \ell^2 \preceq \cdots \preceq \ell^N$ in the notation of Section \ref{Section7.1}. Furthermore, as the tiling is unifomly distributed we conclude that the joint law of $(\ell^1, \dots, \ell^N)$ is given by (\ref{S7eq:measure_k}) with $\theta = 1$, namely
\begin{equation}\label{S7eq:measure_kV2}
\mathbb{P}_N(\ell^1, \dots, \ell^N) =   \frac{\prod_{i = 1}^{N} \Gamma( i ) }{Z_N} \prod_{1 \leq i < j \leq N} (\ell^N_i - \ell^N_j)  \prod\limits_{i=1}^{N}w(\ell_i^{N}; N), 
\end{equation}
where $w(\cdot;N)$ and $Z_N$ are as in (\ref{S72E1}). In particular, by Proposition \ref{PropExtension} we conclude that the joint law of $(\ell^N, \ell^{N-1})$ is given by (\ref{PDef}) with $\theta = 1$, $M =  A$ and $w(\cdot; N)$ as in (\ref{S72E1}). This shows that the measure $\mathbb{P}_N$ fits into the setup of Section \ref{Section3}.\\

We are interested in the scaling limit of a random lozenge tiling of the hexagon as $L \rightarrow \infty$ when $A = \lfloor a \cdot L \rfloor$, $B =   \lfloor b \cdot L \rfloor$, $C =  \lfloor c \cdot L \rfloor$ and $N =  \lfloor n \cdot L \rfloor$ where $a,b,c,n > 0$ and $n \leq \min (b,c)$. We begin by showing that the unduced measure on $(\ell^N, \ell^{N-1})$ satisfies Assumptions 1-5. For convenience we denote $a_1 = a/n$, $b_1 = b/n$ and $c_1 = c/n$.

Assumption 1 is immediate with $ \lM = a_1$ and Assumption 2 follows from Stirling's formula with 
$$V(s) = -\left(s + c_1 - 1 \right)\log \left(s+ c_1 - 1 \right) - \left(a_1 + b_1 - s\right)\log \left(a_1 + b_1- s\right)+ s \log s + \left(a_1 + 1- s\right)\log\left(a_1 + 1 - s\right).$$
One readily observes that 
$$ \frac{w(Nx;N)}{w(Nx- 1;N)} = \frac{(Nx + C - N) (A +N - Nx)}{Nx(A+B-Nx)},$$
which shows that Assumption 3 is satisfied with $\Phi_N^+(Nz) = (z + C/N - 1)(A/N + 1 - z)$ and $\Phi_N^-(Nz) = z (A/N + B/N - z)$, and in particular 
\begin{equation}\label{S72Phis}
\Phi^+(z) = (z + c_1 -1)(a_1 + 1 -z) \mbox{ and } \Phi^-(z) = z(a_1 + b_1 -z),
\end{equation}
which are clearly real analytic and positive on $(0, a/n + 1)$. Since $\Phi_N^-(0) = 0 = \Phi_N^+(A + N - 1) = 0$ we see that Assumption 4 holds as well.

We next show that Assumption 5 is also satisfied. If $\mu$ denotes the equilibrium measure as in Theorem \ref{LLN} then we have that 
\begin{equation}\label{S72QRmu}
\begin{split}
&R_{\mu}(z) = \Phi^-(z) \cdot e^{-  G_{\mu} (z) }+  \Phi^+(z) \cdot e^{  G_{\mu} (z) }, \mbox{ and } Q_{\mu}(z) = \Phi^-(z) \cdot e^{-  G_{\mu} (z) } -  \Phi^+(z) \cdot e^{  G_{\mu} (z) }.
\end{split}
\end{equation}
If we set $c_\mu = \int_{\mathbb{R}} x \mu(x) dx$, use that $G_\mu(z) = \frac{1}{z} + \frac{c_\mu}{z^2} + O(z^{-3})$ as $|z| \rightarrow \infty$ and (\ref{S72Phis}) we see that 
\begin{equation*}
\begin{split}
&R_\mu(z) = \Phi^-(z) \cdot e^{-  G_{\mu} (z) }+  \Phi^+(z) \cdot e^{  G_{\mu} (z) } =  z(a_1 + b_1 -z) \cdot \left[ 1 - \frac{1}{z} - \frac{c_\mu}{z^2} + \frac{1}{2z^2} + O(z^{-3}) \right] + \\
& + (z + c_1 -1)(a_1 + 1 -z)  \cdot \left[ 1 + \frac{1}{z} + \frac{c_\mu}{z^2} + \frac{1}{2z^2} + O(z^{-3})\right] = \\
& -2 z^2 + ( 2 + 2a_1 + b_1 - c_1) z - (a_1+ b_1 + c_1 - c_1(a_1 + 1)) + O(z^{-1}).
\end{split}
\end{equation*}
In \cite{BGG} the authors showed that $R_\mu(z)$ is a degree $2$ polynomial, which in view of the above implies that 
\begin{equation}\label{S72Rmu2}
\begin{split}
&R_{\mu}(z) = A_0z^2 + B_0z + C_0, \mbox{ with } A_0 = -2, B_0 = ( 2 + 2a_1 + b_1 - c_1) \mbox{ and } C_0 = - (a_1+ b_1 ) + a_1 c_1.
\end{split}
\end{equation}
Since $Q^2_\mu(z) = R^2_\mu(z) - 4\Phi^+(z) \Phi^-(z)$ we see that
\begin{equation}\label{S72Qmu2}
\begin{split}
&Q_{\mu}(z) = (b_1 + c_1) \cdot \sqrt{ (z - a^-) (z- a^+)} \mbox{ with } a^{\pm} = \frac{a_1b_1c_1 + a_1c_1^2 + a_1b_1 - a_1c_1 + b_1^2 + b_1 c_1 \pm 2\sqrt{D_1}}{(b_1+c_1)^2}, \\
& \mbox{ where } D_1 = a_1b_1c_1(b_1 + c_1 -1)(a_1 + b_1 + c_1).
\end{split}
\end{equation}
In particular, we see that Assumption 5 holds with $\alpha = a^-$ and $\beta = a^+$ as in (\ref{S72Qmu2}) and $H(z) = b_1 + c_1$. Overall, we see that Assumptions 1-5 in Section \ref{Section3} hold and so Theorem \ref{CLTfun} holds for these measures. Notice that (\ref{S72QRmu}) sets up a quadratic equation $e^{G_\mu(z)}$ from which we obtain
\begin{equation}\label{S72EGmu}
e^{G_\mu(z)} = \frac{R_\mu(z) + (b_1 + c_1) \sqrt{(z - a^-) (z- a^+)}}{2(z + c_1 -1)(a_1 + 1 -z) }.
\end{equation}

As mentioned before the above limit of random lozenge tilings has been considered before in \cite{CLP,BuGo,P1,P2}, where it has been shown that the object is asymptotically described by the pullback of the Gaussian free field (GFF) on $\mathbb{H}$ under a suitable map. Our goal in the remainder of this section is to explain how our result fits into that framework. We will follow the notation in \cite{BuGo} and explain the results there and afterwards we will connect our Theorem \ref{CLTfun} to them. For simplicity of the notation we will assume that $b + c = 1$ in the remainder.

A natural way to interpret a random lozenge tiling is through the so-called {\em height function}, which is an integer-valued function $H_L(y, \eta)$ and counts the number of horizontal lozenges  $\diam$ {\em above} the point $(Ly, L\eta)$, cf. Figure \ref{S7_1}. Notice that we have rescaled the coordinates now by $L$ so that $\eta \in [0,1]$. In \cite{BuGo} the authors established a certain Central limit theorem (CLT) for the random height functions $H_L(y, \eta)$, which involves a certain map to $\mathbb{H}$ that we describe now.

Given a compactly supported measure ${\bf m}$ on $\mathbb{R}$ we let $G_{\bf m}(z) = \int_{\mathbb{R}} \frac{d{\bf m}(x)}{z - x}$ denote its Stieltjes transform and define the map $\Omega_{\bf m}^{-1}:\mathbb{H} \rightarrow \mathbb{R} \times \mathbb{R}$  through $\Omega_{\bf m}^{-1}(z) = (y_{\bf m}(z), \eta_{\bf m}(z))$, where
\begin{equation}
\begin{split}
&y_{\bf m}(z) = z+ \frac{(z - \overline{z})( \exp (G_{\bf m}(\overline{z}) - 1) \exp( G_{\bf m}(z))}{\exp(G_{\bf m}(z)) - \exp ( G_{\bf m}(\overline{z}))} \\
&\eta_{\bf m}(z) = 1+ \frac{(z - \overline{z})( \exp (G_{\bf m}(\overline{z}) - 1)( \exp( G_{\bf m}(z)) - 1)}{\exp(G_{\bf m}(z)) - \exp ( G_{\bf m}(\overline{z}))}.
\end{split}
\end{equation}
We also let $D_{\bf m} \subset \mathbb{R}^2$ denote the image of this map. In \cite[Proposition 3.13]{BuGo} it was shown that $\Omega_{\bf m}^{-1} : \mathbb{H} \rightarrow D_{\bf m}$ is a diffeomorphism and we denote its inverse by $\Omega_{\bf m}:D_{\bf m} \rightarrow \mathbb{H}$.

We define the moments of the random height function as 
\begin{equation}\label{momentDisc}
M^L_{\eta, k} = \int_{\mathbb{R}} y^k \left[ H_L(y,\eta) - \mathbb{E} \left[ H_L(y,\eta)\right] \right]dy, \hspace{5mm} 0 < \eta \leq 1, k \in \mathbb{N}.
\end{equation}
We also define the moments of the pullback of the GFF under the map $\Omega_{\bf m}$ through 
\begin{equation}\label{momentCont}
\mathcal{M}^{\bf m}_{\eta, k} = \int_{z \in \mathbb{H}: \eta_{{\bf m }}(z)  = \eta} \hspace{-5mm} y^k_{\bf m}(z) \mathfrak{G}(z) \frac{dy_{\bf m}(z)}{dz} dz, \hspace{5mm} 0 < \eta \leq 1, k \in \mathbb{N}.
\end{equation}
In the above equation $\mathfrak{G}$ stands for the Gaussian free field on $\mathbb{H}$ -- see \cite[Section 3.3]{BuGo} and the references in there for a definition of this object and the random variables $\mathcal{M}^{\bf m}_{\eta, k} $.

With the above notation \cite[Theorem 3.14]{BuGo} implies the following statement.
\begin{theorem}\label{TBuGo}
Suppose that $a,b,c > 0$ and $b + c = 1$. Let $H_L$ denote the random height function of a uniform random lozenge tiling of the hexagon with sides $A = \lfloor a \cdot L \rfloor, B = \lfloor b \cdot L \rfloor$ and $C = \lfloor c \cdot L \rfloor$. Then as $L \rightarrow \infty$ the sequence of random height functions
$$ \sqrt{\pi} \left( H_L(y,\eta) - \mathbb{E} \left[ H_L(y,\eta) \right] \right)$$
converges to the pullback of the Gaussian free field on $\mathbb{H}$ with respect to the map $\Omega_{\bf m}$, where the measure ${\bf m}$ has density ${\bf 1}\{x \in [0, b]\} + {\bf 1}\{x \in [a+b, a+b+c]\}$ in the following sense. The collection of random variables $\{ \sqrt \pi M^L_{\eta, k}\}_{\eta > 0, k \in \mathbb{Z}_{\geq 0}}$ in (\ref{momentDisc}) converges jointly in the sense of moments to  $\{ \mathcal{M}^L_{\eta, k}\}_{\eta > 0, k \in \mathbb{Z}_{\geq 0}}$ in (\ref{momentCont}).
\end{theorem}
\begin{remark}
We mention that \cite[Theorem 3.14]{BuGo} is formulated for much more general domains than just the hexagon. In particular, to specialize the notation there to our setting one needs to replace $N$ with $L$ in the theorem, and set $\lambda(N)$ to equal the partition with $C$ parts equal to $A$ and all other parts equal to $0$. In addition, we mention that the formulation of  \cite[Theorem 3.14]{BuGo} goes through identifying the pushforward of $H_L$ under the map $\Omega^{-1}_{\bf m}$ with the free field $ \mathfrak{G}(z)$. Instead, in the above theorem we followed the notation from Section \cite[Section 4.5]{BorGor} and formulated the result as identifying $H_L$ with the pullback of the free field $\mathfrak{G}(z)$ under the map $\Omega_{\bf m}$. Of course, the two are equivalent.
\end{remark}

In the remainder of the section we explain how our Theorem \ref{CLTfun} relates to Theorem \ref{TBuGo} and we start by giving a different descritpion of $\Omega_{\bf m}$ and $D_{\bf m}$ above. Since ${\bf m}(x) = {\bf 1}\{x \in [0, b]\} + {\bf 1}\{x \in [a+b, a+b+c]\}$ we know that $\exp (- G_{\bf m}(z)) = \frac{(z-b)(z - a - b - c)}{z (z -a - b)}$. We then define the map
\begin{equation}\label{S72DefF}
\mathcal{F}_{\bf m; \eta} (z) := z + \frac{1 - \eta}{ \exp (- G_{\bf m}(z)) - 1} = z + \frac{(1- \eta)z (z -a - b)}{(z-b)(z - a - b - c) - z (z -a - b)}.
\end{equation}
It follows from \cite[Proposition 3.13]{BuGo} that the equation $\mathcal{F}_{\bf m; \eta} (z) = y$ has either $0$ or $1$ root in $\mathbb{H}$ and moreover there is a root in $\mathbb{H}$ if and only if $(y, \eta) \in D_{\bf m}$, and then $\Omega_{\bf m}(y ,\eta)$ is this root. The equation $\mathcal{F}_{\bf m; \eta} (z) = y$ is equivalent to the quadratic equation
\begin{equation}
\eta z^2 + (ac - a\eta + c\eta - \eta - y)z + (1-c)y (a+1) = 0,
\end{equation}
where we used that $b + c = 1$. In particular, we see that 
$$D_{\bf m} = \{ (y, \eta) \in \mathbb{R}^2: \tilde{A} \eta^2 + \tilde{B} y \eta + \tilde{C} y^2 + \tilde{D} \eta + \tilde{E} y + \tilde{F} < 0 \},$$
where 
\begin{equation}
\begin{split}
&\tilde{A} = (1 + a - c)^2, \tilde{B} = 4ac - 2a + 2c - 2, \tilde{C} = 1, \tilde{D} = -2ac(1 + a -c), \tilde{E} = -2ac \mbox{ and } \tilde{F} = a^2 c^2.
\end{split}
\end{equation}
In particular, $D_{\bf m}$ the region enclosed by an ellipse: one can actually show that this ellipse is inscribed in the rescaled hexagon $a \times b \times c$ and $D_{\bf m}$ is typically referred to as {\em the liquid region}, cf. \cite{P2}. If one looks at the vertical slice through $(0,\eta)$ for $\eta \in (0,1)$ then it will intersect the ellipse at two points given by
\begin{equation}\label{endpoints}
a^{\pm}(\eta) = -2 ac \eta + ac + a \eta - c \eta + \eta \pm 2 \sqrt{ac\eta (1- \eta)(1- c)(1 +a)}.
\end{equation}

\begin{definition}\label{DefK}
In the notation of \cite[Section 4.5]{BorGor} we let $\mathcal{K}(y, \eta)$ denote the pullback $\mathfrak{G} \circ \Omega_{\bf m}$ of the GFF $\mathfrak{G}$ on $\mathbb{H}$ under the map $\Omega_{\bf m}$. $\mathcal{K}(y, \eta)$ is a generalized Gaussian field on $D_{\bf m}$ with covariance 
$$\mathbb{E} \left[ \mathcal{K}(\eta_1, y_1) \mathcal{K}(\eta_2, y_2) \right] = - \frac{1}{2\pi} \log \left| \frac{\Omega_{\bf m}(\eta_1, y_1) -\Omega_{\bf m}(\eta_2, y_2)  }{\Omega_{\bf m}(\eta_1, y_1) -\overline{\Omega}_{\bf m}(\eta_2, y_2)} \right|.$$
We can also extend $\mathcal{K}$ to the whole of $\mathbb{R}_+^2$ by setting it to $0$ outside $D_{\bf m}$.
\end{definition}
With respect to the field $\mathcal{K}$ the variables $\mathcal{M}^{\bf m}_{\eta, k}$ in (\ref{momentCont}) can be re-expressed as 
\begin{equation}\label{momentCont2}
\mathcal{M}^{\bf m}_{\eta, k} = \int_{ a^-(\eta)}^{a^+(\eta)} y^k \mathcal{K}(y, \eta) dy, \hspace{5mm} 0 < \eta \leq 1, k \in \mathbb{N}.
\end{equation}
In this sense Theorem \ref{TBuGo} identifies the macroscopically separated $1$-d slices of the height function $ H_L$ with the $1$-d slices of $\mathcal{K}$. On the other hand, in Theorem \ref{CLTfun} we consider observables formed by two {\em adjacent} slices of the model. Let us introduce a height function formulation of these observables.
\begin{definition}\label{DefW}
For $(y, \eta) \in \mathbb{R}_+ \times [L^{-1},1]$ we define $W_L(y, \eta) = L^{1/2} \cdot \left[ H_L(y, \eta) - H_L(y, \eta - L^{-1}) \right]$. 
\end{definition}
Theorem \ref{CLTfun} then leads to the weak convergence of $W_L$ to a ``renormalized derivative" of the random field $\mathcal{K}$ in the following sense.
\begin{theorem}\label{HexDerCLT} Assume the same notation as in Theorem \ref{TBuGo} and fix $\eta \in (0,\min(b,c))$ and $h \in \mathbb{N}$. Then for any integers $k_1, \dots, k_h \geq 0$ the vector 
\begin{equation}\label{PLD1}
\left( \int_{\mathbb{R}_+} y^{k_i} \left( W_L(y, \eta) - \mathbb{E} \left[ W_L(y, \eta) \right] \right) dy \right)_{i = 1}^h 
\end{equation}
as $L\rightarrow \infty$ converges in distribution to a Gaussian vector, which is the same as the weak limit of 
\begin{equation}\label{PLD2}
\lim_{ \delta \rightarrow 0^+} \delta^{-1/2} \left( \int_{\mathbb{R}_+} y^{k_i} \left( \mathcal{K}(y, \eta + \delta) - \mathcal{K} (y, \eta) \right) dy \right)_{i = 1}^h.
\end{equation}
In addition, 
\begin{equation}\label{PLD3}
\left( \int_{\mathbb{R}_+} y^{k_i} \left( H_L(y, \eta) - \mathbb{E} \left[ H_L(y, \eta) \right] \right) dy \right)_{i = 1}^h 
\end{equation}
and (\ref{PLD1}) jointly converge (in distribution) as $L \rightarrow \infty$, while the limit vectors are independent.
\end{theorem}
\begin{remark}\label{S72R1}
Note that (\ref{PLD2}) is defined as a weak limit and may not exist in the probability space of $\mathcal{K}$.
\end{remark}
\begin{remark}\label{S72R2}
An analogue of Theorem \ref{HexDerCLT} has been established in \cite[Theorem 3.13]{GZ} for a special continuous $\beta$-corners process of the form (\ref{eq:gen_cont_beta}) called the $\beta$-Jacobi corners process. We remark that in \cite{GZ} the authors were successful in identifying the {\em joint distribution} of $1$-d slices of the height function $W_L$ with the ``renormalized derivative" of a certain Gaussian field $\mathcal{K}$ {\em on several levels}. Theorem \ref{HexDerCLT} is weaker since we can only access {\em single} $1$-d slices; however, we remark that it is the first of its kind for discrete corners processes.
\end{remark}
\begin{proof}
We split the proof of the theorem into several steps for clarity. \\

{\raggedleft \bf Step 1.} In this step we compute the covariance of the vectors in (\ref{PLD2}). From \cite[Section 9.1]{BuGo} we have that $\mathcal{M}^{\bf m}_{\eta, k}$ as in (\ref{momentCont2}) are jointly zero-centered Gaussian random variables and for $r \leq t$ and $k_r, k_t \in \mathbb{Z}_{\geq 0}$ we have
\begin{equation}\label{NTL1}
\begin{split}
Cov( \mathcal{M}^{\bf m}_{r, k_r},  \mathcal{M}^{\bf m}_{t, k_t}) =& \frac{1}{(2 \pi i)^2 (k_r + 1)(k_t + 1)} \oint_{|z| = 2C} \oint_{|w| = C} \frac{dz dw}{(z - w)^2} \times \\
& \left ( z + \frac{1 - r}{\exp (-G_{\bf m}(z)) - 1}\right)^{k_r + 1} \left (w + \frac{1 - t}{\exp (-G_{\bf m}(w)) - 1}\right)^{k_t + 1},
\end{split}
\end{equation}
where $C$ is a large enough constant so that the circle of radius $C$ contains all the singularities of the integrand. In particular, using (\ref{S72DefF}) we see that $C > a + 1$ will suffice in our case.

From the above it follows that 
$$\delta^{-1/2} \left( \int_{\mathbb{R}_+} y^{k_i} \left( \mathcal{K}(y, \eta + \delta) - \mathcal{K} (y, \eta) \right) dy \right)_{i = 1}^h = \delta^{-1/2} \left(\mathcal{M}^{\bf m}_{\eta + \delta, k_i} - \mathcal{M}^{\bf m}_{\eta, k_i} \right)_{i = 1}^h$$
is a centered Gaussian vector and the covariance is given by
\begin{equation}\label{NTL2}
\begin{split}
&Cov \left( \delta^{-1/2} \left(\mathcal{M}^{\bf m}_{\eta + \delta, k_i} - \mathcal{M}^{\bf m}_{\eta, k_i} \right) ,\delta^{-1/2} \left(\mathcal{M}^{\bf m}_{\eta + \delta, k_j} - \mathcal{M}^{\bf m}_{\eta, k_j} \right)  \right) =  \frac{I_1(\delta) + I_2(\delta)}{(2 \pi i)^2 (k_i + 1)(k_j + 1)} \mbox{, where } \\
&I_1(\delta) = \delta^{-1} \oint_{|z| = 2C} \oint_{|w| = C} \frac{dz dw}{(z - w)^2} \left ( z + \frac{1 - \eta - \delta}{\exp (-G_{\bf m}(z)) - 1}\right)^{k_i + 1} \left (w + \frac{1 - \eta - \delta }{\exp (-G_{\bf m}(w)) - 1}\right)^{k_j + 1} -  \\
& \delta^{-1} \oint_{|z| = 2C} \oint_{|w| = C} \frac{dz dw}{(z - w)^2} \left ( z + \frac{1 - \eta }{\exp (-G_{\bf m}(z)) - 1}\right)^{k_i + 1} \left (w + \frac{1 - \eta - \delta }{\exp (-G_{\bf m}(w)) - 1}\right)^{k_j + 1};\\
&I_2(\delta) = \delta^{-1} \oint_{|z| = 2C} \oint_{|w| = C} \frac{dz dw}{(z - w)^2} \left ( z + \frac{1 - \eta }{\exp (-G_{\bf m}(z)) - 1}\right)^{k_j + 1} \left (w + \frac{1 - \eta  }{\exp (-G_{\bf m}(w)) - 1}\right)^{k_i + 1} -  \\
& \delta^{-1} \oint_{|z| = 2C} \oint_{|w| = C} \frac{dz dw}{(z - w)^2} \left ( z + \frac{1 - \eta  }{\exp (-G_{\bf m}(z)) - 1}\right)^{k_j + 1} \left (w + \frac{1 - \eta - \delta }{\exp (-G_{\bf m}(w)) - 1}\right)^{k_i + 1}.
\end{split}
\end{equation}
In particular, we see that 
\begin{equation*}
\begin{split}
\lim_{\delta \rightarrow 0^+} I_1(\delta) =  & \oint_{|z| = 2C} \oint_{|w| = C} \left ( z + \frac{1 - \eta }{\exp (-G_{\bf m}(z)) - 1}\right)^{k_i } \left (w + \frac{1 - \eta  }{\exp (-G_{\bf m}(w)) - 1}\right)^{k_j+ 1} \times \\
&  \frac{-(k_i + 1)dz dw}{(z - w)^2(\exp (-G_{\bf m}(z)) - 1)}\\
\lim_{\delta \rightarrow 0^+} I_2(\delta) =  & \oint_{|z| = 2C} \oint_{|w| = C} \left ( z + \frac{1 - \eta }{\exp (-G_{\bf m}(z)) - 1}\right)^{k_j + 1 } \left (w + \frac{1 - \eta  }{\exp (-G_{\bf m}(w)) - 1}\right)^{k_i} \times \\
&  \frac{(k_i + 1)dz dw}{(z - w)^2(\exp (-G_{\bf m}(z)) - 1)}.
\end{split}
\end{equation*}
By the Residue theorem we conclude that
\begin{equation}\label{NTL3}
\begin{split}
&\lim_{\delta \rightarrow 0^+}Cov \left( \delta^{-1/2} \left(\mathcal{M}^{\bf m}_{\eta + \delta, k_i} - \mathcal{M}^{\bf m}_{\eta, k_i} \right) ,\delta^{-1/2} \left(\mathcal{M}^{\bf m}_{\eta + \delta, k_j} - \mathcal{M}^{\bf m}_{\eta, k_j} \right)  \right) = \\
&\frac{1}{2\pi i}  \oint_{|w| = C}[\mathcal{F}_{\bf m; \eta} (w)]^{k_i + k_j} \cdot  \frac{ \partial_w \mathcal{F}_{\bf m; \eta} (w)dw}{(\exp (-G_{\bf m}(w)) - 1)},
\end{split}
\end{equation}
where we recall that $\mathcal{F}_{\bf m; \eta} (w)$ was defined in (\ref{S72DefF}). Since the vectors $ \delta^{-1/2} \left(\mathcal{M}^{\bf m}_{\eta + \delta, k_i} - \mathcal{M}^{\bf m}_{\eta, k_i} \right)_{i = 1}^h$ are zero-centered and Gaussian and their covariances converge we conclude that the weak limit in (\ref{PLD2}) exists and is a zero-centered Gaussian vector with covariance given in (\ref{NTL3}).\\

{\bf \raggedleft Step 2.} In this step we prove the joint convergence of (\ref{PLD1}) and (\ref{PLD3}) by appealing to Theorem \ref{CLTfun}. Denote $N = \lfloor \eta L \rfloor$ and observe that
$$\int_{\mathbb{R}_+} y^{k}  H_L(y, \eta) dy = \sum_{i = 1}^N \int_{\ell_{i}}^{\ell_{i + 1}}\frac{ (N-i + 1) y^k}{L^{k+1}} dy = \sum_{i = 1}^N \frac{(N - i + 1)(\ell^{k+1}_{i} - \ell^{k+1}_{i+1}) }{L^{k+1}(k+1)}= \sum_{i = 1}^N \frac{\ell^{k+1}_{i}}{L^{k+1}(k+1)},$$
where $\ell_1 > \ell_2 > \cdots > \ell_N$ are the locations of the horizontal lozenges on the vertical slice through $(0,N)$ and $\ell_{N+1} = 0$. If we furthermore denote by $m_i$ for $i = 1, \dots, N-1$ the locations of the horizontal lozenges on the vertical line through $(0,N-1)$ then
$$\int_{\mathbb{R}_+} y^{k}  W_L(y, \eta) = L^{1/2} \cdot \left[ \sum_{ i = 1}^N\frac{\ell^{k+1}_{i}}{k+1} -\sum_{ i = 1}^{N-1}\frac{m^{k+1}_{i}}{k+1}  \right].$$
In particular, if we denote $f_i(x) = \frac{x^{k_i+1}}{k_i + 1}$ for $i = 1, \dots, h$ then we see that the random variables in (\ref{PLD1}) and (\ref{PLD3}) have the joint distribution of 
\begin{equation*}
\left( (N/L)^{k_i + 1/2}\mathcal{L}^m_{f_i} \right)_{i = 1}^h \mbox{ and }\left( (N/L)^{k_i + 1}\mathcal{L}^t_{f_i} \right)_{i = 1}^h,
\end{equation*}
where $\mathcal{L}^m_{f}$ and $\mathcal{L}^t_{f}$ are as in the statement of Theorem \ref{CLTfun}. It follows from Theorem \ref{CLTfun} that the above vectors converge jointly and in the sense of moments to a $2h$-dimensional centered Gaussian vector $\xi = (\xi_1^m, \dots, \xi_h^m, \xi_1^t, \dots, \xi_h^t)$ such that $Cov(\xi_i^m, \xi_j^t) = 0$ for all $1 \leq i, j \leq h$ and 
\begin{equation}\label{S72Cov1}
Cov(\xi_i^m, \xi_j^m) =  \frac{\eta^{k_i + k_j + 1}}{(2\pi i )^3} \oint_{\Gamma_1} \oint_{\Gamma_1}\int_{\Gamma}  \frac{ f_i(s) f_j(t)  }{e^{ G_\mu(z)} - 1} \cdot \left[ - \frac{1}{(z-s)^2(z-t)^2}\right]dzdsdt,
\end{equation}
where $e^{G_\mu(z)}$ is as in (\ref{S72EGmu}), $\Gamma_1$ and $\Gamma$ are positively oriented contours such that $\Gamma_1$ contains $\Gamma$ in its interior and $\Gamma$ encloses the segment $[0, a+ 1]$. Using the Residue theorem, the formula for $f_i, f_j$ and (\ref{S72EGmu}) we can rewrite (\ref{S72Cov1}) as 
\begin{equation}\label{S72Cov2}
Cov(\xi_i^m, \xi_j^m) = \frac{1}{2\pi i }\int_{\Gamma}  \frac{\eta^{k_i + k_j + 1} \cdot 2(z + c_1 -1)(a_1 + 1 -z) z^{k_i + k_j}dz }{(2 - b_1 - c_1)z + a_1c_1 - a_1+ b_1 + 2c_1 -2  - (b_1 + c_1) \sqrt{(z - a^-) (z- a^+)} },
\end{equation}
where $a^{\pm}$ are as in (\ref{S72Qmu2}). \\

{\bf \raggedleft Step 3.} In this final step we show that the covariance  (\ref{S72Cov2}) agrees with the one in (\ref{NTL3}). We first observe by (\ref{S72DefF}) that $\mathcal{F}_{\bf m; \eta} (w) $ is invertible for $|w|$ large enough and we have
\begin{equation}\label{S72FInv}
\mathcal{F}^{-1}_{\bf m; \eta} (z) := \frac{z + \eta(1+a - c) - ac + \sqrt{(z - a^-(\eta))(z - a^+(\eta))}}{-2\eta },
\end{equation}
where $a^{\pm}(\eta)$ are as in (\ref{endpoints}). Using the above and (\ref{S72DefF}) we can do a change of variables $z = \mathcal{F}_{\bf m; \eta} (w)$ in the right side of (\ref{NTL3}) and rewrite it as
\begin{equation}\label{NTL4}
\begin{split}
&\frac{1}{2\pi i}  \oint_{\eta \cdot \Gamma}  \left[z + \frac{z + \eta(1+a - c) - ac + \sqrt{(z - a^-(\eta))(z - a^+(\eta))}}{2\eta }  \right]  \frac{z^{k_i + k_j} dz}{1 - \eta} = \\
& \frac{1}{2\pi i}  \oint_{\eta \cdot\Gamma} \frac{z^{k_i + k_j}\sqrt{(z - a^-(\eta))(z - a^+(\eta))}  dz}{2 \eta(1 - \eta)}.
\end{split}
\end{equation}
where we applied Cauchy's theorem to deform the contour to $\eta \cdot\Gamma$ and evaluate the analytic part of the integrand to $0$. On the other hand, starting from (\ref{S72Cov2}) we can rationalize the denominator and obtain
\begin{equation*}
 \frac{1}{2\pi i }\int_{\Gamma}  \frac{\eta^{k_i + k_j + 1} \cdot [(2 - b_1 - c_1)z + a_1c_1 - a_1+ b_1 + 2c_1 -2 + (b_1 + c_1) \sqrt{(z - a^-) (z- a^+)} ] z^{k_i + k_j}dz }{2(b_1+c_1-1)}.
\end{equation*}
Using Cauchy's theorem to integrate the analytic part of the above expression and performing the change of variables $w = \eta z$ we get that  (\ref{S72Cov2}) equals
\begin{equation}\label{S72Cov3}
 \frac{1}{2\pi i }\int_{\eta \cdot \Gamma}  \frac{ (b_1 + c_1) \sqrt{( w - \eta a^-) (w-  \eta a^+)} ] w^{k_i + k_j}dw}{2 \eta (b_1+c_1-1)}.
\end{equation}
Finally, (\ref{S72Cov3}) equals (\ref{NTL4}) since $\eta a^{\pm} = a^{\pm}(\eta)$ and $b_1 + c_1 = \eta^{-1}$ (recall $b+c = 1$).
\end{proof}

%
\subsubsection{Quadratic potential}\label{Section7.2.3} In this section we consider the case when $\mathbb{P}_N$ is the probability measure on $\mathfrak{X}^{\theta}_{N,N} $ as in Proposition \ref{PropExtension} with $w(\ell;N) = \exp \left( - \theta  \ell^2/2N\right)$. The quadratic decay of the weight ensures that $Z$ in Proposition \ref{PropExtension} is indeed finite and we conclude that the projection on the top two levels, which we will denote by $(\ell, m)$, is given by
 \begin{equation}\label{S723E1}
\mathbb{P}^\theta_N (\ell, m) = \frac{\Gamma( N  \theta)}{\Gamma(\theta)} \cdot \frac{1}{Z} \cdot H^t(\ell) \cdot H^b(m) \cdot I(\ell, m), \mbox{ where } 
\end{equation}
\begin{align}\label{S723E2}
\begin{split}
H^t(\ell) = &\prod_{1 \leq i < j \leq N} \frac{\Gamma(\ell_i - \ell_j + 1)}{\Gamma(\ell_i - \ell_j + 1 - \theta)}  \prod_{i = 1}^N e^{-\theta \ell_i^2/2N}, \hspace{2mm} H^b(m) = \prod_{1 \leq i < j \leq N-1} \frac{\Gamma(m_i - m_j + \theta)}{\Gamma(m_i - m_j)}, \\
&I(\ell, m) = \prod_{1 \leq i < j \leq N} \frac{\Gamma(\ell_i - \ell_j + 1 - \theta)}{\Gamma(\ell_i - \ell_j) } \cdot \prod_{1 \leq i < j \leq N-1} \frac{\Gamma(m_i - m_j + 1)}{\Gamma(m_i - m_j + \theta)} \\
&\times \prod_{1 \leq i < j \leq N} \frac{\Gamma(m_i - \ell_j)}{ \Gamma(m_i - \ell_j + 1 - \theta)}  \cdot \prod_{1 \leq i \leq j \leq N-1} \frac{\Gamma(\ell_i - m_j + \theta)}{\Gamma(\ell_i - m_j + 1)}.
\end{split}
\end{align}
The measure in (\ref{S723E1}) can be thought of as a discrete analogue of the measure on $\mathcal{G} = \{(x,y) \in \mathbb{R}^{2N-1}:  x_1 < y_1 < x_2 < y_2 < \cdots < y_{N-1} < x_N \}$ with density
\begin{equation}\label{S7twologgas}
f_\beta(x, y) = \frac{1}{Z_\beta^c} \prod_{1 \leq i < j \leq N} (x_j - x_i) \prod_{1 \leq i < j \leq N-1} (y_j - y_i) \prod_{i = 1}^{N-1} \prod_{j = 1}^N |y_i - x_j|^{\beta/2 - 1} \cdot \prod_{i = 1}^N e^{-\beta x_i^2/4N },
\end{equation}
where $Z_\beta^c$ is a normalization constant such that the integral of $f_\beta(x,y)$ over $\mathcal{G}$ is $1$ (as usual $\beta = 2\theta$). Combining the results in \cite[Section 2.5]{AGZ} and \cite[Proposition 1.1]{Ne} one observes that when $\beta = 1$ and $\beta = 2$ the measure in (\ref{S7twologgas}) precisely describes the joint distribution of the eigenvalues of an $N \times N$ random Hermitian matrix sampled from the GOE and GUE respectively (these are the $x$'s) together with the eigenvalues of its $(N-1) \times (N-1)$ corner (these are the $y$'s). Let us elaborate the latter point a bit. Let $\{ \xi_{i,j}, \eta_{i,j}\}_{i,j = 1}^\infty$ be an i.i.d. family of real mean $0$ and variance $1$ Gaussian random variables. When $\beta = 1$ we define the random $N\times N$ matrix $H$, whose entries are given by
$$H_{i,i} = \sqrt{2N} \xi_{i,i} \mbox{ for $i = 1, \dots, N$ and } H_{i,j} = H_{j,i} = \sqrt{N} \xi_{i,j} \mbox{ for $1 \leq i < j \leq N$}.$$
This gives a random symmetric matrix. Since $H$ and its $(N-1) \times (N-1)$ top left corner are both symmetric real matrices their spectra are real and one can show that their law is given by (\ref{S7twologgas}) with $\beta = 1$. When $\beta = 2$ the entries of $H$ are instead given by
$$H_{i,i} =  \xi_{i,i} \mbox{ for $i = 1, \dots, N$ and } H_{i,j} = \overline{H_{j,i}} = \sqrt{N} \cdot \frac{\xi_{i,j} + i \eta_{i,j}}{\sqrt{2}} \mbox{ for $1 \leq i < j \leq N$}.$$
This gives a random Hermitian matrix. Since $H$ and its $(N-1) \times (N-1)$ top left corner are both Hermitian matrices their spectra are real and one can show that their law is given by (\ref{S7twologgas}) with $\beta = 2$. The measures in (\ref{S7twologgas}) for $\beta = 1$ and $\beta = 2$ were studied in \cite{ES} where the authors established the following result.
\begin{proposition}\label{ThmES} Let $(X_1, \dots, X_N, Y_1, \dots, Y_{N-1})$ be a random vector in $\mathcal{G}$ with density given by $f_\beta$ as in (\ref{S7twologgas}). Then we can find $C> 2$ such that the following holds. For a real polynomial $f$ let
$$\mathcal L^{m,C}_{f}= N^{1/2} \cdot \left[\sum_{i = 1}^N \left( \tilde{f}(X_i/N) - \mathbb{E} \left[\tilde{f}(X_i/N) \right] \right)   -   \sum_{i = 1}^{N-1}  \left( \tilde{f}(Y_i/N) - \mathbb{E} \left[\tilde{f}(Y_i/N)  \right] \right) \right]  ,$$
where $\tilde{f}(x) = {\bf 1}\{x \in [-C,C]\} \cdot f(x)$. Then  if $\beta =1$ or $\beta = 2$, as $N \rightarrow \infty$ the random variables $\mathcal L^{m,C}_{f}$ converge in the sense of moments to a real Gaussian variable $\xi^{\beta}$, with 
\begin{equation}\label{ESE1}
\mathbb{E} \left[ \xi^{\beta} \right] = 0 \mbox{ and } \mathbb{E}\left[(\xi^{\beta})^2 \right] = \frac{2}{\beta}  \cdot \int_{-2}^2 f'(x)^2\rho(x) dx,
\end{equation}
where $\rho(x):= \frac{1}{2\pi} \sqrt{4 -x^2}$ is the density of the semicircle law.
\end{proposition}
\begin{remark}
Proposition \ref{ThmES} is a very special case of \cite[Theorem 2.1]{ES}, which considers much more general Wigner matrices and not just the GOE and GUE. Furthermore, we remark that one can take $C = 10$ above and the function $f$ that we took to be polynomial could be taken in a more general Sobolev space. In Proposition \ref{ThmESUs} below we will see that the variance in (\ref{ESE1}) is different for the discrete measures in (\ref{S723E1}).
\end{remark}

As a discrete analogue to Proposition \ref{ThmES} we prove the following result for the measures (\ref{S723E1}).
\begin{proposition}\label{ThmESUs} Fix $\theta \in (0, \pi)$ and let $\mathbb{P}^\theta_N$ be as in (\ref{S723E1}). Then we can find $D> 2$, depending on $\theta$, such that the following holds. For any real polynomials $f_1, \dots, f_n$ and $k = 1, \dots, n$ define
$$\mathcal L^{m,D}_{f_k}=N^{1/2} \cdot \left[\sum_{i = 1}^N \left( \tilde{f}_k(\ell_i/N) - \mathbb{E} \left[\tilde{f}_k(\ell_i/N) \right] \right)  -  \sum_{i = 1}^{N-1}  \left(\tilde{f}_k(m_i/N) - \mathbb{E} \left[ \tilde{f}_k(m_i/N)  \right]\right) \right],$$
where $\tilde{f}(x) = {\bf 1}\{x \in [-D,D]\} \cdot f(x)$. Then as $N \rightarrow \infty$ the random variables $\mathcal L^{m,D}_{f_k}$ converge jointly in the sense of moments to a mean $0$ Gaussian vector $(\xi^\theta_1, \dots, \xi^\theta_n)$, whose covariance is given by
\begin{equation}\label{ESUsE1}
Cov(\xi^{\theta}_i, \xi^{\theta}_j) = \theta^{-1}  \cdot \int_{-2}^2 f'_i(x) f'_j(x) \rho^{\theta}(x) dx,
\end{equation}
where 
$$\rho^{\theta}(x):=  \frac{\theta}{\pi} \cdot  \frac{ e^{\theta x /2} \sin \left((\theta/2) \sqrt{4-x^2} \right)}{e^{\theta x} + 1 - 2 e^{\theta x/2} \cos\left((\theta/2) \sqrt{4-x^2}\right)}.$$
\end{proposition}
\begin{remark} We remark that even when $\theta = \beta/2 = 1$ or $1/2$ we have $\rho^{\theta}(x) \neq \rho(x)$ from Proposition \ref{ThmES}. The latter might seem surprising since by \cite[Corollary 9.4]{BGG} we have that the asymptotic fluctuations of $\sum_{i = 1}^N \tilde{f}(\ell_i/N)  $ are the same as those of $\sum_{i = 1}^N \tilde{f}(X_i/N) $. In \cite[Theorem A.1]{ES} the authors showed that the variable $\xi^{\beta}$ is given by a pairing of $f$ with a suitably normalized derivative of the Gaussian field that describes all Wigner matrices \cite{BorW}. We believe that the same is true for $\{\xi^{\theta}_i\}_{k= 1}^n$, but that the limiting Gaussian field is different. Thus when restricted to the top level, the two fields are the same, but the full 2D structure is different depending on whether one is dealing with a continuous or a discrete multi-level log gas. 
\end{remark}
\begin{remark} As pointed out by one of the referees, one has $\lim_{\theta \rightarrow 0+} \rho^{\theta}(x) = \rho(x)$. At this time, we do not have a good explanation as to why this limit transition recovers the continuous covariance from the discrete one, and it would be interesting to see if it holds for more general models.
\end{remark}
\begin{remark} One can readily check that when $\theta = \beta/2 \in[0,1]$ we have that $\rho^{\theta}(x) \leq \rho(x)$ for all $x \in[-2, 2]$. The fact that $\rho^{\theta}(x) \leq \rho(x)$ in particular shows that the variance of $\xi^{\theta}_i$ from Proposition \ref{ThmESUs} is strictly smaller than that of $\xi^{\beta}$ from Proposition \ref{ThmES}. It would be nice to get a good physical explanation of why the variance in the discrete model is smaller than that of the continuous one.
\end{remark}
\begin{proof}
We split the proof of the proposition into several steps for clarity. \\

{\raggedleft \bf Step 1.} In this step we reformulate the problem so that it fits into the setup of Section \ref{Section3}.

 By \cite[Theorem 10.1]{BGG} we know that there exists $D_1 > 2$ and $C_1 > 0$, depending on $\theta$ alone, such that for each $N \geq 1$ we have
\begin{equation}\label{S7LE1}
\mathbb{P}_N^{\theta}  \left( - N \cdot D_1  \leq \ell_N \leq \ell_1 \leq N \cdot D_1  \right) > 1 - C_1^{-1} \cdot \exp( - N \cdot C_1).
\end{equation}
Let $D_N = N \cdot \lfloor D_1 + 1  \rfloor$  and $M_N = 2 D_N$.  We also take $D = 2 \lfloor D_1 + 1  \rfloor + \theta^{-1}$ and $f_1, \dots, f_n$ to be any real polynomials. If $E_N = \{ - D_N \leq \ell_N \leq \ell_1 \leq D_N\}$ we see from (\ref{S7LE1}) that
\begin{equation}\label{S7TailP}
\mathbb{P}_N^{\theta} (E^c_N) \leq C_1^{-1} \cdot \exp( - N \cdot C_1).
\end{equation}
In particular, we see that for any fixed $A_1, \dots, A_n \in \mathbb{Z}_{\geq 0}$ we have
\begin{equation}\label{S7Cond1}
\mathbb{E}_N^{\theta} \left[  \prod_{k = 1}^n \left(\mathcal L^{m,D}_{f_k} \right)^{A_k} \right] = \mathbb{E}_N^{\theta} \left[  \prod_{k = 1}^n \left(\mathcal L^{m}_{f_k} \right)^{A_k} {\Big \vert} E_N\right] + O \left( \exp\left( - N \cdot C_1/2 \right) \right),
\end{equation}
where $\mathcal L^{m}_{\hat{f}_k}$ are as in Theorem \ref{CLTfun}. Notice that on the right side of (\ref{S7Cond1}) we no longer cut off the functions $f_k$ since conditional on $E_N$ we have $\mathcal L^{m}_{f_k} = \mathcal L^{m,D}_{f_k}$. In addition, we mention that the $O\left( \exp\left( - N \cdot C_1/2 \right) \right) $ was obtained from the tail estimate (\ref{S7TailP}) and the fact that $\mathcal L^{m,D}_{f_k}$ are almost surely polynomially large in $N$. From (\ref{S7Cond1}) we only need to show
\begin{equation}\label{S7Cond2}
\lim_{N \rightarrow \infty} \mathbb{E}_N^{\theta} \left[  \prod_{k = 1}^n \left(\mathcal L^{m}_{f_k} \right)^{A_k} {\Big \vert} E_N\right] = \mathbb{E} \left[ \prod_{k = 1}^n (\xi^{\theta}_k)^{A_k} \right].
\end{equation}

We subsequently consider the measure
\begin{equation}\label{S7HatP}
\hat{\mathbb{P}}_N^{\theta}(\ell, m) = \frac{1}{\hat{Z}_N}\prod_{1 \leq i < j \leq N} \frac{\Gamma(\ell_i - \ell_j + 1)}{\Gamma(\ell_i - \ell_j + 1 - \theta)} \cdot  \prod_{1 \leq i < j \leq N-1} \frac{\Gamma(m_i - m_j + \theta)}{\Gamma(m_i - m_j)} \cdot  \prod_{i = 1}^N e^{-\theta (\ell_i - D_N)^2/2N},
\end{equation}
which is supported on $\mathfrak{X}^\theta_N$ as in (\ref{GenState}) with $M_N$ as above. Note that we have recentered the measure so that $\ell_N \geq 0$ as required from (\ref{GenState}). We also define $\hat{f}_k$ for $k = 1,\dots, k$ through
$$\hat{f}_k(x) = f_k(x - \lfloor D_1 + 1  \rfloor),$$
and observe that 
$$
 \mathbb{E}_N^{\theta} \left[  \prod_{k = 1}^n \left(\mathcal L^{m}_{f_k} \right)^{A_k} {\Big \vert} E_N\right]= \hat{\mathbb{E}}_N^{\theta} \left[  \prod_{k = 1}^n \left(\mathcal L^{m}_{\hat{f}_k} \right)^{A_k}\right],
$$
where $\hat{\mathbb{E}}_N^{\theta}$ is the expectation with respect to $\hat{\mathbb{P}}_N^{\theta}$. Consequently, we reduced the problem to showing
\begin{equation}\label{S7Cond3}
\lim_{N \rightarrow \infty} \hat{\mathbb{E}}_N^{\theta} \left[  \prod_{k = 1}^n \left(\mathcal L^{m}_{\hat{f}_k} \right)^{A_k}\right] = \mathbb{E} \left[ \prod_{k = 1}^n (\xi^{\theta}_k)^{A_k} \right].
\end{equation}

{\raggedleft \bf Step 2.} In this step we show that the measures in (\ref{S7HatP}) satisfy Assumptions 1-5 in Section \ref{Section3}. Assumption 1 holds trivially with $\lM = 2\lfloor D_1 + 1 \rfloor$, and Assumption 2 holds with $V_N(x) = V(x) = \theta (x - \lfloor D_1 + 1  \rfloor)^2/2$. Next notice that by \cite[Chapter 2]{AGZ} we know that when $V(x) =   \theta (x - \lfloor D_1 + 1  \rfloor)^2/2$ the maximizer of the unconstrained variational problem (\ref{energy}) is given by 
$$\mu(x) = {\bf 1} \{x \in [\lfloor D_1 + 1  \rfloor - 2, \lfloor D_1 + 1  \rfloor + 2] \cdot \frac{1}{2\pi}\sqrt{4-(x- \lfloor D_1 + 1  \rfloor)^2},$$
which is the semicircle law centered at $\lfloor D_1 + 1  \rfloor$. Since $\theta <  \pi $ and $\lM +\theta >\lM > 2+ \lfloor D_1 + 1  \rfloor $ we see that $\mu(x) $ also satisfies the constraints that it is supported in $[0, \lM + \theta]$ and $0 \leq \mu(x) \leq \theta^{-1}$ and so it is also the constrained maximizer of (\ref{energy}) as in Proposition \ref{LLN}. 

We next have that 
$$\frac{w(Nx ;N)}{w(Nx - 1;N)} =\exp \left( -\theta (2Nx - 2D_N - 1)/2N \right) = \frac{\Phi^+_N(Nx)}{\Phi^-_N(Nx)},$$
where 
$$\Phi^+_N(Nx) =  \exp \left( -\theta (2Nx - 2D_N - 1)/2N \right) \mbox{ and } \Phi_N^-(Nx) = 1.$$
In particular, we see that Assumption 3 also holds with the above choice of $\Phi_N^{\pm}$ and then one readily observes that $\Phi^+(x) = \exp \left( - (x- \lfloor D_1 + 1 \rfloor )\right)$ and $\Phi^-(x) = 1$. From (\ref{S7TailP}) we also have that 
$$\hat{\mathbb{P}}_N^{\theta}(\ell_N = 0) = O\left(\exp ( - C_1 \cdot N)\right) = \hat{\mathbb{P}}_N^{\theta}(\ell_1 = M_N + (N-1)\cdot \theta),$$
so that Assumption 4 holds as well. Finally, Assumption 5 was shown to hold in the proof of \cite[Lemma 9.4]{BGG}. \\

 {\raggedleft \bf Step 3.} From Step 2 we know that Assumptions 1-5 hold for the measure $\hat{\mathbb{P}}_N^{\theta}$ and so we conclude from Theorem \ref{CLTfun} that $\left(\mathcal L^{m}_{f_k} \right)_{k = 1}^n$ converge jointly in the sense of moments to a centered Gaussian vector $(\hat{\xi}_1, \dots, \hat{\xi}_n)$ with covariance given by
\begin{align}\label{S7CovUs}
\begin{split}
&Cov(\hat{\xi}_i, \hat{\xi}_j) = \frac{1}{(2\pi \i )^2} \oint_{\Gamma} \oint_{\Gamma} \hat{f}_i(s) \hat{f}_j(t) \Delta \mathcal{C}_{\theta, \mu}(s,t)dsdt, \mbox{ for $1 \leq i, j \leq n$} \mbox{, where }\\
&\Delta \mathcal{C}_{\theta, \mu}(z_1,z_2) = \frac{1}{2\pi \i}\int_{\Gamma_1}  \frac{dz}{e^{\theta G_\mu(z)} - 1} \cdot \left[ - \frac{1}{(z-z_2)^2(z-z_1)^2}\right], 
\end{split}
\end{align}
and $\Gamma_1, \Gamma$ are positively oriented contours such that $\Gamma$ encloses $\Gamma_1$, and $\Gamma_1$ encloses the interval $[0, \lM + \theta]$. What remains is to show that the covariances in (\ref{S7CovUs}) and (\ref{ESUsE1}) agree. By Cauchy's theorem we can evaluate the $\Gamma$ integrals as the residue at $s= z$ and $t = z$, which gives
$$Cov(\hat{\xi}_i, \hat{\xi}_j)  = \frac{-1}{2\pi \i}\int_{\Gamma_1}  \frac{ \hat{f}'_i(z) \hat{f}'_j(z)dz}{e^{\theta G_\mu(z)} - 1}.$$
We next perform the change of variables $w = z -  \lfloor D_1 + 1 \rfloor$ and use that
$$G_\mu(w + \lfloor D_1 + 1 \rfloor) = \int_{-2 +  \lfloor D_1 + 1\rfloor}^{2+  \lfloor D_1 + 1\rfloor} \frac{\rho(x - \lfloor D_1 + 1 \rfloor) dx}{w - (x -  \lfloor D_1 + 1 \rfloor)} = \int_{-2}^2 \frac{\rho(x) dx}{w - x} = G_\rho(w),$$
where $\rho(x)= \frac{1}{2\pi} \sqrt{4 -x^2}$ is the density of the usual semicircle law to get
$$Cov(\hat{\xi}_i, \hat{\xi}_j)  = \frac{-1}{2\pi \i}\int_{\Gamma_2}  \frac{ f'_i(w) f'_j(w)dw}{e^{\theta G_\rho(w)} - 1},$$
where $\Gamma_2$ encloses the interval $[- \lfloor D_1 + 1 \rfloor, \lM + \theta -  \lfloor D_1 + 1 \rfloor]$. From \cite[(2.4.7)]{AGZ} we know that 
$$G_\rho(w) = \frac{w - \sqrt{w^2 - 4}}{2},$$
with the square root as in Section \ref{Section1.5}, and so for each $x \in [-2,2]$ we have
$$\lim_{ \epsilon \rightarrow 0^{\pm}} \frac{1}{e^{\theta G_\rho(x\pm \i \epsilon )} - 1} = \frac{1}{e^{\theta x/2} \cdot  \cos\left((\theta/2) \sqrt{4-x^2}\right) \mp \i e^{\theta x/2} \sin \left((\theta/2) \sqrt{4-x^2} \right) - 1 },$$
while for $|x| > 2$ we have
$$\lim_{ \epsilon \rightarrow 0^{\pm}} \frac{1}{e^{\theta G_\rho(x\pm \i \epsilon )} - 1} = \frac{1}{e^{\theta x/2 - \theta\sqrt{x^2- 4}/2}- 1 }.$$
We may then deform $\Gamma_2$ to a thin rectangle that encloses $[-2,2]$, without affecting the value of the integral by Cauchy's theorem. Shrinking the width of the rectangle to $0$ we traverse the interval $[-2,2]$ once in each direction and we see that the real part is taken with the same sign, and so cancels, while the imaginary part has opposite sign in the two directions that we traverse the interval. Consequently, we obtain
$$Cov(\hat{\xi}_i, \hat{\xi}_j) =  \frac{1}{\pi}  \cdot \int_{-2}^2 \frac{  e^{\theta x/2} \sin \left((\theta/2) \sqrt{4-x^2} \right) f'_i(x) f'_j(x)dx }{\left[e^{\theta x/2} \cos\left((\theta/2) \sqrt{4-x^2}\right) - 1\right]^2 + e^{\theta x} \sin^2 \left((\theta/2) \sqrt{4-x^2} \right)  },$$
which we identify as (\ref{ESUsE1}) once we expand the square in the denominator.

\end{proof}

%
\section{Continuous limit} \label{Section6} The purpose of this section is to derive certain two-level analogues of the loop equations in \cite{BoGu} for natural two-level extensions of the measures considered in that paper. In Section \ref{Section6.1} we formulate the two-level measures we consider and explain how it generalizes the usual $\beta$-log gas. In Section \ref{Section6.2} we derive the continuous measures from Section \ref{Section6.1} as diffuse limits of the measures in Section \ref{Section1.1}. In Section \ref{Section6.3} we derive the loop equations in \cite{BoGu} from the single level Nekrasov's equations -- Proposition \ref{SingleLevelNekrasov}. In Section \ref{Section6.4} we derive the continuous limits of our Nekrasov's equations -- Theorems \ref{TN1} and \ref{TN1Theta1}.

%
\subsection{Two-level log gas}\label{Section6.1}
Let us fix $N \geq 2$, $a_-,a_+ \in \mathbb{R}$ with $a_- < a_+$ and $\theta > 0$. In addition, we let $V^t(z)$ and $V^b(z)$ be two analytic function in a neighborhood $\mathcal{M}$ of $[a_-,a_+]$, which are real-valued on $\mathcal{M} \cap \mathbb{R}$. With this data we define the following probability density function
\begin{equation}\label{twologgas}
f(x, y) = \frac{1}{Z^c} \prod_{1 \leq i < j \leq N} (x_j - x_i) \prod_{1 \leq i < j \leq N-1} (y_j - y_i) \prod_{i = 1}^{N-1} \prod_{j = 1}^N |y_i - x_j|^{\theta - 1} \times\prod_{i = 1}^{N-1} e^{-N\theta V^b(y_i)} \prod_{i = 1}^N e^{-N\theta V^t(x_i)},
\end{equation}
where the density $f(x,y)$ is supported on the set $\mathcal{G} = \{(x,y) \in \mathbb{R}^{2N-1}: a_- < x_1 < y_1 < x_2 < y_2 < \cdots < y_{N-1} < x_N < a_+\}$ and $Z^c$ is a normalization constant such that the integral of $f(x,y)$ over $\mathcal{G}$ is $1$. As mentioned in the introduction, $x_i$, $y_j$ are labeled in increasing order (unlike the $\ell_i$, $m_j$ in Section \ref{Section3}) as is typical in the random matrix literature.

 Observe that, the above density is well defined since by a version of the Dixon-Anderson identity \cite{Dixon, Anderson} (see \cite[Equation (2.2)]{FW08}) we have
\begin{equation}\label{Dix}
\int_{x_1}^{x_2} \cdots \int _{x_{N-1}}^{x_N}  \hspace{-4mm} dy_1 \cdots dy_{N-1} \hspace{-4mm}\prod_{1 \leq i < j \leq N-1} (y_j - y_i) \prod_{i = 1}^{N-1} \prod_{j = 1}^N |y_i - x_j|^{\theta - 1} = \frac{\Gamma(\theta)^{N}}{\Gamma(N \theta)} \cdot \prod_{1 \leq i < j \leq N}(x_j - x_i)^{2\theta - 1},
\end{equation}
which implies that $Z^c < \infty$. The formula (\ref{Dix}) implies further that if $V^b(z) = 0$ then the projection of the measures (\ref{twologgas}) to the top level $(x_1, \dots, x_N)$ has density
\begin{equation}\label{S6Singleloggas}
f(x) ={\bf 1}\{ a_- < x_1 < \cdots < x_N < a_+\} \cdot (Z^t_N)^{-1}\prod_{1 \leq i < j \leq N} (x_j - x_i)^{2\theta} \prod_{i = 1}^N e^{-N\theta V^t(x_i)}.
\end{equation}
The measures in (\ref{S6Singleloggas}) are the same as those studied in \cite{BoGu} once one sets $\theta = \beta/2$ and so the ones in (\ref{twologgas}) can be thought of as their natural generalizations. 

Let $(X_1, \dots, X_N, Y_1, \dots, Y_{N-1})$ be a random $2N-1$ dimensional vector with density given by (\ref{twologgas}). For $z \in \mathbb{C} \setminus [a_-, a_+]$ we denote 
\begin{equation}\label{S6GcontDef}
G^t_c(z) =  \sum_{ i = 1}^N \frac{1 }{z - X_i} \mbox{ and } G^b_c(z) =  \sum_{ i = 1}^{N-1} \frac{1 }{z - Y_i} .
\end{equation}

We recall \cite[Theorem 3.1 and Theorem 3.2]{BoGu} below as Proposition \ref{SingleLevelLoppThm}. The identification is made once we set $\theta = \beta/2$. Below we write $\llbracket p, q \rrbracket$ to mean the set $\{p, p+1, \dots ,q\}$ for integers $p \leq q$.
\begin{proposition}\label{SingleLevelLoppThm}
Fix $\theta > 0$. Let $(X_1, \dots, X_N, Y_1, \dots, Y_{N-1})$ be a random $2N-1$ dimensional vector with density given by (\ref{twologgas}) with $V_b \equiv 0$ so that $(X_1, \dots, X_N)$ has density (\ref{S6Singleloggas}). Given $ v_1, \dots, v_m \in \mathbb{C} \setminus [a_-, a_+]$ we define
\begin{equation}\label{S6ContCumSL}
 \M( v_1, \dots, v_m) = M(G^t_c(v_1), \dots, G^t_c(v_m)) ,
\end{equation}
where we recall that for $m$ bounded random variables $\xi_1, \dots, \xi_m$, $M(\xi_1, \dots, \xi_m)$ stands for their joint cumulant if $m \geq 2$ and $\mathbb{E}[\xi_1]$ if $m = 1$. Then for any $v \in \mathbb{C} \setminus [a_-, a_+]$ the following rank $1$ loop equation holds
\begin{align}\label{S6rank1}
\begin{split}
& 0 = \frac{ N[1 -\theta^{-1}] - N^2}{(v-a_-)(v -a_+)} + \M(v,v) + \M(v)^2 +  [ 1- \theta^{-1}]  \partial_{v} \M(v)- \\
& - \frac{N}{2\pi i} \int_\Gamma dz  \frac{   (z-a_-)(z -a_+)  \M(z) \partial_zV^t(z) }{(v-a_-)(v -a_+) (v- z) }.
\end{split}
\end{align}
Also for $m \geq 1$ and $v, v_1, \dots, v_m \in \mathbb{C} \setminus [a_-, a_+]$ the following rank $(m+1)$ loop equation holds
\begin{align}\label{S6rankm}
\begin{split} 
&0 = \M(v, v,\llbracket 1, m \rrbracket) + \sum_{J \subseteq \llbracket 1, m \rrbracket} \M(v, J) \cdot \M(v, \llbracket 1,m \rrbracket \setminus J) +  [1-  \theta^{-1}] \partial_{v}\M(   v ,\llbracket 1, m \rrbracket ) - \\
&  - \frac{ N}{2\pi i}\int_{\Gamma} dz \frac{(z - a_-)(z-a_+)  \partial_zV^t(z) \M \left(  z,  \llbracket 1, m \rrbracket \right) }{(v - z)(v - a_-)(v - a_+)} + \\
&+ \theta^{-1} \sum_{a = 1}^m  \partial_{v_a}  \left[   \frac{\M\left(  v,  \llbracket 1,m \rrbracket \setminus \{a \} \right)}{v-v_a}  - \frac{ ( v_a-a_-)(v_a - a_+) \M\left( v_a, \llbracket 1,m \rrbracket \setminus \{a \} \right) }{(v - v_a)(v - a_-)(v - a_+) } \right],
\end{split}
\end{align}
where $\Gamma$ is a positively oriented contour, which encloses the segment $[a_-, a_+]$, is contained in $\mathcal{M}$ as in the beginning of Section \ref{Section6.1} and excludes the points $v, v_1, \dots, v_m$.
\end{proposition}
In Section \ref{Section6.3} we deduce Proposition \ref{SingleLevelLoppThm} from a limit of the single level Nekrasov's equations -- Proposition  \ref{SingleLevelNekrasov}. We remark that Proposition \ref{SingleLevelLoppThm} was proved in \cite{BoGu} using different techniques. Nevertheless we present our proof using single level Nekrasov's equations as it is new and in our opinion of sufficient conceptual importance. \\

We next state the main result in this section, which is a certain two-level analogue of the above loop equations for the measures in (\ref{twologgas}). 
\begin{theorem}\label{TwoLevelLoopThm}
Fix $\theta > 0$. Let $(X_1, \dots, X_N, Y_1, \dots, Y_{N-1})$ be a random $2N-1$ dimensional vector with density given by (\ref{twologgas}). Given $m,n \geq 0$ such that $m+n \geq 1$ points $v^1_1, \dots, v^1_m ,v^2_1,\dots, v^2_n \in \mathbb{C} \setminus [a_-, a_+]$ we define
\begin{equation}\label{S6CumTwo}
\begin{split}
&\M(v^1_1, \dots, v^1_m;v^2_1, \dots, v^2_n ) = M(G^t_c(v^1_1), \dots, G^t_c(v^1_m), G^b_c(v^2_1), \dots, G^b_c(v^2_n)).
\end{split}
\end{equation}
For any $v \in \mathbb{C} \setminus [a_-, a_+]$ the following rank $(0,0)$ loop equation holds
\begin{align}\label{TLRank00}
\begin{split} 
& 0 = \frac{N\theta}{2\pi \i}\int_{\Gamma}dz \frac{ (z - a_-)(z - a_+) }{(z-v)(v - a_-) (v-a_+)} \left[  \M (z; \varnothing)  \partial_z V^t(z) +   \M (\varnothing; z)  \partial_z V^b(z) \right]     \\
& -\frac{N^2 - (1-\theta)N(N-1)}{(v - a_-) (v-a_+)}  +  \frac{\M ( v,v; \varnothing ) + \M( v; \varnothing)^2}{2} +  \frac{ \M (\varnothing; v,v) + \M(\varnothing; v)^2}{2} \\
&   - \frac{\partial_v \M (\varnothing; v)}{2}  - \frac{\partial_v \M (v;\varnothing)}{2} -  (1-\theta) [\M(v;v) + \M(v;\varnothing)\M(\varnothing;v)].
\end{split}
\end{align}
Also for $v, v^1_1, \dots, v^1_m ,v^2_1,\dots, v^2_n \in \mathbb{C} \setminus [a_-, a_+]$ the following rank $(m,n)$-loop equation holds
\begin{align}\label{TLRankmn}
\begin{split} 
& 0  =\frac{N\theta }{2\pi \i}\int_{\Gamma} dz \frac{ (z - a_-)(z - a_+) [\partial_z V^t(z) \M \left(z,\llbracket 1,m \rrbracket ; \llbracket 1,n \rrbracket\right) +\partial_z V^b(z) \M \left(\llbracket 1,m \rrbracket ;z, \llbracket 1,n \rrbracket\right)] }{(z-v)(v - a_-)(v - a_+)}   \\
& - \frac{\partial_v \M(v, \llbracket 1,m \rrbracket; \llbracket 1,n \rrbracket) + \partial_v \M( \llbracket 1,m \rrbracket; v, \llbracket 1,n \rrbracket)}{2} + \frac{\M(v,v, \llbracket 1,m \rrbracket; \llbracket 1,n \rrbracket) + \M( \llbracket 1,m \rrbracket; v,v, \llbracket 1,n \rrbracket)}{2}  \\
& -(1-\theta) \cdot \M(v, \llbracket 1,m \rrbracket; v, \llbracket 1,n \rrbracket) +  \frac{1}{2} \sum_{J^t \subseteq \llbracket 1,m \rrbracket} \sum_{J^b\subseteq \llbracket 1,n \rrbracket} \M\left(v, J^t; J^b \right) \cdot \M \left(v, \llbracket 1,m \rrbracket \setminus J^t;  \llbracket 1,m \rrbracket \setminus J^b \right) \\
&  +  \M\left( J^t; v, J^b \right) \cdot \M \left( \llbracket 1,m \rrbracket \setminus J^t; v,  \llbracket 1,m \rrbracket \setminus J^b \right) - 2 (1-\theta) \M\left( v, J^t; J^b \right) \cdot \M \left( \llbracket 1,m \rrbracket \setminus J^t; v,  \llbracket 1,m \rrbracket \setminus J^b \right)  \\
& + \sum_{a = 1}^m  \partial_{v^1_a}  \left[   \frac{\M\left( v, \llbracket 1,m \rrbracket \setminus \{a \}; \llbracket 1, n \rrbracket \right)}{v-v_a^1}  + \frac{ ( v_a^1-a_-)(v^1_a - a_+) \M\left(    \llbracket 1,m \rrbracket ; \llbracket 1, n \rrbracket \right) }{( v^1_a - v)(v - a_-)(v - a_+) } \right] \\
& + \sum_{b = 1}^n  \partial_{v^2_b}  \left[   \frac{\M\left( \llbracket 1,m \rrbracket; v, \llbracket 1, n \rrbracket \setminus \{b \} \right)}{v-v_b^2}  + \frac{ ( v_b^2-a_-)(v_b^2 - a_+) \M\left(    \llbracket 1,m \rrbracket , \llbracket 1, n \rrbracket  \right) }{( v_b^2 - v)(v - a_-)(v - a_+) } \right].
\end{split}
\end{align}
where $\Gamma$ is a positively oriented contour, which encloses the segment $[a_-, a_+]$, is contained in $\mathcal{M}$ as in the beginning of Section \ref{Section6.1} and excludes the points $v, v^1_1, \dots, v^1_m ,v^2_1,\dots, v^2_n $. In (\ref{TLRankmn}) a set $A$ that appears before the semi-colon in $\M$ should be replaced with $\{v_a^1\}_{a \in A}$ and, similarly, a set $B$ that appears after the semi-colon should be replaced with $\{v_b^2\}_{b \in B}$.
\end{theorem}

%
\subsection{Diffuse limits}\label{Section6.2}

In this section we derive the measures in (\ref{twologgas}) as diffuse limits of the measures from Section \ref{Section1.1}. We start by introducing some notation.
Let $L \in \mathbb{N}$ be sufficiently large so that $(a_+ - a_-) \cdot L > 1$. For all such $L$ we define the measures
\begin{equation}\label{PLDef}
\mathbb{P}_{N,L}(\ell, m) = (Z^d_L)^{-1}\cdot H^t(\ell) \cdot H^b(m) \cdot I(\ell, m), \mbox{ where }
\end{equation}
\begin{align}\label{PLDef2}
\begin{split}
H^t(\ell) = &\prod_{1 \leq i < j \leq N} \frac{\Gamma(\ell_i - \ell_j + 1)}{\Gamma(\ell_i - \ell_j + 1 - \theta)}  \prod_{i = 1}^Nw(\ell_i;L) \mbox{ with } w(\ell_i;L)  = e^{-N \theta V^t(\ell_i/L)}, \\
 H^b(m) = &  \prod_{1 \leq i < j \leq N-1} \frac{\Gamma(m_i - m_j + \theta)}{\Gamma(m_i - m_j)} \prod_{i = 1}^{N-1} \tau(m_i;L ) \mbox{ with }\tau(m_i;L )  = e^{-N \theta V^b(m_i/L)}, \\
I(\ell, m) = &\prod_{1 \leq i < j \leq N} \frac{\Gamma(\ell_i - \ell_j + 1 - \theta)}{\Gamma(\ell_i - \ell_j) } \cdot \prod_{1 \leq i < j \leq N-1} \frac{\Gamma(m_i - m_j + 1)}{\Gamma(m_i - m_j + \theta)} \times\\
&\prod_{1 \leq i < j \leq N} \frac{\Gamma(m_i - \ell_j)}{ \Gamma(m_i - \ell_j + 1 - \theta)}  \cdot \prod_{1 \leq i \leq j \leq N-1} \frac{\Gamma(\ell_i - m_j + \theta)}{\Gamma(\ell_i - m_j + 1)}.
\end{split}
\end{align}
In the above formula $\ell_i = \lambda_i + (N  - i) \theta$ for $i = 1, \dots, N$ and $m_i = \mu_i + (N - i) \theta$ for $i = 1, \dots, N-1$, with $\lambda_i, \mu_j \in \mathbb{Z}$ and the measure is supported on $(2N - 1)$-tuples such that $\lceil a_- L \rceil \leq \lambda_N \leq \mu_{N-1} \leq \lambda_{N-2} \leq \cdots \leq \mu_1 \leq \lambda_1 \leq \lfloor a_+ L \rfloor$.
If $(\ell, m)$ satisfy the above inequalities we write $\ell \succeq m$ and denote the set of such tuples by $\mathfrak{X}_{N,L}$. Throughout this section we will frequently switch from $\ell_i$'s to $\lambda_i$'s and from $m_i$'s to $\mu_i$'s without mention using the formulas 
\begin{equation}\label{S6Coord}
\ell_i = \lambda_i + (N- i)\cdot\theta \mbox{ and } m_i = \mu_i + (N  - i) \cdot \theta. 
\end{equation}

We turn to the main result of the section.
\begin{proposition} \label{prop_cont_limit} Fix $\theta > 0$ and $N \geq 2$. Let $\left( \ell^L_1, \cdots, \ell^L_N ,  m_1^L, \cdots, m^L_{N-1} \right)$ be a sequence of random $2N-1$ dimensional vectors, whose probability distribution is $\mathbb{P}_{N,L}$ as in (\ref{PLDef}). Then the sequence 
$$\left(L^{-1} \cdot \ell^L_1, \cdots, L^{-1} \cdot \ell^L_N , L^{-1} \cdot m_1^L, \cdots, L^{-1} \cdot m^L_{N-1} \right)$$
 converges weakly as $L \rightarrow \infty$ to $(X_N, \cdots, X_1, Y_{N-1} \cdots, Y_N)$ where $(X_1, \dots, X_N, Y_1, \dots, Y_{N-1})$ is a random $(2N-1)$-dimensional vector with density given by (\ref{twologgas}).
\end{proposition}
\begin{proof} 
Throughout the proof we use that for $x \geq \min(\theta, 1)$ 
\begin{equation}\label{Sandwich}
\frac{\Gamma(x+ \theta)}{\Gamma(x)} = x^{\theta} \cdot \exp(O(x^{-1})),
\end{equation}
where the constant in the big $O$ depends on $\theta$ alone, see \cite{TE}. For clarity we split the proof into several steps.\\

{\bf \raggedleft Step 1.} In this step we show that we can find a constant $\hat{C}$ and depending on $(a_+ - a_-), N , \theta$ such that if $L (a_+ - a_-) > 1$ and $(\ell, m) \in \mathfrak{X}_{N,L}$ we have
\begin{align}\label{S6HBound}
\begin{split}
&H(\ell, m) \cdot L^{-[(N-1)^2  + (\theta - 1) \cdot N(N-1)]} \leq \hat{C} \mbox{ if } \theta \geq 1 \mbox{ and }H(\ell, m) \cdot L^{-[(N-1)^2  + (\theta - 1) \cdot N(N-1)]} \leq  \\
& \leq \hat{C}   \prod_{i = 1}^{N-1} \left[ \left(\frac{m_i - \ell_{i+1} + 1 - \theta}{L} \right)^{\theta - 1} + \left( \frac{\ell_i - m_i+1}{L} \right)^{\theta -1} \right] \mbox{ if } \theta \in (0,1), \mbox{ where }
\end{split}
\end{align}
\begin{equation}\label{HDef}
H(\ell, m) =  \prod_{1 \leq i < j \leq N}\frac{ (\ell_i - \ell_j)  \Gamma(m_i - \ell_j)}{ \Gamma(m_i - \ell_j + 1 - \theta)} \prod_{1 \leq i < j \leq N-1} (m_i - m_j)   \prod_{1 \leq i \leq j \leq N-1} \frac{\Gamma(\ell_i - m_j + \theta)}{\Gamma(\ell_i - m_j + 1)}.
\end{equation}

 Using $\Gamma(z+1) = z \Gamma(z)$ and (\ref{Sandwich}) we conclude that for $1 \leq i < j \leq N$ we have
\begin{equation}\label{mixed1}
 \frac{\Gamma(m_i - \ell_j)}{ \Gamma(m_i - \ell_j + 1 - \theta)}  = \frac{1}{m_i - \ell_j }   \frac{\Gamma(m_i - \ell_j + 1)}{ \Gamma(m_i - \ell_j + 1 - \theta)} = \frac{(m_i - \ell_j + 1 - \theta)^{\theta}}{m_i - \ell_j}  \exp (O(|m_i - \ell_j|^{-1} ) .
\end{equation}
Analogous considerations show that if $1 \leq i \leq j \leq N-1$ we have
\begin{equation}\label{mixed2}
  \frac{\Gamma(\ell_i - m_j + \theta)}{\Gamma(\ell_i - m_j + 1)} = \frac{1}{\ell_i - m_j + \theta} \cdot    \frac{\Gamma(\ell_i - m_j + \theta + 1)}{\Gamma(\ell_i - m_j + 1)} = \frac{(\ell_i - m_j + 1)^{\theta }}{\ell_i - m_j+ \theta} \cdot \exp (O(|\ell_i - m_j  + 1|^{-1} ) .
\end{equation}

If $\theta \geq 1$ we have from (\ref{HDef}), (\ref{mixed1}) and (\ref{mixed2}) that for some $\hat{C}_1 > 0$ depending on $\theta$ and $N$
\begin{equation}\label{HUB1}
H(\ell, m) \leq \hat{C}_1 \cdot  \prod_{1 \leq i < j \leq N} (\ell_i - \ell_j)  (m_i - \ell_j)^{\theta - 1}  \prod_{1 \leq i < j \leq N-1} (m_i - m_j)   \prod_{1 \leq i \leq j \leq N-1} (\ell_i - m_j + 1)^{\theta -1 }.
\end{equation}
Now each of the factors on the right side of (\ref{HUB1}) is upper bounded by $(a_+ - a_-)L + N\theta + 1$ and so (\ref{HUB1}) implies (\ref{S6HBound}) when $\theta \geq 1$.\\

We next suppose that $\theta \in (0,1)$. Observe that in this case we have
$$\frac{m_i - \ell_j + 1 - \theta}{m_i - \ell_j} \leq 1 + \frac{1 - \theta}{\theta} \mbox{ for $1 \leq i < j \leq N$ and } \frac{\ell_i - m_j + 1}{\ell_i - m_j+ \theta} \leq 1 + \frac{1 - \theta}{\theta}  \mbox{ for $1 \leq i \leq j \leq N-1$}.$$ 
Combining the latter with (\ref{HDef}), (\ref{mixed1}) and (\ref{mixed2}) we see that we can find a constant $\hat{C}_2 > 0$ depending on $\theta$ and $N$ such that 
\begin{align}\label{HUB2}
\begin{split}
&H(\ell, m) \leq \hat{C}_2 \cdot H_N(\ell,m) \mbox{ where }H_N(\ell,m) =  \prod_{1 \leq i < j \leq N} (\ell_i - \ell_j) (m_i - \ell_j + 1 - \theta)^{\theta - 1} \\
&  \times \prod_{1 \leq i < j \leq N-1} (m_i - m_j)  \prod_{1 \leq i \leq j \leq N-1} (\ell_i - m_j + 1)^{\theta -1 }.
\end{split}
\end{align}
We will prove that for each $N \geq 2$ we can find a constant $C_N > 0$ depending on $(a_+ - a_-), N , \theta$ such that if $(a_+ - a_-) \cdot L > 1$ and $(\ell, m) \in \mathfrak{X}_{N,L}$ we have
\begin{equation}\label{HUB3}
\begin{split}
 L^{-[(N-1)^2  + (\theta - 1) \cdot N(N-1)]}  H_N(\ell,m)  \leq C_N  \prod_{i = 1}^{N-1} \left[ \left(\frac{m_i - \ell_{i+1} + 1 - \theta}{L} \right)^{\theta - 1} \hspace{-5mm}+ \left( \frac{\ell_i - m_i+1}{L} \right)^{\theta -1} \right].
\end{split}
\end{equation}
The $\theta \in (0,1)$ case in (\ref{S6HBound}) is then a consequence of (\ref{HUB2}) and (\ref{HUB3}) and in the remainder we establish (\ref{HUB3}).\\

We prove (\ref{HUB3}) by induction on $N \geq 2$. If $N = 2$ we have by definition
\vspace{-1mm}
$$ H_2(\ell, m) \leq (\ell_1 - \ell_2 + 2 - \theta) \cdot (m_1 - \ell_2 + 1 - \theta)^{\theta - 1} \cdot (\ell_1 - m_1+1)^{\theta - 1} = $$
$$(m_1 - \ell_2 + 1 - \theta)^{\theta } \cdot (\ell_1 - m_1+1)^{\theta - 1} + (m_1 - \ell_2 + 1 - \theta)^{\theta - 1} \cdot (\ell_1 - m_1+1)^{\theta }.$$
We note that we can find a constant $C > 0$ depending on $\theta, (a_+ - a_-)$ such that for all $L \geq 1$
$$\left( \frac{m_1 - \ell_2 + 1 - \theta}{L} \right)^{\theta } \leq C \mbox{ and }\left( \frac{\ell_1 - m_1+1}{L} \right)^{\theta } \leq C.$$
Combining the last two statements then gives (\ref{HUB3}) for $N = 2$ with $C_2 = 2 C$.

Suppose we have proved (\ref{HUB3}) for $N$ and wish to establish it for $N+1$. Let us fix $(\ell, m) \in \mathfrak{X}_{N+1,L}$ and write $\ell = ( \ell_1, \tilde{\ell})$, $m = (m_1, \tilde{m})$ with $\tilde{\ell} = (\ell_2, \dots, \ell_{N+1})$ and $\tilde{m} = (m_2 , \dots, m_{N})$. Observe that $(\tilde{\ell}, \tilde{m}) \in \mathfrak{X}_{N,L}$ and also 
\vspace{-1mm}
\begin{align}\label{S6recursion}
\begin{split}
&H_{N+1}(\ell, m) = H_{N} (\tilde{\ell}, \tilde{m}) \cdot (\ell_1 - \ell_2) \cdot(m_1 - \ell_2 + 1 - \theta)^{\theta - 1} \cdot (\ell_1 - m_1 + 1)^{\theta - 1} \times \\
& \prod_{i = 2}^{N} (\ell_1 - \ell_{i+1})(m_1 - m_i)(m_1 - \ell_{i+1} + 1 - \theta)^{\theta - 1}(\ell_1 - m_i + 1)^{\theta - 1}  .
\end{split}
\end{align}
We now observe by the interlacing condition $\ell \succeq m$ we have
$$0 \leq (\ell_1 - m_1)(m_i - \ell_{i+1}) + (m_i - \ell_{i+1} + 1 - \theta) + (1-\theta)(\ell_1 - m_i ),$$
which implies
\begin{equation}\label{S6Interlace}
\frac{(\ell_1 - \ell_{i+1})(m_1 - m_i)}{(m_1 - \ell_{i+1} + 1 - \theta)(\ell_1 - m_i + 1)} \leq 1.
\end{equation}
Combining (\ref{S6recursion}) and (\ref{S6Interlace}) with the induction hypothesis for $H_{N}(\tilde{\ell}, \tilde{m})$ we conclude that 
\vspace{-1mm}
\begin{align}\label{HUB4}
\begin{split}
 &L^{-[N^2  + (\theta - 1) \cdot (N+1)N]}  H_{N+1}(\ell,m)  \leq C_N  \prod_{i = 2}^{N} \left[ \left(\frac{m_i - \ell_{i+1} + 1 - \theta}{L} \right)^{\theta - 1} \hspace{-5mm} + \left( \frac{\ell_i - m_i+1}{L} \right)^{\theta -1} \right] \times \\
&\left(\frac{\ell_1 - \ell_2}{L} \right)  \left(\frac{m_1 - \ell_{2} + 1 - \theta}{L} \cdot \frac{\ell_1 - m_1+1}{L} \right)^{\theta -1} \prod_{i = 2}^{N}\left(\frac{m_1 - \ell_{i+1} + 1 - \theta}{L}\cdot \frac{\ell_1 - m_i + 1}{L}\right)^{\theta}.
\end{split}
\end{align}
Also, we have
\vspace{-2mm}
$$\left(\frac{\ell_1 - \ell_2}{L} \right)  \left(\frac{m_1 - \ell_{2} + 1 - \theta}{L} \right)^{\theta - 1} \left( \frac{\ell_1 - m_1+1}{L} \right)^{\theta -1} \leq $$
$$  \left(\frac{m_1 - \ell_{2} + 1 - \theta}{L} \right)^{\theta -1} \left( \frac{\ell_1 - m_1+1}{L} \right)^{\theta} + \left(\frac{m_1 - \ell_{2} + 1 - \theta}{L} \right)^{\theta } \left( \frac{\ell_1 - m_1+1}{L} \right)^{\theta -1}.$$
The above inequality shows that we can find a constant $C > 0$ depending on $(a_+ - a_-), N , \theta$ such that the second line in (\ref{HUB4}) is upper bounded by
$$C \cdot \left[\left(\frac{m_1 - \ell_{2} + 1 - \theta}{L} \right)^{\theta -1} +\left(  \frac{\ell_1 - m_1+1}{L} \right)^{\theta -1}  \right],$$
which together with (\ref{HUB4}) proves (\ref{HUB3}) for $N+1$ with $C_{N+1} = C_N \cdot C$. The general result now proceeds by induction.\\

{\bf \raggedleft Step 2.} In this step we show that 
\begin{equation}\label{PFConv}
\lim_{L \rightarrow \infty} L^{-[2N - 1 + (N-1)^2  + (\theta - 1) \cdot N(N-1)]}  \cdot   Z_L^d = Z^c.
\end{equation}

Using the functional equation for the Gamma function $\Gamma(z + 1) = z\Gamma(z)$ we can write
\vspace{-1mm}
\begin{equation}\label{PLDefv2}
\begin{split}
&\mathbb{P}_{N,L}(\ell, m) = (Z^d_L)^{-1}\cdot H(\ell, m) \cdot \prod_{i = 1}^N e^{-N \theta V^t(\ell_i/L)} \cdot  \prod_{i = 1}^{N-1} e^{-N \theta V^b(m_i/L)},
\end{split}
\end{equation}
with $H(\ell,m)$ as in (\ref{HDef}). Equations (\ref{mixed1}) and (\ref{mixed2}) imply that for each fixed $(x,y) \in \mathcal{G}$  
we have
\begin{equation}\label{S6probAss}
\lim_{L \rightarrow \infty} L^{-[ (N-1)^2  + (\theta - 1) \cdot N(N-1)]} H(\ell^L,m^L) =  \hspace{-2mm} \prod_{1 \leq i < j \leq N} \hspace{-2mm}(x_j - x_i)  \hspace{-4mm}\prod_{1 \leq i < j \leq N-1}  \hspace{-4mm}(y_j - y_i) \prod_{i = 1}^{N-1} \prod_{j = 1}^N |y_i - x_j|^{\theta - 1}, 
\end{equation}
where $(\ell^L, m^L)$ is a sequence of elements such that $\lambda_{N - i+ 1}^L = \lfloor x_i L \rfloor$ for $i = 1, \dots ,N$ and $\mu_{N - i}^L = \lfloor y_i L \rfloor$ for $i = 1, \dots, N-1$.

Using (\ref{PLDef}) and (\ref{S6Coord}) we have
\begin{equation}\label{discretePF}
\begin{split}
&L^{-[2N - 1 + (N-1)^2  + (\theta - 1) \cdot N(N-1)]}  \cdot Z_L^d  =  \int_{ \mathbb{R}^{2N - 1}} f_L(x,y) dx dy,
\end{split}
\end{equation}
where $dxdy$ stands for $dx_1dx_2\cdots dx_N dy_1\cdots dy_{N-1}$ and is the Lebesgue measure on $\mathbb{R}^{2N - 1}$, and also
\begin{equation}\label{IntegrandCont}
\begin{split}
&f_L(x,y) =  L^{-[(N-1)^2  + (\theta - 1) \cdot N(N-1)]}H(\ell,m) \cdot \prod_{i = 1}^N e^{-N \theta V^t(\ell_i/L)} \cdot  \prod_{i = 1}^{N-1} e^{-N \theta V^b(m_i/L)} \\
\end{split}
\end{equation}
 if there exist $\lambda_1, \dots, \lambda_N$, $\mu_1, \dots, \mu_{N-1} \in \mathbb{Z}$ with  $\lceil a_- L \rceil \leq \lambda_N \leq \mu_{N-1} \leq \lambda_{N-2} \leq \cdots \leq \mu_1 \leq \lambda_1 \leq \lfloor a_+ L \rfloor$  and $\lambda_i \leq Lx_{N- i +1} < \lambda_i + 1$ for $i = 1,\dots, N$ and $\mu_i \leq Ly_{N-i} < \mu_i + 1$ for $i = 1,\dots, N-1$. If they do not exist we set $f_L(x,y) = 0$. It follows from (\ref{S6probAss}) that $f_L(x,y)$ converge pointwise to $Z_c \cdot f(x,y)$ almost everywhere on $\mathbb{R}^{2N - 1}$. 

If $\theta \geq 1$ then from (\ref{S6HBound}) and the boundedness of $V^t, V^b$ we see that there is $C > 0$ such that
\begin{equation}\label{S6GUBv1}
C   \geq f_L(x,y).
\end{equation}
So (\ref{PFConv}) follows from (\ref{discretePF}) and the a.e. pointwise convergence of $f_L(x,y)$ to $Z_c \cdot f(x,y)$ after an application of the Bounded convergence theorem.\\

Suppose next that $\theta \in (0,1)$. Let $g(x,y)$ be compactly supported on $ \in [a_- - 1, a_+ + 1]^{2N - 1}$, where it is given by
\begin{equation}\label{S6Defg}
g(x,y) = \prod_{i = 1}^{N-1} \left[ |y_i - x_{i+1} |^{\theta - 1} + |x_i - y_i|^{\theta -1} \right] .
\end{equation}
An application of Fubini's theorem and the integrability of $|x|^{\theta - 1}$ near $0$ implies that $g(x,y) \in L^1(\mathbb{R}^{2N-1}).$
We claim that there is a $C > 0$ depending on $N, \theta, (a_+ - a_-)$ such that for all $(x,y) \in \mathbb{R}^{2N-1}$ and $L$ such that $L (a_+ - a_-) > 1$ we have
\begin{equation}\label{S6GUB}
C \cdot g(x,y)  \geq f_L(x,y).
\end{equation}
We prove (\ref{S6GUB}) in the next step. For now we assume its validity and conclude the proof of (\ref{PFConv}). 

As before, $f_L(x,y)$ converges pointwise a.e. to $Z_c \cdot f(x,y)$ and by (\ref{S6GUB}) it is bounded by the integrable function $C \cdot g(x,y)$. Consequently, (\ref{PFConv}) follows after an application of the Dominated convergence theorem.\\

{\bf \raggedleft Step 3.} In this step we prove (\ref{S6GUB}). Let us fix $\lambda$ and $\mu$ such that $\lceil a_- L \rceil \leq \lambda_N \leq \mu_{N-1} \leq \lambda_{N-2} \leq \cdots \leq \mu_1 \leq \lambda_1 \leq \lfloor a_+ L \rfloor$. We also let $\ell$ and $m$ be as in (\ref{S6Coord}) for this choice of $\lambda, \mu$. Sicne $V^t$ an $V^b$ are bounded we see that to show (\ref{S6GUB}) it suffices to prove that for some $\tilde{C} > 0$ we have
\begin{align}\label{S6probBoundP1}
\begin{split}
& L^{-[(N-1)^2  + (\theta - 1) \cdot N(N-1)]} \cdot H(\ell,m)  \leq \tilde{C} \cdot g(x,y), \mbox{ provided $(x,y) \in Q_L = Q_L(\lambda, \mu)$ where }\\
&L \cdot Q_L = [\lambda_N , \lambda_{N} + 1] \times \cdots \times [\lambda_1 , \lambda_{1} + 1]   \times   [\mu_{N-1} , \mu_{N-1}+1] \times \cdots \times [\mu_1, \mu_{1}+1].
\end{split}
\end{align}
Indeed, (\ref{S6probBoundP1}) implies (\ref{S6GUB}) whenever $(x,y) \in Q_L$ and for $(x,y)$ not belonging to such a cube $f_L(x,y) = 0$ by definition so that (\ref{S6GUB}) holds trivially.\\

By (\ref{S6HBound}) we have for some $\hat{C} > 0$ that
\begin{equation}\label{S6HBoundP2}
\begin{split}
H(\ell, m) \cdot L^{-[(N-1)^2  + (\theta - 1) \cdot N(N-1)]} \leq  \hat{C}   \prod_{i = 1}^{N-1} \left[ \left(\frac{\mu_i - \lambda_{i+1} + 1}{L} \right)^{\theta - 1} + \left( \frac{\lambda_i - \mu_i+1}{L} \right)^{\theta -1} \right].
\end{split}
\end{equation}
Notice that for each $i = 1, \dots, N-1$
\begin{equation*}
\begin{split}
&\inf_{(x,y) \in Q_L} |y_i - x_{i+1}|^{\theta - 1} = \left(\frac{  \lambda_{N  - i} - \mu_{N-i} + 1}{L} \right)^{\theta - 1} \hspace{-3mm} \mbox{, } \inf_{(x,y) \in Q_L} |x_i - y_i|^{\theta - 1} = \left(\frac{\mu_{N-i} - \lambda_{N  - i + 1} + 1}{L} \right)^{\theta - 1}.\\
\end{split}
\end{equation*}
The above and (\ref{S6HBoundP2}) imply  (\ref{S6probBoundP1}) with $\tilde{C} = \hat{C}$.\\

{\bf \raggedleft Step 4.} Let $h(x,y)$ be a bounded continuous function on $\mathcal{G}$ and define $h(x,y) = 0$ outside of $\mathcal{G}$. We claim
\begin{equation}\label{LimitRectangle}
\lim_{ L \rightarrow \infty} \mathbb{E}_{N,L} \left[ h(L^{-1} \cdot \ell, L^{-1} \cdot m) \right] = \int_{\mathbb{R}^{2N-1}} f(x,y) h(x,y) dx dy.
\end{equation}
The weak convergence follows from (\ref{LimitRectangle}).

By definition we have 
\begin{equation}\label{LimitRectangle2}
\mathbb{E}_{N,L} \left[ h(L^{-1} \cdot \ell, L^{-1} \cdot m) \right]  =\frac{1}{Z_L^d \cdot L^{-[2N - 1 + (N-1)^2  + (\theta - 1) \cdot N(N-1)]} }  \cdot \int_{\mathbb{R}^{2N - 1}} f_L(x,y) \cdot h_L(x, y) ,
\end{equation}
where $h_L(x,y)$ is defined through $h_L(x,y) = h(L^{-1} \cdot \ell, L^{-1} \cdot m),$  $(\ell, m)$ are as in (\ref{S6Coord}) for $\lambda_i = \lfloor x_{N+1 - i} L \rfloor$ for $i = 1,\dots, N$ and and $\mu_i = \lfloor y_{N - i} L \rfloor$ for $i = 1, \dots, N-1$. Using (\ref{PFConv}) and (\ref{LimitRectangle2}) we see that (\ref{LimitRectangle}) would follow if we can show that 
\begin{equation}\label{LimitRectangle3}
\lim_{ L \rightarrow \infty} \int_{\mathbb{R}^{2N - 1}} f_L(x,y) \cdot h_L(x, y)  = \int_{\mathbb{R}^{2N-1}} f(x,y) h(x,y) dx dy.
\end{equation}
It follows from (\ref{S6probAss}) that $f_L(x,y) \cdot h_L(x,y)$ converge pointwise to $Z_c \cdot f(x,y) \cdot h(x,y)$ almost everywhere on $\mathbb{R}^{2N - 1}$, while by (\ref{S6GUBv1}) and (\ref{S6GUB}) we know that $|f_L(x,y) \cdot h_L(x,y)|$ is upper bounded by $C \cdot (1 + g(x,y))$ for a sufficiently large $C > 0$, where $g(x,y)$ is as in (\ref{S6Defg}). We can thus conclude (\ref{LimitRectangle3}) from the Dominated convergence theorem. 
\end{proof}

%

\subsection{Single level loop equations} \label{Section6.3}
In this section we deduce Proposition \ref{SingleLevelLoppThm} from a limit of the single level Nekrasov's equations -- Proposition  \ref{SingleLevelNekrasov}. We start with some notation that will be useful also in the next section.

%
\subsubsection{Deformed measures}\label{Section6.3.1}
We introduce a similar construction to the one from Section \ref{Section3.2}. The essential difference is that here we rescale the particle locations by $L$, which is now decoupled from $N$ (the number of particles on the top level). 

Take $2m+2n$ parameters $\t^1 = (t^1_1,\dots,t^1_m)$, $\vm^1 = (v^1_1,\dots,v^1_m)$, $\t^2 = (t_1^2, \dots, t^2_n)$, $\vm^2 = (v^2_1, \dots, v^2_n)$ and such that $v^i_a + t^i_a - y \neq 0$ for all meaningful $i,a$ and all $y \in [a_- ,  a_+ + N \theta L^{-1} ]$, and let the deformed distribution $\mathbb{P}_{N,L}^{\t, \vm}$  be defined as
\begin{equation} \label{eq:distrgen_deformed}
\begin{split}
\mathbb{P}_{N,L}^{\t, \vm}(\ell, m)=Z(\t, \vm)^{-1} \mathbb{P}_{N,L}(\ell, m)\prod_{i =1}^{N} \prod^m_{a=1}  \left(  \hspace{-1mm} 1+ \frac{t^1_a}{v^1_a-\ell_i/L} \right) \cdot \prod_{i =1}^{N-1}\prod^n_{a=1}  \left(  \hspace{-1mm} 1+ \frac{t^2_a}{v^2_a-m_i/L} \right),
\end{split}
\end{equation}
where $\mathbb{P}_{N,L}$ is as in (\ref{PLDef}). If $m = n = 0$ we have $\mathbb P^{\t, \vm}_{N,L} = \mathbb{P}_{N,L}$. In general, $\mathbb P^{\t, \vm}_{N,L}$ may be a complex-valued measure but we always choose the normalization constant $Z(\t, \vm)$ so that $\sum_{\ell, m} \mathbb P^{\t, \vm}_{N,L}(\ell, m) = 1$. In addition, we require that the numbers $t^i_a$ are sufficiently close to zero so that $Z(\t, \vm) \neq 0$. 

If $(\ell, m)$ is distributed according to \ref{eq:distrgen_deformed} we denote 
\begin{equation}\label{GLDef}
 G^{t}_L(z)= \sum_{i = 1}^N \frac{1}{z- \ell_i/ L } \mbox{ and } G^b_L(z) = \sum_{i = 1}^{N-1} \frac{1}{z - m_i/L}. 
\end{equation}

The definition of the deformed measure $\mathbb P_{N,L}^{\t, \vm}$ is motivated by the following observation. 
\begin{lemma}\label{S6LemCum1} Let $\xi$ be a bounded random variable. For any $m, n\geq 0$ we have
\begin{equation}\label{S6eq:derivative_k}
\frac{\partial^{m+n}}{\partial t^1_1 \cdots \partial t^1_m \partial t^2_1 \cdots \partial t^2_n}\mathbb E_{\mathbb P_{N,L}^{\t, \vm}}\left[\xi\right]\bigg\rvert_{t^i_a = 0} = M( \xi, G^t_L(v^1_1),\dots ,G^t_L(v^1_m),G^b_L(v^2_1), \dots, G^b_L(v^2_n)),
\end{equation}
where the right side is the joint cumulant of the given random variables with respect to $\mathbb P_{N,L}$.
\end{lemma}
The proof is the same as that of Lemma \ref{LemCum1} so we omit it.

%
\subsubsection{Asymptotic expansions}\label{Section6.3.2}
In this section we derive asymptotic expansions that are analogues of those in Section \ref{Section4.2}. Below we will write $\xi_L(z)$ to mean a generic random analytic function on $\mathbb{C} \setminus [a_- ,  a_+ ]$, which is almost surely $O(1)$ over compact subsets of $\mathbb{C} \setminus [a_- ,  a_+ ]$. For $x, y \in [-\theta - 1, \theta + 1]$ by a direct Taylor series expansion we have
\begin{equation}\label{S6AEP1v1}
\begin{split}
 \prod_{ i = 1}^N \frac{Lz - \ell_i + x}{Lz - \ell_i + y} = 1 + \frac{(x-y) G^t_L(z) }{L} + \frac{(x^2 - y^2) \partial_z G^t_L(z)}{2L^2} + \frac{(x-y)^2[G^t_L]^2(z)}{2L^2}  + \frac{\xi_L(z)}{L^3}; 
\end{split}
\end{equation}
\begin{equation}\label{S6AEP1v2}
\begin{split}
 \prod_{ i = 1}^{N-1} \frac{Lz - m_i + x}{Lz - m_i + y} = 1 + \frac{(x-y) G^b_L(z) }{L} + \frac{(x^2 - y^2) \partial_z G^b_L(z)}{2L^2} + \frac{(x-y)^2[G^b_L]^2(z)}{2L^2}  + \frac{\xi_L(z)}{L^3}; 
\end{split}
\end{equation}
\begin{equation}\label{S6AEP1v3}
\begin{split}
\sum_{i = 1}^N \frac{1}{Lz - \ell_i + x} - \sum_{i = 1}^{N-1} \frac{1}{Lz - m_i + y}  = \frac{G_L^t(z)}{L} + \frac{x \partial_z G^t_L(z)}{L^2} - \frac{G_L^b(z)}{L} - \frac{y \partial_z G^b_L(z)}{L^2} + \frac{\xi_L(z)}{L^3}.
\end{split}
\end{equation}

%
\subsubsection{Single level Nekrasov's equation}\label{Section6.3.3} Let us define
\begin{align}\label{S6SLNE}
\begin{split}
&R_L(Lz) = P_L(z)  \Phi_L^-(Lz) A^{\t,\vm}_1(z)  \mathbb{E}_{\mathbb P^{\t,\vm}_{N,L}}\left[ \prod_{ i = 1}^N \frac{Lz - \ell_i - \theta}{Lz - \ell_i} \right]  \\
&+P_L(z) \Phi_L^+(Lz)  B^{\t,\vm}_1(z)  \mathbb{E}_{\mathbb P^{\t,\vm}_{N,L}}\left[ \prod_{ i = 1}^N \frac{Lz - \ell_i - 1 + \theta} {Lz - \ell_i - 1} \right],  
\end{split}
\end{align}
where $P_L(z) =  (Lz - \lceil a_-L \rceil )(Lz - \lfloor a_+ L \rfloor - (N-1) \theta)$, the expectations are with respect to the deformed measures in Section \ref{Section6.3.1} with $V^b \equiv 0$, $n = 0$, $v^1_a = v_a$ for $a = 1, \dots, m$ and
$$A^{\t,\vm}_1(z) = \prod_{a = 1}^m \left[ v_a +t_a - z + \frac{1}{L} \right] \left[ v_a - z \right], 
B^{\t,\vm}_1(z) =  \prod_{a = 1}^m \left[ v_a +t_a - z \right] \left[ v_a - z + \frac{1}{L}\right].$$
Moreover the functions $\Phi_L^{+}$ and $\Phi_L^{-}$ are given by
\begin{equation}\label{S6DefPhi}
\Phi_L^{-}(z) = 1 \mbox{ and }\Phi_L^{+}(z) = \exp \left[ - N\theta (V^t(z/L) - V^t(z/L - 1/L)) \right].
\end{equation}

We claim that $R_L(Lz)$ is analytic in $\mathcal{M}$. Observe that the above choice ensures
\begin{equation*}
\begin{split}
&\frac{w(Lz;L)}{w(Lz-1;L)}=\frac{\Phi_L^+(Lz)}{\Phi_L^-(Lz)},
\end{split}
\end{equation*} 
and so we are in the setup of Proposition \ref{SingleLevelNekrasov} (upto a trivial shift). We conclude that
$$\frac{R_L(Lz) }{P_L(z)} = \tilde{R}_L(Lz) + \frac{r^-}{Lz - \lceil a_-L \rceil } + \frac{r^+}{Lz - \lfloor a_+ L \rfloor - (N-1) \theta},$$
with $\tilde{R}_L(Lz)$ analytic in $\mathcal{M}$ and $r^{\pm}$ are as in the statement of the proposition. But now multiplying the above by $P_L(z)$ cancels the possible poles at $\lceil a_-L \rceil$ and $\lfloor a_+ L \rfloor + (N-1) \theta$ and so $R_L(Lz)$ is also analytic in $\mathcal{M}$.

%
\subsubsection{Proof of Proposition \ref{SingleLevelLoppThm}} We continue with the same notation as in Sections \ref{Section6.3.1} -\ref{Section6.3.3} and in Proposition \ref{SingleLevelLoppThm}. We fix $m \geq 0$ and points $v_0, \dots, v_m \in \mathbb{C} \setminus [a_-, a_+]$. For a set $A \subseteq \llbracket 1, m\rrbracket$ and a bounded random variable $\xi$ we write $M(\xi; A)$ for the joint cumulant of $\xi$ and $G^t_L(v_a)$ for $a \in {A}$. If $A = \varnothing$ the latter notation stands for $\mathbb{E}[\xi]$. 

We start by dividing both sides of (\ref{S6SLNE}) by $2\pi \i  \cdot (z - v_0)\cdot B^{\t,\vm}_1 $ and integrating over $\Gamma$. This gives
\begin{align*}
\begin{split}
&\frac{1}{2\pi \i }\int_{\Gamma} \frac{dz R_L(Lz) }{ (z - v_0)\cdot B_1 }= \frac{1}{2\pi \i }\int_{\Gamma} \frac{dz P_L(z) \Phi^{-}_L(Lz) A_1}{(z - v_0) \cdot B_1}  \cdot \mathbb{E} \left[ \prod_{i = 1}^N\frac{Lz- \ell_i -\theta}{Lz - \ell_i} \right]  \\
&+ \frac{1}{2\pi \i }\int_{\Gamma} \frac{dz P_L(z) \Phi^{+}_L(Lz)  }{ z - v_0} \cdot   \mathbb{E} \left[ \prod_{i = 1}^{N}\frac{Lz- \ell_i + \theta - 1}{Lz - \ell_i - 1} \right],
\end{split}
\end{align*}
where the expectation is with respect to $\mathbb P^{\t,\vm}_{N,L}$ and we have suppressed the dependence on $A_1$ and $B_1$ on $\t, \vm$. By Cauchy's theorem the left side above is zero. We next apply the operator $\mathcal{D} : = \partial_{t_1} \cdots \partial_{t_m}$ to both sides and set $t_a = 0$ for $a = 1, \dots, m$. Notice that when we perform the differentiation to the right side some of the derivatives could land on $\frac{A_1}{B_1}$ and some on the (measure inside of the) expectation. We will split the result of the differentiation based on subests $A$, where $A$ consists of indices $a$ in $\{1, \dots, m\}$ such that $\partial_{t_a} $ differentiates the expectation. The result of this procedure is as follows
\begin{align*}
\begin{split} 
& 0 =  \frac{1}{2\pi \i }\int_{\Gamma} dz  \sum_{A \subseteq \llbracket 1,m \rrbracket} \prod_{a \in A^c}  \left[ \frac{- L^{-1}}{ (v_a -z)(v_a - z +L^{-1})}\right]   \frac{P_L(z)  \Phi_L^-(Lz) }{z - v_0}M \left(\prod_{i = 1}^N\frac{Lz- \ell_i -\theta}{Lz - \ell_i} ; A\right)  \\
&+  \frac{P_L(z)  \Phi_L^+(Lz) }{ z - v_0 }  M\left( \prod_{i = 1}^{N}\frac{Lz- \ell_i +\theta - 1}{Lz - \ell_i - 1} ; \llbracket 1, m \rrbracket \right),
\end{split}
\end{align*}
where $A^c = \llbracket 1,m \rrbracket \setminus A$. We may now use (\ref{S6AEP1v1}) to rewrite the above as
\begin{align*}
\begin{split} 
& O(L^{-1}) = \frac{1}{2\pi \i}\int_{\Gamma} dz \frac{\Phi_L^+(Lz)  P_L(z) }{ z - v_0 } M\left( \frac{\theta G^t_L(z)}{L}  + \frac{[\theta^2 -2 \theta] \partial_z G^t_L(z)}{2L^2} + \frac{\theta^2 [G^t_L]^2(z)}{2L^2} ; \llbracket 1, m \rrbracket \right)  \\
& + \hspace{-3mm} \sum_{A \subseteq \llbracket 1,m \rrbracket} \prod_{a \in A^c}\hspace{-1mm}   \left[ \frac{- L^{-1}}{ (v_a -z)(v_a - z +L^{-1})}\right]  \hspace{-1mm}  \frac{\Phi_L^-(Lz) P_L(z) }{ z - v_0}M \left( \hspace{-1mm}  -\frac{\theta G^t_L(z) }{L} + \frac{\theta^2 \partial_z G^t_L(z)}{2L^2} + \frac{\theta^2[G^t_L]^2(z)}{2L^2} ; A\right)\hspace{-1mm} ,
\end{split}
\end{align*}
where we have removed the constants $1$ from the cumulants using the following rationale. If the joint cumulant is of two or more variables, we can remove the $1$ as joint cumulants remain unchanged by shifts by constants. If the joint cumulant is of one variable then the term involving $1$ integrates to $0$ by Cauchy's theorem.

Notice that the integral of the terms on the second line above are $O(L^{-1})$ unless $|A^c| = 0$ or $|A^c| = 1$. Consequently we can simplify the above as
\begin{align*}
\begin{split} 
& O(L^{-1}) = \frac{1}{2\pi \i}\int_{\Gamma}dz  \frac{\Phi_L^+(Lz)  P_L(z) }{ z - v_0 } M\left( \frac{\theta G^t_L(z)}{L}  + \frac{[\theta^2 -2 \theta] \partial_z G^t_L(z)}{2L^2} + \frac{\theta^2 [G^t_L]^2(z)}{2L^2} ; \llbracket 1, m \rrbracket \right) \\
& + \frac{\Phi_L^-(Lz) P_L(z) }{ z - v_0}M \left(  -\frac{\theta G^t_L(z) }{L} + \frac{\theta^2 \partial_z G^t_L(z)}{2L^2} + \frac{\theta^2[G^t_L]^2(z)}{2L^2} ; \llbracket 1, m \rrbracket \right)  \\
&  - \sum_{a = 1}^m      \hspace{-1mm} \frac{\Phi_L^-(Lz) P_L(z) }{ L(v_a -z)(v_a - z +L^{-1}) (z - v_0)}M \left(   \hspace{-1mm}-\frac{\theta G^t_L(z) }{L} + \frac{\theta^2 \partial_z G^t_L(z)}{2L^2} + \frac{\theta^2[G^t_L]^2(z)}{2L^2} ; \llbracket 1, m \rrbracket \setminus \{a\} \right).
\end{split}
\end{align*}
In view of (\ref{S6DefPhi}) we have $\Phi_L^{-}(z) = 1$,
\begin{equation*}
\Phi_L^{+}(Lz)  = 1 -  \frac{N \theta \partial_z V^t(z)}{L} + O(L^{-2}) \mbox{ and } \frac{- L^{-1}}{ (v_a -z)(v_a - z +L^{-1})} = \frac{1}{L(v_a -z)^2} + O(L^{-2}),
\end{equation*}
which allows us to simplify the above to
\begin{align*}
\begin{split} 
 O(L^{-1}) = &\frac{1}{2\pi \i}\int_{\Gamma}dz  \frac{ \theta^2  P_L(z)  }{ (z - v_0)\cdot L^2 } M\left(  - N \partial_z V^t(z) + [1- \theta^{-1}]\partial_z G^t_L(z) + [G^t_L]^2(z) ; \llbracket 1, m \rrbracket \right)  \\
&  + \sum_{a = 1}^m       \frac{ \theta  P_L(z) M \left(   G^t_L(z)  ; \llbracket 1, m \rrbracket \setminus \{a\} \right) }{(v_a -z)^2 (z - v_0) \cdot  L^2}.
\end{split}
\end{align*}
After taking the limit $L \rightarrow \infty$ above (and applying Proposition \ref{prop_cont_limit}) we arrive at
\begin{align}\label{S6LNE}
\begin{split} 
&0 = \frac{- N\theta^2 }{2\pi \i}\int_{\Gamma}dz  \frac{(z - a_-)(z-a_+) \partial_zV^t(z)  }{z - v_0}  \cdot \M \left(  z, \llbracket 1, m \rrbracket \right)  + \\
&  \frac{\theta^2}{2\pi \i}\int_{\Gamma}dz  \frac{(z - a_-)(z-a_+)}{z - v_0}\left[ [1 - \theta^{-1}]\partial_z \M \left(  z,\llbracket 1, m \rrbracket \right) + \M \left( [ G_c^t(z)]^2; \llbracket 1, m \rrbracket \right)  \right] + \\
&  \frac{\theta}{2\pi \i}\int_{\Gamma}dz  \sum_{a = 1}^m    \frac{(z - a_-)(z-a_+) }{ (z - v_0)(v_a -z)^2} \cdot \M \left(  z , \llbracket 1, m \rrbracket \setminus \{a\} \right),
\end{split}
\end{align}
where we write $\M (\xi; A)$ for the joint cumulant of $\xi$ and the variables $G^t_c(v_a)$ for $a \in A$. As usual if $A = \varnothing$ this stands for $\mathbb{E}[\xi]$. \\

If $m = 0$ then the third line in (\ref{S6LNE}) is zero and we can evaluate the second line as minus the residues at $z = v_0$ and $z = \infty$, using that 
$$\frac{(z - a_-)(z-a_+)}{z - v_0} \partial_z \M \left(  z \right)  \sim - N z^{-1} \mbox{ and }\frac{(z - a_-)(z-a_+)}{z - v_0} \M \left( [ G_c^t(z)]^2\right) \sim N^2 z^{-1} \mbox{ as } |z| \rightarrow \infty,$$
to get
\begin{align}\label{S6AFR}
\begin{split}
0 = &\frac{- N\theta^2 }{2\pi \i}\int_{\Gamma}dz  \frac{(z - a_-)(z-a_+) \partial_zV^t(z)  }{z - v_0}  \cdot \M \left(  z, \llbracket 1, m \rrbracket \right) \\
& - \theta^2 (v_0 - a_-)(v_0-a_+) \left( [1- \theta^{-1}]\partial_z \M \left(  v_0 \right) +  \M \left( [ G_c^t(v_0)]^2\right) \right) + \theta^2 N^2 - N[\theta^2- 1], 
\end{split}
\end{align}
which is the same as (\ref{S6rank1}) once we divide by $-\theta^2(v_0 - a_-)(v_0-a_+)$. \\

In the remainder we assume $m \geq 1$. We may now compute the integrals on the second and third lines of (\ref{S6LNE}) as minus the residues at $z = v_0$ and $z = v_a$ -- notice there are no residues at infinity. Further we can divide both sides by $- \theta^2 (v_0 - a_-)(v_0 - a_+)$. The result is
\begin{align}\label{S6LNEv2}
\begin{split} 
&0 = \frac{- N}{2\pi \i}\int_{\Gamma}  \frac{(z - a_-)(z-a_+)}{(v_0 - z)(v_0 - a_-)(v_0 - a_+)}  \M \left( \partial_zV^t(z) G^t_c(z); \llbracket 1, m \rrbracket \right) + \\
&  + \M(G_c^t(v_0)^2;\llbracket 1, m \rrbracket)  +  [1-  \theta^{-1}] \M(  \partial_z G(v_0)  ;\llbracket 1, m \rrbracket ) - \\
& + \theta^{-1} \sum_{a = 1}^m  \partial_{v_a}  \left[   \frac{\M\left(  v_0,  \llbracket 1,m \rrbracket \setminus \{a \} \right)}{v_0-v_a}  - \frac{ ( v_a-a_-)(v_a - a_+) \M\left(  v_a, \llbracket 1,m \rrbracket \setminus \{a \} \right) }{(v_0 - v_a)(v_0 - a_+)(v_0 - a_+) } \right].
\end{split}
\end{align}
We see that (\ref{S6LNEv2}) is the same as (\ref{S6rankm}) once we use
$$ \M(G_c^t(v_0)^2;\llbracket 1, m \rrbracket) = \M(v_0, v_0,\llbracket 1, m \rrbracket) + \sum_{J \subseteq \llbracket 1, m \rrbracket} \M(v_0, J) \cdot \M(v_0, \llbracket 1,m \rrbracket \setminus J) $$
which follows from the more general statement
\begin{equation}\label{RemainingCLE2}
M(XY, X_1, \cdots, X_m) =  M(X,Y, X_1, \cdots, X_m) + \sum_{J \subseteq \llbracket 1, m \rrbracket} M(X; J) \cdot M(Y; \llbracket 1,m \rrbracket \setminus J),
\end{equation}
which in turn is a special case of Malyshev's formula, see e.g. \cite[equation (3.2.8)]{Taqqu} (one needs to set $b = \{1,2\}, \{3\}, \dots, \{m+2\}$ in that formula). This suffices for the proof.

%

\subsection{Two level loop equations}\label{Section6.4} We go back to the setup of Section \ref{Section6.1}, i.e. $V^b(z)$ is again an arbitrary analytic function in $\mathcal{M}$.

Let us as before set $P_L(z) = (Lz - \lceil a_-L \rceil )(Lz - \lfloor a_+ L \rfloor - (N-1) \theta)$. For $m, n \geq 0$ and $v_1^1, \dots, v_m^1, v^2_1, \dots, v_n^2 \in \mathbb{C} \setminus [a_-, a_+]$ and $\theta > 0$ we define
\begin{align}\label{TLNE1}
\begin{split}
& R^\theta_{L}(Lz) = P_L(z) \phi^{t}_L(Lz) \cdot A^{\t, \vm}_1(z) \cdot B^{\t, \vm}_2(z) \cdot \mathbb{E}_{\mathbb P_{N,L}^{\t,\vm}} \left[ \prod_{i = 1}^N\frac{Lz- \ell_i -\theta}{Lz - \ell_i} \right] + \\
& P_L(z)  \phi^{b}_L(Lz) \cdot B^{\t, \vm}_1(z) \cdot A^{\t, \vm}_2(z) \cdot  \mathbb{E}_{\mathbb P^{\t,\vm}_{N,L}} \left[ \prod_{i = 1}^{N-1}\frac{Lz- m_i + \theta - 1}{Lz - m_i - 1} \right] + \\
& +  P_L(z)  \phi_L^{m}(Lz) \cdot B^{\t, \vm}_1 (z) \cdot B^{\t, \vm}_2(z) \cdot \mathbb{E}_{\mathbb P_{N,L}^{\t,\vm}} \left[  \Pi^\theta_1(Lz)  \right],
\end{split}
\end{align}
where the expectations are with respect to the deformed measures in Section \ref{Section6.3.1} and 
\begin{align}
\begin{split}
&A^{\t, \vm}_1(z) = \prod_{a = 1}^m \left[ v_a^1 +t_a^1 - z + \frac{1}{L} \right] \left[ v_a^1 - z \right], B^{\t, \vm}_1(z) =  \prod_{a = 1}^m \left[ v_a^1 +t_a^1 - z \right] \left[ v_a^1 - z + \frac{1}{L}\right] \\
&A^{\t, \vm}_2(z) = \prod_{a = 1}^n \left[ v_a^2 +t_a^2 - z \right] \left[ v_a^2 - z + \frac{1}{L} \right], B^{\t, \vm}_2(z) =  \prod_{a = 1}^n \left[ v_a^2 +t_a^2 - z + \frac{1}{L} \right] \left[ v_a^2 - z\right].
\end{split}
\end{align}
Moreover, $\phi_L^t$, $\phi_L^b$ and $\phi_L^m$ are given by 
\begin{align}\label{S6Phi2L}
\begin{split}
&\phi_L^{t}(z) =  \exp \left[ N\theta (V^t(z/L) - V^t(z/L- L^{-1}))  \right], \phi_L^{m}(z) = 1,\\
& \phi_L^{b}(z) =  \exp \left[ N\theta (V^b(z/L -  L^{-1}) - V^b(z) )  \right],
\end{split}
\end{align}
and 
\begin{equation}\label{S6Pi}
\begin{split}
\Pi^\theta_1(z) =  \begin{cases} \frac{\theta}{1-\theta} \prod_{i = 1}^N \frac{z- \ell_i -\theta}{z - \ell_i - 1}  \prod_{i = 1}^{N-1}\frac{z- m_i + \theta - 1}{z - m_i} &\mbox{ if } \theta \neq 1, \\
 \sum_{i = 1}^N \frac{1}{z - \ell_i - 1} - \sum_{i = 1}^{N-1} \frac{1}{z - m_i} &\mbox{ if } \theta = 1.
\end{cases} \\
\end{split}
\end{equation}

Observe that the above choice ensures
\begin{equation*}
\begin{split}
&\frac{\phi_L^t(Lz)}{ \phi_L^m(Lz)} =   \frac{w(Lz-1;L)}{w(Lz;L)}, \hspace{2mm}  \frac{\phi_L^b({Lz})}{\phi_L^m({Lz})} =  \frac{\tau({Lz};L)}{\tau({Lz}-1;L)},
\end{split}
\end{equation*} 
and so we conclude by Theorems \ref{TN1} and \ref{TN1Theta1} that $R^\theta_L(Lz)$ is analytic in $\mathcal{M}$ from the beginning of Section \ref{Section6.1}. To be more specific (\ref{TLNE1}) is obtained from (\ref{S2mN}) (when $\theta \neq 1$) and (\ref{S2mNTheta1}) (when $\theta = 1$) by multiplying both sides by $P_L(z)$, and moving the terms 
$$\frac{r^-_1 P_L(z)}{Lz - \lceil a_-L \rceil } + \frac{r^+_1 P_L(z)}{Lz - \lfloor a_+ L \rfloor -( N-1) \theta},$$
that are both analytic by the definition of $P_L$ to the right side of the equation.

The goal of this section is to use the $L \rightarrow \infty$ limit of the above equations and prove Theorem \ref{TwoLevelLoopThm}. We remark that we will only take the limit of one of our two-level Nekrasov's equations. It turns out that if one takes the analogous limit of the other Nekrasov's equation, the same limit is obtained. As the second limit does not lead to any new results we will omit it.\\

\begin{definition}\label{S6CumDef}\vspace{-2mm}
We summarize some notation in this definition. Let $K$ be a compact subset of $\mathbb{C} \setminus [a_-, a_+]$. In addition, we fix integers $m, n \geq 0$ and points $\{v^1_a\}_{a = 1}^m, \{v^2_b \}_{b = 1}^n \subseteq K$. In addition, we fix $v \in K$ and let $\Gamma$ be a positively oriented contour, which encloses the segment $[a_-, a_+]$, is contained in $\mathcal{M}$ as in the beginning of Section \ref{Section6.1} and avoids $K$.

For a bounded random variable $\xi$ and sets $A,B,C$ we let  $M(\xi; A,B)$ be the joint cumulant of the random variables $\xi$, $G^t_L(v^1_a)$, $G_L^b(v^1_b)$ for $a \in A$,  for $b \in B$, where we recall that $G_L^t$ and $G_L^b$ were defined in (\ref{GLDef}). If $A = B = \varnothing$ then $M(\xi; A,B) = \mathbb{E}[\xi]$. 
\end{definition}

We will ease our notation by dropping the $\t,\vm$ dependence from the notation. We start by dividing both sides of (\ref{TLNE1}) by $2\pi \i \cdot (z - v) \cdot B_1 \cdot B_2$ and inetgrating over $\Gamma$. This gives
\begin{align*}
\begin{split}
&\frac{1}{2\pi \i }\int_{\Gamma} \frac{ R^\theta_L(Lz) dz}{ (z - v) \cdot B_1 \cdot B_2}= \frac{1}{2\pi \i }\int_{\Gamma}dz  \frac{ P_L(z)\phi^{m}_L(Lz)}{z - v}  \cdot  \mathbb{E} \left[  \Pi_1^{\theta}(Lz) \right]   \\
& + \frac{ P_L(z) \phi^{t}_L(Lz)A_1}{ B_1 \cdot (z - v)}  \cdot \mathbb{E} \left[ \prod_{i = 1}^N\frac{Lz- \ell_i -\theta}{Lz - \ell_i} \right] + \frac{P_L(z)\phi^{b}_L(Lz) A_2 }{ B_2 \cdot (z - v) } \cdot   \mathbb{E} \left[ \prod_{i = 1}^{N-1}\frac{Lz- m_i + \theta - 1}{Lz - m_i - 1} \right].
\end{split}
\end{align*}
By Cauchy's theorem the left side of the above expression vanishes. We next apply the operator 
$$\mathcal{D} : =  \partial_{t_{1}^{1}} \cdots  \partial_{t_{m}^{1}} \cdot \partial_{t_{1}^{2}} \cdots  \partial_{t_{n}^{2}}  $$
to both sides and set $t_a^1 = 0$ for $a = 1, \dots, m$ and $t_a^2 = 0$ for $a = 1, \dots ,n$. Notice that when we perform the differentiation to the second line above some of the derivatives could land on the products and some on the expectation. We will split the result of the differentiation based on subsets $A,B$. The set $A$ consists of indices $a$ in $\{1, \dots, m\}$ such that $\partial_{t_a^1} $ differentiates the expectation. Similarly, $B$ denotes the set of indices $b$ in $\{1, \dots, n\}$ such that $\partial_{t^2_{b}}$ differentiates the expectation. The result of this procedure is as follows 
\begin{align*}
\begin{split} 
& 0 = \frac{1}{2\pi \i}\int_{\Gamma}   \frac{ P_L(z)\phi^{m}_L(Lz)}{z - v}   M\left( \Pi_1^{\theta}(Lz) ; \llbracket 1,m \rrbracket, \llbracket 1,n \rrbracket \right) \\
& +\sum_{A \subseteq \llbracket 1,m \rrbracket} \prod_{b \in A^c}  \left[ \frac{- L^{-1}}{ (v_a^1 -z)(v_a^1- z^-)}\right]   \frac{ P_L(z) \phi^{t}_L(Lz)}{ z - v}M \left(\prod_{i = 1}^N\frac{Lz- \ell_i -\theta}{Lz - \ell_i} ; A,  \llbracket 1, n \rrbracket \right) \\
&+ \sum_{B \subseteq \llbracket 1,n \rrbracket} \prod_{b \in B^c} \left[ \frac{ L^{-1}}{ (v^2_b - z) (v^2_b -z^-)}\right]  \frac{P_L(z)\phi^{b}_L(Lz) }{z - v}M\left( \prod_{i = 1}^{N-1}\frac{Lz- m_i +\theta - 1}{Lz - m_i - 1} ;  \llbracket 1,m \rrbracket, B\right),
\end{split}
\end{align*}
where $z^- = z - L^{-1}$. We may now use (\ref{S6AEP1v1}, \ref{S6AEP1v2}, \ref{S6AEP1v3}) to rewrite the above as
\begin{align}\label{SExpandTLNE}
\begin{split} 
& O(L^{-1}) = \frac{1}{2\pi \i}\int_{\Gamma}dz  \frac{\theta P_L(z)\phi^{m}_L(Lz)}{z - v}   M\Bigg{(} \frac{\bf{1}\{\theta \neq 1\}}{1- \theta} +  \frac{G^t_L(z) -G^b_L(z)  }{L} \\
&-\frac { (1 +\theta) \partial_z G^t_L(z)}{2L^2}+ \frac{(1-\theta)(G^t_L(z) - G_L^b(z))^2}{2L^2}+ \frac{(1 - \theta) \partial_z G^b_L(z)}{2L^2}   ; \llbracket 1,m \rrbracket, \llbracket 1,n \rrbracket \Bigg{)} \\
& +\sum_{A \subseteq \llbracket 1,m \rrbracket} \hspace{-0.5mm} \prod_{a \in A^c}  \left[ \frac{- L^{-1}}{ (v_a^1 -z)(v_a^1- z^-)}\right] \hspace{-1mm}  \frac{ \theta P_L(z) \phi^{t}_L(Lz)}{ z - v} M \left( \hspace{-1mm} -\frac{ G^t_L(z) }{L} \hspace{-0.5mm} +  \hspace{-0.5mm} \frac{\theta [\partial_z G^t_L(z) + [G^t_L]^2(z)]}{2L^2}  ; A,  \llbracket 1, n \rrbracket \hspace{-1mm} \right) \\
&+ \sum_{B \subseteq \llbracket 1,n \rrbracket}  \hspace{-0.5mm} \prod_{b \in B^c}  \hspace{-1mm}\left[ \frac{ L^{-1}}{ (v^2_b - z) (v^2_b -z^-)}\right] \hspace{-1.5mm}  \frac{\theta P_L(z)\phi^{b}_L(Lz) }{z - v}M\left(  \hspace{-1mm} \frac{ G^b_L(z)}{L} \hspace{-1mm} +\hspace{-1mm} \frac{[\theta - 2] \partial_z G^b_L(z)  \hspace{-1mm}+ \hspace{-1mm} \theta [G^b_L]^2(z)}{2L^2}  ; \hspace{-0.5mm} \llbracket 1,m \rrbracket, \hspace{-0.5mm} B  \hspace{-1mm}\right).
\end{split}
\end{align}
We can remove the term ${\bf 1}\{\theta \neq 1\}$ in (\ref{SExpandTLNE}) using the following rationale. If the joint cumulant is of two or more variables, we can remove it as joint cumulants remain unchanged by shifts by constants. If the joint cumulant is of one variable then this term integrates to $0$ by Cauchy's theorem.

Setting $m = n = 0$ in (\ref{SExpandTLNE}) we obtain
\begin{align*}
\begin{split} 
& O(L^{-1}) = \frac{1}{2\pi \i}\int_{\Gamma}dz \frac{\theta P_L(z)\phi^{m}_L(Lz)}{z - v}   \mathbb{E}\Bigg{[} \frac{G^t_L(z) -G^b_L(z) }{L} + \frac{ (1-\theta)( G^t_L(z)- G^b_L(z))^2}{2L^2} \\
&- \frac {(1+\theta) \partial_z G^t_L(z)}{2L^2} + \frac{(1-\theta) \partial_z G^b_L(z)}{2L^2}   \Bigg{]}  +  \frac{ \theta P_L(z) \phi^{t}_L(Lz)}{ z - v} \mathbb{E} \left[  -\frac{ G^t_L(z) }{L} + \frac{\theta \partial_z G^t_L(z)}{2L^2} + \frac{\theta[G^t_L]^2(z)}{2L^2} \right] \\
&+  \frac{\theta P_L(z)\phi^{b}_L(Lz) }{z - v} \mathbb{E} \left[ \frac{ G^b_L(z)}{L} + \frac{[\theta - 2] \partial_z G^b_L(z)}{2L^2} + \frac{\theta [G^b_L]^2(z)}{2L^2} \right].
\end{split}
\end{align*}
From (\ref{S6Phi2L}) we get 
\begin{equation}\label{S6Phi2Lv2}
\phi_L^{m}(Lz) = 1, \phi_L^{t}(Lz) = 1 + N \theta\frac{\partial_z V^t(z)}{L} + O(L^{-2}) \mbox{ and } \phi^b_L(Lz) = 1 - N\theta \frac{\partial_z V^b(z)}{L} + O(L^{-2}).
\end{equation}
Substituting (\ref{S6Phi2Lv2}) above we obtain
\begin{align*}
\begin{split} 
& O(L^{-1}) = \frac{1}{2\pi \i}\int_{\Gamma}dz \frac{N\theta^2 \cdot P_L(z)}{L^2 \cdot (z - v)}   \mathbb{E} \left[  - G_L^t(z) \cdot \partial_z V^t(z) - G_L^b(z) \cdot \partial_z V^b(z)\right]  \\
& +  \frac{ \theta P_L(z) }{2L^2(z - v)} \mathbb{E} \left[  - \partial_z G^t_L(z)-  \partial_z G^b_L(z) +[G^t_L]^2(z)  +  [G^b_L]^2(z) -  2(1-\theta) G^t_L(z)G^b_L(z) \right],
\end{split}
\end{align*}
We may now send $L \rightarrow \infty$ above and apply Proposition \ref{prop_cont_limit} to get
\begin{align*}
\begin{split} 
& 0 = \frac{1}{2\pi \i}\int_{\Gamma} \frac{N\theta^2 \cdot (z - a_-)(z - a_+) }{z - v}   \mathbb{E} \left[  - G_c^t(z) \cdot \partial_z V^t(z) - G_c^b(z) \cdot \partial_z V^b(z)\right]    \\
& +  \frac{ \theta (z - a_-)(z - a_+) }{2(z - v)} \mathbb{E} \left[  - \partial_z G^t_c(z)-  \partial_z G^b_c(z) + [G^t_c]^2(z) +  [G^b_c]^2(z) - 2(1-\theta) G^t_c(z)G^b_c(z) \right],
\end{split}
\end{align*}
Finally, we can compute the integral of the terms on the second line as minus the residues at $v$ and infinity, using that $G^t_c(z) = N/z + O(z^{-2})$ and $G^b_c(z) = (N-1)/z +O(z^{-2})$ as $|z| \rightarrow \infty$, and divide the whole expression by $ - \theta \cdot (v - a_-)(v - a_+)$ to get
\begin{align*}\label{2LoopRank1}
\begin{split} 
& 0 = \frac{1}{2\pi \i}\int_{\Gamma} dz \frac{N\theta \cdot (z - a_-)(z - a_+) }{(z-v)(v - a_-) (v-a_+)}   \mathbb{E} \left[  G_c^t(z) \cdot \partial_z V^t(z) + G_c^b(z) \cdot \partial_z V^b(z)\right] - \frac{N^2 - (1-\theta)N(N-1)}{(v - a_-) (v-a_+)}   \\
& + (1/2) \cdot \mathbb{E} \left[  - \partial_z G^t_c(v) -  \partial_z G^b_c(v)  + [G^t_c]^2(v) +  [G^b_c]^2(v) -  2(1-\theta) G^t_c(v)G^b_c(v) \right].
\end{split}
\end{align*}
This proves (\ref{TLRank00}) and next we focus on (\ref{TLRankmn}). \\

We next suppose that $m, n \geq 0$ are such that $m + n \geq 1$. Notice that we can restrict the sums in (\ref{SExpandTLNE}) to be over sets such that $|A^c| \leq 1$ and $|B^c|\leq 1$ as all other terms can be absorbed into the $O(L^{-1})$ part of the equation. These simplifications combined with (\ref{S6Phi2Lv2}) yield
\begin{align*}
\begin{split} 
& O(L^{-1}) =\frac{1}{2\pi \i}\int_{\Gamma}dz  \frac{ \theta P_L(z)}{L^2 \cdot (z - v)}   M \Bigg{(} -N\theta G_L^t(z) \cdot \partial_z V^t(z)  -  N\theta G_L^b(z) \cdot \partial_z V^b(z)    \\
& -  \frac{ \partial_z G^t_L(z)}{2} - \frac{ \partial_z G^b_L(z)}{2} + \frac{[G^t_L]^2(z)}{2}  + \frac{ [G^b_L]^2(z)}{2} -  (1-\theta) G^t_L(z)G^b_L(z);\llbracket 1,m \rrbracket, \llbracket 1,n \rrbracket  \Bigg{)}  \\
& +\sum_{a = 1}^m    \frac{\theta P_L(z)M \left( G^t_L(z) ; \llbracket 1, m \rrbracket \setminus \{a\},  \llbracket 1, n \rrbracket \right)}{( z - v)(v_a^1 -z)^2 \cdot L^2}  + \sum_{b = 1}^n   \frac{\theta P_L(z)M\left(   G^b_L(z) ;  \llbracket 1,m \rrbracket, \llbracket 1, n \rrbracket \setminus \{b\}\right)}{(z - v) (v^2_b - z)^2 \cdot L^2}  .
\end{split}
\end{align*}
We may now send $L \rightarrow \infty$ above and apply Proposition \ref{prop_cont_limit} to get
\begin{align*}
\begin{split} 
& 0  =\frac{1}{2\pi \i}\int_{\Gamma}  \frac{ \theta (z - a_-)(z - a_+)}{z - v}   \M \Bigg{(} - N\theta G_c^t(z)  \partial_z V^t(z)  - N\theta   G_c^b(z)  \partial_z V^b(z)    \\
&  -\frac{ \partial_z G^t_c(z)}{2} - \frac{ \partial_z G^b_c(z)}{2} + \frac{[G^t_c]^2(z)}{2} + \frac{ [G^b_c]^2(z)}{2} -  (1-\theta) G^t_c(z)G^b_c(z);\llbracket 1,m \rrbracket, \llbracket 1,n \rrbracket  \Bigg{)}    \\
& + \frac{ \theta (z - a_-)(z - a_+)}{z- v}\left[ \sum_{a = 1}^m    \frac{\M \left( G^t_c(z) ; \llbracket 1, m \rrbracket \setminus \{a\},  \llbracket 1, n \rrbracket \right)}{(v_a^1 -z)^2 }  + \sum_{b = 1}^n   \frac{ \M\left(  G^b_c(z) ;  \llbracket 1,m \rrbracket, \llbracket 1, n \rrbracket \setminus \{b\}\right)}{ (v^2_b - z)^2 } \right],
\end{split}
\end{align*}
where we write $\M(\xi; A, B)$ to mean the joint cumulant of $\xi$ and the variables $G^t_c(v_a^1)$ for $a \in A$ and $G^b_c(v_b^2)$ for $b \in B$. We may now evaluate the integrals of the terms on the second and third lines above as minus the residue at $z =v$ (there is no residue at infinity). After doing this we divide both sides by $- \theta \cdot (v - a_-)(v-a_+)$ and obtain
\begin{align*}
\begin{split} 
& 0  =\frac{N\theta}{2\pi \i}\int_{\Gamma} dz  \frac{ (z - a_-)(z - a_+)\M \left( G_c^t(z) \cdot \partial_z V^t(z) +  G_c^b(z) \cdot \partial_z V^b(z) ;\llbracket 1,m \rrbracket, \llbracket 1,n \rrbracket \right)}{(z-v)(v - a_-)(v - a_+)} +  \\
& +\M\left(  \frac{- \partial_z G^t_c(v)}{2} + \frac{-  \partial_z G^b_c(v)}{2} + \frac{[G^t_c]^2(v)}{2} + \frac{ [G^b_c]^2(v)}{2} -  (1-\theta) G^t_c(v)G^b_c(v) ;\llbracket 1,m \rrbracket, \llbracket 1,n \rrbracket \right)  \\
& + \sum_{a = 1}^m  \partial_{v^1_a}  \left[   \frac{\M\left(  G^t_c(v);  \llbracket 1,m \rrbracket \setminus \{a \}, \llbracket 1, n \rrbracket \right)}{v-v_a^1}  + \frac{ ( v_a^1-a_-)(v^1_a - a_+) \M\left(  G^t_c(v^1_a);  \llbracket 1,m \rrbracket \setminus \{a \}, \llbracket 1, n \rrbracket \right) }{(v^1_a- v)(v - a_-)(v - a_+) } \right] \\
& + \sum_{b = 1}^n  \partial_{v^2_b}  \left[   \frac{\M\left(  G^b_c(v);\llbracket 1,m \rrbracket , \llbracket 1, n \rrbracket \setminus \{b \} \right)}{v-v_b^2}  + \frac{ ( v_b^2-a_-)(v_b^2 - a_+) \M\left(  G^t_c(v_b^2);  \llbracket 1,m \rrbracket , \llbracket 1, n \rrbracket \setminus \{b \} \right) }{( v_b^2 - v)(v - a_-)(v - a_+) } \right].
\end{split}
\end{align*}
The latter equation is the same as  (\ref{TLRankmn}) once we invoke Malyshev's formula (\ref{RemainingCLE2}).

%
\section{Appendix A}\label{Section9}
In this section we give the proof of Lemmas \ref{S3AnalRQ}, \ref{S3Lsupp} and \ref{S3NonVanish}, recalled here as  Lemmas \ref{AnalRQ}, \ref{Lsupp} and \ref{NonVanish}, respectively. In what follows we assume the same notation as in Section \ref{Section3.1} and our work in this section will rely solely on Proposition \ref{LLN}, which as we mentioned earlier is \cite[Theorem 5.3]{BGG}, and Proposition \ref{SingleLevelNekrasov}, whose proof is given in Section \ref{Section10}. We begin by recalling a certain large deviation estimate for the measures $\mu_N$ in (\ref{EmpMeas}).\\

Take any two compactly supported absolutely continuous probability measures with uniformly bounded densities $\nu(x)dx$ and $\rho(x)dx$ and define $\mathcal{D}(\nu(x), \rho(x))$ through
\begin{equation}\label{B2}
\mathcal{D}^2(\nu(x), \rho(x)) = - \int_{\mathbb{R}} \int_{\mathbb{R}} \log|x-y| (\nu(x) - \rho(x))(\nu(y) - \rho(y)) dx dy.
\end{equation}
There is an alternative formula for $\mathcal{D}(\nu(x), \rho(x))$ in terms of Fourier transforms, cf. \cite{BeGu}:
\begin{equation}\label{B3}
\mathcal{D}(\nu(x), \rho(x)) = \sqrt{\int_{0}^\infty \frac{dt}{t} \left| \int_{\mathbb{R}} e^{-{\i} tx} (\nu(x) - \rho(x)) dx\right|^2 }.
\end{equation}

Fix a parameter $p > 2$ and let $\tilde{\mu}_N$ denote the convolution of the empirical measure $\mu_N$, see (\ref{EmpMeas}), with the uniform measure on the interval $[0, N^{-p}]$.
With the above notation we have the following result.

\begin{proposition}\label{BP1}
Suppose that Assumptions 1 and 2 hold and let $\mu$ be as in Proposition \ref{LLN}. Then there exists a constant $C > 0$ such that for all $x > 0$ and $N \geq 2$ 
$$\mathbb{P}_N \left( \mathcal{D}(\tilde{\mu}_N, \mu ) \geq x \right) \leq \exp\left( CN \log(N)^2 - \theta \cdot x^2 N^2\right).$$
The constant $C$ depends on the constants $A_1,A_2, A_3, A_4$ in Assumptions 1 and 2  as well as $ \lM, \theta$.
\end{proposition}
\begin{remark}\label{R10.1}
Proposition \ref{BP1} is essentially \cite[Proposition 5.6]{BGG}. A careful analysis of the proof of that proposition shows that the constant $C$ can be taken sufficiently large depending on $A_1,A_2, A_3, A_4$ as stated. We remark that \cite[Proposition 5.6]{BGG} has a missing $\theta$ in front of $\gamma^2 N^2$, which comes from the fact that when $\theta \neq 1$ equation (40) in \cite{BGG} should have $\theta$ in front of the $-\mathcal{D}^2$. See also \cite[Proposition 3.1.3]{DK1}.
\end{remark}

\begin{corollary}\label{BC1}
Assume the same notation as in Proposition \ref{BP1}. For a compactly supported Lipschitz function $g(x)$  define 
$$ \|g \|_{ 1/2} = \left( \int_{-\infty}^\infty |s| \left| \int_{-\infty}^\infty e^{-{\i}sx} g(x)dx \right|^2 ds \right)^{1/2}, \hspace{3mm} \|g\|_{\mbox{\em \Small Lip}} = \sup_{ x \neq y} \left|\frac{g(x) - g(y)}{x- y} \right|.$$
Fix any $p > 2$. Then for all $a > 0$, $N \geq 2$ and $g$ we have
\begin{equation}\label{B16}
\mathbb{P}_N \left( \left| \int_{\mathbb{R}} \hspace{-2mm} g(x) \mu_N(dx) -\hspace{-1mm}  \int_{\mathbb{R}}  \hspace{-2mm}  g(x) \mu(dx) \right| \geq a\|g  \|_{ 1/2}\hspace{-0.5mm} +\hspace{-0.5mm} \frac{ \|g\|_{\mbox{\em \Small Lip}}}{N^p} \right) \leq \exp\left( CN \log(N)^2 - 2\pi^2 \theta    a^2N^2\right) \hspace{-1mm},
\end{equation}
where $C$ is as in Proposition \ref{BP1}.
\end{corollary}
The above lemma is proved in \cite[Corollary 5.7]{BGG}, see also \cite[Corollary 3.1.4]{DK1}. We remark that (\ref{B16}) differs from (94) in \cite{BGG} in that there is an extra $\theta$, whose origin is described in Remark \ref{R10.1} and a $2\pi$, which comes from misapplication of Parseval's identity in the proof of Corollary 2.17 in \cite{BGG}.\\

We turn to the first main result of the section.
\begin{lemma}\label{AnalRQ}
Suppose that Assumptions 1-4 from Section \ref{Section3.1} hold. Then the functions $R_\mu$ and $Q_\mu^2$ from (\ref{QRmu}) are analytic on $\mathcal{M}$ and real-valued on $\mathcal{M} \cap \mathbb{R}$. 
\end{lemma}
\begin{remark}
In \cite[Section 5]{BGG} the authors prove the above lemma under Assumptions 1-3 and $\Phi_N^-(0) = 0$ and $\Phi^+_N(M_N + 1 +(N-1) \cdot \theta) = 0$. Below we give the proof when the last assumption above is replaced by the weaker Assumption 4 we have, and remark that this statement is implicit in \cite[Section 8]{BGG}. 
\end{remark}
\begin{proof}

For clarity we split the proof into several steps.

{\bf \raggedleft Step 1.} By Proposition \ref{SingleLevelNekrasov} we know that for all large $N$ the following function is analytic in $\mathcal{M}$ 
\begin{align}\label{SingleLevelEquationS10}
\begin{split}
R_N(Nz) = &\Phi^-_N(Nz)  \mathbb{E}_{\mathbb{P}_N} \left[ \prod_{i = 1}^N  \frac{Nz - \ell_i - \theta}{Nz - \ell_i}\right] \\
&+ \Phi^+_N(Nz)  \mathbb{E}_{\mathbb{P}_N}  \left[ \prod_{i = 1}^N  \frac{Nz - \ell_i  + \theta - 1}{Nz - \ell_i - 1} \right] - \frac{r^-(N)}{Nz} - \frac{r^+(N)}{Nz - s_N}, 
\end{split}
\end{align}
where $s_N = M_N+ 1 + (N-1) \cdot \theta$ and 
\begin{equation}\label{RemSLS10}
\begin{split}
&r^-(N) = \Phi^-_N(0) \cdot (-\theta) \cdot \mathbb{P}_N(\ell_N = 0) \cdot \mathbb{E}_{\mathbb{P}_N} \left[ \prod_{i= 1}^{N-1}  \frac{ \ell_i + \theta }{ \ell_i}\Big{\vert}  \ell_N = 0 \right] ;\\
&r^+(N) = \Phi^+_N(s_N) \cdot \theta \cdot \mathbb{P}_N(\ell_1 = s_N -1 ) \cdot \mathbb{E}_{\mathbb{P}_N} \left[\prod_{i = 2}^N  \frac{s_N - \ell_i + \theta - 1}{s_N - \ell_i - 1} \Big{\vert}  \ell_1 = s_N - 1 \right].
\end{split}
\end{equation}
In this step we show that for some $\tilde{C} > 0$ and $c,a > 0$ as in Assumption 4 and all large $N$
\begin{equation}\label{ResidueDecay}
|r^{\pm}(N)| \leq \tilde{C}N e^{-cN^a}.
\end{equation} 
Since $\ell_{i} \geq (N-i) \cdot \theta$ for $i = 1, \dots, N-1$ , we know that \vspace{-2mm}
$$0 \leq \prod_{i= 1}^{N-1}  \frac{ \ell_i + \theta }{ \ell_i} \leq \prod_{i= 1}^{N-1}  \frac{ (N- i + 1) \cdot \theta }{ (N-i) \cdot \theta} = N.$$
Similarly, if $\ell_1 = s_N - 1 =  M_N + (N-1) \cdot \theta$ then $s -\ell_i - 1 \geq (i-1) \cdot \theta$ and so \vspace{-2mm}
$$0 \leq \prod_{i = 2}^N  \frac{s_N - \ell_i + \theta - 1}{s_N - \ell_i - 1} \leq \prod_{i = 2}^N \frac{ i \cdot \theta}{(i -1) \cdot \theta} = N.$$
\vspace{-3mm}

{\raggedleft The above} inequalities together with (\ref{RemSLS10}) and Assumptions 3 and 4 imply (\ref{ResidueDecay}).\\

{\bf \raggedleft Step 2.} Let us fix a compact set $K \subset \mathcal{M} \setminus [0, \lM + \theta]$. We claim that 
\vspace{-2mm}
\begin{align}\label{convST}
\begin{split}
&\limsup_{N \rightarrow \infty} \sup_{v \in K} \left| \mathbb{E}_{\mathbb{P}_N} \left[  \prod_{i = 1}^N  \frac{Nv - \ell_i - \theta}{Nv - \ell_i}\right]  - e^{-\theta G_{\mu}(v)} \right| = 0,\\
&\limsup_{N \rightarrow \infty} \sup_{v \in K} \left| \mathbb{E}_{\mathbb{P}_N} \left[  \prod_{i = 1}^N  \frac{Nv - \ell_i + \theta - 1}{Nv - \ell_i - 1} \right]  - e^{\theta G_{\mu}(v)} \right| = 0.
\end{split}
\end{align}
\vspace{-2mm}

{\raggedleft We defer} the proof of (\ref{convST}) to Step 4. For now we assume it and finish the proof of the lemma.\\

In view of (\ref{convST}), (\ref{ResidueDecay}), (\ref{SingleLevelEquationS10}) and Assumption 3, we conclude that for any compact $K \subset \mathcal{M} \setminus [0, \lM + \theta]$
\begin{equation}\label{convAway}
\limsup_{N \rightarrow \infty} \sup_{v \in K}  \left| R_N(Nv) -R_\mu(v) \right| = 0.
\end{equation}
Let $\gamma$ be a fixed thin positively oriented rectangle that encloses the segment $[0, \lM + \theta]$ and is contained in $\mathcal{M}$. We also let $d_\gamma = \mbox{dist}(\gamma, [0, \lM + \theta]) > 0$. We let $R$ denote the open rectangular region enclosed by $\gamma$ and let $\phi_\gamma: [0,1] \rightarrow \mathbb{C}$ be a fixed piecewise linear parametrization of $\gamma$. For $v \in R$ we define 
\begin{equation}\label{Ralt}
\tilde{R}_\mu(v) = \frac{1}{2\pi \i} \int_0^1 \frac{ R_\mu(\phi_\gamma(s)) }{v - \phi_\gamma(s)} \phi_\gamma'(s) ds.
\end{equation}
Observe that since $G_\mu(z)$ is analytic in $\mathcal{M} \setminus [0, \lM + \theta]$ (as the support of $\mu$ is contained in $[0, \lM + \theta]$) the same is true for $R_{\mu}(v)$. This implies that $R_\mu(\phi_\gamma(s))$ is continuous on $[0,1]$ and so by \cite[Chapter 2, Theorem 5.4]{SS} we conclude that $\tilde{R}_\mu(v)$ is analytic in $R$. Fix a compact set $K_1 \subset R$. We claim that 
\begin{equation}\label{Raltconv}
\limsup_{N \rightarrow \infty} \sup_{v \in K_1} \left| R_N(Nv) - \tilde{R}_\mu(v) \right| = 0.
\end{equation}
We will prove (\ref{Raltconv}) in Step 3 below. For now we assume it and finish the proof of the lemma.\\

In view of (\ref{convAway}) and (\ref{Raltconv}) we see that $R_{\mu}(v) =  \tilde{R}_\mu(v)$ for all $v \in (\mathcal{M} \setminus [0, \lM + \theta]) \cap R$. The latter implies that $R_\mu(z)$ can be analytically extended to $\mathcal{M}$ by setting $R_\mu(v):= \tilde{R}_\mu(v)$ for $v \in [0, \lM + \theta]$. This proves the analyticity of $R_\mu(z)$. Since by definition 
$$Q^2_\mu(z) = R^2_\mu(z) - 4 \Phi^-(z) \Phi^+(z),$$
the analyticity of $R_\mu(z)$ and $\Phi^{\pm}(z)$ in $\mathcal{M}$ implies that of $Q^2_\mu(z)$. Finally, we note that $G_\mu(z)$ is real if $z < 0$ and so $R_\mu(z)$ is real-valued on $(-d_\gamma, 0)$, which implies that $R_\mu(z)$ is real-valued on $\mathcal{M} \cap \mathbb{R}$ and then so is $Q^2_\mu(z)$ since $\Phi^{\pm}$ in Assumption 3 are analytic in $\mathcal{M}$ and real-valued on $\mathcal{M} \cap \mathbb{R}$. This proves the lemma.\\

{\bf \raggedleft Step 3.} In this step we continue to assume (\ref{convST}) and finish the proof of (\ref{Raltconv}). Let us denote 
$$F_N(z) :=  \Phi^-_N(Nz)  \mathbb{E}_{\mathbb{P}_N} \left[ \prod_{i = 1}^N  \frac{Nz - \ell_i - \theta}{Nz - \ell_i}\right] +  \Phi^+_N(Nz)  \mathbb{E}_{\mathbb{P}_N} \left[ \prod_{i = 1}^N  \frac{Nz - \ell_i  + \theta - 1}{Nz - \ell_i - 1} \right].$$
Using (\ref{SingleLevelEquationS10}) and Cauchy's integral formula we have for any $v \in K_1$ that
\begin{equation}\label{CauchyRN}
R_N(Nv) = \frac{1}{2\pi \i} \int_{\gamma} \frac{R_N(Nz)}{v -z} dz =  \frac{1}{2\pi \i}\int_0^1 \frac{F_N(\phi_\gamma(s))  }{v - \phi_\gamma(s)} \phi_\gamma'(s) ds -  A_N(v),
\end{equation}
where 
$$A_N(v) = \frac{r^-(N)}{N} \frac{1}{2\pi \i}\int_0^1 \frac{ \phi_\gamma'(s) ds}{(v - \phi_\gamma(s)) \phi_\gamma(s) } + \frac{r^+(N)}{N} \frac{1}{2\pi \i}\int_0^1 \frac{\phi_\gamma'(s) ds}{(v - \phi_\gamma(s)) (\phi_\gamma(s) - s_N \cdot N^{-1}) } .$$
Since $s_N \cdot N^{-1} \rightarrow \lM + \theta$ we conclude that we can find a constant $C_1 > 0$ such that for all large $N$  
$$ \sup_{v\in K_1}\left| A_N (v) \right| \leq C_1 \cdot N^{-1}  \cdot d_\gamma^{-1} \cdot d_1^{-1} \cdot \left[ |r^-(N)| + |r^+(N)| \right],$$
where $d_1 = \mbox{dist}( \gamma, K_1) > 0$. Combining the above with (\ref{ResidueDecay}) we conclude that 
\begin{equation}\label{ANBound}
\limsup_{N \rightarrow \infty} \sup_{v \in K_1} \left| A_N (v) \right|  = 0.
\end{equation}
Combining (\ref{ANBound}), (\ref{CauchyRN}) and (\ref{Ralt}) we see that 
$$\limsup_{N \rightarrow \infty} \sup_{v \in K_1}  \left| R_N(Nv) - \tilde{R}_\mu(v) \right|  \leq \frac{d_1^{-1}}{2\pi } \sup_{s \in [0,1]} |\phi'_\gamma(s)| \cdot \limsup_{N \rightarrow \infty} \sup_{z \in \gamma} |F_N( z) - R_\mu(z)| = 0,$$
where the last inequality used (\ref{convAway}) as $\gamma$ is compactly supported in $ \mathcal{M} \setminus [0, \lM + \theta]$. This proves (\ref{Raltconv}).\\

{\bf \raggedleft Step 4.} In this step we prove (\ref{convST}) and we fix a non-empty compact set $K \subset \mathcal{M} \setminus [0, \lM + \theta]$. Firstly, one readily observes the following asymptotic expansions
\begin{align}\label{S10AE1}
\begin{split}
&\prod_{i = 1}^N  \frac{Nz - \ell_i - \theta}{Nz - \ell_i} = \exp \left( - \theta \cdot N^{-1} \cdot G_N^t(z) +O(N^{-1})  \right)\\
& \prod_{i = 1}^N  \frac{Nz - \ell_i + \theta - 1}{Nz - \ell_i - 1} = \exp \left( \theta \cdot N^{-1} \cdot G_N^t(z) +O(N^{-1})\right),
\end{split}
\end{align}
where the constants in the big $O$ notation are deterministic and uniform as $z$ varies over $K$ and (\ref{S10AE1}) holds $\mathbb{P}_N$-almost surely. We recall that $G^t_N(z)$ was defined in Section \ref{Section3.2}. Let $\eta > 0$ be sufficiently small so that $I_\eta := [-\eta , \lM + \theta + \eta] \subset \mathcal{M}$ and $ I_\eta \cap K = \varnothing$. Let $h(x)$ be a smooth function such that $0 \leq h(x) \leq 1$,  $h(x) = 1$ if $x \in [-\eta/2, \lM + \theta + \eta/2]$, $h(x) = 0$ if $x \leq -\eta$ or $x \geq \lM + \theta + \eta$ and $\sup_{ x \in I_\eta} |h'(x)| \leq \eta^{-1} \cdot 10$. Since $M_N \cdot N^{-1} \rightarrow \lM$ as $N \rightarrow \infty$ we know that for all large $N$ and $v \in K$ we have $\mathbb{P}_N$ almost surely
\begin{equation}\label{compactify}
\left| \int_{\mathbb{R}}  g_v(x) \cdot h(x)  \mu_N(dx) -  \int_{\mathbb{R}} g_v(x) \cdot h(x)  \mu(dx) \right|  = \left| N^{-1} G_N^t(v) - G_\mu(v)\right|,
\end{equation}
where $g_v(x) = (v -x)^{-1}$. Let us denote 
\begin{equation}\label{c1c2}
c_1(K) := \sup_{v \in  K} \| g_v \cdot h \|_{1/2} \mbox{ and } c_2(K) := \sup_{v \in  K} \| g_v \cdot h \|_{\mbox{ \Small Lip}}.
\end{equation} 
It is clear that $c_1(K) > 0$ and $c_2(K) > 0$ and we claim that
\begin{equation}\label{c1c2Bound}
c_1(K) < \infty \mbox{ and } c_2(K) < \infty.
\end{equation} 
We will prove (\ref{c1c2Bound}) in Step 5. For now we assume its validity and finish the proof of (\ref{convAway}).\\

From (\ref{S10AE1}) we see that for all large enough $N$ and $v \in K$
\begin{align}\label{ANKBNK}
\begin{split}
 &\left| \mathbb{E}_{\mathbb{P}_N} \left[  \prod_{i = 1}^N  \frac{Nv - \ell_i - \theta}{Nv - \ell_i} \cdot e^{\theta G_\mu(v) } \right] - 1 \right| \leq A_N(K) + B_N(K), \mbox{ where }\\
& A_N(K) =   \Big{\vert} \exp \left( \theta \cdot  \left( N^{-1/4} c_1(K) + N^{-3} c_2(K)  \right) + N^{-1/2}  \right) \cdot \\
& \mathbb{P}_N\left(  \left| N^{-1} G_N^t(v) - G_\mu(v)\right| \leq N^{-1/4} c_1(K) + N^{-3} c_2(K)   \right) - 1 \Big{\vert}, \\
&B_{N}(K) =  N \cdot  \mathbb{P}_N\left(  \left| N^{-1} G_N^t(v) - G_\mu(v)\right| > N^{-1/4} c_1(K) + N^{-3} c_2(K)   \right),
\end{split}
\end{align}
where we used that from (\ref{S10AE1}) for all large $N$ and $v \in K$ we have $\mathbb{P}_N$-almost surely
$$\left| \prod_{i = 1}^N  \frac{Nv - \ell_i - \theta}{Nv - \ell_i} \cdot e^{\theta G_\mu(v) } \right| \leq \exp \left(  \theta \cdot | N^{-1} \cdot G_N^t(v) - G_\mu(v)| + N^{-1/2}\right) \leq N $$
and also we split the expectation in the first line of (\ref{ANKBNK}) over the events where $\left| N^{-1} G_N^t(v) - G_\mu(v)\right| \leq N^{-1/4} c_1(K) + N^{-3} c_2(K)$ and $\left| N^{-1} G_N^t(v) - G_\mu(v)\right| > N^{-1/4} c_1(K) + N^{-3} c_2(K)$.

Combining (\ref{compactify}) and (\ref{B16}) for $p = 3$, $a = N^{-1/4}$ we conclude that 
\begin{align}\label{ANKBNKUB}
\begin{split}
&\limsup_{N \rightarrow \infty} A_N(K) \leq  \limsup_{N \rightarrow \infty}   \Big{\vert} \exp \left( \theta \cdot  \left( N^{-1/4} c_1(K) + N^{-3} c_2(K)  \right) + N^{-1/2}  \right) - 1 \Big{\vert}  \\
& + \limsup_{N \rightarrow \infty}  \exp \left( \theta \cdot  \left( N^{-1/4} c_1(K) + N^{-3} c_2(K)  \right) + N^{-1/2}  \right) \\
&\times  \mathbb{P}_N\left(  \left| N^{-1} G_N^t(v) - G_\mu(v)\right| > N^{-1/4} c_1(K) + N^{-3} c_2(K)   \right) \leq  \\
&   \limsup_{N \rightarrow \infty}   \exp\left( \theta \cdot  \left( N^{-1/4} c_1(K) + N^{-3} c_2(K)  \right) + N^{-1/2}   + CN \log(N)^2 - 2\pi^2 \theta \cdot   N^{3/2} \right) = 0,\\
&\limsup_{N \rightarrow \infty} B_{N}(K) \leq   \limsup_{N \rightarrow \infty}  N \cdot  \exp\left( CN \log(N)^2 - 2\pi^2 \theta \cdot   N^{3/2} \right) = 0.
\end{split}
\end{align}
From (\ref{ANKBNK}) and (\ref{ANKBNKUB}) we conclude the first line in (\ref{convST}). The second line in (\ref{convST}) is derived in the same way -- we only need to replace the left side of the first line in (\ref{ANKBNK}) by 
$$\left| \mathbb{E}_{\mathbb{P}_N} \left[ \prod_{i = 1}^N  \frac{Nv - \ell_i + \theta - 1}{Nv - \ell_i - 1}  \cdot e^{-\theta G_\mu(v) } \right] - 1 \right|.$$

{\bf \raggedleft Step 5.} In this step we establish (\ref{c1c2Bound}). We first note that $g_v(x)$ is analytic and so in particular we can find a constant $C(K) > 0 $ such that 
$$\sup_{x \in [-\eta, \lM + \theta + \eta]}[ |g_v'(x)| + |g_v(x)|] = \sup_{x \in [-\eta, \lM + \theta + \eta]} \left[ \frac{1}{|v -x|^2} + \frac{1}{|v-x|} \right]\leq C(K).$$
Furthermore, by assumption $|h'(x)| \leq 10 \cdot \eta^{-1}$ and $|h(x)| \leq 1$. The latter implies that 
$$c_2(K) = \sup_{v \in  K} \| g_v \cdot h \|_{\mbox{ \Small Lip}} \leq \sup_{v \in  K}  \sup_{x \in [-\eta, \lM + \theta + \eta]} 2 \left| \frac{d}{dx} [g_v(x)  h(x)] \right| \leq 2C(K) \cdot \left(1 + \frac{10}{\eta} \right),$$
which proves that $c_2(K) < \infty$. 

Observe by definition that 
$$[g_v \cdot h]_{H^{1/2}(\mathbb{R})} := \int_{\mathbb{R}} \int_{\mathbb{R}} \frac{| g_v(x) \cdot h(x) -  g_v(y) \cdot h(y)|^2}{|x-y|^2} \leq c^2_2(K) \cdot [ \lM + \theta + 2\eta]^2.$$
On the other hand, as can be deduced from the proof of \cite[Proposition 3.4]{NPV} we have
$$[g_v \cdot h]_{H^{1/2}(\mathbb{R})}  = 2 C(1,1/2)^{-1} \cdot \|g_v \cdot h \|_{1/2}^2, \mbox{ where }C(1,1/2) = \int_{\mathbb{R}} \frac{1 - \cos (x)}{x^2}dx \in (0, \infty).$$ 
Combining the last two equations and $c_2(K) < \infty$ shows that $c_1(K) < \infty$. This concludes the proof of (\ref{c1c2Bound}) and hence the lemma.
\end{proof}

Our next goal is to give a formula for the equilibrium measure $\mu$ in Proposition \ref{LLN} in terms of the functions $R_\mu$ and $\Phi^{\pm}$.
\begin{lemma}\label{Lsupp} Suppose that Assumptions 1-4  from \ref{Section3.1} hold. Then $\mu$ has density
\begin{equation}\label{eqMForm}
 \mu(x) =  \frac{1}{\theta \pi } \cdot \mathrm{arccos} \left( \frac{R_\mu(x)}{2 \sqrt{\Phi^-(x)  \Phi^+(x)}}\right),
\end{equation}
for $x \in [0, \lM + \theta]$ and $0$ otherwise. In particular, $\mu(x)$ is continuous in $[0, \lM + \theta]$. 
\end{lemma}
\begin{proof}
Let us denote $g(x) = {\bf 1}\{x \in [0, \lM + \theta]\} \cdot \mu(x)$ for $x \in \mathbb{R}$ and note that $g \in L^2(\mathbb{R})$. Following \cite[Chapter 5, Theorem 91]{Tit} and its proof we have that the limit
$$- \frac{1}{\pi} \lim_{\epsilon \rightarrow 0^+} \int_{|t| > \epsilon} \frac{g(t)dt}{t - x} =: -\frac{1}{\pi} P \int_\mathbb{R} \frac{g(t)dt}{t -x},$$
exists almost everywhere and defines a function $f(x) \in L^2(\mathbb{R})$. $P$ means that we take the integral in the principal value sense. Furthermore for $z \in \mathbb{H}$ we have
\begin{equation}\label{PhiCDef}
\Phi(z):= \frac{1}{\i \pi} \int_{\mathbb{R}} \frac{f(t)dt}{t-z} = - \frac{1}{\pi} \int_{\mathbb{R}} \frac{g(t)dt}{t-z}
\end{equation}
and for almost every $x \in \mathbb{R}$ we have
\begin{equation}\label{PhiLim}
\lim_{y \rightarrow 0^+} \Phi(x + \i y) = f(x) - \i g(x) \mbox{ and }\lim_{y \rightarrow 0^+} \Phi(x - \i y) = f(x) + \i g(x).
\end{equation}

Recall by Assumption 3 that $\Phi^{\pm}(x)$ are analytic functions on $\mathcal{M}$, real-valued on $\mathcal{M} \cap \mathbb{R}$, and also $\Phi^{\pm}(x) > 0$ for $x \in (0, \lM + \theta)$. In addition, by Lemma \ref{AnalRQ} we know that $R_\mu(x)$ is analytic in $\mathcal{M}$ and real-valued on $\mathcal{M} \cap \mathbb{R}$. We may thus define the function
$$F(x) := \frac{R_\mu(x)}{2\sqrt{\Phi^-(x) \Phi^+(x)}}$$ 
for $x \in (0, \lM + \theta)$ and note that this function is smooth in $(0, \lM + \theta)$ and for each $(\lM + \theta)/2 > \delta > 0$ it is analytic in a complex neighborhood of $[\delta, \lM + \theta - \delta]$ and real-valued on its restriction to $\mathbb{R}$. Let us denote 
\begin{align*}
& S_b := \{x \in (0, \lM + \theta): -1 < F(x) < 1\}, \hspace{2mm}S_v := \{x \in (0, \lM + \theta): F(x) \geq 1\}, \\
& S_s := \{x \in (0, \lM + \theta): F(x) \leq -1\}.
\end{align*}

Recall that for $z \in \mathcal{M} \cap \mathbb{H}$ we have from (\ref{QRmu}) and (\ref{PhiCDef}) that 
\begin{equation}\label{RmuAlt}
R_\mu(z) = \Phi^-(z) \cdot e^{-\theta G_\mu(z)} + \Phi^+(z) \cdot e^{\theta G_\mu(z)} = \Phi^-(z) \cdot e^{-\theta \pi \Phi(z)} + \Phi^+(z) \cdot e^{\theta \pi \Phi(z)}. 
\end{equation}
Let $x \in S_b$ be such that the limit (\ref{PhiLim}) exists, it is finite and $g(x) \in [0, \theta^{-1}]$. Then taking the limit $\varepsilon \rightarrow 0^+$ with $z = x \pm \i \varepsilon$ in (\ref{RmuAlt}) we conclude 
\begin{equation*}
R_\mu(x) = \Phi^-(x) \cdot e^{-\theta \pi [f(x) \mp  \i g(x)]} + \Phi^+(x) \cdot e^{\theta \pi [f(x) \mp  \i g(x)]}.
\end{equation*}
The above implies that $e^{-\theta \pi [f(x) \mp \i g(x)]} $ are roots of 
\begin{equation}\label{quadEqn}
P(X) := \Phi^-(x) \cdot X^2 - R_\mu(x) X + \Phi^+(x) = 0
\end{equation}
 and we conclude that 
\begin{equation}\label{Roots}
\{ e^{-\theta \pi [f(x) \pm \i g(x)]}\} =  \left\{ \frac{ R_\mu(x) \pm \sqrt{  R^2_\mu(x) - 4\Phi^-(x) \Phi^+(x) }}{2 \Phi^-(x)} \right\},
\end{equation}
where the square root is with respect to the principal branch and assumed in $\mathbb{H}$ for negative values. In (\ref{Roots}) we have that the set (of at most two numbers) on the left side is the same as that on the right. Since $x \in S_b$ we know that $F^2(x) \in (0,1)$ and so
$$R^2_\mu(x) < 4 \Phi^-(x) \Phi^+(x).$$
The latter and the fact that $g(x) \in [0, \theta^{-1}]$ imply that $e^{-\theta \pi [f(x) - \i g(x)]}$ lies in $\mathbb{H}$ and so we conclude 
\begin{equation}\label{Roots2}
\begin{split}
&e^{-\theta \pi [f(x) \pm \i g(x)]}=  \frac{ R_\mu(x) \mp \sqrt{  R^2_\mu(x) - 4\Phi^-(x) \Phi^+(x) }}{2 \Phi^-(x)}\\
\end{split}
\end{equation}
Taking absolute values on both sides of the above we get
$$e^{-\theta \pi f(x)} = \sqrt{\Phi^+(x) / \Phi^-(x)},$$
and then taking the real part on both sides we get
\begin{equation}\label{BandEquality}
\sqrt{\Phi^+(x) / \Phi^-(x)} \cdot \cos ( \theta \pi g(x)) = \frac{ R_\mu(x) }{2 \Phi^-(x)} \iff g(x) = \frac{1}{\pi \theta} \cdot \mathrm{arccos} \left( F(x)\right).
\end{equation}
Since the latter is true for a.e. $x \in S_b$, we conclude (\ref{eqMForm}) for $x \in S_b$. \\

Suppose next that $x\in S_v$ is such that (\ref{PhiLim}) exists, it is finite and $g(x) \in [0, \theta^{-1}]$. We still have that $e^{-\theta \pi [f(x) - \i g(x)]}
$ is a root of $P(X)$ in (\ref{quadEqn}). If $F(x) \geq 1$ then the roots of $P(X)$ are still given by the right side of (\ref{Roots}) and so both are positive and real. Since $g(x) \in [0, \theta^{-1}]$ we conclude that $g(x) = 0$. We see that for a.e. $x \in S_v$ we have (\ref{eqMForm}). 

Suppose next that $x\in S_s$ is such that (\ref{PhiLim}) exists, it is finite and $g(x) \in [0, \theta^{-1}]$. We still have that $e^{-\theta \pi [f(x) - \i g(x)]}
$ is a root of $P(X)$ in (\ref{quadEqn}). If $F(x) \leq  -1$ then the roots of $P(X)$ are negative and real and since $g(x) \in [0, \theta^{-1}]$ we conclude that $g(x) = \theta^{-1}$. We see that for a.e. $x \in S_s$ we have (\ref{eqMForm}).

Combining all of the above work and the fact that $S_b \cup S_v \cup S_s = (0, \lM + \theta)$ we conclude (\ref{eqMForm}). In the remainder we focus on the last statement in the lemma. Clearly, $\mu(x)$ is continuous on $(0, \lM + \theta)$. We show that it can be continuously extended to the endpoints $0$ and $\lM + \theta$ as well.\\

We will only show that $\mu(x)$ can be continuously extended to $0$ and remark that a similar argument shows that the same can be done for $\lM + \theta$. In view of Assumption 3 we know that there exist non-negative integers $m, n$ and reals $\{a_k\}_{k = n}^\infty$ and $\{b_k\}_{k = m}^\infty$ with $a_n > 0$ and $b_m > 0$ such that $\Phi^{+}(z)$ and $\Phi^-(z)$ have the following absolutely convergent power series expansion near $0$
$$\Phi^{-}(z) = \sum_{k = n}^\infty a_k z^k \mbox{ and } \Phi^{-}(z) = \sum_{k = m}^\infty b_k z^k.$$
We further know by Lemma \ref{AnalRQ} we know that there is a non-negative integer $d$ and reals $\{c_k\}_{k = d}^\infty$ such that $c_d \neq 0$ and $R_\mu(z)$ has the following absolutely convergent power series expansion near $0$
$$R_{\mu}(z) = \sum_{k = d}^\infty c_k z^k.$$
We observe that
$$\lim_{\epsilon \rightarrow 0^+} \epsilon^{(m+n)/2 - d} \cdot F(\epsilon) = \frac{c_d}{2\sqrt{a_n b_m}}.$$
Suppose first that $c_d > 0$. If $d < (m + n) /2$ then we see that $F(\epsilon) > 1$ and so $\mu(\epsilon) = 0$ for all small enough $\epsilon > 0$, which means we can continuously extend $\mu(x)$ to $0$ by setting it to $0$ there. If $d> (m+n)/2$ then $F(x)$ continuously extends to $0$, where it equals $0$, which means we can continuously extend $\mu(x)$ to $0$ by setting it to $\theta^{-1}$ there. If $d = (m+n)/2$ then $F(x)$ continuously extends to $0$, where it equals $ \frac{c_d}{2\sqrt{a_n b_m}}$ and so we can continuously extend $\mu(x)$ to $0$ by setting it equal to 
$$\frac{1}{\theta \pi} \cdot \mathrm{arccos} \left(  \frac{c_d}{2\sqrt{a_n b_m}} \right).$$
A similar argument applies if $c_d < 0$. If $d < (m + n) /2$ then we can continuously extend $\mu(x)$ to $0$ by setting it to $\theta^{-1}$ there. If $d> (m+n)/2$ then we can continuously extend $\mu(x)$ to $0$ by setting it to $(2\theta)^{-1}$ there. If $d = (m+n)/2$ then we can continuously extend $\mu(x)$ to $0$ by setting it equal to 
$$\frac{1}{\theta \pi} \cdot \mathrm{arccos} \left(  \frac{c_d}{2\sqrt{a_n b_m}} \right).$$
\end{proof}

We end this section by proving the following result.
\begin{lemma}\label{NonVanish}  If Assumptions 1-5  from \ref{Section3.1} hold then $\Phi^-(x) + \Phi^+(x) - R_\mu(x) \neq 0$ for all $x \in [0, \lM + \theta]$. 
\end{lemma}
\begin{proof}
We split the proof into two parts for clarity.\\

{\bf \raggedleft Part I.} We continue with the same notation as in Lemma \ref{Lsupp}. Recall that by Assumption 5 we know that 
\begin{equation}\label{QmuRecall}
Q_\mu(z) = H(z) \cdot \sqrt{(z-\alpha)(z - \beta)},
\end{equation}
where $0 \leq \alpha < \beta \leq \lM + \theta$ and $H(z)$ is analytic in a complex neighborhood of $[0, \lM + \theta]$ and does not vanish in $[0, \lM + \theta]$. In this part we show that $H(x)$ is analytic in a complex neighborhood of $[0, \lM + \theta]$ and is strictly positive (and in particular real-valued) on $[0, \lM + \theta]$. We also show that
\begin{equation} \label{BANDS}
S_b =(\alpha, \beta) \mbox{ and } S_v \cup S_s = (0, \lM + \theta) \setminus (\alpha, \beta).
\end{equation}

Observe that for all $z \in \mathcal{M} \cap \mathbb{H}$ we have from (\ref{QRmu}) and (\ref{PhiCDef}) that 
\begin{equation}\label{QmuAlt}
Q_\mu(z) = \Phi^-(z) \cdot e^{-\theta G_\mu(z)} - \Phi^+(z) \cdot e^{\theta G_\mu(z)} = \Phi^-(z) \cdot e^{-\theta \pi \Phi(z)} - \Phi^+(z) \cdot e^{\theta \pi \Phi(z)}. 
\end{equation}

Let $x \in S_b$ be such that the limit (\ref{PhiLim}) exists, it is finite and $g(x) \in [0, \theta^{-1}]$. Then taking the limit $\varepsilon \rightarrow 0^+$ with $z = x + \i \varepsilon$ in (\ref{QmuAlt}) and using (\ref{Roots2}) we get
\begin{equation}\label{Qlim}
\lim_{\varepsilon \rightarrow 0^+} Q_\mu(x \pm \i \varepsilon) =  \Phi^-(x) \cdot   \frac{ R_\mu(x) \pm \sqrt{  R^2_\mu(x) - 4\Phi^-(x) \Phi^+(x) }}{2 \Phi^-(x)} -  \frac{2 \Phi^-(x) \Phi^+(x) }{ R_\mu(x) \pm \sqrt{  R^2_\mu(x) - 4\Phi^-(x) \Phi^+(x) }}.
\end{equation}
Rationalizing the second term and using that $x \in S_b$ we get
\begin{equation}\label{Qlimpm}
\lim_{\varepsilon \rightarrow 0^+} Q_\mu(x \pm \i \varepsilon) =  \pm \i \sqrt{  4\Phi^-(x) \Phi^+(x) -R^2_\mu(x) }.
\end{equation}
The above equation implies that a.e. on $S_b$ we have
$$\lim_{\varepsilon \rightarrow 0^+} Q_\mu(x + \i \varepsilon) - Q_\mu(x - \i \varepsilon) =2 \i \sqrt{  4\Phi^-(x) \Phi^+(x) -R^2_\mu(x) } \neq 0,$$
which in view of (\ref{QmuRecall}) implies that $S_b \subseteq (\alpha, \beta)$. 

On the other hand, suppose $x \in S_v$ is such that the limit (\ref{PhiLim}) exists, it is finite and $g(x) = 0$. Then taking the limit $\varepsilon \rightarrow 0^+$ with $z = x + \i \varepsilon$ in (\ref{QmuAlt}) we again obtain
\begin{equation}\label{Qlimpm2}
\lim_{\varepsilon \rightarrow 0^+} Q_\mu(x \pm \i \varepsilon) =   \Phi^-(x) \cdot  e^{-\theta \pi f(x)} -   \Phi^+(x) \cdot  e^{\theta \pi f(x)}.
\end{equation}
In view of Lemma \ref{Lsupp} we conclude that a.e. on $S_v$ we have
$$\lim_{\varepsilon \rightarrow 0^+} Q_\mu(x + \i \varepsilon) - Q_\mu(x - \i \varepsilon) =0,$$
which in view of (\ref{QmuRecall}) implies that $S_v \subseteq (0, \lM + \theta) \setminus (\alpha, \beta)$. An analogous argument shows that $S_s \subseteq (0, \lM + \theta) \setminus (\alpha, \beta)$. But now $S_b, S_v, S_s$ are pairwise disjoint and their union is $(0, \lM + \theta)$ and the same is true for $(\alpha, \beta)$ and $ (0, \lM + \theta) \setminus (\alpha, \beta)$. Consequently, we conclude (\ref{BANDS}). 

Combining (\ref{Qlimpm}) and (\ref{QmuRecall}) we see that for a.e. $x \in (\alpha, \beta)$ we have
$$H(x) \cdot \i \cdot \sqrt{(x - \alpha)(\beta - x)} = \i \cdot\sqrt{  4\Phi^-(x) \Phi^+(x) -R^2_\mu(x) },$$
and so we conclude that $H(x)$ is analytic in a complex neighborhood of $[0, \lM + \theta]$, non-negative on $[0, \lM + \theta]$, and since it does not vanish in $[0, \lM + \theta]$ we conclude it is strictly positive there.\\

{\bf \raggedleft Part II.} In this part we give the proof of the lemma. For the sake of contradiction suppose that $\Phi^+(x_0) + \Phi^-(x_0) = R_\mu(x_0) $ for some $x_0 \in [0, \lM + \theta]$.

 Suppose first that $ x_0 \in (\alpha, \beta)$. We then have that $F(x_0) \in (-1, 1)$ and so 
$$(\Phi^+(x_0) + \Phi^-(x_0) )^2 = R_\mu^2(x_0) < 4 \Phi^+(x_0) \Phi^-(x_0) \implies (\Phi^+(x_0) - \Phi^-(x_0) )^2 < 0,$$
which is clearly impossible. 

Suppose next that $x_0 \in (\beta, \lM + \theta) \neq \varnothing$. From (\ref{BANDS}) we know that $(\beta, \lM + \theta) \subseteq S_v \cup S_s $ and by the continuity of $F(x)$ we conclude that $(\beta, \lM + \theta) \subseteq S_s$ or $(\beta, \lM + \theta) \subseteq S_v$. In the former case we have $F(x_0) \leq -1$ and so $R_\mu(x_0) < 0$, which is a contradiction as $\Phi^+(x_0) + \Phi^-(x_0) > 0$. We thus conclude that $(\beta, \lM + \theta) \subseteq S_v$ and so $g(x) = 0$ for all $x \in (\beta, \lM + \theta).$ The latter implies that 
\begin{equation}\label{Gmupos}
G_\mu(z) = \int_0^{\lM + \theta} \frac{\mu(x)dx}{z - x}
\end{equation}
is analytic near $x_0$, $G_\mu(x_0) > 0$ and also
$$0 < H(x_0) \cdot \sqrt{(x_0 - \alpha)(x_0 - \beta)} = Q_\mu(x_0) = \Phi^-(x_0) \cdot e^{-\theta G_\mu(x_0)} - \Phi^+(x_0) \cdot e^{\theta G_\mu(x_0)} .$$
In particular, we see that 
$$\Phi^-(x_0)  >  \Phi^+(x_0) \cdot e^{2\theta G_\mu(x_0)} > \Phi^+(x_0) \cdot e^{\theta G_\mu(x_0)} > \Phi^+(x_0).$$
In view of 
\begin{equation*}
R_\mu(x_0) = \Phi^-(x_0) \cdot e^{-\theta G_\mu(x_0)} + \Phi^+(x_0) \cdot e^{\theta G_\mu(x_0)} .
\end{equation*}
we conclude that
\begin{equation}\label{Rmuneg}
0 =R_\mu(x_0) - \Phi^-(x_0) - \Phi^+(x_0) = \left[ \Phi^{-}(x_0) - \Phi^{+}(x_0)e^{\theta G_\mu(x_0)} \right] \cdot \left[e^{-\theta G_\mu(x_0)} - 1\right] < 0,
\end{equation}
which is again a contradiction. An analogous argument leads to a contradiction if $x_0 \in (0, \alpha) \neq \varnothing$.\\

The above considerations show that $x_0 \in \{\alpha, \beta, 0, \lM + \theta\}$. Suppose next that $x_0 = \beta$ and let
$$\Phi^-(x) = \sum_{k =0}^\infty A_k (x- \beta)^k, \Phi^+(x) = \sum_{k = 0}^\infty B_k(x-\beta)^k \mbox{ and }R_\mu(x) =  \sum_{k = 0}^\infty C_k(x-\beta)^k$$
be the power series expansion of $\Phi^{\pm}$ and $R_\mu(x)$ near $ \beta$. For $x \in (\alpha, \beta)$ we know that 
$$ R_\mu^2(x) < 4 \Phi^+(x) \Phi^-(x)$$
and  taking the limit as $x \rightarrow \beta$, we conclude that $(A_0+B_0)^2 = C_0^2 \leq 4A_0B_0$, which implies that $C_0/2 = A_0 = B_0 $. Suppose next that $\beta \neq \lM + \theta$. Then from our previous work we know that $\Phi^-(x) + \Phi^{+}(x) - R_\mu(x) > 0$ if $x \in(\alpha, \beta)$ and if $x \in (\beta, \lM + \theta)$. Consequently, we conclude that $C_1 = A_1 + B_1$. The latter implies that near $\beta$ we have
\begin{equation}\label{Qmubeta}
Q^2_\mu(x) = R^2_\mu(x) - 4 \Phi^+(x) \Phi^-(x) = O(|x-\beta|^2),
\end{equation}
which implies that $H(\beta) = 0$, which is a contradiction. 

Suppose instead that $\beta = \lM + \theta$. Notice that as $\Phi^{\pm}(x) > 0$ on $(\alpha, \beta)$ we have $A_0 \geq 0$. If $A_0 = 0$ then again (\ref{Qmubeta}) holds, leading to a contradiction. We may thus assume $A_0 > 0$. Notice that for $x = \beta + \epsilon$ and $\epsilon > 0$ small we may verbatim repeat the arguments from (\ref{Gmupos}) to (\ref{Rmuneg}) and conclude that $\Phi^-(x) + \Phi^{+}(x) - R_\mu(x) > 0$ for all such $x$. Since $\Phi^-(x) + \Phi^{+}(x) - R_\mu(x) > 0$ both to the left and right of $\beta$ we conclude as before that $C_1 = A_1 + B_1$, which implies (\ref{Qmubeta}) leading to the same contradiction. Summarizing the last two paragraphs, we see that $x_0 \neq \beta$ and an analogous argument shows that $x_0 \neq \alpha$.\\

What remains is to investigate the cases when $x_0 = 0 < \alpha$ and $x_0 = \lM + \theta > \beta$. Suppose that $x_0 = \lM + \theta > \beta$. As before we have that $(\beta, \lM + \theta) \subseteq S_s$ or $(\beta, \lM + \theta) \subseteq S_v$. In the former case we have $R_\mu(x) < 0$ for $x \in (\beta, \lM + \theta)$ and by continuity we conclude $R_\mu(x_0) \leq 0$. 
On the other hand, $\Phi^{\pm}(x) \geq 0$ for $x \in (\beta, \lM + \theta)$ and by continuity we conclude $\Phi^{\pm}(x_0) \geq 0$. We thus see that $R_\mu(x_0) = \Phi^{\pm}(x_0) = 0$. But then $Q^2_\mu(x_0) = R^2_\mu(x_0) - 4 \Phi^+(x_0) \Phi^-(x_0) = 0$, and so $H(\lM + \theta) = 0$, which is a contradiction. We thus conclude that $(\beta, \lM + \theta) \subseteq S_v$. Arguing as before we have for $x\in (\beta, x_0]$ that $G_\mu(x) > 0$ and 
\begin{equation*}
R_\mu(x) - \Phi^-(x) - \Phi^+(x) = \left[ \Phi^{-}(x) - \Phi^{+}(x)e^{\theta G_\mu(x)} \right] \cdot \left[e^{-\theta G_\mu(x)} - 1\right],
\end{equation*}
which implies that 
$$\Phi^{-}(x_0) =  \Phi^{+}(x_0)e^{\theta G_\mu(x_0)}.$$
Since $\Phi^-(x) > 0$ for $x \in (\beta, x_0)$ we conclude by continuity that $\Phi^-(x_0) \geq 0$ and so
$$0 < H(x_0) \cdot \sqrt{(x_0 - \alpha)(x_0 - \beta)}= \Phi^-(x_0) \cdot e^{-\theta G_\mu(x_0)} - \Phi^+(x_0) \cdot e^{\theta G_\mu(x_0)} \leq  \Phi^-(x_0)  - \Phi^+(x_0) \cdot e^{\theta G_\mu(x_0)} = 0,$$
which is a contradiction. We conclude that $x_0 \neq \lM + \theta$ and an analogous argument shows that $x_0 \neq 0 $. Overall, we reach a contradiction in all cases, which concludes the proof of the lemma.
\end{proof}

%
\section{Appendix B}\label{Section10}
In this section we prove Propositions \ref{SingleLevelNekrasov}, \ref{CLT} and \ref{MomentBoundSingleLevel}. We use the notation from Section \ref{Section3}. 

%
\subsection{Proof of Proposition \ref{SingleLevelNekrasov}}\label{Section10.0}
The function $R_N(z)$ has possible poles at $s = a +(N- i) \cdot \theta$ where $i= 1, \cdots, N$ and $a\in \{0,\dots, M_N + 1\}$. Note that all of these poles are simple, since $\ell_i$ are {\em strictly} increasing. We will write $\ell^{i,\pm}$ for the $N$-tuple $(\ell_1, \dots, \ell_{i-1}, \ell_i {\pm 1}, \ell_{i+1}, \dots \ell_N)$. 

Fix a possible pole $s$ and assume $s\neq 0, M_N+1+(N-1) \cdot  \theta.$ The expectation $\mathbb{E}_{\mathbb{P}_N} $ is a sum over elements  $\ell \in \mathfrak{X}^t$. Such an element
contributes to a residue if $\ell_i=s \text{ or } \ell_i=s-1,$ for some $i = 1, \dots , N$. Then the residue at $s$ is given by
\begin{align}\label{S3BigSum}
\begin{split}
&\sum_{i = 1}^N \sum_{ \ell \in \mathcal{G}^i_1 } \Phi^-_N(s) \mathbb{P}_N(\ell) \cdot (-\theta) \cdot \prod_{j \neq i}^N  \frac{s - \ell_j - \theta}{s - \ell_j} +  \Phi^+_N(s)  \mathbb{P}_N(\ell^-) \cdot (\theta) \cdot \prod_{j \neq i}^N  \frac{s - \ell_j + \theta - 1}{s - \ell_j - 1} \\
&+ \sum_{i = 1}^N \sum_{ \ell\in \mathcal{G}^i_2} \hspace{-1mm}\Phi^-_N(s) \mathbb{P}_N(\ell) \cdot (-\theta) \cdot \hspace{-1mm} \prod_{j \neq i}^N  \frac{s - \ell_j - \theta}{s - \ell_j} + \sum_{i = 1}^N \Phi^+_N(s)   \sum_{ \ell\in \mathcal{G}^i_3}\hspace{-1mm} \mathbb{P}_N(\ell^-) \cdot (\theta) \cdot \hspace{-1mm} \prod_{j \neq i}^N  \frac{s - \ell_j + \theta- 1}{s - \ell_j - 1},
\end{split}
\end{align}
$$\mbox{ where } \mathcal{G}^i_1 = \{ \ell: \ell_i = s \mbox{ and } \ell,  \ell^{i, -}  \in \mathbb{W}^\theta_{N,N}\}, \hspace{2mm}  \mathcal{G}^i_2 = \{ \ell: \ell_i = s, \ell \in \mathbb{W}^\theta_{N,N},  \ell^{i, -} \not \in \mathbb{W}^\theta_{N,N}\},$$
$$ \mathcal{G}^i_3 = \{ \ell: \ell_i = s, \ell \not \in \mathbb{W}^\theta_{N,N},  \ell^{i, -}  \in \mathbb{W}^\theta_{N,N}\}.$$
Notice that the first sum vanishes term-wise as can be seen from (\ref{SingleLevelMeasure}) and Assumption 3. We next note that if $\ell \in \mathcal{G}^i_2$ then either $s = 0$ in the case $i = N$, which we ruled out, or $i \neq N$ and $\ell_{i+1} = s -\theta$. This means that the product in the second sum vanishes and so we get no contribution to the residue from this sum. Similarly, if $\ell \in \mathcal{G}^i_3$ then either $i = 1$ and $s = M_N + 1 + (N-1)\cdot \theta$, which we ruled out or $i \neq 1$ and $\ell_{i-1} = s - 1 + \theta$. This means that the product in the third sum vanishes and so we get no contribution to the residue from this sum. Overall, the residue is zero provided $s \neq 0$ and $s \neq M_N + 1 + (N-1)\cdot \theta$. 

We finally consider the residues at $s = 0$ and $s = M_N + 1 + (N-1)\cdot  \theta$ starting with the former. If $s = 0$ then we get no contribution from the second expectation and the first expectation in (\ref{SingleLevelEquation}) only contributes if $\ell_N = 0$. Consequently, the residue is given by
$$  \Phi^-_N(0) \cdot (-\theta) \cdot \mathbb{P}_N(\ell_N = 0) \cdot \mathbb{E}_{\mathbb{P}_N} \left[ \prod_{i= 1}^{N-1}  \frac{  \ell_i + \theta }{ \ell_i}\Big{\vert}  \ell_N = 0 \right] -r^-(N),$$
which is zero by the definition of $r^-(N)$. 

If $s = M_N + 1 + (N-1)\cdot \theta$ then we get no contribution from the first expectation, while the second expectation in (\ref{SingleLevelEquation}) only contributes if $\ell_1 = M_N + (N-1)\cdot \theta$. Consequently, the residue is given by
$$  \Phi^+_N(\theta) \cdot \theta \cdot \mathbb{P}_N(\ell_1 = s - 1) \cdot \mathbb{E}_{\mathbb{P}_N} \left[\prod_{i = 2}^N  \frac{s - \ell_i + \theta-1 }{s - \ell_i - 1}\Big{\vert}  \ell_1 = s - 1\right] -r^+(N),$$
which is zero by the definition of $r^+(N)$.

%
\subsection{Applications of Nekrasov's equations}\label{Section10.1}
We assume that we have a sequence of probability measures $\mathbb{P}_N$ that satisfy Assumptions 1-5 as in Section \ref{Section3}.

Let us fix a compact subset $K$ of $\mathbb{C} \setminus [0, \lM + \theta]$ and suppose $\epsilon > 0$ is sufficiently small so that $K$ is at least distance $\epsilon$ from $[0, \lM + \theta]$. We also let $v_a$ for $a = 1, \dots, m$ be any $m$ points in $K$. We apply Nekrasov's equations, Proposition \ref{SingleLevelNekrasov}, to the measures $\mathbb{P}_N^{{\bf t}, {\bf v}}$ from Section \ref{Section3.2} (here $n = 0$) and obtain the following statement. Let $R_N$ be given by
\begin{align}\label{S9SLNEv1}
\begin{split}
R_N(Nz) = &\Phi^{-}_N(Nz)   A_1(z) \cdot   \mathbb{E} \left[ \prod_{i = 1}^N\frac{Nz- \ell_i -\theta}{Nz - \ell_i} \right] \\
&+ \Phi^{+}_N(Nz)    B_1(z)  \cdot  \mathbb{E} \left[ \prod_{i = 1}^{N}\frac{Nz- \ell_i + \theta - 1}{Nz - \ell_i - 1} \right] + W_1(z),
\end{split}
\end{align}
where 
\begin{align}\label{S9As}
\begin{split}
&A_1(z) = \prod_{a = 1}^m \left[ v_a +t_a - z + \frac{1}{N} \right] \left[ v_a- z \right], B_1(z) =  \prod_{a = 1}^m \left[ v_a +t_a- z \right] \left[ v_a - z + \frac{1}{N}\right], \\
& W_1(z) =   C_1(z)   \mathbb{E}\left[  \prod_{i= 1}^{N-1}  \frac{ \ell_i + \theta }{ \ell_i} \cdot {\bf 1} \{\ell_N = 0\} \right] - C_2(z)  \mathbb{E}\left[\prod_{i = 2}^N  \frac{s_N - \ell_i -1 + \theta}{s_N - \ell_i - 1}\cdot {\bf 1}\{ \ell_1 = s_N - 1\} \right] \\
&C_1(z) = \frac{\theta  \Phi^-_N(0)}{N \cdot z}\prod_{a = 1}^m \hspace{-1mm} \left[v_a + t_a + \frac{1}{N} \right] \hspace{-1mm}v_a,  C_2(z) =   \frac{\theta  \Phi_N^+(s_N)}{Nz- s_N} \prod_{a = 1}^m\hspace{-1mm} \left[v_a + t_a - \frac{s_N}{N}\right] \hspace{-1mm} \left[v_a - \frac{s_N}{N} + \frac{1}{N}\right],
\end{split}
\end{align}
 $\Phi^{\pm}_N$ are as in Assumption 3 and $s_N = M_N + 1 + (N-1) \cdot\theta$. All the expectations above are with respect to the measure $\mathbb{P}_N^{{\bf t}, {\bf v}}$. Then for all large $N$ function $R_N(Nz)$ is holomorphic in $\mathcal{M}$ as in Assumption 3, provided that the $\max_{i,j}|t_j| < \epsilon/2$. 

For integers $p \leq  q$, we will denote by $\llbracket p, q \rrbracket$ the set of integers $\{p, p+1, \dots ,q\}$. If $A \subseteq \llbracket 1, m\rrbracket$ and $\xi$ is a bounded random variable we write $M(\xi; A)$ for the joint cumulant of $\xi$ and $X_N^t(v_a)$ for $a \in A$, where we recall that $X^t_N$ was defined in (\ref{DefX}). If $A = \varnothing$ then $M(\xi;A) = \mathbb{E}[\xi]$. In addition, we fix $v \in K$ and let $\Gamma$ be a positively oriented contour, which encloses the segment $[0 , \lM + \theta]$, is contained in $\mathcal{M}$ as in Assumption 3 and avoids $K$.\\

Our goal in this section is to derive the following result
\begin{align}\label{SLExpand}
\begin{split} 
& M(X^t_N(v); \llbracket 1, m \rrbracket) =  \frac{N \cdot \theta^{-1}}{2\pi \i \sqrt{(v - \alpha) (v- \beta)}}\int_{\Gamma}  dz\frac{\Phi_N^+(Nz) e^{\theta G_\mu(z)} {\bf 1 }\{m = 0\}}{H(z) \cdot (z - v)} \cdot \left[ 1 +  \frac{[\theta^2 - 2\theta] \partial_z G_{\mu}(z)}{2N} \right]\\
& + N^{-1} \hspace{-2mm} \sum_{A \subseteq \llbracket 1,m \rrbracket}  \hspace{-2mm}    M \left(  \xi^{\Gamma}_N(z); A\right) +  M(\xi_N^{\Gamma}(z)  [X_N^t(z)]^2;A) +M( \xi_N^{\Gamma}(z)  \partial_z X_N^t(z);A) + M( \xi_N^{\Gamma}(z)  X_N^t(z);A)\\
& + \prod_{a = 1}^m  \left[ \frac{- N^{-1}}{ (v_a -z)(v_a - z^-)}\right]   \frac{\Phi_N^-(Nz) e^{-\theta G_\mu(z)} }{H(z) \cdot (z - v)} \mathbb{E}\left[1 - {\bf 1}\{m > 0\} \cdot \frac{\theta X_N^t(z)}{N} + \frac{\theta^2 \partial_z G_{\mu}(z) }{2N}\right] \hspace{-1mm} + O(N^{-1}), \\
\end{split}
\end{align}
where $z^- = z - N^{-1}$, all the expectations above are with respect to $\mathbb{P}_N$, $\xi_N^{\Gamma}(z)$ stands for a generic random functions that is $\mathbb{P}_N$ - almost surely $O(1)$ and the constants in the big $O$ notations are uniform as $z$ varies over $\Gamma$ and $v, v_1, \dots, v_m$ vary over $K$.

We start by dividing both sides of (\ref{S9SLNEv1}) by $2\pi \i \cdot H(z) \cdot (z - v)\cdot B_1(z)$ and integrating over $\Gamma$, where we recall that $H(z)$ was defined in Assumption 5. This gives
\begin{align*}
&\frac{1}{2\pi \i }\int_{\Gamma} \frac{R_N(Nz) dz}{H(z) \cdot (z - v) \cdot B_1(z) }= \frac{1}{2\pi \i }\int_{\Gamma} \frac{W_1(z) dz}{H(z) \cdot (z - v) \cdot B_1(z) }  \\
&+ \frac{1}{2\pi \i }\int_{\Gamma} \frac{ \Phi^{-}_N(Nz) A_1(z) dz}{H(z)  (z - v)  B_1(z)} \cdot  \mathbb{E} \left[ \prod_{i = 1}^N\frac{Nz- \ell_i -\theta}{Nz - \ell_i} \right] + \frac{1}{2\pi \i }\int_{\Gamma} \frac{\Phi^{+}_N(Nz)  dz}{H(z)  (z - v)}  \cdot   \mathbb{E} \left[ \prod_{i = 1}^{N}\frac{Nz- \ell_i + \theta - 1}{Nz - \ell_i - 1} \right].
\end{align*}
Notice that by Assumption 5, we have $H(z) \neq 0$ in a neighborhood of the region enclosed by $\Gamma$ and so by Cauchy's theorem the left side of the above expression vanishes. Furthermore, the integrand of the first term on the right has two simple poles in the region enclosed by $\Gamma$ at $z = 0$ and $z = s_N \cdot N^{-1}$. The last two observations show
\begin{align*}
\begin{split}
&0  = \frac{\theta \cdot \Phi_N^-(0)}{N \cdot H(0) \cdot(- v)} \cdot \prod_{a = 1}^m \frac{(v_a + t_a + 1/N)v_a}{(v_a + t_a)(v_a + 1/N)}\cdot  \mathbb{E}\left[  \prod_{i= 1}^{N-1}  \frac{ \ell_i + \theta }{ \ell_i} \cdot {\bf 1} \{\ell_N = 0\} \right]   \\
&- \frac{\theta \cdot \Phi_N^+(s_N)}{N \cdot H(s_N/N) \cdot (s_N/N -v)} \cdot  \mathbb{E}\left[\prod_{i = 2}^N  \frac{s_N - \ell_i + \theta -1 }{s_N - \ell_i - 1}\cdot {\bf 1}\{ \ell_1 = s_N - 1\} \right]  \\
&+ \frac{1}{2\pi \i }\int_{\Gamma} \frac{ \Phi^{-}_N(Nz) A_1(z) dz}{H(z) (z - v)  B_1(z)}  \cdot \mathbb{E} \left[ \prod_{i = 1}^N\frac{Nz- \ell_i -\theta}{Nz - \ell_i} \right] + \frac{1}{2\pi \i }\int_{\Gamma} \frac{\Phi^{+}_N(Nz)  dz}{H(z)  (z - v)} \cdot   \mathbb{E} \left[ \prod_{i = 1}^{N}\frac{Nz- \ell_i + \theta - 1}{Nz - \ell_i - 1} \right].
\end{split}
\end{align*}

We next apply the operator $\mathcal{D} : = \partial_{t_1} \cdots \partial_{t_m}$ to both sides and set $t_a = 0$ for $a = 1, \dots, m$. Notice that when we perform the differentiation to the right side some of the derivatives could land on $\frac{A_1(z)}{B_1(z)}$ or $\prod_{a = 1}^m \frac{(v_a + t_a + 1/N)v_a}{(v_a + t_a)(v_a + 1/N)}$ and some on the corresponding expectation. We will split the result of the differentiation based on subsets $A$, where $A$ consists of indices $a$ in $\{1, \dots, m\}$ such that $\partial_{t_a} $ differentiates the expectation. The result of this procedure is as follows
\begin{align}\label{S10Expand1}
\begin{split} 
& 0 = \sum_{A \subseteq \llbracket 1,m \rrbracket} \prod_{a \in A^c}  \left[ \frac{- N^{-1}}{ v_a(v_a  +N^{-1})}\right]   \frac{\theta \cdot \Phi_N^-(0) }{N \cdot H(0) \cdot (- v)}M \left(\prod_{i= 1}^{N-1}  \frac{ \ell_i + \theta }{ \ell_i} \cdot {\bf 1} \{\ell_N = 0\} ; A\right)    \\
& - \frac{\theta \cdot \Phi_N^+(s_N)}{N \cdot H(s_N/N) \cdot (s_N/N -v)} \cdot M\left(\prod_{i = 2}^N  \frac{s_N - \ell_i  + \theta-1}{s_N - \ell_i - 1}\cdot {\bf 1}\{ \ell_1 = s_N - 1\}  ; \llbracket 1, m \rrbracket \right)  \\
&+ \frac{1}{2\pi \i }\int_{\Gamma}dz  \sum_{A \subseteq \llbracket 1,m \rrbracket} \prod_{a \in A^c}  \left[ \frac{- N^{-1}}{ (v_a -z)(v_a - z^-)}\right]   \frac{\Phi_N^-(Nz) }{H(z) \cdot (z - v)}M \left(\prod_{i = 1}^N\frac{Nz- \ell_i -\theta}{Nz - \ell_i} ; A\right)  \\
&+  \frac{\Phi_N^+(Nz) }{H(z) \cdot (z - v) }  M\left( \prod_{i = 1}^{N}\frac{Nz- \ell_i +\theta - 1}{Nz - \ell_i - 1} ; \llbracket 1, m \rrbracket \right),
\end{split}
\end{align}
where $A^c = \llbracket 1,m \rrbracket \setminus A$. We now observe that since $\ell_i \geq (N-i)\theta$ and $s_N - \ell_i - 1 \geq (i-1) \theta$ we have
$$1 \leq  \prod_{i= 1}^{N-1}  \frac{ \ell_i + \theta }{ \ell_i} \leq N \mbox{ and }1 \leq \prod_{i = 2}^N  \frac{s_N - \ell_i  + \theta-1}{s_N - \ell_i - 1}\leq N.$$
Let us denote 
$$\xi^-_N = \frac{ 1}{N } \cdot \prod_{i= 1}^{N-1}  \frac{ \ell_i + \theta }{ \ell_i} \cdot {\bf 1} \{\ell_N = 0\} \mbox{ and } \xi_N^{+} =  \frac{ 1}{N} \cdot \prod_{i = 2}^N  \frac{s_N - \ell_i  + \theta-1}{s_N - \ell_i - 1}\cdot {\bf 1}\{\ell_1 = s_N - 1 \},$$
and observe that $\xi^{\pm}_N= O(1)$ and $X_N^t(v_a)$ are $O(N)$ almost surely. The latter two statements, together with Assumption 4, imply that the expressions on the first and second line in (\ref{S10Expand1}) are $O(N^m \exp(-cN^a))$ and so in particular $O(N^{-2})$. Combining the latter with (\ref{AEP1v2}) we may rewrite (\ref{S10Expand1}) as
\begin{align*}
\begin{split} 
& 0 = \frac{1}{2\pi \i}\int_{\Gamma} dz \sum_{A \subseteq \llbracket 1,m \rrbracket} \prod_{a \in A^c}  \left[ \frac{- N^{-1}}{ (v_a -z)(v_a - z^-)}\right]   \frac{\Phi_N^-(Nz) e^{-\theta G_\mu(z)} }{H(z) \cdot (z - v)}M \left(1 - \frac{\theta X_N^t(z)}{N} + \frac{\theta^2 \partial_z G_{\mu}(z) }{2N}; A\right)  \\
& + \frac{\Phi_N^+(Nz) e^{\theta G_\mu(z)} }{H(z) \cdot (z - v) } M\left( 1 +  \frac{\theta X^t_N(z)}{N} + \frac{[\theta^2 - 2\theta] \partial_z G_{\mu}(z)}{2N} ; \llbracket 1, m \rrbracket \right) + O(N^{-2}) \\
& + N^{-2} \sum_{A \subseteq \llbracket 1,m \rrbracket}   M \left(  \xi^{\Gamma}_N(z); A\right) +  M(\xi_N^{\Gamma}(z)  [X_N^t(z)]^2;A) +M( \xi_N^{\Gamma}(z)  \partial_z X_N^t(z);A)  ,
\end{split}
\end{align*}
where $\xi_N^{\Gamma}(z)$ is a random function that is almost surely $O(1)$ for $z \in \Gamma$ and $v, v_1, \dots, v_m \in K$.

We can now use the linearity of cumulants together with the fact that the joint cumulant of any non-empty collection of  bounded random variables and a constant is zero. In addition, we know that uniformly as $z$ varies over $\Gamma$ and $v \in K$ we have
$$\frac{1}{ (v -z)(v - z +N^{-1})} = \frac{1}{ (v -z)^2} + O(N^{-1}) \mbox{ and } \Phi^{\pm}_N(Nz) = \Phi^{\pm}(z) + O(N^{-1}).$$
Applying the last two statements we get
\begin{align*}
\begin{split} 
& 0 = \frac{1}{2\pi \i}\int_{\Gamma}dz \frac{\theta \cdot [ -\Phi^-(z) e^{-\theta G_\mu(z)} + \Phi^+(z) e^{\theta G_\mu(z)}] }{H(z) \cdot (z - v) \cdot N } \cdot M \left( X_N^t(z); \llbracket 1, m \rrbracket \right)  \\
& + {\bf 1 }\{m = 0\} \cdot \frac{\Phi_N^+(Nz) e^{\theta G_\mu(z)}}{H(z) \cdot (z - v)} \cdot \left[ 1 +  \frac{[\theta^2 - 2\theta] \partial_z G_{\mu}(z)}{2N} \right] + O(N^{-2})  \\
& + \prod_{a = 1}^m  \left[ \frac{- N^{-1}}{ (v_a -z)(v_a - z +N^{-1})}\right]   \frac{\Phi_N^-(Nz) e^{-\theta G_\mu(z)} }{H(z) \cdot (z - v)} \mathbb{E}\left[1 - {\bf 1}\{m > 0\} \cdot \frac{\theta X_N^t(z)}{N} + \frac{\theta^2 \partial_z G_{\mu}(z) }{2N}\right]  \\
& + N^{-2}\hspace{-2mm} \sum_{A \subseteq \llbracket 1,m \rrbracket}  \hspace{-2mm}  M \left(  \xi^{\Gamma}_N(z); A\right) +  M(\xi_N^{\Gamma}(z)  [X_N^t(z)]^2;A) +M( \xi_N^{\Gamma}(z)  \partial_z X_N^t(z);A) + M( \xi_N^{\Gamma}(z)  X_N^t(z);A).
\end{split}
\end{align*}
Using  (\ref{QRmu}) and Assumption 4 we see that the first term on the right above equals
$$\frac{\theta \cdot \sqrt{(z - \alpha) (z- \beta)} M \left( X_N^t(z); \llbracket 1, m \rrbracket \right) }{  (z - v) \cdot N }, $$
and so we can compute this integral as minus the residue at $z = v$ (notice that there is no residue at infinity). Substituting this above and multiplying the result by $(\sqrt{(z - \alpha) (z- \beta)})^{-1} \cdot N \cdot \theta^{-1}$ we arrive at (\ref{SLExpand}).

%
\subsection{Proof of Proposition \ref{MomentBoundSingleLevel} }\label{Section10.2}
In this section we prove Proposition \ref{MomentBoundSingleLevel}. We want to show that for each $k \geq 1$ 
\begin{equation}\label{S9E0}
  \mathbb{E} \left[ | X^t_N(z)] |^k \right] = O(1), 
\end{equation}
where the constants in the big $O$ notation depend on $k$ but not on $N$ (provided it is sufficiently large) and are uniform as $z$ varies over compact subsets of $\mathbb{C} \setminus [0, \lM + \theta]$.

The proof we present below is similar to the one given in Section \ref{Section5.2}. For the sake of clarity we split the proof into several steps.\\

{\bf \raggedleft Step 1.} From (\ref{SLExpand}) for $m = 0$ we get
\begin{align*}
\begin{split}
 \mathbb{E} [X^t_N(v)] = &\frac{N \cdot \theta^{-1}}{2\pi \i \sqrt{(v - \alpha) (v- \beta)}}\int_{\Gamma}dz  \frac{\Phi_N^+(Nz) e^{\theta G_\mu(z)}}{H(z) \cdot (z - v)} \cdot \left[ 1 +  \frac{[\theta^2 - 2\theta] \partial_z G_{\mu}(z)}{2N} \right]  \\
& +   \frac{\Phi_N^-(Nz) e^{-\theta G_\mu(z)} }{H(z) \cdot (z - v)} \mathbb{E}\left[1  + \frac{\theta^2 \partial_z G_{\mu}(z) }{2N}\right]  + O(N^{-1}) .
\end{split}
\end{align*}
We may now use that $\Phi_N^+(Nz)  = \Phi^{\pm}(z) + O(N^{-1})$ and (\ref{QRmu}) to rewrite the above as
\begin{equation*}
\begin{split}
& \mathbb{E} [X^t_N(v)] = \frac{N \cdot \theta^{-1}}{2\pi \i \sqrt{(v - \alpha) (v- \beta)}}\int_{\Gamma} dz  \frac{R_\mu(z)}{H(z) \cdot (z - v)}   + O(1) .
\end{split}
\end{equation*}
By Cauchy's theorem the above integral is zero and so we conclude 
\begin{equation}\label{S9E1}
\begin{split}
& \mathbb{E} [X^t_N(v)] = O(1) .
\end{split}
\end{equation}

{\bf \raggedleft Step 2.} In this step we reduce the proof of (\ref{S9E0}) to the establishment of the following self-improvement estimate claim.\\

{\bf \raggedleft Claim:} Suppose that for some $n, M \in \mathbb{N}$ we have that 
\begin{equation}\label{S9E2}
\mathbb{E} \left[ \prod_{a = 1}^m |X^t_N(v_a)| \right]= O(1) + O(N^{m/2 + 1 - M/2}) \mbox{ for $m = 1, \dots, 4n + 4$,}
\end{equation}
where the constants in the big $O$ notation are uniform as $v_a$ vary over compacts in $\mathbb{C} \setminus [0, \lM + \theta]$. Then
\begin{equation}\label{S9E3}
\mathbb{E} \left[ \prod_{a = 1}^m |X^t_N(v_a)| \right]= O(1) + O(N^{m/2 + 1 - (M+1)/2}) \mbox{ for $m = 1, \dots, 4n $.}
\end{equation}
We prove the above claim in the following steps. For now we assume its validity and finish the proof of (\ref{S9E0}). 

 Let $\eta > 0$ be sufficiently small so that $I_\eta := [-\eta , \lM + \theta + \eta] \subset \mathcal{M}$ and $ I_\eta \cap K = \varnothing$. Let $h(x)$ be a smooth function such that $0 \leq h(x) \leq 1$,  $h(x) = 1$ if $x \in [-\eta/2, \lM + \theta + \eta/2]$, $h(x) = 0$ if $x \leq -\eta$ or $x \geq \lM + \theta + \eta$ and $\sup_{ x \in I_\eta} |h'(x)| \leq \eta^{-1} \cdot 10$. Since $M_N \cdot N^{-1} \rightarrow \lM$ as $N \rightarrow \infty$ we know that for all large $N$ and $v \in K$ we have $\mathbb{P}_N$-almost surely
\begin{equation*}
\left| \int_{\mathbb{R}}  g_v(x) \cdot h(x)  \mu_N(dx) -  \int_{\mathbb{R}} g_v(x) \cdot h(x)  \mu(dx) \right|  = \left|X_N^t(v)\right|,
\end{equation*}
where $g_v(x) = (v -x)^{-1}$. Let us denote 
\begin{equation*}
c_1(K) := \sup_{v \in  K} \| g_v \cdot h \|_{1/2} \mbox{ and } c_2(K) := \sup_{v \in  K} \| g_v \cdot h \|_{\mbox{ \Small Lip}},
\end{equation*} 
and recall that in (\ref{c1c2Bound}) we showed that $c_1(K), c_2(K)$ are finite positive constants. 

We can apply Corollary \ref{BC1} for the function $g_v \cdot h$ with $a = r \cdot N^{1/2n - 1/2}$, $r > 0$ and $p = 3$ to get
$$\mathbb{P}_N \left( \left|  X_N^t(v) \right| \geq c_1(K) r N^{1/2+1/2n} + c_2(K) N^{-2} \right) \leq \exp\left( CN \log(N)^2 - 2\theta \pi^2 r^2 N^{1+1/n}\right),$$
which implies
\begin{equation*}
\mathbb{E} \left[ \left|X_N^t(v)\right|^n \right] =O \left( N^{n/2 + 1/2} \right).
\end{equation*}
Using H{\"o}lder's inequality, the above implies that (\ref{S9E2}) holds for the for the pair $n = 2k$ and $M = 1$. The conclusion is that (\ref{S9E2}) holds for the pair $n = 2k- 1$ and $M = 2$. Iterating the argument an additional $k$ times we conclude that (\ref{S9E2}) holds with $n = k - 1$ and $M = k+2$, which implies (\ref{S9E0}).\\

{\bf \raggedleft Step 3.} In this step we prove that 
\begin{equation}\label{S9E4}
M( X^t_N(v_0), X^t_N(v_1), \dots, X^t_N(v_m)) = O(1) + O( N^{m/2 + 1 - M/2}) \mbox{ for $m = 1, \dots, 4n+2$}.
\end{equation}
The constants in the big $O$ notation are uniform over $v_0, \dots, v_m$ in compact subsets of $\mathbb{C} \setminus [0, \lM + \theta]$. 

We start by fixing $\mathcal{V}$ to be a compact subset of $\mathbb{C} \setminus [0, \lM + \theta]$, which is invariant under conjugation. We also fix $\Gamma$ to be a positively oriented contour, which encloses the segment $[0, \lM + \theta]$, is contained in $\mathcal{M}$ as in Assumption 3 and excludes the set $\mathcal{V}$. 

From (\ref{SLExpand}) for $ m = 1, \dots, 4n+2$ and $v_0, v_1 \dots, v_m \in \mathcal{V}$ we have
\begin{align}\label{S9E5}
\begin{split} 
& M\left( X^t_N(v_0), X^t_N(v_1), \dots, X^t_N(v_m)  \right)  =  O(N^{-1}) + \frac{ \theta^{-1}}{2\pi \i \sqrt{(v_0 - \alpha) (v_0- \beta)}}  \\
& \times \int_{\Gamma} dz \prod_{a = 1}^m  \left[ \frac{- N^{-1}}{ (v_a -z)(v_a - z +N^{-1})}\right]  \frac{N \cdot \Phi_N^-(Nz) e^{-\theta G_\mu(z)} }{H(z) \cdot (z - v_0)} \mathbb{E}\left[1 - \frac{\theta X_N^t(z)}{N} + \frac{\theta^2 \partial_z G_{\mu}(z) }{2N}\right]  \\
& + N^{-1} \hspace{-3mm} \sum_{A \subseteq \llbracket 1,m \rrbracket}  \hspace{-2mm}  M \left(  \xi^{\Gamma}_N(z); A\right) +  M(\xi_N^{\Gamma}(z)  [X_N^t(z)]^2;A) +M( \xi_N^{\Gamma}(z)  \partial_z X_N^t(z);A) + M( \xi_N^{\Gamma}(z)  X_N^t(z);A).
\end{split}
\end{align}

We next use the fact that cumulants can be expressed as linear combinations of products of moments. This means that $M(\xi_1, \dots, \xi_r)$ can be controlled by the quantities $1$ and $\mathbb{E} \left[ |\xi_i|^k \right]$ for $ i = 1, \dots, k$. We use the latter and (\ref{S9E2}) to get
\begin{align}\label{S9E6}
\begin{split}
&N^{-1} \sum_{A \subseteq \llbracket 1,m \rrbracket}  \hspace{-2mm}  M \left(  \xi^{\Gamma}_N(z); A\right) +  M(\xi_N^{\Gamma}(z)  [X_N^t(z)]^2;A) + \\
& +M( \xi_N^{\Gamma}(z)  \partial_z X_N^t(z);A) + M( \xi_N^{\Gamma}(z)  X_N^t(z);A) = O(1) + O(N^{m/2 + 1 - M/2}).
\end{split}
\end{align}
One might be cautious about the term involving $\partial_z X^t(z)$; however, by Cauchy's inequalities, see e.g. \cite[Corollary 4.3]{SS}, the moment bounds we have for $\mathbb{E} \left[ |X^t_N(z)|^k \right]$ in (\ref{S9E2}) imply analogous ones for $\mathbb{E} \left[ |\partial_z X^t_N(z)|^k \right]$. Putting (\ref{S9E6}) into (\ref{S9E5}) we obtain (\ref{S9E4}).\\

{\bf \raggedleft Step 4.} In this step we will prove (\ref{S9E3}) except for a single case, which will be handled separately in the next step. Notice that by H{\"o}lder's inequality we have
\begin{equation*}
\sup_{v_1, \dots, v_m \in \mathcal{V}} \mathbb{E} \left[ \prod_{a = 1}^m |X^t_N(v_a)| \right]\leq \sup_{v \in \mathcal{V}} \mathbb{E} \left[ |X^t_N(v)|^m \right],
\end{equation*}
and so to finish the proof it suffices to show that for $m = 1, \dots, 4n$ we have
\begin{equation}\label{S9E11}
 \mathbb{E} \left[ |X^t_N(v)|^m \right] = O(1)  + O(N^{m/2 + 1 - (M+1)/2}).
\end{equation}
Using the fact that for centered random variables one can express joint moments as linear combinations of products of joint cumulants we deduce from (\ref{S9E4}) that
\begin{equation*}
\sup_{ v_0, v_1, \dots, v_{m-1} \in \mathcal{V}} \mathbb{E} \left[ \prod_{a = 0}^{m-1}[X^t_N(v_a) - \mathbb{E}[ X^t_N(v_a)]] \right] = O(1) + O( N^{(m-1)/2 + 1 - M/2}) \mbox{ for $m = 1, \dots, 4n+2$}.
\end{equation*}
Combining the latter with (\ref{S9E1}) we conclude that 
\begin{equation}\label{S9E12}
\sup_{ v_0, v_1, \dots, v_{m-1} \in \mathcal{V}} \mathbb{E} \left[ \prod_{a = 0}^{m-1}X^t_N(v_a) \right] = O(1) + O( N^{(m-1)/2 + 1 - M/2}) \mbox{ for $m = 1, \dots, 4n+2$}.
\end{equation}

If $m = 2m_1$ then we set $v_0 = v_1 = \cdots = v_{m_1 - 1} = v$ and $v_{m_1} = \cdots = v_{2m_1 - 1} = \overline{v}$ in (\ref{S9E12}), which yields 
\begin{equation}\label{S9E13}
\sup_{ v \in \mathcal{V}} \mathbb{E} \left[ |X^t_N(v)|^m \right] = O(1) + O( N^{m/2 + 1/2 - M/2}) \mbox{ for $m = 1, \dots, 4n+2$}.
\end{equation}
In deriving the above we used that $X^t_N(\overline{v}) = \overline{X^t_N(v)}$ and so $X^t_N(v) \cdot X^t_N(\overline{v}) = |X^t_N(v)|^2$. 

We next let $m = 2m_1 + 1$ be odd and notice that by the Cauchy-Schwarz inequality and (\ref{S9E13}) 
\begin{align}\label{S9E14}
\begin{split}
&\sup_{ v \in \mathcal{V}} \mathbb{E} \left[ |X^t_N(v)|^{2m_1 + 1} \right] \leq \sup_{ v \in \mathcal{V}} \mathbb{E} \left[ |X^t_N(v)|^{2m_1 + 2} \right]^{1/2} \cdot \mathbb{E} \left[ |X^t_N(v)|^{2m_1} \right]^{1/2} =  \\  
&O(1) + O( N^{m/2 + 1 - M/2}) + O(N^{m_1/2 + 3/4 - M/4}) .
\end{split}
\end{align}
We note that the bottom line of (\ref{S9E14}) is $O(1) + O(N^{m_1 + 1 - M/2})$ except when $M = 2m_1 + 2$, since
$$m_1 + 3/4 - M/4 \leq \begin{cases} m_1 + 1 - M/2 &\mbox{ when $M \leq 2m_1 + 1$,} \\ 0 &\mbox{ when $M \geq 2m_1 + 3$.} \end{cases}$$
Consequently, (\ref{S9E13}) and (\ref{S9E14}) together imply (\ref{S9E11}), except when $M = 2m_1 +2$ and $m = 2m_1 + 1$. We will handle this case in the next step.\\

{\bf \raggedleft Step 5.} In this step we will show that (\ref{S9E11}) even when $M = 2m_1 + 2$ and $4n > m = 2m_1 + 1$. In the previous step we showed in (\ref{S9E11}) that
$\sup_{v \in \mathcal{V}} \mathbb{E} \left[ |X^t_N(v)|^{2m_1 + 2} \right] = O(N^{1/2})$, and below we will improve this estimate to
\begin{equation}\label{S9E15}
\sup_{v \in \mathcal{V}} \mathbb{E} \left[ |X^t_N(v)|^{2m_1 + 2} \right] = O(1).
\end{equation}
The trivial inequality $x^{2m_1 + 2} + 1 \geq |x|^{2m_1 + 1}$ together with (\ref{S9E15}) imply
$$\sup_{v \in \mathcal{V}} \mathbb{E} \left[ |X^t_N(v)|^{2m_1 + 1} \right]  = O(1).$$
Consequently, we have reduced the proof of the claim to establishing (\ref{S9E15}). \\

Let us list the relevant estimates we will need
\begin{align}\label{S9E16}
\begin{split}
& \mathbb{E} \left[ \prod_{a = 1}^{2m_1 + 2} |Y_N(v_a)| \right] = O(N^{1/2})\mbox{, } \mathbb{E} \left[ \prod_{a = 1}^j |Y_N(v_a)| \right] = O(1) \mbox{ for $0 \leq j \leq 2m_1$,} \\\
&\mathbb{E} \left[ \prod_{a = 1}^{2m_1 + 3} |Y_N(v_a)| \right] = O(N) , \mathbb{E} \left[ \prod_{a = 1}^{2m_1 + 1} |Y_N(v_a)| \right] = O(N^{1/2}).
\end{split}
\end{align}
The first three identities follow from (\ref{S9E11}), which we showed to hold unless $m = 2m_1 + 1$ in the previous step. The last identity is a consequence of the first one and the inequality $x^{2m_1 + 2} + 1 \geq |x|^{2m_1 + 1}$. All constants are uniform over $v_a \in \mathcal{V}$. Below we feed the improved estimates of (\ref{S9E16}) into Steps 3. and 4., which ultimately yield (\ref{S9E15}).\\

In Step 3. we have the following improvement over (\ref{S9E6}) using the estimates of (\ref{S9E16})
\begin{equation*}
\begin{split}
N^{-1} \hspace{-3mm}\sum_{A \subseteq \llbracket 1,m \rrbracket}  \hspace{-2mm}  M \left(  \xi^{\Gamma}_N(z); A\right) +  M(\xi_N^{\Gamma}(z)  [X_N^t(z)]^2;A) +M( \xi_N^{\Gamma}(z)  \partial_z X_N^t(z);A) + M( \xi_N^{\Gamma}(z)  X_N^t(z);A) =  O(1).
\end{split}
\end{equation*}
Substituting the above into (\ref{S9E5}) we obtain the following improvement over (\ref{S9E4})
\begin{equation}\label{S9E19}
\begin{split} 
& M\left( X^t_N(v_0), X^t_N(v_1), \dots, X^t_N(v_{2m_1 + 1})  \right) = O(1),
\end{split}
\end{equation}
We may now repeat the arguments in Step 4. and note that by using (\ref{S9E19}) in place of (\ref{S9E4}) we obtain the following improvement over (\ref{S9E12})
\begin{equation}\label{S9E20}
\sup_{ v_0, v_1, \dots, v_{2m_1 + 1} \in \mathcal{V}} \mathbb{E} \left[ \prod_{a = 0}^{2m_1 + 1}Y_N(v_a) \right] = O(1).
\end{equation}
Setting $v_0 = v_1 = \cdots = v_{m_1} = v$ and $v_{m_1+1} = \cdots = v_{2m_1 + 1} = \overline{v}$ in (\ref{S9E20}) we get (\ref{S9E15}).\\

\begin{remark} Proposition \ref{MomentBoundSingleLevel} is implied in \cite{BGG} for general $\theta$; however, it is only stated and proved when $\theta = 1$ as \cite[Proposition 2.18]{BGG}. Moreover, when $\theta = 1$ the expansion of the Nekrasov's equations, see \cite[equation (44)]{BGG}, is missing the terms corresponding to $[X_N^t(z)]^2$ and $X_N^t(z)$ in (\ref{SLExpand}) and we believe they should be present. Of course, one can introduce these extra terms in their proof and they can be handled in the same way we have handled them. A more serious oversight in the proof is near the end of the proof of \cite[Proposition 2.18]{BGG}, where the special case we encountered in Step 4. was not recognized. The way it can be overcome, is through an extra dummy step of the self improving estimates, which is what we did in Step 5. For these reasons we decided to include the proof of this proposition in the present paper.
\end{remark}

\subsection{Proof of Proposition \ref{CLT}} In this section we prove Proposition \ref{CLT}. The proof we present contains many of the same ideas as in \cite{BGG} and we include it for the sake of completeness.

Since we are dealing with centered random variables it suffices to show that second and higher order cumulants of $G^t_N(z) - \mathbb{E} [G^t_N(z)]$ converge to those specified in the statement of the proposition. Moreover, since cumulants remain unchanged upon shifts by constants, we can replace $G^t_N(z) - \mathbb{E} [G^t_N(z)]$ with $X^t_N(z)$ and establish the convergence of second and higher order cumulants for the latter instead. In the sequel we fix a compact set $K \subset \mathbb{C} \setminus [0, \lM + \theta] $ and a positively oriented contour $\Gamma$ that contains $[0, \lM + \theta]$, is contained in $\mathcal{M}$ as in Assumption 3 and excludes $K$. 

From (\ref{SLExpand}) for $m = 1$ and Proposition \ref{MomentBoundSingleLevel} we get
\begin{equation*}
\begin{split} 
& Cov(X^t_N(v)  ; X^t_N(v_1)) = \frac{ \theta^{-1}}{2\pi \i \sqrt{(v - \alpha) (v- \beta)}}\int_{\Gamma}   \frac{- \Phi_N^-(Nz) e^{-\theta G_\mu(z)}dz }{ H(z) \cdot (z - v) (v_1 -z)(v_1 - z +N^{-1})}+O(N^{-1}).
\end{split}
\end{equation*}
We may now use that $\Phi^{\pm}_N(Nz) = \Phi^{\pm}(z) + O(N^{-1})$, (\ref{QRmu}) to obtain

\begin{equation*}
\begin{split} 
& Cov(X^t_N(v)  ; X^t_N(v_1)) = \frac{ \theta^{-1}}{2\pi \i \sqrt{(v - \alpha) (v- \beta)}}\int_{\Gamma}   \frac{-(1/2) [R_\mu(z) - Q_\mu(z) ]dz}{ H(z) \cdot (z - v) (v_1 -z)^2}+O(N^{-1}).
\end{split}
\end{equation*}
The term involving $R_\mu$ integrates to $0$ by Cauchy's theorem. By Assumption 5 the remainder is
\begin{equation*}
\begin{split} 
& Cov(X^t_N(v)  ; X^t_N(v_1)) = \frac{ \theta^{-1}}{4\pi \i \sqrt{(v - \alpha) (v- \beta)}}\int_{\Gamma}   \frac{ \sqrt{(z - \alpha)(z- \beta)} dz}{ (z - v) (v_1 -z)^2}+O(N^{-1}).
\end{split}
\end{equation*}
Evaluating the above integral as minus the sum of the residues at $z = v$ and $z =v_1$ we obtain (\ref{eq:GField}).

Furthermore, from (\ref{SLExpand}) for $m \geq 2$ and Proposition \ref{MomentBoundSingleLevel} we get
\begin{equation*}
\begin{split} 
& M(X^t_N(v); \llbracket 1, m \rrbracket) =O(N^{-1}),
\end{split}
\end{equation*}
which concludes the proof the proposition.

\bibliographystyle{amsplain}
\bibliography{refs}

\providecommand{\bysame}{\leavevmode\hbox to3em{\hrulefill}\thinspace}
\providecommand{\MR}{\relax\ifhmode\unskip\space\fi MR }
\providecommand{\MRhref}[2]{%
  \href{http://www.ams.org/mathscinet-getitem?mr=#1}{#2}
}
\providecommand{\href}[2]{#2}
\begin{thebibliography}{10}

\bibitem{FA}
A.~Alastuey and P.J. Forrester, \emph{Correlations in two-component log-gas
  systems}, J. Stat. Phys. \textbf{81} (1995), 579--627.

\bibitem{AGZ}
G.~Anderson, A.~Guionnet, and O.~Zeitouni, \emph{Introduction to random
  matrices}, Cambridge Studies in Advanced Mathematics, 2009.

\bibitem{Anderson}
G.~W. Anderson, \emph{A short proof of {S}elberg's generalized beta formula},
  Forum Math. \textbf{3} (1991), 415--417.

\bibitem{AY1}
P.~W. Anderson and G.~G.~Yuval, \emph{Exact results in the {K}ondo problem:
  Equivalence to a classical one-dimensional {C}oulomb gas}, Phys. Rev. Lett.
  \textbf{23} (1969), no.~2, 1522--1528.

\bibitem{AY2}
P.~W. Anderson, G.~G.~Yuval, and Hamann~D. R., \emph{Exact results in the
  {K}ondo problem. {I}{I}. {S}caling theory, qualitatively correct solution,
  and some new results on one-dimensional classical statistical models}, Phys.
  Rev. B \textbf{1} (1970), 4464--4473.

\bibitem{Bar}
Y.~Baryshnikov, \emph{G{U}{E}s and queues}, Probab. Theory and Relat. Fields
  \textbf{119} (2001), 256--274.

\bibitem{BFG}
F.~Bekerman, A.~Figalli, and A.~Guionnet, \emph{Transport maps for
  $\beta$-matrix models and universality}, Comm. Math. Phys. \textbf{338}
  (2015), 589--619.

\bibitem{BeGu}
G.~Ben~Arous and A.~Guionnet, \emph{Large deviations for {W}igner's law and
  {V}oiculescu's non-commutative entropy}, Probab. Theory Relat. Fields
  \textbf{108} (1997), 517--542.

\bibitem{BorW}
A.~Borodin, \emph{C{L}{T} for spectra of submatrices of {W}igner random
  matrices}, Mosc. Math. J. \textbf{14} (2014), 29--38.

\bibitem{BorCor}
A.~Borodin and I.~Corwin, \emph{Macdonald processes}, Probab. Theory Relat.
  Fields \textbf{158} (2014), 225--400.

\bibitem{BorGor}
A.~Borodin and V.~Gorin, \emph{General $\beta$-{J}acobi corners process and the
  {G}aussian free field}, Comm. Pure and Applied Math \textbf{68} (2015),
  no.~10, 1774--1844.

\bibitem{BGG}
A.~Borodin, V.~Gorin, and A.~Guionnet, \emph{Gaussian asymptotics of discrete
  $\beta$-ensembles}, Publications math{\' e}matiques de l'IH{\' E}S
  \textbf{125} (2017), 1--78.

\bibitem{BP}
A.~Borodin and L.~Petrov, \emph{Integrable probability: {F}rom representation
  theory to {M}acdonald processes}, Probab. Surv. \textbf{11} (2014), 1--58.

\bibitem{BoGu2}
G.~Borot and A.~Guionnet, \emph{Asymptotic expansion of $\beta$ matrix models
  in the multi-cut regime},  (2013), Preprint, arXiv:1303.1045.

\bibitem{BoGu}
\bysame, \emph{Asymptotic expansion of $\beta$ matrix models in the one-cut
  regime}, Comm. Math. Phys. \textbf{317} (2013), 447--483.

\bibitem{BEY}
P.~Bourgade, L.~Erd{\H o}s, and H.T. Yau, \emph{Edge universality of beta
  ensembles}, Comm. Math. Phys. \textbf{332} (2014), 261--353.

\bibitem{Buf}
A.~Bufetov, \emph{Kerov's interlacing sequences and random matrices}, J. Math.
  Phys. \textbf{54} (2013), 113302.

\bibitem{BuGo}
A.~Bufetov and V.~Gorin, \emph{Fluctuations of particle systems determined by
  {S}chur generating functions}, Adv. Math. \textbf{338} (2018), 702--781.

\bibitem{CLP}
H.~Cohn, M.~Larsen, and J.~Propp, \emph{The shape of a typical boxed plane
  partition}, New York J. Math. \textbf{4} (1998), 137--165.

\bibitem{NPV}
E.~Di~Nezza, G.~Palatucci, and E.~Valdinoci, \emph{Hitchhiker's guide to the
  fractional {S}obolev spaces}, Bull. Sci. Math. \textbf{136} (2012), 521--573.

\bibitem{DK1}
E.~Dimitrov and A.~Knizel, \emph{Log-gases on quadratic lattices via discrete
  loop equations and q-boxed plane partitions}, J. Funct. Anal. \textbf{276}
  (2019), 3067--3169.

\bibitem{DK2}
\bysame, \emph{in preparation},  (2020).

\bibitem{Dixon}
A.~L. Dixon, \emph{Generalizations of {L}egendre's formula $ke' - (k-e)k' =
  \frac{1}{2} \pi$}, Proc. London Math. Soc. \textbf{2} (1905), 206--224.

\bibitem{TE}
A.~Erd\'elyi and F.~G. Tricomi, \emph{The asymptotic expansion of a ratio of
  gamma functions}, Pacific J. Math \textbf{1} (1951), no.~1, 133--142.

\bibitem{ES}
L.~Erd{\H o}s and D.~Schr\"oder, \emph{Fluctuations of rectangular {Y}oung
  diagrams of interlacing {W}igner eigenvalues}, Int. Math. Res. Not.
  \textbf{2018} (2018), 3255--3298.

\bibitem{FW08}
P.~Forrester and S.V.E.N. Warnaar, \emph{The importance of the {S}elberg
  integral}, Bull. Amer. Math. Soc. (2008), no.~4, 489--534.

\bibitem{Forr}
P.J. Forrester, \emph{Log-gases and random matrices}, Princeton University
  Press, Princeton, 2010.

\bibitem{GelN}
I.~M. Gelfand and M.~A. Naimark, \emph{Unitary representations of the classical
  groups}, Trudy Mat. Inst. Steklov. Izdat. Akad. Nauk USSR, Moscow (Russian)
  \textbf{36} (1950).

\bibitem{Ahn}
G.~Goel and A.~Ahn, \emph{Discrete derivative asymptotics of the
  $\beta$-{H}ermite eigenvalues}, Combin. Probab. Comp. (2019).

\bibitem{Go}
V.~Gorin, \emph{Nonintersecting paths and the {H}ahn orthogonal polynomial
  ensemble}, Funct. Anal. Appl. \textbf{42} (2008), 180--197.

\bibitem{GS}
V.~Gorin and M.~Shkolnikov, \emph{Multilevel {D}yson {B}rownian motions via
  {J}ack polynomials}, Probab. Theory Relat. Fields \textbf{163} (2015),
  413--463.

\bibitem{GZ}
V.~Gorin and L.~Zhang, \emph{Interlacing adjacent levels of $\beta$-{J}acobi
  corners processes}, Probab. Theory Relat. Fields \textbf{172} (2018),
  915--981.

\bibitem{GH}
A.~Guionnet and J.~Huang, \emph{Rigidity and edge universality of discrete
  $\beta$-ensembles},  (2017), Preprint, arXiv:1705.05527.

\bibitem{Hu}
J.~Huang, \emph{Law of large numbers and central limit theorems by {J}ack
  generating functions},  (2018), Preprint: arXiv:1807.09928.

\bibitem{JL}
K.~Johansson, \emph{On fluctuations of eigenvalues of random {H}ermitian
  matrices}, Duke Math. J. \textbf{91} (1998), 151--204.

\bibitem{J}
\bysame, \emph{Non-intersecting, simple, symmetric random walks and the
  extended {H}ahn kernel}, Ann. Inst. Fourier \textbf{55} (2005), 2129--2145.

\bibitem{JN}
K.~Johansson and E.~Nordenstam, \emph{Eigenvalues of {G}{U}{E} minors},
  Electron. J. Probab. \textbf{11} (2006), 1342--1371.

\bibitem{KOO}
S.~Kerov, A.~Okounkov, and G.~Olshanski, \emph{The boundary of {Y}oung graph
  with {J}ack edge multiplicities}, Int. Math. Res. Not. (1998), 173--199.

\bibitem{KS}
T.~Kriecherbauer and M.~Shcherbina, \emph{Fluctuations of eigenvalues of matrix
  models and their applications},  (2010), Preprint, arXiv:1003.6121.

\bibitem{Mac}
I.~G. Macdonald, \emph{Symmetric functions and {H}all polynomials, second
  edition}, Oxford University Press, London, 1999.

\bibitem{Meh}
M.L. Mehta, \emph{Random matrices, 3rd edition}, Elsevier/Academic Press,
  Amsterdam, 2004.

\bibitem{N}
N.~Nekrasov, \emph{{B}{P}{S}/{C}{F}{T} correspondence: non-perturbative
  {D}yson-{S}chwinger equations and $qq$-characters}, J. High Energy Phys.
  \textbf{3} (2016), Article 181.

\bibitem{Ne}
Yu. Neretin, \emph{Rayleigh triangles and non-matrix interpolation of matrix
  beta integrals}, Sbornik: Mathematics \textbf{194} (2003), no.~4, 515--540.

\bibitem{OO}
A.~Okounkov and G.~Olshanski, \emph{Shifted {J}ack polynomials, binomial
  formula, and applications}, Math. Res. Let. (1997), 68--79.

\bibitem{PS}
L.~Pastur and M.~Shcherbina, \emph{Eigenvalue distribution of large random
  matrices}, AMS, Providence, 2011.

\bibitem{Taqqu}
G.~Peccati and M.~Taqqu, \emph{Wiener {C}haos: {M}oments, {C}umulants and
  {D}iagrams}, Springer-Verlag Italia, Italy, 2011.

\bibitem{P1}
L.~Petrov, \emph{Asymptotics of random lozenge tilings via {G}elfand-{T}setlin
  schemes}, Probab. Theory Relat. Fields \textbf{160} (2014), 429--487.

\bibitem{P2}
\bysame, \emph{Asymptotics of uniformly random lozenge tilings of polygons.
  {G}aussian free field}, Ann. Probab. \textbf{43} (2015), 1--43.

\bibitem{Ker}
Kerov S., \emph{Interlacing measures}, Amer. Math. Soc. Transl. (1998), 35--84.

\bibitem{Sa}
L.~{\v{S}}amaj, \emph{Thermodynamic properties of the one-dimensional
  two-component log-gas}, J. Stat. Phys. \textbf{105} (2001), 173--191.

\bibitem{S}
M.~Shcherbina, \emph{Fluctuations of linear eigenvalue statistics of $\beta$
  matrix models in the multi-cut regime}, J. Stat. Phys. \textbf{151} (2013),
  1004--1034.

\bibitem{SS}
E.~Stein and R.~Shakarchi, \emph{Complex analysis}, Princeton University Press,
  Princeton, 2003.

\bibitem{Sun16}
Y.~Sun, \emph{Matrix models for multlevel {H}eckman-{O}pdam and multivariate
  {B}essel measures},  (2016), Preprint: , arXiv:1609.09096.

\bibitem{Tit}
E.~C. Titchmarsh, \emph{Introduction to the theory of {F}ourier integrals},
  Oxford University Press, London, 1937.

\end{thebibliography}
\end{document}